\documentclass[10pt,a4paper]{article}

\usepackage[utf8]{inputenc} %
\usepackage[T1]{fontenc} %
\usepackage[pdftex=true,hyperindex=true,colorlinks=true,linkcolor=gray,citecolor=blue]
{hyperref}
\usepackage{amssymb}
\usepackage{a4wide}
\usepackage{t1enc}
\usepackage{graphics}
\usepackage{graphicx}
\usepackage{epsfig}
\usepackage{verbatim}
\usepackage{times}
\usepackage{caption}
\usepackage{subcaption}
\usepackage{pifont}
\usepackage{shadow}
\usepackage{listings}
\usepackage{latexsym}
\usepackage{pdfsync}
\usepackage{amsthm}
\usepackage{bbm}
\usepackage{amsmath, amsthm, amsfonts}
\usepackage{fancyhdr}
\usepackage{dsfont}  
\usepackage{flafter}  
\usepackage{placeins}
\usepackage{mathrsfs}
\usepackage{xcolor}
\usepackage{paralist}
\usepackage{geometry} 
\usepackage[normalem]{ulem}
\usepackage{xspace}
\usepackage{nicefrac}

\def\be{\mathbf{e}}

\def\R{\mathbb{R}}

\def\N{\mathbb{N}}
\def\Z{\mathbb{Z}}

\def\mx{\mathbf{x}}
\def\mX{\mathbf{X}}
\def\ttv{\mathtt{v}}
\def\tta{\mathtt{a}}
\def\ttb{\mathtt{b}}
\def\match{\text{match}}
\def\mod{\text{mod}}

\def\dsp{\displaystyle}

\newtheorem{hypo}{Assumption}[section]

\newtheorem{theo}[hypo]{Theorem}
\newtheorem{lema}[hypo]{Lemma}

\newtheorem{prop}[hypo]{Proposition}
\newtheorem{rem}[hypo]{Remark}

\newcommand{\bn}{{\mathbf n}}
\newcommand{\bx}{{\mathbf x}}
\newcommand{\bX}{{\mathbf X}}

\newcommand{\hole}{\text{\rm hole}}
\newcommand{\IR}{{\mathbb R}}

\newcommand{\IN}{{\mathbb N}}
\newcommand{\IZ}{{\mathbb Z}}
\newcommand{\macro}{\text{macro}}

\newcommand\xnorm[2]{\left\lVert #1 \right\rVert_{#2}}

\newcommand{\OmegaTop}   {\Omega_{\mathrm{T}}}
\newcommand{\OmegaBottom}{\Omega_{\mathrm{B}}}
\newcommand{\LTop}       {L'}
\newcommand{\LBottom}    {L }
\newcommand{\HTop}       {H_{\mathrm{T}}}
\newcommand{\HBottom}    {H_{\mathrm{B}}}
\newcommand{\eOne}       {\mathrm{e}_1}
\newcommand{\eTwo}       {\mathrm{e}_2}
\newcommand{\Hone}       {\mathrm{H}^1}
\newcommand{\HoneGammaD} {\Hone_{\Gamma_D}}
\newcommand{\HonehalfG}   {\mathrm{H}^{\nicefrac12}     (\Gamma)}
\newcommand{\HonehalfzzG} {\mathrm{H}^{\nicefrac12}_{00}(\Gamma)}

\newcommand{\eg}{\textit{e.\,g\mbox{.}}\xspace}
\newcommand{\ie}{\textit{i.\,e\mbox{.}}\xspace}

\newcommand{\cf}{\textrm{cf.}\xspace}

\definecolor{mydarkcyan}{rgb}{0,0.5,0.5}
\definecolor{mygreen}{RGB}{0,204,255}
\newcommand{\bds}[1]{{#1}}

\makeatletter

\@addtoreset{equation}{section}
\makeatother

\begin{document}
\begin{center}
{\Large{When a thin periodic layer meets corners: asymptotic analysis of a singular Poisson problem}}\\[0.5cm]

{\large{B\'erang\`ere Delourme$^{a,}$\footnote{Part of this work was
      carried out where the author was on research leave at Laboratoire
    POEMS, INRIA-Saclay, ENSTA, UMR CNRS 2706, France}, Kersten Schmidt$^{b,c}$, Adrien Semin$^c$}}\\[0.5cm]

{\small $a$: Universit\'e Paris 13, Sorbone Paris Cit\'e, LAGA, UMR
  7539, 93430 Villetaneuse, France} \\
{\small $b$: Research center Matheon, 10623 Berlin, Germany}\\
{\small $c$: Institut f\"ur Mathematik, Technische Universit\"at Berlin, 10623 Berlin, Germany} \\
\end{center}

\noindent \textbf{Abstract}\\
\noindent The present work deals with the resolution of the Poisson equation in
a bounded domain made of a thin and periodic layer of finite
length placed into a homogeneous medium. %
We provide and justify a high order asymptotic expansion which takes into
account the boundary layer effect occurring in the vicinity of
the periodic layer as well as the corner singularities appearing in
the neighborhood of the extremities of the layer. %
Our approach combines 
the method of matched
asymptotic expansions and the method of periodic surface homogenization, and a
complete justification is included in the paper or its appendix. \\



\tableofcontents

\section*{Introduction}

The present work is dedicated  to the construction of a high order
asymptotic expansion of the solution to a Poisson problem posed in a
polygonal domain which excludes a set of similar small obstacles equi-spaced 
along the line between two re-entrant corners.
The distance between two consecutive obstacles, which appear to be holes in the domain,
and the diameter of the obstacles are of the same order of
magnitude $\delta$, which is supposed to be small compared to the 
dimensions of the domain. 
The presence of this thin periodic layer of holes is responsible for
the appearance of two different kinds of singular
behaviors. First, a highly oscillatory boundary layer appears in the vicinity of the
periodic layer. Strongly localized, it decays exponentially fast as the distance to the
periodic layer increases. Additionaly, since the thin periodic layer has a finite
length and ends in corners of the boundary, corners singularities come up in the neighborhood of its
extremities. The objective of this work is to provide a sophisticated asymptotic
expansion that takes into account these two types of singular behaviors. \\

\noindent The boundary layer effect occurring in
the vicinity of the periodic layer is well-known. It can be described
using a two-scale asymptotic expansion (inspired by the periodic homogenization
theory) that
superposes slowly varying macroscopic terms and periodic
correctors that have a two-scale behavior: these functions are the
combination of highly
oscillatory and decaying functions (periodic of period $\delta$ with
respect to the
tangential direction of the periodic interface and exponentially
decaying with respect to $d/\delta$, $d$ denoting the distance to the
periodic interface) multiplied by slowly varying functions. This
boundary layer effect has been widely investigated since the work of
Sanchez-Palencia \cite{RapportSanchezPalencia,SanchezPalencia},
Achdou~\cite{Achdou,AchdouCR} and
Artola-Cessenat~\cite{ArtolaCessenat, ArtolaCessenat2}. In particular,
high order asymptotics have been derived in
\cite{AchdouPironneauValentin,Madureira,CiupercaJaiPoignard,BreschMilisic2010} for the
Laplace equation and in   
\cite{poirier2006impedance,Poirier} for the Helmholtz
equation.  \\

\noindent On the other hand, corner singularities appearing when
dealing with singularly perturbed boundaries have also been widely investigated. Among the
numerous examples of such singularly perturbed problems, we can mention 
the cases of small inclusions~(see \cite[chapter 2]{LivreNazarov1}
for the case of one inclusion and
\cite{MR2573145} for the case of several inclusions), perturbed
corners~\cite{DaugeTordeuxVialVersionLongue}, 
propagation of waves in  
thin slots~\cite{fente1,fente2}, the diffraction by wires~\cite{XavierArticle}, or  the mathematical investigation of
patched antennas~\cite{BendaliMakhloufTordeux}.  
Again, this 
effect can be depicted using two-scale asymptotic
expansion methods that are the method of multiscale expansion 
(sometimes called compound method) and the method of matched asymptotic expansions
(see \cite{VanDyke,LivreNazarov1,Ilin}). 
Following these
methods, the solution of the perturbed problem may be seen as the
superposition of slowly varying macroscopic terms that do not see directly the
perturbation and microscopic terms that take into account the local
perturbation.  \\

\noindent Recently, Vial and co-authors~\cite{VialThese,CalozVial} investigated a Poisson problem 
in a polygonal domain surrounded by a thin and homogeneous layer, while Nazarov~\cite{Nazarov400}
studied the resolution of a general elliptic problem in a polygonal domain with periodically changing boundary.
In their studies they have combined the two different kinds of asymptotic expansions 
mentioned above in order to deal with both corner singularities and the boundary layer effect. %
Based on the multiscale method, the authors of~\cite{VialThese,CalozVial} constructed and justified a complete
asymptotic expansion for the case of the homogeneous layer. %
For the periodic boundary in~\cite{Nazarov400} the first terms of the asymptotic expansion have been constructed and
error estimates have been carried out. 
This asymptotic expansion relies on a sophisticated analysis of solution behavior at infinity for the Poisson problem in an infinite cone with oscillating boundary with Dirichlet boundary conditions by Nazarov~\cite{Nazarov143},
where he published an analysis for Neumann boundary conditions in~\cite{Nazarov205}.
In the present paper, we are going to extend the work for the homogeneous layer and the periodic boundary 
by constructing explicitely and rigorously justifying asymptotic
expansion for the above mentioned periodic layer transmission problem to any order (with Neumann boundary conditions on the perforations of the layer).\\


\section{Description of the problem and main results}
\label{sec:intr-sett-probl}

\subsection{Description of the problem}

In this section we are going to define the domain of interest $\Omega^\delta \in \IR^2$, its limit when $\delta\to0$ and the problem considered. 
With the coordinates $\mathbf{x} = (x_1, x_2)$ of $\IR^2$ let $\OmegaBottom$  and  $\OmegaTop$ be the two adjacent rectangular domains defined by
\begin{align*}
\OmegaBottom &= (-L,L)\times(-H_B,0)\ , &
\OmegaTop &= 
(-\LTop, \LTop) \times (0, \HTop)\ ,
\end{align*}
where $\LTop > \LBottom$, $\HBottom$ and $\HTop$ are positive numbers.
We denote by $\Gamma$ the common interface of $\OmegaBottom$ and $\OmegaTop$, \ie,
\begin{align*}
\overline{\Gamma} = \overline{\partial \OmegaBottom} \cap \overline{\partial \OmegaTop} 
\quad \text{and} \quad 
\Gamma = (-\LBottom, \LBottom) \times \{ 0 \}.
\end{align*}
and we consider the (non-convex) polygonal domain (see Fig.~\ref{fig:Omega})
\begin{align*}
  \Omega = \OmegaBottom \cup \OmegaTop \cup \Gamma\ ,
\end{align*}
which has
two reentrant corners at $\mathbf{x}_{O}^\pm = ( \pm\LBottom,0)$ with both an angle of $\frac{3\pi}{2}$.

\begin{figure}[htbp]
        \centering
        \begin{subfigure}[b]{0.42\textwidth}
                \includegraphics[width=\textwidth]{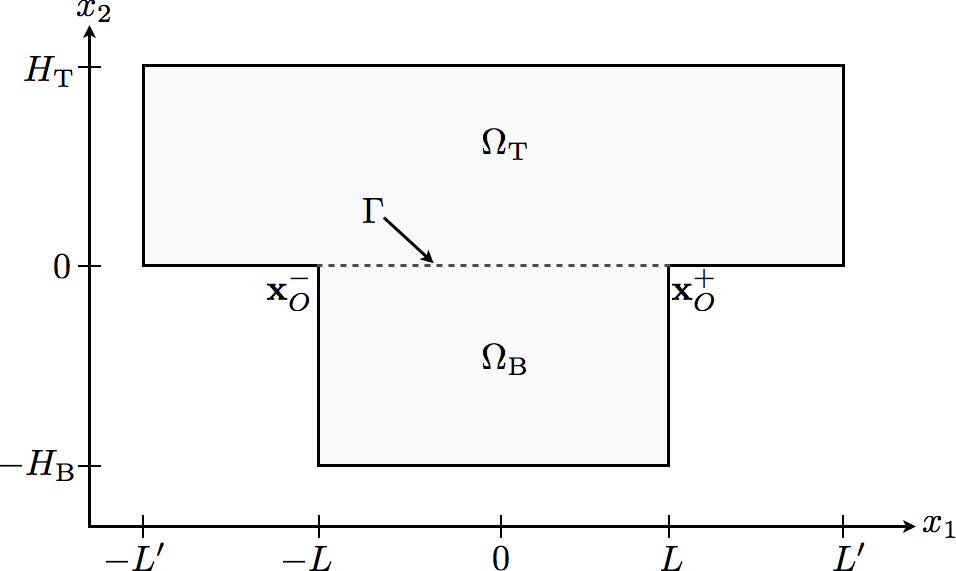}
                \caption{The domain $\Omega = \OmegaTop \cap \OmegaBottom
                  \cap \Gamma$.}
                \label{fig:Omega}
        \end{subfigure}
        \qquad 
        \begin{subfigure}[b]{0.42\textwidth}
                \includegraphics[width=\textwidth]{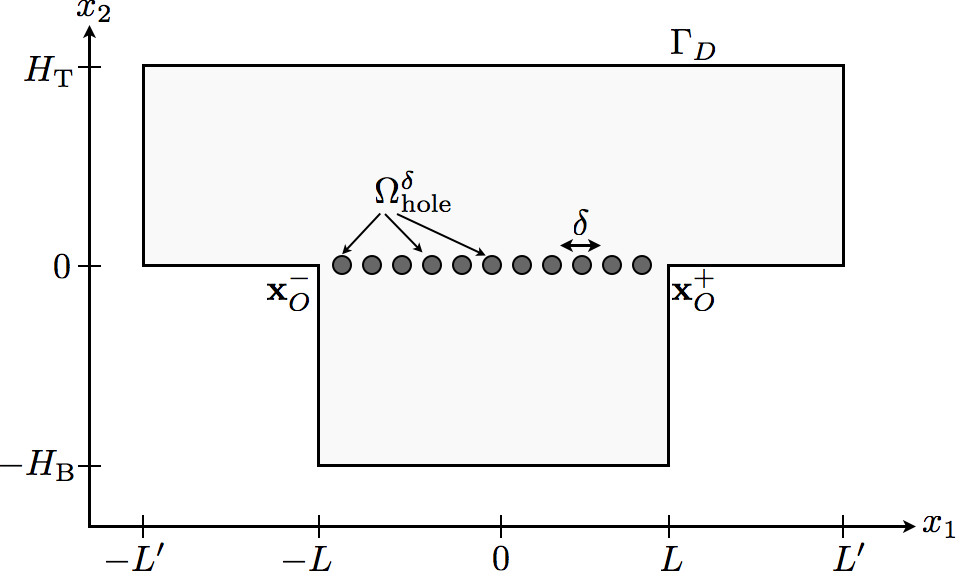}
                \caption{The domain of interest $\Omega^\delta = \Omega \backslash \overline{\Omega^\delta_\hole}$.}
                \label{fig:OmegaDelta}
        \end{subfigure}
        
        \caption{Illustration of the polygonal domain $\Omega$ and the domain of interest $\Omega^\delta$.}\label{fig:dessinsDomaines}
\end{figure}

\noindent Besides, let $\widehat{\Omega}_{\hole} \in \IR^2$ be a {\emph smooth}
canonical bounded open set
(not necessarily connected) strictly included in the domain $(0,1)
\times (-1,1)$. Then, let $N^\ast := \N \setminus \{0 \}$ denote the set of positive integers
and let
 $\delta$ be a positive real number (that is
supposed to be small) such that 
\begin{equation}
\frac{2\LBottom}{\delta} = q \in \N^\ast.
\end{equation}

\noindent %
Now, let $\Omega^\delta_{\hole}$ be 
a thin (periodic) layer consisting 
of $q$ equi-spaced similar obstacles
which can be defined 
by scaling and shifting the canonical obstacle $\widehat{\Omega}_{\hole}$
(see Fig.~\ref{fig:OmegaDelta}):
\begin{equation}
  \label{eq:layer_Omega_hole}
  \Omega^\delta_{\hole} = \bigcup_{\ell = 1}^{q} 
  \left\{ -\LBottom \eOne + \delta \{ \widehat{\Omega}_{\hole} +
        (\ell-1) \eOne \} \,  \right\}.
\end{equation}
Here, $\eOne$ and $\eTwo$ denote the unit vectors of $\IR^2$ and $\delta$ is assumed to be smaller
than $\HTop$ and $\HBottom$ such that $\Omega^\delta_\hole$ does not touch the top or bottom boundaries of $\Omega$. %
Finally, we define our domain of interest as
\begin{align*}
  \Omega^\delta = (\OmegaBottom \cup \OmegaTop \cup \Gamma) \backslash \overline{\Omega^\delta_{\hole}}.
\end{align*}
Its boundary $\partial\Omega^\delta$ consists of the boundary of the set of holes $\Gamma^\delta = \partial\Omega^\delta_{\hole}$
and $\Gamma_D = \partial \Omega^\delta \setminus \Gamma^\delta = \partial \Omega$, the boundary of $\Omega$.
Here and in what follows,  we denote by $\bn$ the outward unit normal
vector of $\partial\Omega^\delta$. %
Note, that in the limit $\delta\to 0$ the repetition of holes degenerates to the interface $\Gamma$, 
the domain $\Omega^\delta$ to the domain $\OmegaTop \cup \OmegaBottom$ and
its boundary $\partial\Omega^\delta$ to $\partial\Omega \cup \Gamma$.  \\


\noindent The domain $\Omega^\delta$  being defined, we can introduce the
problem to be considered in this article: Seek $u^\delta$
solution to 
\begin{equation}
  \label{eq:perturbed_laplace}
  \left\lbrace\quad
    \begin{aligned}
      - \Delta u^\delta &= f, &\quad&
      \text{in }
      \Omega^\delta\ ,\\
      \nabla u^\delta \cdot \bn &= 0, &&\text{on } \Gamma^\delta ,\\
      u^\delta &= 0, &&\text{on }\Gamma_D\ ,
    \end{aligned}
  \right.
\end{equation}
where $f \in L^2(\Omega^\delta)$. It is natural to search for $u^\delta \in \HoneGammaD(\Omega^\delta)$
where 
\begin{align}
\label{eq:HoneGammaD}
\HoneGammaD(\Omega^\delta) = \left\{ u \in H^1(\Omega^\delta) \; \mbox{such that}
  \; u =0 \; \mbox{on} \; \Gamma_D \right\}.
\end{align}
The well-posedness of problem (\ref{eq:perturbed_laplace}) in
$\HoneGammaD(\Omega^\delta)$ directly
follows from Lax-Milgram theorem:
\begin{prop}[Existence, uniqueness and stability]
  \label{prop:existence_uniqueness_u_delta}
  Let $f \in L^2(\Omega^\delta)$.  Then, for any $\delta>0$ there exists a
  unique solution $u^ \delta$ of problem~\eqref{eq:perturbed_laplace} in
  $\HoneGammaD(\Omega^\delta)$, and with a constant $C$ (independent of~$\delta$) it holds
  \begin{equation}
    \label{eq:prop_norm_u_norm_f}
    \xnorm{u^\delta}{\Hone(\Omega^\delta)} \leqslant
    C\,\xnorm{f}{\mathrm{L}^2(\Omega^\delta)}\, .
  \end{equation}
\end{prop}
\noindent The objective of this paper is to describe the behavior of $u^\delta$
as $\delta$ tends to $0$. For the sake of simplicity, we shall assume
that $f$ has a compact support in a subset of $\OmegaTop$
with distance $\delta_0 > 0$ to $\Gamma$. %
Our work relies on a construction of an asymptotic expansion of $u^\delta$ as $\delta$
tends to $0$. %

\begin{rem}
  \label{rem:geometrical}
  The construction is for simplicity for the specific geometrical setting,
  where $\Gamma$ is a straight line ending in two corners of the polygonal
  boundary $\partial\Omega$, where the angles between $\Gamma$ and
  $\partial\Omega$ are both ends at angles $\frac{\pi}{2}$ or $\pi$,
  respectively. Nevertheless, the study may be extended to a polygon
  $\Omega$ of different angles.
\end{rem}
\begin{rem} It is worth noting that the choice of the boundary condition
  imposed on the small obstacles $\Gamma^\delta$ constituting the periodic layer, here homogeneous Neumann
  boundary conditions, has a strong impact on the asymptotic expansion.
  A homogeneous Dirchlet condition would yield to a completely different asymptotic expansion (see for instance
  Appendix A in \cite{DelourmeThese}, or \cite{TrefethenChapmanHewett}).
\end{rem}

\begin{rem} The smoothness of $\widehat{\Omega}_{\hole}$ is not required
  for the existence of $u^\delta \in H^1(\Omega^\delta)$. The
  well-posedness result remains valid if $\widehat{\Omega}_{\hole}$ is a
  Lipschitz domain. However, we use
this assumption in the forthcoming analysis
  (In particular in Proposition~\ref{propositionAsymptoticNearField} and Section~\ref{SectionErrorAnalysis}).   
  %
\end{rem}

\subsection{Ansatz of the asymptotic expansion}

As mentioned in the introduction, due to the periodic layer, it seems not possible to write a simple
asymptotic expansion valid in the whole domain. We have to take into account both the boundary
layer effect in the vicinity of $\Gamma$ and the additional corner
singularities appearing in the neighborhood of the two reentrant
corners. To do so, we shall distinguish a {\em far field area} located ' far'
from the reentrant corners $\mx_{O}^ \pm$ and two {\em near
field zones} located in the vicinity of them (see Fig.~\ref{DessinAsymptotic}). 
\begin{figure}[h!]
  \centering
      \includegraphics[width=0.6\textwidth]{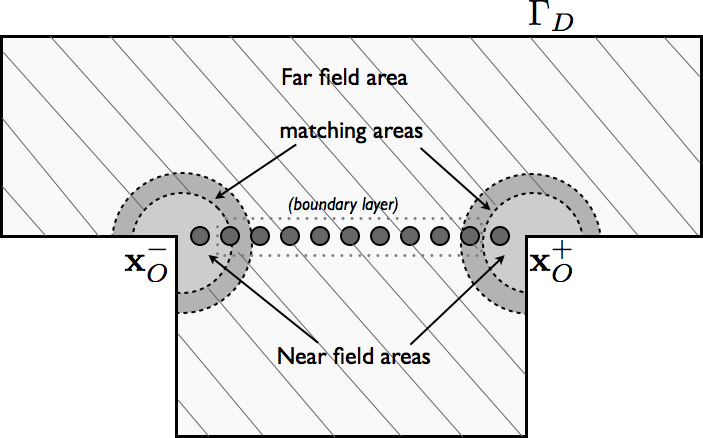}
  \caption{Schematic representation of the overlapping subdomains for the asymptotic expansion.
  The far field area ({\em hatched}) away from the corners
  $\mx_{O}^ \pm$ is overlapping the near field area ({\em light grey}) in the matching zone ({\em dark grey}).}
\label{DessinAsymptotic}
\end{figure}

\subsubsection{Far field expansion} 

\noindent Far from the two corners $\mx_{O}^ \pm$ (hatched area in~Fig.~\ref{DessinAsymptotic}), we shall see that
$u^\delta$ is the superposition of a macroscopic part (that is not
oscillatory) and a boundary layer localized in the neighborhood of
the thin periodic layer. More precisely, we choose 
the following ansatz:   
\begin{equation}\label{FFExpansion}
u^\delta (\mathbf{x}) =  \sum_{(n,q) \in \N^ 2} \delta^{\frac{2}{3} n
   + q} \, u_{FF,n,q}^\delta (\mathbf{x}),
\end{equation}
 where $\mathbf{x} = (x_1, x_2)$, and for $(n,q) \in \N^2$
\begin{equation} \label{def_uFFnq}
u_{FF,n,q}^\delta (\mathbf{x}) =  
\begin{cases}
u_{n,q}^\delta (\mathbf{x}) & \mbox{if} \; |x_1| > \LBottom, \\[1ex]
\chi\left( \frac{x_2}{\delta} \right) u_{n,q}^\delta (\mathbf{x}) +
\Pi_{n,q}^\delta(x_1, \frac{\mathbf{x}}{\delta}) &
\mbox{if} \; |x_1|< \LBottom.  
\end{cases} 
\end{equation}
Here $\chi : \R \mapsto (0,1)$ denotes a smooth cut-off function satisfying 
\begin{equation}\label{defchi}
\chi(t) = \begin{cases}
1 & \mbox{if} \; |t| >2, \\[1ex]
0 & \mbox{if} \; |t|<1.
\end{cases}
\end{equation}
The {\em macroscopic terms} $u_{n,q}^\delta$ are defined in the limit domain
$\OmegaTop \cup \OmegaBottom$. A priori, they are not
continuous across~$\Gamma$.  As for the {\em boundary layer correctors} 
$\Pi_{n,q}^\delta(x_1,X_1,X_2)$ (also sometimes denoted {\em periodic correctors}), and as usual in the periodic
homogenization theory, there are $1$-periodic with respect to
the scaled tangential variable $X_1$. Consequently, they are defined in $(-\LBottom ,
\LBottom) \times \mathcal{B}$, where $\mathcal{B}$ is the infinite
periodicity cell (see Fig.~\ref{fig:PeriodicityCell}): 
\begin{equation}\label{PeriodicityCellDef}
\mathcal{B} = \left\{ (0,1) \times \R \right\} \setminus \overline{\widehat{\Omega}_{\hole}}.
\end{equation}
Moreover, the periodic correctors are super-algebraically decaying as the scaled variable
$X_2$ tends to $\pm \infty$ (they decay faster than any power of $X_2$), more precisely, for any $(k, \ell) \in \N^2$,  
\begin{equation}
\label{eq:exponential_decaying}
\lim_{|X_2| \rightarrow + \infty}%
X_2^k \partial_{X_2}^\ell
\Pi_{n,q}^\delta =0.
\end{equation}
The macroscopic terms as well as the boundary layer corrector terms
might have a polynomial dependence with respect to $\ln \delta$: %
there is $N(n,q)\in \N$ such that
\begin{equation*}
u_{n,q}^\delta =\sum_{s=0}^{N(n,q)} (\ln \delta)^s \, u_{n,q,s}, \quad \mbox{and} \quad \Pi_{n,q}^\delta =\sum_{s=0}^{N(n,q)} (\ln \delta)^s \Pi_{n,q,s},
\end{equation*}
where $u_{n,q,s}$ and $\Pi_{n,q,s}$ do not depend on $\delta$.
\begin{rem}Here
and in what follows, although it might be surprising at first glance, we call far
field expansion the expansion~\eqref{FFExpansion}, \ie, the  superposition of the
macroscopic terms and the boundary layer correctors.  Besides, it should also
be noted that, for any $k \in \N$, we consider $\delta^{\frac{2(n+3k)}{3}+q}$ and
    $\delta^{\frac{2n}{3}+(q+2k)}$ as different scales as they would be different powers of $\delta$. In
    fact, we shall see that $n$ and $q$ play a 
    different role in the asymptotic procedure.  Finally, following
    Remark~\ref{rem:geometrical},  the consideration of the more
    general case of two angles of
    measure $\alpha$,  would yield to an expansion
    of the form~ (\ref{FFExpansion}) substituting $\delta^{\frac{2 n}{3}+q}$
    for $\delta^{\frac{\pi n}{\alpha}+q}$ (see~\cite{CalozVial}).
  \end{rem}

\subsubsection{Near field expansions}
In the vicinity of the two
corners $\mx_{O}^ \pm$ (light grey areas
in~Fig.~\ref{DessinAsymptotic}), the solution varies rapidly in all directions. %
Therefore, we shall see that 
\begin{equation}\label{NFExpansion}
u^\delta(\mathbf{x}) =   \sum_{(n,q) \in \N^2} \delta^{\frac{2n}{3} +
  q  \,}U_{n,q, \pm}^\delta\left(\frac{\mathbf{x}-\mathbf{x}_{O}^\pm}{\delta}\right)\, ,
\end{equation}
for some near field terms $U_{n,q,\pm}^\delta$ defined in the fixed unbounded domains
\begin{equation}\label{definitionOmegaHatpm}
\widehat{\Omega}^- =\mathcal{K}^- \setminus  \bigcup_{\ell \in \N}
\left\{ \overline{\widehat{\Omega}_{\hole}} + \ell \be_{1} \right\}, \quad  \widehat{\Omega}^+ =\mathcal{K}^+ \setminus  \bigcup_{\ell \in \N^\ast}
\left\{  \overline{\widehat{\Omega}_{\hole}} - \ell \be_{1} \right\}
\end{equation}
shown in Figure~\ref{fig:HatOmegaMoins} and \ref{fig:HatOmegaPlus},
where $\mathcal{K}^\pm$ are the unbounded angular domains 
\begin{align*}
 \mathcal{K}^\pm &= \left\{ \mathbf{X} = R^\pm  ( \cos \theta^\pm, \sin \theta^\pm),  R^\pm\in \R^\ast_+, \theta^\pm \in I^\pm\right\} \in \R^2
\end{align*}
of angular sectors $I^+ = (0, \frac{3\pi}{2})$ and $I^- = (-\frac{\pi}{2}, \pi)$. 
If the domain $\widehat{\Omega}_{\hole}$ is symmetric with respect to the
axis $X_1 = 1/2$, then the domain $\widehat{\Omega}^-$ is nothing but the domain
$\widehat{\Omega}^+$ mirrored with respect to the axis $X_1=0$.
However, this is not the case in general. 
Similarly to the far field terms the
 near field terms might also have a polynomial dependence with respect to $\ln
\delta$, \ie, for all $(n,q) \in \N^2$,  there is $N(n,q)\in \N$ such that
\begin{equation*}
U_{n,q,\pm}^\delta=\sum_{s=0}^{N(n,q)} (\ln \delta)^s U_{n,q,\pm,s}\ ,
\end{equation*}
where the functions $U_{n,q,\pm,s}$ do not depend on $\delta$.
\begin{figure}[htbp]
        \centering
        \begin{subfigure}[b]{0.22\textwidth}
          \centering
                {\includegraphics[height=3.7cm]{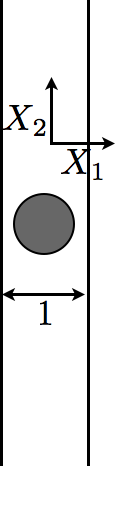}}
                \caption{The periodicity cell $\mathcal{B}$.}
                \label{fig:PeriodicityCell}
        \end{subfigure}
        \begin{subfigure}[b]{0.38\textwidth}
          \centering
                {\includegraphics[height=3.7cm]{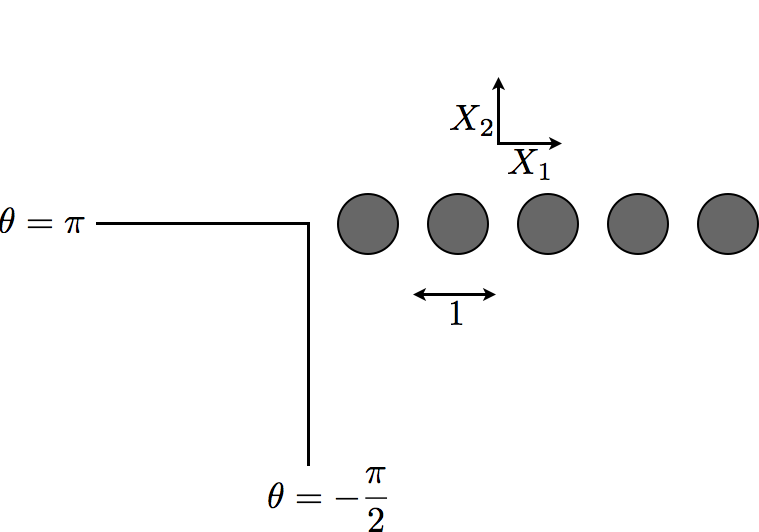}}
                \caption{The domain $\widehat{\Omega}^-$.}
                \label{fig:HatOmegaMoins}
        \end{subfigure}
        \begin{subfigure}[b]{0.38\textwidth}
          \centering
                {\includegraphics[height=3.7cm]{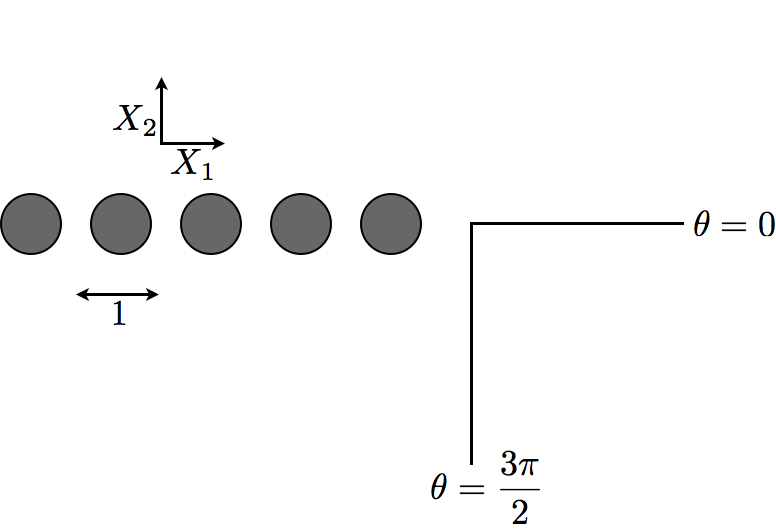}}
                \caption{The domain $\widehat{\Omega}^+$.}
                \label{fig:HatOmegaPlus}
        \end{subfigure}
        \caption{The periodicity cell $\mathcal{B}$ and the normalized
          domains $\widehat{\Omega}^\pm$.}\label{fig:dessinsDomaines2}
\end{figure}

\subsubsection{Matching principle} To link the two different expansions,
we assume that they are both valid in two intermediate areas
  ${\Omega}^{\delta,\pm}_\mathcal{M}$ (dark shaded in Fig.~\ref{DessinAsymptotic}) of the following form:
$$
{\Omega}^{\delta,\pm}_\mathcal{M} = \left\{ \mathbf{x} = (x_1,x_2) \in \Omega,
 \sqrt{\delta} \leq d(\mathbf{x}, \mathbf{x}_O^\pm) \leq 2\,\sqrt{\delta} \right\},
$$
where $d$ denotes the usual Euclidian distance. The precise definition of the matching areas is not important. The
reader might just  keep in mind  that they correspond to
a neighborhood of the corners $\mathbf{x}_{O}^\pm$ of the reentrant
corners for the far field terms (macroscopic and boundary layer correctors)
and to $R^\pm$ going to $+\infty$ for
the near field terms (expressed in the scaled variables).

\subsection{Far and near field equations}
The 'ansatz' being assumed, the next objective is to construct the terms
$u_{n,q}^\delta$, $\Pi_{n,q}^\delta$ and $U_{n,q,\pm}^\delta$ of the far and near field expansions (the asymptotic expansions are justified
later, by
proving error estimates). This is by far the longest part of the
work (Sections \ref{sec:intr-sett-probl} to~\ref{SectionMatchingProcedure}).  The usual starting point of this construction
consists in the formal derivation of the near field and far field
problems, that is to say problems satisfied by the near and far field terms.
\subsubsection{Far field equations: macroscopic  and boundary
  layer correctors equations}
Inserting the far field expansion into the
initial problem~\eqref{eq:perturbed_laplace} and separating the
different powers of $\delta$ (the complete procedure, based on the
separation of the scales, is explained in
Appendix~\ref{SectionScaleSeparation}) gives a collection of equations for the
macroscopic terms and the boundary layer terms:
\paragraph{Macroscopic equations} The macroscopic terms $u_{n,q}^\delta$ satisfy
\begin{equation}\label{FFVolum}
- \Delta u_{n,q}^\delta =  \begin{cases} 
f \; &  \mbox{if} \; n = q = 0, \\
0 \; & \mbox{otherwise,}  \end{cases} \mbox{in} \; \OmegaTop \cup \OmegaBottom,
\end{equation}
together with homogeneous Dirichlet boundary conditions on $\Gamma_D$
\begin{equation}\label{FFDirichlet}
 u_{n,q}^\delta = 0  \; \mbox{on} \; \, \Gamma_D.\\
\end{equation}
\paragraph{Boundary layer corrector equations} The boundary layer
correctors satisfy 
\begin{equation}
    \label{PeriodicCorrectorEquations}
    \left\lbrace
      \begin{aligned}
         - \Delta_{\bX} \Pi_{n,q}^\delta(x_1,\bX)  & = &
         G_{n,q}^\delta & \quad \mbox{in} \;
        \mathcal{B}, \\
        \partial_n \Pi_{n,q}^{\delta} &  = & \;\; 0 & \quad \mbox{on} \;  \partial \widehat{\Omega}_{\hole}.
\end{aligned}
        \right.
  \end{equation}
where, for any $p \in \N$, 
\begin{multline}
  \label{eq:def_Gnq_delta}
  G_{n,q}^\delta  = \sum_{p=0}^q \frac{\partial^p
          u_{n,q-p}^\delta}{{\partial x_2}^p}(x_1,0^+)  [ \Delta, \chi_+] \left(
   \frac{(X_2)^p}{p!} \right)  +
 \sum_{p=0}^q \frac{\partial^p
          u_{n,q-p}^\delta}{{\partial x_2}^p}(x_1,0^-)  [ \Delta, \chi_-] \left(
   \frac{(X_2)^p}{p!} \right) \\
 + \partial_{x_1}^2 \Pi_{n, q-2}^{\delta}(x_1,\bX)
        + 2 \partial_{x_1} \partial_{X_1} \Pi_{n, q-1}^{\delta}(x_1,\bX).  
\end{multline}
Here, for any sufficiently smooth function $v$, $[\Delta, v]$ denotes the commutator between $\Delta$ and $v$, that is to say
$$
[\Delta, v] u = \Delta (v  u)  - v \Delta u = \nabla v \cdot \nabla u
+ u \Delta v.  
$$
Moreover, the smooth truncation functions $\chi_\pm$ are defined by
\begin{equation}\label{DefinitionChiplusmoins}
\chi_\pm(X_2) = \begin{cases}
\, \chi(X_2), & \mbox{if} \pm X_2 >0, \\
0 & \mbox{otherwise},
\end{cases}  \qquad (\chi_\pm(X_2) = 1_{(\pm X_2>0)} \chi(X_2)),
\end{equation}
and, for $p \in \N$,
$$
\frac{\partial^p
          u_{n,q-p}^\delta}{{\partial x_2}^p}(x_1,0^\pm) = \lim_{h\rightarrow
          0^\pm}  \frac{\partial^p
          u_{n,q-p}^\delta}{{\partial x_2}^p}(x_1,h).
$$
Note that equations~(\ref{FFVolum}-\ref{FFDirichlet}), posed in
the domains $\OmegaTop$ and $\OmegaBottom$, do not define entirely the
macroscopic terms. Indeed, we first have to prescribe transmission
conditions
across the interface $\Gamma$ (for instance the jump of their
  trace and the jump of their normal
trace across $\Gamma$). This information will appear to be a consequence of
the boundary layer equations (Section~\ref{SectionTransmissionConditionsBoundaryLayer}). Then, we also have to prescribe the
behavior of the macroscopic terms in the vicinity of the two corner
points $\mathbf{x}_O^\pm$.  This information will be provided by the
matching conditions (Section~\ref{SectionMatchingProcedure}). 
 
\subsubsection{Near field equations}
The near field equations are obtained in a much more direct way. Inserting the near field ansatz~\eqref{NFExpansion} into the Laplace
equation~\eqref{eq:perturbed_laplace} and separating
formally the different powers of $\delta$, it is
easily seen that the near field term $U_{n,q}^\delta$ satisfy
\begin{equation}\label{NearFieldEquation}
\left \lbrace
\begin{aligned}
- \Delta_{\mX} U_{n,q}^\delta & = & 0 & \quad \mbox{in} \;
\widehat{\Omega}^{\pm}, \\
U_{n,q}^\delta & = & 0 & \quad \mbox{on} \; \partial \mathcal{K}^\pm, \\
\partial_n U_{n,q}^\delta & = &0  & \quad \mbox{on}\,
\partial\widehat{\Omega}^{\pm}_{\hole} = \partial
\widehat{\Omega}^\pm  \setminus \partial \mathcal{K}^\pm.
\end{aligned}
\right.
\end{equation}  

\subsection{Outlook of the paper}
\label{sec:intr-sett-probl:outlook}

The remainder of the paper is organized as follows. In
Section~\ref{SectionTransmissionConditionsBoundaryLayer},  we
investigate the boundary layer problems. We derive transmission
condition for the macroscopic term $u_{n,q}^\delta$ up to any
order (Proposition~\ref{PropFormeGeneraleTransmission}). We also obtain an explicit formula for the
periodic correctors $\Pi_{n,q}^\delta$
(see~\eqref{definitionPiq}). In particular, we shall see that the
periodic corrector $\Pi_{n,q}^\delta$ is completely determined 
providing that the macroscopic terms $u_{n,p}^\delta$ are defined for $p\leq q$.\\ 

\noindent Then, Section~\ref{SectionFFS} is
dedicated to the analysis of the far field problems (consisting of the
far field equations~\eqref{FFVolum} together with the transmission
conditions~\eqref{SautFFq},\eqref{SautNormalFFq}). 
We first introduce two families of so-called macroscopic singularities $s_{m,q}^\pm$
(Proposition~\ref{PropExistenceUniquenesmMF} and Proposition~\ref{PropDefinitionsmq}). These functions
are particular solutions of the homogeneous Poisson equations (with
prescribed jump conditions across the interface $\Gamma$) that blow up
in the vicinity of the reentrant corners. These two
families are then used to
derive a quasi-explicit formula for the far field terms
(Proposition~\eqref{Definitionunq}). This quasi-explicit formula
defines the macroscopic terms $u_{n,q}^\delta$ up to the prescription of $2n$ constants called
$\ell_{-m}^\pm(u_{n,q}^\delta)$, $1\leq m \leq n$. \\

\noindent Section~\ref{SectionNF} deals with the resolution of the near field
problems~\eqref{NearFieldEquation}. As done for the macroscopic terms,
we define two families of near field singularities $S_m^\pm$, that are
particular solutions of the homogeneous Poisson problem posed in
$\widehat{\Omega}^{\pm}$  that blow up at infinity
(Proposition~\ref{PropositionExistenceSmplus}). Based on
these near field singularities, we then derive a quasi-explicit formula
for the near field terms~\eqref{FormeExplicitEUnqpm}. Here again,
this quasi-explicit  formula defines the near field terms $U_{n,q,\pm}^\delta$ up to the prescription of $n$ constants called
$\mathscr{L}_{m}(U_{n,q,\pm}^\delta)$, $1\leq m \leq n$. \\
   
\noindent Section~\ref{SectionMatchingProcedure} is dedicated to the derivation of the matching
conditions and the definition of the terms of the asymptotic
expansions. Based on an asymptotic representation of the far field terms close to
the reentrant corners and of the near field terms at infinity, we obtain a
collection of matching
conditions~\eqref{MactchingChampProche},\eqref{matchingChampLointain},\eqref{MactchingChampProcheMoins} and~\eqref{matchingChampLointainMoins} that permit to determine the constants
$\mathscr{L}_{m}(U_{n,q,\pm}^\delta)$ for the near fields and  the
constants $\ell_{-m}^\pm(u_{n,q}^\delta)$ for the macroscopic fields. As a
consequence, all the terms of the asymptotic expansion are then 
constructed (through an iterative procedure).\\

\noindent Finally, Section~\ref{SectionErrorAnalysis} deals with the justification of the
asymptotic. 

\section{Analysis of the boundary layer problems: transmission
  conditions}\label{SectionTransmissionConditionsBoundaryLayer}
This section is dedicated to the analysis of the boundary layers
problems~\eqref{NearFieldEquation}. It permits us to derive
(necessary) transmission
conditions for the macroscopic fields
$u_{n,q}^\delta$ across $\Gamma$ (Proposition~\ref{PropFormeGeneraleTransmission}). For a given $n \in \N$, we shall
propose a recursive procedure to write the jump of the trace and of
the normal trace of
$u_{n,q}^\delta$ as linear combinations of the mean values of the
macroscopic fields of lower order $u_{n,k}^\delta$ ($k<q$) and their tangential derivatives. This procedure is done
by induction on the index $q$ and is completely independent of the
index $n$ and of the superscript $\delta$ (of
  $u_{n,q}^\delta$). That is why we shall omit the index $n$ and the
superscript $\delta$ in this
section. \\

\noindent For any sufficiently smooth function $u$ defined in $\OmegaTop\cup \OmegaBottom$, we denote by $[u(x_1,0)
]_{\Gamma}$ and $\langle u(x_1,0) \rangle_{\Gamma}$ 
its jump and mean values
across $\Gamma$: for $|x_1| < \LBottom$,
\begin{align}\label{definitionMeanJump}
\dsp \left[ u(x_1,0) \right]_{\Gamma} = \lim_{h\rightarrow 0^+}
\left(  u(x_1, h) -  u(x_1,
 -h) \right), \quad 
\dsp \langle u(x_1,0) \rangle_{\Gamma} = \frac{1}{2} \lim_{h\rightarrow 0^+
  } \left(  u(x_1,h) +  u(x_1,
 -h) \right).
\end{align}

\noindent Let $(f_q)_{q\in \N}$ be a sequence of functions belonging to
$L^2(\OmegaTop)\cap L^2(\OmegaBottom)$ that are compactly supported in $\OmegaTop$.  We consider the following sequence of coupled problems (obtained by
rewriting 
\eqref{FFVolum},\eqref{FFDirichlet},\eqref{PeriodicCorrectorEquations} and
omitting the index~$n$):
\begin{equation}\label{ProblemeSectionTransmission}
\left \lbrace
\begin{aligned}
- \Delta u_{q} & = &  f_q &\quad \mbox{in} \; \OmegaTop \cup \OmegaBottom, \\
u_{q} & = & 0 &\quad  \mbox{on} \;  \Gamma_D,
\end{aligned} 
\right.
\quad
\left \lbrace
\begin{array}{rcll}
-\Delta_{\mX} \Pi_{q}(x_1, \mX) &= & G_q(x_1, \mX)  &\quad \mbox{in} \; \mathcal{B}\ , \\
 \partial_n \Pi_{q} &= & 0  &\quad \mbox{on} \;  \partial \widehat{\Omega}_\hole\ , 
\end{array}
\right.
\end{equation}
where
\begin{multline}\label{SecondMembreGq}
 G_q(x_1,\mX)   =   \sum_{p=0}^q \Big( 2 \langle g_p(\mX) \rangle \langle \partial^p_{x_2} u_{q-p}(x_1,0) \rangle_{\Gamma} +
 \frac{1}{2} \left[ g_p(\mX) \right] \left[  \partial^p_{x_2} u_{q-p}(x_1,0) \right]_{\Gamma} \Big) \\[-0.5em]
   + 2 \partial_{x_1} \partial_{X_1} \Pi_{q-1} + \partial_{x_1}^2 \Pi_{q-2}.
\end{multline}
Here, we use  
\begin{equation}\label{SautMoyennegp}
\langle  g_p(\mX) \rangle := \tfrac12 [  \Delta, \chi_+ + \chi_-] \left( \frac{X_2^p}{p!}\right)\ ,
\qquad
\left[  g_p(\mX)   \right] := [  \Delta, \chi_+ - \chi_-] \left( \frac{X_2^p}{p!}\right),
\; p \in \N,
\end{equation}
and later also $g_p^\pm(\mX)   \rangle :=  [\Delta, \chi^\pm] \left( \frac{X_2^p}{p!}\right)$ will be needed.
As previously, we impose $\Pi_q$ to be $1$-periodic with respect to
$X_1$ and to be super-algebraically decaying as $|X_2|$ tends to $+\infty$: for any $(k,
\ell) \in \N^2$, 
\begin{equation}
\label{eq:super-algebraically}
\lim_{|X_2| \rightarrow +\infty}
X_2^k \partial_{X_2}^\ell
\Pi_{q} = 0.
\end{equation}
\noindent Note that the right-hand side $G_q$ in~\eqref{SecondMembreGq}
corresponds to the right-hand side $G_{n,q}^\delta$ of Problem~\eqref{PeriodicCorrectorEquations} (for a given $n$).
The problems for $\Pi_q$, $q \in \IN$ are coupled to the others by the source terms. In difference, the 
problems for $u_q$, $q \in \IN$ are not complete and their coupling to other problems will be exposed in following.\\

 \noindent The present section is organized as follows: in
 Section~\ref{SubSectionPropBL}, we give a standard existence and
   uniqueness result
 (Proposition~\ref{prop:layer_existence_uniqueness_problem_strip}), which shows
 that under two compatibility conditions the boundary layer
 problems for $\Pi_q$ in~(\ref{ProblemeSectionTransmission}) have a unique
 decaying solution. In Section~\ref{SubSectionDerivationu0u1}, we use
 Proposition~\ref{prop:layer_existence_uniqueness_problem_strip} to derive
 transmission conditions for the first two terms $u_0$ and $u_1$ (see
 \eqref{Sautu0}-\eqref{Sautdu0}-\eqref{Sautu1}-\eqref{Sautdu1}), and we obtain
 an explicit tensorial representation for the associated boundary layer
 correctors (cf.
 \eqref{decompositionPi0}-\eqref{decompositionPi1}
 ). Finally the
 approach is extended in Section~\ref{SubSectionTransmissionGeneral} to
 obtain transmission conditions up to any order for the macroscopic fields
 $u_q$.

\begin{rem} The asymptotic construction described in this section
 is entirely similar to the construction of a multi-scale expansion for an
 infinite periodic thin layer (without corner singularity). A
 complete description of this case may be found in
 \cite{RapportSanchezPalencia}, \cite{AchdouCR}, \cite{Ammari}, \cite{Poirier} and
 references therein. 
\end{rem}
\subsection{Preliminary step : existence result for the boundary layer problem}\label{SubSectionPropBL}
In this subsection, we give a standard result of existence for the
boundary layer corrector problems for $\Pi_q$ in
(\ref{ProblemeSectionTransmission}). It will be subsequently
used to construct exponentially decaying boundary layer correctors. 
\noindent Let us introduce the two weighted Sobolev spaces
 \begin{equation}
    \label{eq:prop_layer_existence_uniqueness_space}
    \mathcal{V}^\pm(\mathcal{B}) = \left\lbrace \Pi \in \mathrm{H}_{\text{loc}}^1(\mathcal{B}),
      \Pi(0,X_2) = \Pi(1,X_2),  \text{ and }  \left(\Pi\, w_e^\pm\right)  \in
      \Hone(\mathcal{B}) \right\rbrace, 
    \end{equation}
where the weighting functions $w_e^\pm(X_1,X_2)= \chi(X_2)  \exp( \pm
\frac{| X_2|}{2})$. The functions of $\mathcal{V}^-(\mathcal{B})$ correspond to the
periodic (w.r.t. $X_1$) functions of $H^1_{\text{loc}}(\mathcal{B})$  that grow slower than $\exp(
\frac{| X_2|}{2})$ as $X_2$ tends to $\pm \infty$. By contrast, the
functions of $\mathcal{V}^+(\mathcal{B})$ correspond to the periodic
functions of $H^1_{\text{loc}}(\mathcal{B})$ decaying faster than $\exp(-
\frac{| X_2|}{2})$ as $X_2$ tends to $\pm \infty$. As a
consequence, they are super-algebraically decaying, means they satisfy
\eqref{eq:super-algebraically} for $\ell \leq 1$. Note also
that $\mathcal{V}^+(\mathcal{B}) \subset \mathcal{V}^-(\mathcal{B})$.\\

\noindent Based on this functional framework, we consider the following problem: 
for given $g \in \left(\mathcal{V}^-(\mathcal{B})\right)'$
find $\Pi \in \mathcal{V}^-(\mathcal{B})$ such that
  \begin{equation}
    \label{CanoniqueBoundaryLayer}
\left \lbrace
\begin{array}{rcll}
    \dsp - \Delta_\bX \Pi & = & g & \quad \text{in } \mathcal{B}, \\[1ex]
\dsp \partial_n \Pi 
    & = &  0 &\quad \text{on } \partial \widehat{\Omega}_{\hole}, \\[1ex]
\dsp \partial_{X_1} \Pi(0,X_2) & =&  \partial_{X_1} \Pi(1,X_2), & \quad X_2 \in
\R.
\end{array}
\right.
  \end{equation}
  \begin{prop}
  \label{prop:layer_existence_uniqueness_problem_strip}
 \begin{enumerate}
  \item Problem (\ref{CanoniqueBoundaryLayer})
    has a finite dimensional
    kernel of dimension $2$, spanned by the functions
    $\mathcal{N} = \mathds{1}_{\mathcal{B}}$ and $\mathcal{D}$, where
    $\mathcal{D}$ is the unique harmonic function of
    $\mathcal{V}^-(\mathcal{B})$ such that there exists
    $\mathcal{D}_\infty \in \R$
$$\widetilde{\mathcal{D} }(X_1, X_2) = \mathcal{D}(X_1, X_2)-\chi_+(X_2) ( X_2 + 
    \mathcal{D}_\infty)  -   \chi_-(X_2) (X_2 - \mathcal{D}_\infty) $$
  belongs to $\mathcal{V}^+(\mathcal{B})$. The constant
  $\mathcal{D}_\infty$ only depends on the geometry of the periodicity
  cell $\mathcal{B}$.

  \item If $f$ is orthogonal to $\mathcal{D}$ and $\mathcal{N}$ in
    the $\mathrm{L}^2(\mathcal{B})$ sense, meaning that
    \begin{align}
      \label{eq:prop_layer_existence_uniqueness_compatibility_D}
      \tag{$\mathcal{C}_{\mathcal{D}}$}
      \int_{\mathcal{B}} g(\bX) \mathcal{D}(\bX) d\bX & = 0 \\
      \label{eq:prop_layer_existence_uniqueness_compatibility_N}
      \tag{$\mathcal{C}_{\mathcal{N}}$}
      \int_{\mathcal{B}} g(\bX) \mathcal{N}(\bX) d\bX & = 0
    \end{align}
    then, there exists a unique solution $\Pi \in \mathcal{V}^+(\mathcal{B})$.
\item  Conversely,
  if problems (\ref{CanoniqueBoundaryLayer}) admits a
    solution $\Pi \in \mathcal{V}^+(\mathcal{B})$, then it satisfies the compatibility conditions
    \eqref{eq:prop_layer_existence_uniqueness_compatibility_D},
   \eqref{eq:prop_layer_existence_uniqueness_compatibility_N}.
  \end{enumerate}
\end{prop}
\noindent For the proof of the previous proposition we refer the reader to~\cite[Prop.~2.2]{Nazarov205} and~\cite[Sec.~5]{ArticleXavierDelourme}. 
General results on the elliptic problems in infinite cylinder can be found in~\cite{MR1469972} (Chapter 5).
Note that all these results remain the same with a different exponential growth or decay constant in the definition of $w^\pm_e$ unless it does not exceed
an upper bound which is determined by the least exponentially decaying or growing functions in the kernel of $-\Delta$ in $\mathcal~B$.\\

\noindent Based on the previous proposition, we shall construct
$\Pi_q(x_1, \mX)$
in $\mathcal{C}\left((-\LBottom, \LBottom), \mathcal{V}^+(\mathcal{B}) \right)$. Transmission conditions for the macroscopic terms $[u_q]_\Gamma$ and $[\partial_{x_2} u_q]_\Gamma$ will
directly follow from the compatibility conditions (\ref{eq:prop_layer_existence_uniqueness_compatibility_D}),
   (\ref{eq:prop_layer_existence_uniqueness_compatibility_N}) 
applied to Problem (\ref{ProblemeSectionTransmission}) for $\Pi_q$, 
$q \in \N$. It will guarantee 
that the boundary layer correctors $\Pi_q$ are exponentially decaying. 
Let us give a couple of useful 
relations, which are easy to obtain by direct calculations
(noting that $g_0^\pm = [\Delta,\chi_\pm]1$),  and will be extensively used in the
next subsections:
\begin{lema}\label{CalculConditionCompatibilite} The following
  relations hold:
\begin{equation} \label{CalculsPratiques}
\begin{array}{ll}
\dsp \int_{\mathcal{B}} \langle g_0(\mX) \rangle  \mathcal{D}(\mX) d\mX  = 0, &
\dsp \int_{\mathcal{B}} \left[ g_0(\mX) \right]  \mathcal{D}(\mX) d\mX
= -2,  \\[2ex]
\dsp \int_{\mathcal{B}} \langle g_0(\mX) \rangle  \mathcal{N}(\mX) d\mX  = 0,& 
\dsp \int_{\mathcal{B}} \left[ g_0(\mX) \right] \mathcal{N}(\mX) d\mX  =
0,\\[2ex]
\dsp \int_{\mathcal{B}} \langle g_1(\mX) \rangle  \mathcal{D}(\mX) d\mX  = \mathcal{D}_\infty, &
\dsp \int_{\mathcal{B}} \left[ g_1(\mX) \right]  \mathcal{D}(\mX) d\mX  = 0,\\[2ex]
\dsp \int_{\mathcal{B}} \langle g_1(\mX) \rangle  \mathcal{N}(\mX) d\mX  = 0, & 
\dsp \int_{\mathcal{B}} \left[ g_1(\mX) \right] \mathcal{N}(\mX) d\mX
= 2.
\end{array}  
\end{equation}
\end{lema}

\subsection{Derivation of the first terms}\label{SubSectionDerivationu0u1}
We can now turn to the formal computation of the first solutions of the
sequence of Problems~\eqref{ProblemeSectionTransmission}. We emphasize that the upcoming
iterative procedure is formal in the sense that we shall provide necessary transmission conditions for the
macroscopic terms $u_q$ but we shall not adress the question of
their existence in this part (this question will be investigated in
Section~\ref{SectionFFS}). Throughout this section, we assume
that the macroscopic terms exist and are smooth above and below the
interface $\Gamma$.

\subsubsection{Step 0: $[u_0]_\Gamma$  and $\Pi_0$}\label{SubsubStep0}
The limit boundary layer term (or periodic corrector) $\Pi_0$ is solution
of
\begin{equation}\label{ProblemPi0}
\left\lbrace
\begin{array}{rcll}
- \Delta_{\mX} \Pi_{0}(x_1, \mX) &= & G_0(x_1, \mX)  &\; \mbox{in} \; \mathcal{B}, \\
\partial_n \Pi_{0} & = & 0  &\; \mbox{on} \;  \partial \widehat{\Omega}_\hole,
\end{array}
\right.
\end{equation}
where $
 G_0(x_1, \mX) =  2  \langle g_0(\mX) \rangle \langle u_0(x_1,0) \rangle_{\Gamma} +
 \frac{1}{2} \left[ g_0(\mX) \right] \left[ u_0(x_1,0)
 \right]_{\Gamma}$.
Problem~\eqref{ProblemPi0} is a partial differential equation with
respect to the microscopic variables $X_1$ and $X_2$, wherein the macroscopic variable $x_1$ plays the role
of a parameter. For a fixed $x_1$ in $(-\LBottom, \LBottom)$ (considered as a parameter), $G_0(x_1, \cdot)$ belongs to
$(\mathcal{V}^-(\mathcal{B}))'$ since it is compactly supported.
Then, in view of
Proposition~\ref{prop:layer_existence_uniqueness_problem_strip},
there exists an exponentially decaying solution $\Pi_{0}(x_1, \cdot) \in \mathcal{V}^+(\mathcal{B})$ if and only if the two compatibility conditions  (\ref{eq:prop_layer_existence_uniqueness_compatibility_D},
   \ref{eq:prop_layer_existence_uniqueness_compatibility_N}) (Prop.~\ref{prop:layer_existence_uniqueness_problem_strip}) are
   satisfied. Thanks to the second line of Lemma~\ref{CalculConditionCompatibilite},
\begin{equation}
\int_{\mathcal{B}}  G_0(x_1, \mX)  \mathcal{N}(\mX)  d\mX= 0,
\end{equation}
which means that
(\ref{eq:prop_layer_existence_uniqueness_compatibility_N}) is always
satisfied. Besides, in view of the first line of Lemma~\ref{CalculConditionCompatibilite},
\begin{equation}
\int_{\mathcal{B}}  G_0(x_1, \mX)  \mathcal{D}(\mX)  d\mX = - 
\left[ u_0(x_1,0) 
 \right]_{\Gamma}. 
\end{equation}
As a consequence, we obtain a necessary and sufficient condition for $\Pi_0$ to be
exponentially decaying:
\begin{equation}\label{Sautu0}
\left[ u_0(x_1,0)
 \right]_{\Gamma} =0.
\end{equation}
This condition provides a first transmission condition for the limit
macroscopic term $u_0$. Under the previous condition, $ G_0(x_1, \mX) =  2  \langle
g_0(\mX) \rangle \langle u_0(x_1,0) \rangle_{\Gamma}$, and, using the
linearity of Problem~\eqref{ProblemPi0}, we can obtain a tensorial
representation of $\Pi_0 \in \mathcal{C}\left((-\LBottom, \LBottom),
  \mathcal{V}^+(\mathcal{B})\right)$, in which macroscopic and microscopic variables
are separated:
\begin{equation}\label{decompositionPi0}
\Pi_0(x_1, X) =   \langle u_0(x_1,0) \rangle_{\Gamma} \,  W_{0}^{\mathfrak{t}}(\mX).
\end{equation}
Here the profile function $W_{0}^{\mathfrak{t}}(\mX)$ is the unique function of
$\mathcal{V}^+(\mathcal{B})$ satisfying  
 \begin{equation}\label{ProblemW0t}
\left \lbrace
\begin{array}{rcll}
- \Delta_{\mX} W_{0}^{\mathfrak{t}}(\mX) &= & F_{0}^{\mathfrak{t}}(\mX)  &\; \mbox{in} \; \mathcal{B}, \\
\partial_n W_{0}^{\mathfrak{t}} &= & 0  &\; \mbox{on} \;  \partial
\widehat{\Omega}_\hole,\\
\partial_{X_1} W_{0}^{\mathfrak{t}} (0,X_2) &= &\partial_{X_1} W_{0}^{\mathfrak{t}}(1,X_2),& \; X_2 \in
\R,
\end{array} 
\right. \quad F_{0}^{\mathfrak{t}}(\mX) = 2  \, \langle g_0(\mX) \rangle.\\
\end{equation}
A direct calculation shows that 
\begin{equation}\label{ProprieteW0t}
W_{0}^{\mathfrak{t}}(\mX) = (1  - \chi(X_2)).
\end{equation}
Note that the continuity of $\Pi_0$ with respect to $x_1$ is a
consequence of the continuity of $G_0$ with respect to $x_1$. 
\subsubsection{Step 1: $[\partial_{x_2}
  u_0]_\Gamma$, $[u_1]_\Gamma$, and $\Pi_1$}\label{SubsubStep1}
In view of the general sequence of problems~\eqref{ProblemeSectionTransmission}, the second boundary
layer (or periodic corrector) $\Pi_1$ satisfies
\begin{equation}\label{ProblemPi1}
\left \lbrace
\begin{array}{rcll}
\dsp - \Delta_{\mX} \Pi_{1}(x_1, \mX) & = & G_1(x_1, \mX)  & \; \mbox{in} \; \mathcal{B}, \\
\dsp \partial_n \Pi_{1} & = & 0  &\; \mbox{on} \;  \partial \widehat{\Omega}_\hole.
\end{array} 
\right.
\end{equation}
where, thanks to~\eqref{ProprieteW0t} ($\partial_{X_1} W_{0}^{\mathfrak{t}}=0$),
\begin{multline}\label{G1}
   G_1(x_1, \mX) =   \frac{1}{2} \left[ g_1(\mX) \right]
 \left[\partial_{x_2} u_0(x_1,0) \right]_\Gamma +
 \frac{1}{2} \left[ g_0(\mX) \right] \left[ u_1(x_1,0)
 \right]_{\Gamma}  \\+ F_{0}^{\mathfrak{t}}(\mX)   \langle u_1(x_1,0)
 \rangle_{\Gamma} +  2  \langle g_1(\mX) \rangle\, \langle \partial_{x_2}
 u_0(x_1,0) \rangle_{\Gamma}.
\end{multline}
As for $\Pi_0$,  Problem~\eqref{ProblemPi1} is a partial differential equation with
respect to the microscopic variables $X_1$ and $X_2$, where the macroscopic variable $x_1$ plays the role
of a parameter. For a fixed $x_1$ in $(-\LBottom, \LBottom)$,
$G_1(x_1, \cdot)$ is compactly
supported in $\mathcal{B}$, and, consequently, belongs
to $(\mathcal{V}^-(\mathcal{B}))'$. 
Then, thanks to
Proposition~\ref{prop:layer_existence_uniqueness_problem_strip}, there
exists an exponentially decaying solution $\Pi_1(x_1,\cdot) \in \mathcal{V}^+(\mathcal{B})$ if and only if the two compatibility conditions~(\ref{eq:prop_layer_existence_uniqueness_compatibility_D}),
   (\ref{eq:prop_layer_existence_uniqueness_compatibility_N}) are
   satisfied. In view of
Lemma~\ref{CalculConditionCompatibilite}, $F_{0}^{\mathfrak{t}}(\mX)$, $[g_0(\mX)]$,  $\langle g_1
(\mX) \rangle$ are orthogonal to $\mathcal{N}$. Then, the second
formula of the fourth line
of Lemma~\ref{CalculConditionCompatibilite} gives
\begin{equation*}
\int_{\mathcal{B}} G_1(x_1, \mX)
\mathcal{N}(\mX) d\mX = \left[\partial_{x_2} u_0(x_1,0) \right]_\Gamma.
\end{equation*}
Therefore, the compatibility condition
\eqref{eq:prop_layer_existence_uniqueness_compatibility_N} is
fulfilled if and only if
\begin{equation}\label{Sautdu0}
\left[\partial_{x_2} u_0(x_1,0) \right]_\Gamma = 0.
\end{equation}
Next, using the first and third lines of
Lemma~\ref{CalculConditionCompatibilite}, we obtain
\begin{equation*}
\int_{\mathcal{B}} G_1(x_1, \mX)
\mathcal{D}(\mX) d\mX =  -  \left[ u_1(x_1,0)
 \right]_{\Gamma}   +  2 \mathcal{D}_\infty \, \langle \partial_{x_2}
 u_0(x_1,0) \rangle_{\Gamma}.
\end{equation*}
Therefore, the compatibility condition
\eqref{eq:prop_layer_existence_uniqueness_compatibility_D} is
fulfilled if and only if
\begin{equation}\label{Sautu1}
\left[ u_1(x_1,0) \right]_\Gamma =  2 \mathcal{D}_\infty
\, \langle \partial_{x_2}
 u_0(x_1,0) \rangle_{\Gamma}.
\end{equation}
Under the two conditions~\eqref{Sautdu0}-\eqref{Sautu1},
Problem~\eqref{ProblemPi1} has a unique solution $\Pi_1$ in
$\mathcal{C}\left( (-\LBottom, \LBottom), \mathcal{V}^+(\mathcal{B})\right)$ (the continuity of $\Pi_1$ with respect to $x_1$
results from the continuity of $G_1$ with respect to $x_1$). Using (here again) the linearity of
Problem~\eqref{ProblemPi1}, we can write $\Pi_1$ as a tensorial product between profile functions
that only depend on the microscopic variables $X_1$ and $X_2$, and functions that
only depend on the macroscopic variable $x_1$ (more precisely, the latter functions
consist of the average
traces of the macroscopic terms of order $0$ and $1$ on $\Gamma$):
\begin{equation}\label{decompositionPi1}
\Pi_1(x_1, \mX) = \langle u_1(x_1,0) \rangle_\Gamma\,
W_{0}^{\mathfrak{t}}(\mX)  \;+ \; \langle \partial_{x_2} u_0(x_1,0)
\rangle_\Gamma \, W_{1}^{\mathfrak{n}}(\mX),
\end{equation}
where $W_{0}^{\mathfrak{t}}$ is defined by~\eqref{ProblemW0t} and $W_{1}^{\mathfrak{n}}
\in \mathcal{V}^+(\mathcal{B})$ is the unique decaying solution to the
following problem:
\begin{equation}\label{ProblemW1n}
\left \lbrace
\begin{array}{rcll}
\dsp - \Delta_{\mX} W_{1}^{\mathfrak{n}}(\mX) & = & F_{1}^{\mathfrak{n}}(\mX) + \frac{\mathcal{D}_{1}^{\mathfrak{n}}  }{2}
\dsp [g_0(\mX)] &\; \mbox{in} \; \mathcal{B}, \\[1ex]
\partial_n W_{1}^{\mathfrak{n}} &= &0  &\; \mbox{on} \;  \partial
\widehat{\Omega}_\hole,\\[1ex]
\partial_{X_1} W_{1}^{\mathfrak{n}} (0,X_2) &= &\partial_{X_1} W_{1}^{\mathfrak{n}}(1,X_2), &\; X_2 \in
\R,
\end{array}
\right. 
\end{equation}
where, 
\begin{equation}\label{DefF1n}
F_{1}^{\mathfrak{n}}(\mX)= 2 \langle g_1(\mX) \rangle_\Gamma \quad  \mbox{and}
\quad 
\mathcal{D}_{1}^{\mathfrak{n}} = \int_{B} F_{1}^{\mathfrak{n}}(\mX) \mathcal{D}(\mX)d\mX = 2 \mathcal{D}_\infty.
\end{equation}
It is easily seen that the right-hand side of~\eqref{ProblemW1n} is
orthogonal to both $\mathcal{N}$ and $\mathcal{D}$. A direct computation shows that 
\begin{equation}\label{W1nEgaleDtilde}
W_{1}^{\mathfrak{n}}(\mX) = \widetilde{\mathcal{D}}(\mX),
\end{equation} 
the function $\widetilde{\mathcal{D}}$ being defined in the first point of Proposition~\ref{prop:layer_existence_uniqueness_problem_strip}.

\subsubsection{Step 2: $[\partial_{x_2}
  u_1]_\Gamma$ ($[u_2]_\Gamma$ and $\Pi_2$)}
\label{SubsubStep2}
We can continue the iterative procedure started in the two previous
steps as follows. The periodic corrector $\Pi_2$ satisfies the following
equation
\begin{equation}\label{ProblemPi2}
\left \lbrace
\begin{array}{rcll}
\dsp - \Delta_{\mX} \Pi_{2}(x_1, \mX) & = & G_2(x_1, \mX)  &\; \mbox{in} \; \mathcal{B}, \\
\dsp \partial_n \Pi_{2} &=&  0  &\;\mbox{on} \;  \partial \widehat{\Omega}_\hole.
\end{array}
\right. 
\end{equation}
Here,
\begin{multline}\label{G2}
   G_2(x_1, \mX)  =  
 \frac{1}{2} \left[ g_0(\mX) \right] \left[ u_2
 \right]_{\Gamma}   +   \frac{1}{2} \left[ g_1(\mX) \right]
 \left[\partial_{x_2} u_1 \right]_\Gamma  
 + F_{0}^{\mathfrak{t}}(\mX) \langle u_2
 \rangle_{\Gamma} \; 
\\+ \;  F_{1}^{\mathfrak{n}}(\mX) \, \, \langle \partial_{x_2}
 u_1 \rangle_{\Gamma}  \; +  \; F_{2}^{\mathfrak{t}}(\mX) \, \partial_{x_1}^2 \langle
u_{0}\rangle_{\Gamma} 
\; + \;  F_{2}^{\mathfrak{n}}(\mX) \, \partial_{x_1} \langle \partial_{x_2}
u_{0}\rangle_{\Gamma}.
\end{multline}
$F_{0}^{\mathfrak{t}}(\mX)$ and $F_{1}^{\mathfrak{n}}(\mX)$ are given by~\eqref{ProblemW0t}-\eqref{DefF1n}, and,
\begin{equation}\label{DefF2tF2n}
F_{2}^{\mathfrak{t}}(\mX) =  - 2 \langle g_2(\mX) \rangle + W_{0}^{\mathfrak{t}}(\mX)\ , \quad
F_{2}^{\mathfrak{n}}(\mX) = 2 \, \partial_{X_1} \, W_{1}^{\mathfrak{n}}(\mX). 
\end{equation}
In formula~\eqref{G2}, for the sake of concision, we have omitted the
dependence on $x_1$ 
of the macroscopic terms. To obtain this formula, we have replaced $\Pi_0$ and
$\Pi_1$ with their tensorial representations~(\ref{decompositionPi0}),(\ref{decompositionPi1}),  we have substituted 
$-\partial_{x_2}^2 u_0(x, 0^\pm)$ by $\partial_{x_1}^2u_0(x, 0^\pm)$ using the
macroscopic equation~\eqref{ProblemeSectionTransmission}  ($\Delta
u_{0} = 0$ in the vicinity of $\Gamma$) and we have taken into account
the jump conditions~\eqref{Sautu0},\eqref{Sautdu0} for $u_0$. \\

\noindent  For a fixed $x_1 \in (-\LBottom, \LBottom)$, it is easily verified that $G_2(x_1, \cdot)$ belongs to
$(\mathcal{V}^-(\mathcal{B}))' $: indeed, the first five  terms
of~\eqref{G2} are compactly supported and the last one is
exponentially decaying (more precisely, $w^+_e F_2^{\mathfrak{t}}$ and $w^+_e F_2^{\mathfrak{n}}$
belong to $L^2(\mathcal{B})$). %
Then again, the existence of an exponentially
decaying corrector $\Pi_2(x_1, \cdot) \in \mathcal{V}^+(\mathcal{B})$
results from the orthogonality of
$G_2(x_1, \cdot)$ with $\mathcal{N}$ and
$\mathcal{D}$. As previously, enforcing the compatibility condition
\eqref{eq:prop_layer_existence_uniqueness_compatibility_N}  provides
the transmission condition for the jump
 of the normal trace of $u_1$ across $\Gamma$:
\begin{equation}\label{Sautdu1}
[\partial_{x_2} u_1 ]_\Gamma =  \mathcal{N}_{2}^{\mathfrak{t}} \, \partial_{x_1}^2 \langle
u_0 \rangle_{\Gamma} + \mathcal{N}_{2}^{\mathfrak{n}} \, \partial_{x_1} \langle \partial_{x_2}
u_0 \rangle_{\Gamma},
\end{equation}
where
\begin{equation}\label{DefN2tN2n}
\mathcal{N}_{2}^{\mathfrak{t}} = - \int_{\mathcal{B}} F_{2}^{\mathfrak{t}}(\mX)
\mathcal{N}(\mX) d\mX, \quad \mathcal{N}_{2}^{\mathfrak{n}} = - \int_{\mathcal{B}} F_{2}^{\mathfrak{n}}(\mX)
\mathcal{N}(\mX) d\mX. 
\end{equation}
Then, enforcing the compatibility condition~\eqref{eq:prop_layer_existence_uniqueness_compatibility_N}  provides the jump $[u_2]_\Gamma$, and the
existence of $\Pi_2$ is proved.  Naturally an explicit expression of
$[u_2]_\Gamma$ and a tensorial representation of $\Pi_2$ can be
written (see the upcoming formulas~\eqref{SautFFq}-\eqref{definitionPiq}), but, for the sake of concision, we do not write it here.
\subsection{Transmission conditions up to any
  order}\label{SubSectionTransmissionGeneral}
We are now in a position to extend the previous approach up to any order. For each $q\in \N$, similarly to the first steps, our global iterative
approach relies on the following procedure:
\begin{enumerate}
\item We compute the right-hand side $G_q(x_1,\mX)$  of the
  periodic corrector problem~\eqref{ProblemeSectionTransmission}  of
  order $q$: we write $G_q$ as a tensorial product between functions that only
  depend on the microscopic variables $X_1$ and $X_2$ and functions
  that only depend on the macroscopic variable $x_1$.  More
  specifically, the latter functions consist of 
  the trace and normal trace of the macroscopic terms of order lower than $q$ and their
  tangential derivatives (see \eqref{ProblemPi0},\eqref{G1},\eqref{G2}).
%

\item We compute the normal jump $[\partial_{x_2} u_{q-1}(x_1, 0)]_\Gamma$ by
  enforcing $G_q$ to be orthogonal to $\mathcal{N}$, \ie, to satisfy the compatibility condition
  \eqref{eq:prop_layer_existence_uniqueness_compatibility_N} (see \eqref{Sautu0},\eqref{Sautu1}).
\item We compute the jump $[ u_{q}(x_1, 0)]_\Gamma$ by
  imposing $G_q$ to satisfy the compatibility condition
  \eqref{eq:prop_layer_existence_uniqueness_compatibility_D} ensuring that
  $G_q$ is orthogonal to $\mathcal{D}$ (see  \eqref{Sautdu0},\eqref{Sautdu1}).
\item We write a tensorial representation of the periodic corrector
  $\Pi_q$ introducing at most two new profile functions (see formulas \eqref{decompositionPi0},\eqref{decompositionPi1}).
\end{enumerate}
 Applying this general scheme, we can prove the following proposition,
whose complete proof is postponed in Appendix~\ref{ProofPropositionGeneralTransmission}.
\begin{prop}\label{PropFormeGeneraleTransmission}
Assume that the macroscopic terms $u_q$ satisfying~\eqref{ProblemeSectionTransmission}
exist. Then, there exists four sequences of real constants $\mathcal{N}_{p}^{\mathfrak{n}}$,
$\mathcal{N}_{p}^{\mathfrak{t}}$, $\mathcal{D}_{p}^{\mathfrak{n}}$,
$\mathcal{D}_{p}^{\mathfrak{t}}$ such that
\begin{subequations}
\label{SautFFq+NormalFFq}
\begin{align}\label{SautFFq}
[u_q(x_1,0)]_{\Gamma} &= 
\sum_{p=1}^q \mathcal{D}_p^{\mathfrak{t}} \, \partial_{x_1}^p \langle
u_{q-p}(x_1,0)
\rangle_{\Gamma} \hspace{1.8em} + \sum_{p=1}^{q} 
\mathcal{D}_p^{\mathfrak{n}} \,  \partial_{x_1}^{p-1} \langle \partial_{x_2}
u_{q-p}(x_1,0) \rangle_{\Gamma}\ ,\\
\label{SautNormalFFq}
[\partial_{x_2}u_{q}(x_1,0)]_{\Gamma} &= 
\sum_{p=1}^{q} \mathcal{N}_{p+1}^{\mathfrak{t}} \, \partial_{x_1}^{p+1} \langle
u_{q-p}(x_1,0)
\rangle_{\Gamma}+\sum_{p=1}^{q} 
\mathcal{N}_{p+1}^{\mathfrak{n}} \,  \partial_{x_1}^{p} \langle \partial_{x_2}
u_{q-p}(x_1,0) \rangle_{\Gamma}\ .
\end{align}
\end{subequations}
\end{prop} 
\noindent In the previous definition,we have used the superscript $\mathfrak{t}$
(in $\mathcal{D}_p^{\mathfrak{t}}$, $\mathcal{N}_p^{\mathfrak{t}}$) to refer to
some constants associated with tangential derivatives of the average trace
of the macroscopic terms. Similarly, the superscript $\mathfrak{n}$ (in $\mathcal{D}_p^{\mathfrak{n}}$,
$\mathcal{N}_p^{\mathfrak{n}}$) is used for the constants associated with tangential
derivatives of the average of the normal trace of the macroscopic terms. 
\begin{rem} \label{RemPropFormeGeneraleTransmission} In the proof of
  Proposition~\ref{PropFormeGeneraleTransmission} (Appendix~\ref{ProofPropositionGeneralTransmission}), we also prove
  simultaneously that there exist two families of decaying profile functions
  $W_{p}^{\mathfrak{t}}$ and $W_{p}^{\mathfrak{n}}$ belonging to $\mathcal{V}^+(\mathcal{B})$
  such that the periodic corrector $\Pi_q \in \mathcal{C}\left( (-\LBottom,
  \LBottom), \mathcal{V}^+(\mathcal{B})\right)$ admits the following representation:
 \begin{equation}\label{definitionPiq}
\Pi_q(x_1, \mX) = \sum_{p=0}^q \partial_{x_1}^p \langle
u_{q-p}(x_1,0)
\rangle_{\Gamma} W_{p}^{\mathfrak{t}}(\mX) +\sum_{p=1}^{q} 
\partial_{x_1}^{p-1} \langle \partial_{x_2}
u_{q-p}(x_1,0) \rangle_{\Gamma} W_{p}^{\mathfrak{n}}(\mX).
\end{equation}
The definitions of the functions $W_p^{\mathfrak{t}}$ and $W_p^{\mathfrak{n}}$ and of the constants $\mathcal{D}_q^{\mathfrak{t}}$, $\mathcal{N}_q^{\mathfrak{t}}$,
$\mathcal{D}_q^{\mathfrak{n}}$, $\mathcal{N}_q^{\mathfrak{n}}$ are given explicitely in \eqref{ProblemWpt},\eqref{DefDptNpt},\eqref{ProblemWpn},
and~\eqref{DefDpnNpn}. 
\end{rem}
\noindent We point out that the periodic correctors $\Pi_q$ do not appear
(explicitly) in~\eqref{SautNormalFFq}: they have been eliminated. In
other words, the resolution of macroscopic and boundary layer problems
are decoupled and the construction of $\Pi_q$ can be made a posteriori.

\section{Analysis of the macroscopic problems (macroscopic singularities)}\label{SectionFFS}

Thanks to the previous section (see in particular
Proposition~\ref{PropFormeGeneraleTransmission}, reminding that the
index $n$ and the superscript $\delta$ have been deliberately omitted
in the previous section), we can see that if the
macroscopic terms $u_{n,q}^\delta$ (solution to \eqref{FFVolum})
exist, they satisfy the
following transmission problems: for any $(n,q)  \in \N^2$,
\begin{subequations}
\label{MacroscopicProblem1}
\begin{equation}
\left \lbrace
\begin{array}{r@{\;}ll}
-\Delta u_{n,q}^\delta & =  f_{n,q} &\quad  \mbox{in} \; \OmegaTop \cup \OmegaBottom,\\[0.5em]
\left[ u_{n,q}^\delta (x_1, 0) \right]_{\Gamma} &= g_{n,q}^\delta ,\\[0.5em]
\left[ \partial_{x_2} u_{n,q}^\delta (x_1, 0) \right]_{\Gamma} & = h_{n,q}^\delta, \\[0.5em]
u_{n,q}^\delta & = 0 &\quad \mbox{on} \, \Gamma_D,
\end{array}
\right. \quad \quad  f_{n,q} = \begin{cases}
f & \mbox{if} \; n=q=0, \\
0 & \mbox{otherwise,}
\end{cases} 
\end{equation}
where 
\begin{align}
g_{n,q}^\delta(x_1) &= \sum_{p=1}^q \mathcal{D}_p^{\mathfrak{t}} \, \partial_{x_1}^p \langle
u_{n,q-p}^\delta(x_1,0)  \rangle_{\Gamma} \hspace{1.8em} + \sum_{p=1}^{q} 
\mathcal{D}_p^{\mathfrak{n}} \,  \partial_{x_1}^{p-1} \langle \partial_{x_2}
u_{n,q-p}^\delta(x_1,0) \rangle_{\Gamma}\ , \label{Defgnq}\\
h_{n,q}^\delta(x_1) &= 
\sum_{p=1}^{q} \mathcal{N}_{p+1}^{\mathfrak{t}} \, \partial_{x_1}^{p+1} \langle
u_{n,q-p}^\delta(x_1,0)
\rangle_{\Gamma}+\sum_{p=1}^{q} 
\mathcal{N}_{p+1}^{\mathfrak{n}} \,  \partial_{x_1}^{p} \langle \partial_{x_2}
u_{n,q-p}^\delta(x_1,0) \rangle_{\Gamma}\ . \label{Defhnq}
\end{align}
\end{subequations}
 As previously mentioned, the constants $\mathcal{D}_q^{\mathfrak{t}}$, $\mathcal{N}_q^{\mathfrak{t}}$,
$\mathcal{D}_q^{\mathfrak{n}}$, $\mathcal{N}_q^{\mathfrak{n}}$, which only depend on the geometry of
the periodicity cell $\mathcal{B}$, are defined in \eqref{DefDptNpt}-\eqref{DefDpnNpn}. \\

\noindent The present section is dedicated to the analysis of 
Problems~\eqref{MacroscopicProblem1}. In
Subsection~\ref{SubsectionVariationalFramework},  we give general
results of well-posedness for transmission problems: we first introduce a variational framework,
then we present an alternative functional framework based on weighted
Sobolev spaces. 
In Subsection~\ref{SubsectionFirstSingularity}, we explain the reason
why the variational framework is not adapted for the resolution 
of Problem~\eqref{MacroscopicProblem1} for $q=1$ (and higher). This leads us to
consider singular (extra-variational) macroscopic terms that may blow
up in the vicinity of the two corners.  In Subsection~\ref{MacroscopicSingularities},
we construct several sequences of singular functions that are used in
Subsection~\ref{SubsectionExpliciteMacro} to write a general formula
for the macroscopic terms (Proposition~\ref{PropExplicitMacro}).
\subsection{General results of existence for transmission
  problem}\label{SubsectionVariationalFramework}
The problems under consideration can be investigated using the general
framework for transmission problems posed in polygonal domains developed in
\cite{Nicaise}. 
In the present paper, we first recall a classical
well-posedness result based on a variational form of the problem. Then, based
on weighted Sobolev spaces, we describe the behavior of the solutions close to
the two reentrant corners.
\subsubsection{Variational framework}
Let us introduce the classical Hilbert spaces associated with our problems  
\begin{align*} 
\HoneGammaD(\OmegaTop\cup\OmegaBottom) &= \left\{ u \in \Hone(\OmegaTop\cup\OmegaBottom),
 \; \mbox{s.t.} \; u
= 0  \text{ on }  \Gamma_D\right\}\ ,
\end{align*}
which incorporates discontinuous functions over $\Gamma$ (see Figure~\ref{fig:Omega}).
Its restrictions to $\OmegaTop$ and $\OmegaBottom$
are denoted by $\HoneGammaD(\OmegaTop)$ and $\HoneGammaD(\OmegaBottom)$.
We denote by $\HonehalfzzG$ the restriction of
the trace of the function $\HoneGammaD(\OmegaTop)$ to $\Gamma$ (for a complete description of the trace of functions, we refer the reader to~\cite{Grisvard85}.), \ie,
\begin{equation*}
\HonehalfzzG = \left\{ \mu \in
    \HonehalfG, \mbox{s.t.} \,\exists\, v\in \HoneGammaD(\OmegaTop): 
    v = \mu \text{ on } \Gamma \right\}.
\end{equation*}
Naturally, the space $\HonehalfzzG$ is also the restriction of 
the trace of the functions of  $\HoneGammaD(\OmegaBottom)$ to $\Gamma$. Based on a variational formulation, and thanks to the Lax-Milgram lemma, we can prove the following 
well-posedness result: 
\begin{prop}\label{PropExistenceUniquenessFF} Let $\mathfrak{f} \in
  L^2(\Omega)$, $\mathfrak{g} \in \HonehalfzzG$,  and $\mathfrak{h} \in
  L^2(\Gamma)$. Then, the following problem has a unique solution $u$
  belonging to $\HoneGammaD(\OmegaTop\cup\OmegaBottom)$:
\begin{equation}\label{MacroscopicProblemModel}
\left \lbrace
\begin{aligned}
      - \Delta u &= \mathfrak{f} &\quad& \text{in } \OmegaTop \cup \OmegaBottom\ ,\\
      \left[ u \right]_{\Gamma} &= \mathfrak{g} &\quad&\text{on }\Gamma\ ,\\
      \left[ \partial_{x_2} u \right]_{\Gamma}  &= \mathfrak{h} &\quad&\text{on }\Gamma\ .  
    \end{aligned}
\right.
\end{equation}
\end{prop} 

\subsubsection{Weighted Sobolev spaces and asymptotic behaviour}
 In the next subsections, we shall study the behavior of the macroscopic
 terms in the neighborhood of the two reentrant corners. It is
 well-known that the Hilbert spaces $H^m(\OmegaBottom)$ (resp. $H^m(\OmegaTop)$) are not well-adapted to
 this investigation. By contrast, the
 weighted Sobolev spaces provide a more convenient functional
 framework. We refer the reader to the Kondrat'ev theory 
(see~\cite{Kondratev1963}, ~\cite[Chap.~5 and Chap.~6]{MR1469972} for a
complete presentation of these spaces and their applications). In this
part, we introduce the weighted Sobolev spaces associated with our
problem following the presentation 
of~\cite[Chap.~6]{MR1469972}. Let us first define the polar coordinates $(r^\pm, \theta^\pm)$
centered at the vertex $\mx_O^\pm$, \ie,
\begin{equation}
x_1 - (\mathbf{x}_{O}^\pm)_1 =  r^\pm \cos(\theta^\pm), \quad x_2 - (\mathbf{x}_{O}^\pm)_2 =  r^\pm \sin(\theta^\pm).
\end{equation}
Next, we consider the two infinite angular (or conical) domains 
$\mathcal{K}_{\mathbf{x}_{O}^\pm}$ centered at $\mathbf{x}_{O}^\pm$  of
opening $\frac{3 \pi}{2}$
\begin{equation}\label{DefinitionKpmW}
\mathcal{K}_{\mathbf{x}_{O}^\pm} = \left\{ (r^\pm \cos \theta^\pm, r^\pm \sin \theta^\pm) \in \R^2,  r^\pm > 0, \theta^\pm
  \in I^\pm \right\}, \quad I^+ = (0, \frac{3\pi}{2}), \quad I^-=(-\frac{\pi}{2}, \pi),
\end{equation}
and, for $\ell \in \{0,1,2\}$, we define the space
$V_{2,\beta}^\ell(\mathcal{K}_{\mathbf{x}_{O}^\pm})$ as the closure of
$\mathcal{C}_0^\infty(\overline{\mathcal{K}_{\mathbf{x}_{O}^\pm}} \setminus \{0\})$ with
respect to the norm
\begin{equation}\label{definitionV2betal}
\| u\|_{V_{2,\beta}^\ell(\mathcal{K}_{\mathbf{x}_{O}^\pm})} = \left( \int_{\mathcal{K}_{\mathbf{x}_{O}^\pm}}
\sum_{|\alpha| \leq \ell}
(r^\pm)^{ 2 (\beta - \ell + |\alpha|) }|\partial_{x_1}^{\alpha_1} \partial_{x_2}^{\alpha_2} u|^2 d\mathbf{x} \right)^{1/2},
 \quad  \alpha = (\alpha_1, \alpha_2) \in \N^2, |\alpha| = \alpha_1 + \alpha_2.
\end{equation}
Then, let
\begin{equation}
  \label{eq:definition_chi_LBottom_pm}
  \chi_{\LBottom}^\pm(\mx)  = (1 -\chi(2 r^\pm/\LBottom))
\end{equation}
be the cut-off function equal to one in the vicinity of $\mx_O^\pm$ and
vanishing in the vicinity of $\mx_O^\mp$ (the support of
  $\chi_{\LBottom}^\pm$ is localized in the neighborhood of the vertice
  $\mx_O^\pm$), and let $\chi_{\LBottom}^0 = 1 - \chi_{\LBottom}^+ -
\chi_{\LBottom}^-$. We remind that the truncation function $\chi$ is defined by~\eqref{defchi}.  For $\ell \in \{ 0,1,2\}$, we introduce
the space $V_{2,\beta}^\ell(\Omega)$
\begin{equation}\label{DefV2beta1}
V_{2,\beta}^{\ell}(\Omega) = \left\{ u \in
  H^\ell_{\text{loc}}(\Omega), \quad \|\chi_{\LBottom}^- u
  \|_{V_{2,\beta}^\ell(\mathcal{K}_{\mathbf{x}_{O}^-})} + \|\chi_{\LBottom}^+ u
  \|_{V_{2,\beta}^\ell(\mathcal{K}_{\mathbf{x}_{O}^+})} + \| \chi_{\LBottom}^0 u
  \|_{H^\ell(\Omega)} < +\infty \right\},
\end{equation}
equipped with the following norm
\begin{equation}\label{DefV2betaNorme}
\| u \|_{V_{2,\beta}^\ell(\Omega) } = \| \chi_{\LBottom}^- u
\|_{V_{2,\beta}^\ell(\mathcal{K}_{\mathbf{x}_{O}^-})} + \| \chi_{\LBottom}^+ u
\|_{V_{2,\beta}^\ell(\mathcal{K}_{\mathbf{x}_{O}^+})} + \| \chi_{\LBottom}^0 u
  \|_{H^\ell(\Omega)}\ .
\end{equation}
Here, we have used the convention $H^0(\Omega) = L^2(\Omega)$. 
Note, that the space $V_{2,\beta}^\ell(\Omega)$ is independent of the exact choice of $\chi$ and so the
truncation functions $\chi_{\LBottom}^\pm$ and that 
\begin{align}
\label{DefV2beta:inclusion}
V_{2,\beta'}^\ell(\Omega) \subset V_{2,\beta}^\ell(\Omega) \text{ for any } \beta' < \beta\ .
\end{align}
In the same way, we also define $V_{2,\beta}^\ell(\OmegaTop)$ (resp. $V_{2,\beta}^\ell(\OmegaBottom)$) as well as their
associated norm $\| \cdot \|_{V_{2,\beta}^\ell(\OmegaTop)}$ (resp. $\|\cdot\|_{V_{2,\beta}^\ell(\OmegaBottom)}$) replacing
$\Omega$ with $\OmegaTop$ (resp. $\OmegaBottom$) in the definitions~\eqref{DefV2beta1} and \eqref{DefV2betaNorme}. Finally, for $\ell \in \{ 1, 2\}$,
we introduce the space $V_{2,\beta}^{\ell- 1/2}(\Gamma)$ of the trace of the functions in $V_{2,
  \beta}^\ell(\OmegaTop)$ on the interface $\Gamma$. As norm in
$V_{2,\beta}^{\ell -1/2}(\Gamma)$ we take
\begin{equation}
\| u \|_{V_{2,\beta}^{1/2}(\Gamma)}   = \inf \left\{ \| v
\|_{V_{2,\beta}^{1}(\OmegaTop)} : v \in V_{2,\beta}^1 , v_{\left|\Gamma \right.} = u \right\}.
 \end{equation}

\noindent When studying the behavior of the far field terms close to the
reentrant corners, the set 
\begin{equation}\label{EnsembleDesCoefficientDeSingularite}
\Lambda = \left \{ \lambda_m \in \R, \; \mbox{such that} \; \lambda_m = \frac{2
    m}{3}, \; m \in \Z \setminus \{ 0\} \right\}
\end{equation}
of  {\em singular exponents} will play a crucial role (see~\cite[Chap.~1 -- 4]{Grisvard85}). 
It consists of the real numbers $\lambda$
whose square $\lambda^2$ is an eigenvalue of the
operator
\begin{equation*}
\mathcal{A} : 
\left\{
\begin{array}{r@{\;}l}
 \mathcal{D}(\mathcal{A}) = H^1_0(0, \frac{3\pi}{2}) \cap H^2(0, \frac{3\pi}{2})  \subset L^2(0,\frac{3\pi}{2}) &\rightarrow L^2(0,\frac{3\pi}{2})\ , \\
u &\mapsto \mathcal{A} u = - u''\ .
\end{array}
\right.
\end{equation*} 
Note that the associated eigenvectors are given by
\begin{equation}\label{DefintionEigenvectorwm}
w_{m}(t) = \sin( \lambda_m t), \quad m \in \Z \setminus \{ 0 \}.
\end{equation}
The following proposition, which is a standard result in the literature on elliptic problems in angular domains (cf.~\cite[Theorem~3.6 and Corollary~4.4]{Nicaise} for the proof), provides
an explicit asymptotic representation of the solution of the transmission problems
in a neighbourhood of the corners (see also~\cite[Chap.~6]{MR1469972} for a complete and detailed explanation of the
overall approach):
\begin{prop}\label{PropositionAsymptoticEspaceAPoids}
Let $\beta <0$ be a real number such that $1-\beta \notin \Lambda$. Assume that $\mathfrak{f}
\in V_{2, \beta}^0(\OmegaTop) \cap V_{2, \beta}^0(\OmegaBottom) \subset L^2(\Omega)$, %
$\mathfrak{g} \in V_{2,\beta}^{3/2}(\Gamma) \subset H^{3/2}(\Gamma)$ 
and $\mathfrak{h} \in  V_{2,\beta}^{1/2}(\Gamma) \subset H^{1/2}(\Gamma)$.
Then, the unique solution $u \in \HoneGammaD(\OmegaTop)\cap \HoneGammaD(\OmegaBottom)$ of Problem~\eqref{MacroscopicProblemModel}
admits the following decomposition:
\begin{equation}\label{asymptoticExpansionPropFF}
u = \sum_{1 \leq  q < \frac{3}{2} (1-\beta)} c_q^\pm\, 
(r^\pm)^{\lambda_q} \,
w_{q,0,\pm}(\theta^\pm) \; + \; w^\pm\ ,
\end{equation}
where $w^\pm \in V_{2,\beta}^2(\OmegaTop) \cap V_{2,\beta}^2(\OmegaBottom)$, $
w_{q,0,+}(t) = w_q(t)$ and $w_{q,0,-}(t) = w_q(t-\frac{\pi}{2})$ (where $w_q$ were given in~\eqref{DefintionEigenvectorwm}). %
Moreover, there exists a constant $C$ independent
of $u$ such that
\begin{equation}
\| w^\pm \|_{V_{2,\beta}^2(\OmegaTop) } + \| w^\pm
\|_{V_{2,\beta}^2(\OmegaBottom) } + \sum_{1 \leq  q < \frac{3(1-\beta)}{2}}
|c_q|^\pm \leq C \left( \|\mathfrak{f} \|_{V_{2,\beta}^0(\OmegaTop) } + \| \mathfrak{f}
\|_{V_{2,\beta}^0(\OmegaBottom) } \right).
\end{equation}
\end{prop}
\noindent The expansion~\eqref{asymptoticExpansionPropFF} is nothing but a modal expansion of the solution $u$ is the vicinity of
the two corners. %
Without doubt a similar expansion could be obtained using the
technique of separation of variables (see~\cite[Chap.~2]{Grisvard92}).
The sum $\sum_{1 \leq  q < \frac{3}{2} (1-\beta)} c_q^\pm\, 
(r^\pm)^{\lambda_q} \, w_{q,0,\pm}(\theta^\pm)$ is an asymptotic expansion for $r^\pm\to0$ whose remainder $w^\pm$ decays faster to zero as any term in the sum. %
Obviously, due the embedding~\eqref{DefV2beta:inclusion} asymptotic expansions of higher order in $r^\pm$
are obtained when $\beta$ is decreased (or $|\beta|$ increased).

%

\subsection{The necessary introduction of singular macroscopic terms}\label{SubsectionFirstSingularity}
\subsubsection{The limit macroscopic term and its behavior in the
  vicinity of the corners}\label{SubsubLimitMacro}
The limit macroscopic term $u_{0,0}^\delta$ satisfies
Problem~\eqref{MacroscopicProblemModel} with $\mathfrak{f}= f_{0,0} = f \in
L^2(\OmegaTop)\cap L^2(\OmegaBottom)$, $\mathfrak{g} = 0$ and
$\mathfrak{h}=0$. In view of Proposition~\ref{PropExistenceUniquenessFF}, there
exists a unique solution $u_{0,0}^\delta$ belonging to $H^1_{0,
  \Gamma_D}(\OmegaTop)\cap \HoneGammaD(\OmegaBottom)$. Indeed, $u_{0,0}^\delta$ is
independent of $\delta$ (it will be denoted by $u_{0,0}$) and belongs to
$H^1_0(\Omega)$, since its trace does not jump across $\Gamma$.\\

\noindent The existence and uniqueness of $u_{0,0}$ being granted, we can investigate its
behavior in the neighborhood of the two reentrant corners. Since we have
  assumed that $f$ is compactly supported in $\OmegaTop$, $f \in
  V_{2,\beta}^0(\OmegaTop) \cap V_{2,\beta}^0(\OmegaBottom) $ for any $\beta \in \R$. Then, in view of
  Proposition~\ref{PropositionAsymptoticEspaceAPoids}, $u_{0,0}$ has the following asymptotic
  expansion in the vicinity of the two corners vertices
  $\mathbf{x}_{O}^\pm$: for any $k \in \N$, there exists $u_{0,0,k} \in
  V^2_{2,\beta}(\OmegaTop)\cap  V^2_{2,\beta}(\OmegaBottom)$ for any $\beta >
  1 - \frac{2(k+1)}{3}$, such that 
\begin{equation}\label{Expansionu00}
u_{0,0} = \sum_{m=1}^{k} c_m^\pm \,
\left(r^\pm\right)^{\frac{2m}{3}} w_{m,0,\pm}(\theta^\pm) + u_{0,0,k},
\end{equation}
where $c_m^\pm$ are real constants continuously depending on $\|f\|_{V^0_{2,\beta}(\Omega)}$. Here again, 
the expansion~\eqref{Expansionu00}
could also be obtained using the method of separation  of
variables. 
\subsubsection{A singular problem defining $u_{0,1}^{\delta}$}

To illustrate the fact that the macroscopic terms of higher orders
cannot always be variational (\ie belonging to $H^1(\OmegaTop)\cap H^1(\OmegaBottom)$), let us
consider the problem satisfied by $u_{0,1}^\delta$,
investigating the regularity of $g_{0,1}^\delta$ and $h_{0,1}^\delta$ defined in~\eqref{Defgnq}
and \eqref{Defhnq} (we
deliberately omit the term $u_{1,0}^\delta$ for a while). In view of the
asymptotic expansion~\eqref{Expansionu00} of $u_{0,0}$,
$$ g_{0,1}^\delta  \sim c_{0,1,\pm} (r^\pm)^{-1/3} \quad  \mbox{and}\quad h_{0,1}^\delta  \sim d_{0,1,\pm} (r^\pm)^{-4/3}   $$
as $r^\pm$ tends to zeros. The constants $c_{0,1,\pm}$ and $d_{0,1,\pm}$ can be
explicitly determined (but, there is not need to write their
complete expression).  As a consequence, $g_{0,1}^\delta$ does not belong to
$H^{1/2}_{0,0}(\Gamma)$ and $h_{0,1}^\delta$ is not in $L^2(\Gamma)$. It
follows that we are not able to construct $u_{0,1}^\delta \in
H^1(\OmegaTop\cup\OmegaBottom)$.  However, we shall see that it is
possible to build a function $u_{0,1}^\delta$ that blows up as $(r^\pm)^{-1/3}$
as $r^\pm$ tends to $0$. Since this function is not in $H^1(\OmegaTop \cup
\OmegaBottom)$, we say that this function is {\em singular}. %
To distinguish from singular functions, we denote functions in $H^1(\OmegaTop \cup
\OmegaBottom)$ as {\em regular} (so not meaning $C^\infty$-regular functions).

\begin{rem} The previous analysis explains why, contrary to the case of
  an infinite thin periodic layer (see \cite{poirier2006impedance}, \cite{ArticleXavierDelourme}), it is not possible to construct
an asymptotic expansion of the form
\begin{equation*}
u^\delta(x_1,x_2) = \sum_{n\in \N} \delta^n \left( u_{n}(\mx) +
\Pi_n(x_1,\mX)  \right),
\end{equation*}
where $u_n \in H^1_{0,
    \Gamma_D}(\OmegaTop)\cap \HoneGammaD(\OmegaBottom)$ and $\Pi_n$ are periodic functions with
respect to $X_1$ exponentially decaying as $X_2$ tends to $\pm \infty$. 

\end{rem}
\begin{rem}
Since it is not possible to construct regular macroscopic terms, we
shall construct singular ones.  Nevertheless, the exact solution $u^\delta$ is
not singular. As a consequence the far field expansion~\eqref{FFExpansion}, which contains
singular terms, can not be valid in the immediate surrounding of the two corners.
Here, a near field expansion \eqref{NFExpansion} has to be introduced,
which replace the singular solution behavior towards the corners in their immediate neighborhood.
\end{rem}

\subsection{Two families of macroscopic singularities  $s_{-m,q}^\pm$}\label{MacroscopicSingularities}
\noindent In this section, we introduce two families of functions, that are $s_{-m,q}^+$ and $s_{-m,q}^-$ for the right and left corner,
that will facilitate the definition of the macroscopic terms. %
The functions are defined recursively in $q$ for each $m \in \N \setminus \{0\}$.
The following subsection is dedicated to the definition of $s_{-m,0}^\pm$, where the functions $s_{-m,q}^\pm$, $q \in \N \setminus \{0\}$ are defined by induction afterwards.

\subsubsection{Harmonic singularities $s_{-m,0}^\pm$ ($m \in \N \setminus \{0\}$)}
For any positive integer $n$, the terms $u_{n,0}^\delta$ are harmonic in $\Omega$. It does not imply
that they vanish because we allow for singular behaviors in the
vicinity of the two corners. The present subsection is dedicated to the
definition of a set of harmonic functions that admit singularities in the
vicinity of the two corners $\mx_{0}^\pm$. The forthcoming
analysis is done for the right corner $\mx_O^+$ but a strictly similar
approach may be carried out for the left corner. To start with, we 
exhibit a sequence of harmonic functions $s_{-m,0}^+$ that behave
like $r^{-2m/3}$ in the vicinity of $\mx_{0}^+$, and which are regular in the vicinity of $\mx_{0}^-$. 

\begin{prop}\label{PropExistenceUniquenesmMF}
Let $m \in \N\setminus\{ 0 \}$. There exists a unique harmonic
function $s_{-m,0}^+$  vanishing on $\Gamma_D$ of the form
\begin{equation}\label{Defsmplus}
s_{-m,0}^+ = (r^+)^{-\frac{2 m}{3}} w_{-m,0,+}(\theta^+) \chi_{\LBottom}^+ +
\widetilde{s}_{-m,0}^+,
\end{equation}
where $\widetilde{s}_{-m,0}^+$  belongs to $H^1_{0}(\Omega)$.
\end{prop} 
\begin{proof}
Remarking that $\Delta \widetilde{s}_{-m,0}^\pm$
belongs to $L^2(\Omega)$,  the proof of Proposition~\ref{PropExistenceUniquenesmMF} directly follows from
the Lax-Milgram lemma.
\end{proof}
 \noindent It is worth noting that $s_{-m,0}^+$ does not depend
on the cut off function $\chi_{\LBottom}^+$. Besides, it is easily verified that $s_{-m,0}^+$ belongs to
$V_{2,\beta}^2(\OmegaTop)\cap V_{2,\beta}^2(\OmegaBottom)$ for any $\beta > 1 +  \frac{2 m}{3}$. For instance,
it belongs to $V_{2, \frac{2m}{3} +\frac{7}{6}}^2(\OmegaTop) \cap V_{2,
  \frac{2m}{3} +\frac{7}{6}}^2(\OmegaBottom)$. Naturally, for $m \in \N \setminus \{0 \}$, we can also prove the existence of
a set of functions $s_{-m,0}^-$ of the form
\begin{equation}\label{Defsmmoins}
s_{-m,0}^- = (r^-)^{-\frac{2 m}{3}} w_{-m,0,-}(\theta^-) \chi_{\LBottom}^- +
\widetilde{s}_{-m,0}^-, \quad  \widetilde{s}_{-m,0}^- \in H^1_{0}(\Omega).
\end{equation}  

\noindent As for $u_{0,0}$, we shall write an explicit asymptotic expansion of
${s}_{-m,0}^+$ in the vicinity of the two corners. Applying
Proposition~\eqref{PropositionAsymptoticEspaceAPoids} to the function
$\widetilde{s}_{-m,0}^+$ (noting that $\Delta \widetilde{s}_{-m,0}^+$
vanishes  for $r^\pm <\LBottom/2$), we can prove
the following 
\begin{prop}\label{PropExistenceUniquenesmMF2}
Let $m \in \N\setminus \{0\}$ and $k \in \N$. Then, there exist a function
$r_{-m,k,+}^+$ belonging to  $V_{2,\beta}^2(\OmegaTop)\cap V_{2,\beta}^2(\OmegaBottom)$ for any
$\beta > 1-\frac{2 (k+1)}{3}$, and $k$ real
coefficients $\ell_{q}^+(s_{-m,0}^+)$, $1\leq q\leq k$, such that
\begin{equation}\label{Asymptoticsm0CoinPositif}
s_{-m,0}^+ = (r^+)^{-\frac{2m}{3}} w_{-m,0,+}(\theta^+)  +
\sum_{q=1}^k \ell_{q}^+(s_{-m,0}^+) (r^+)^{\lambda_q}  w_{q,0,+}(\theta^+) + r_{-m,k,+}^+.
\end{equation}
Analogously,
there exist a function $r_{-m,k,-}^+$ belonging to $V_{2,\beta}^2(\OmegaTop)\cap V_{2,\beta}^2(\OmegaBottom)$
for any $\beta > 1 -\frac{2 (k+1)}{3}$ and $k$ real
coefficients $\ell_{q}^-(s_{-m,0}^+)$, $1\leq q\leq k$, such that
\begin{equation}\label{Asymptoticsm0CoinNegatif}
s_{-m,0}^+ =
\sum_{q=1}^k \ell_{q}^-(s_{-m,0}^+) (r^-)^{\lambda_q}  w_{q,0,-}(\theta^-) + r_{-m,k,-}^+.
\end{equation}
\noindent Moreover, for any $\beta > 1 -\frac{2 (k+1)}{3}$, there exists a constant $C$ such that
\begin{multline}\label{Estimationsm0}
\sum_{q=1}^k \left(  | \ell_{q}^+(s_{-m,0}^+) | +   |\ell_{q}^-(s_{-m,0}^+) | \right) 
  + \|r_{-m,k,+}^+ \|_{V_{2,\beta}^{2}(\OmegaTop)} %
  + \|r_{-m,k,+}^+ \|_{V_{2,\beta}^2(\OmegaBottom)} \\
  + \|r_{-m,k,-}^+ \|_{V_{2,\beta}^2(\OmegaTop)}  %
  + \|r_{-m,k,-}^+ \|_{V_{2,\beta}^2(\OmegaBottom)}  
  \; \leq \; C \,\left(  \| s_{-m,0}^+\|_{V_{2,\frac{2m}{3}
      +\frac{7}{6}}^2(\OmegaTop)} +  \|
  s_{-m,0}^+\|_{V_{2,\frac{2m}{3} +\frac{7}{6}}^2(\OmegaBottom)} \right).
\end{multline}
\end{prop}
\noindent The formulas~\eqref{Asymptoticsm0CoinPositif},\eqref{Asymptoticsm0CoinNegatif}
provide asymptotic expansions of $s_{m,0}^+$ in the neighborhood of
$\mx_{O}^\pm$. Again, despite their apparent complexity, they are essentially modal expansions of
$\widetilde{s}_{-m,0}^+$ that can be also obtained using the separation of
variables. 
Note that the remainder $r_{-m,k,+}^+$ is orthogonal to the functions $w_{q,0,+}$, for $q\leq k$:
  $$\int_{I^+} r_{-m,k,+}^+(r^+, \theta^+) w_{q,0,+}(\theta^+)
    d\theta^+ = 0 \quad \forall q \leq k\ , $$
if $r^\pm$ is small enough (\ie, where $\chi^\pm_L = 1$). 
In this case, the coefficients $\ell_q^\pm(s_{-m,0}^+)$
can  be computed as
%
\begin{equation}\label{eq:lq:sm0}
\ell_q^\pm(s_{-m,0}^+) =  (r^\pm)^{-\frac{2q}{3}}  \int_{I^\pm}
\widetilde{s}_{-m,0}^+(r^\pm, \theta^\pm) w_{q,0,\pm}(\theta^\pm) d\theta^\pm.
\end{equation}
\begin{rem}\label{remarqueUnicitesm0} %
It is known~\cite[Chap.~6]{MR1469972} that any function $v \in
  V_{2,\beta}^2(\OmegaTop)\cap V_{2,\beta}^2(\OmegaBottom)$ for $\beta >
  1 +  \frac{2 m}{3}$ satisfying $\Delta v = 0$ in $\Omega$ is a linear
  combination of the functions $s_{-k,0}^\pm$, $1 \leq k \leq m$.
\end{rem}
\subsubsection{The families $s_{m,q}^\pm$, $m \in \N \setminus \{0\} $, $q
\in \N \setminus \{0\} $}
In order to construct the macroscopic terms, it is useful to introduce the
family of functions $s^+_{-m,q}$, $(m,q) \in  (\N^\ast)^2$ (remember that $\N^\ast = \IN \setminus \{0\}$), corresponding to the 'propagation' of $s^+_{-m,0}$ 
(recursively) through the
transmissions conditions~\eqref{Defgnq},\eqref{Defhnq}:

\begin{prop}\label{PropDefinitionsmq}
For any $(m,q) \in  (\N^\ast)^2$ there exists a unique function $s^+_{-m,q} \in V^{2}_{2, \beta}(\OmegaTop)\cap V^{2}_{2, \beta}(\OmegaBottom)$ 
for $\beta> 1+\frac{2m}{3}+q$
 satisfying
\begin{equation}\label{Problemsmq}
\left \lbrace
\begin{aligned}
\dsp {- \Delta s^+_{-m,q} } &\; =  & & 0 \quad \text{in } \OmegaTop
\cap \OmegaBottom,\\
s^+_{-m,q} & \;  = &&  0 \quad \mbox{on} \, \Gamma_D\ ,\\
\dsp [s^+_{-m,q} (x_1,0)]_{\Gamma} &\; = && 
\sum_{p=1}^q \mathcal{D}_p^{\mathfrak{t}} \, \partial_{x_1}^p \langle
s^+_{-m,q-p}(x_1,0)
\rangle_{\Gamma} \hspace{1.8em} +\sum_{p=1}^{q} 
\mathcal{D}_p^{\mathfrak{n}} \,  \partial_{x_1}^{p-1} \langle \partial_{x_2}
s^+_{-m,q-p} (x_1,0) \rangle_{\Gamma}\ , \\
\dsp [\partial_{x_2}s^+_{-m,q} (x_1,0)]_{\Gamma} & \;= & &
\sum_{p=1}^{q} \mathcal{N}_{p+1}^{\mathfrak{t}} \, \partial_{x_1}^{p+1} \langle
s^+_{-m,q-p} (x_1,0)
\rangle_{\Gamma}+\sum_{p=1}^{q} 
\mathcal{N}_{p+1}^{\mathfrak{n}} \,  \partial_{x_1}^{p} \langle \partial_{x_2}
s^+_{-m,q-p} (x_1,0) \rangle_{\Gamma},
\end{aligned}
\right.
\end{equation}
which admits the following decompositions 
\begin{itemize} 
\item For any $k\in \N$, there exists a function  $r_{-m,q,k,+}^+$
belonging to $V_{2,\beta'}^2(\OmegaBottom) \cap V_{2,\beta'}^2(\OmegaTop)$
for any $\beta' > 1 -\frac{2(k+1)}{3}$
and real constants $\ell_n^+(s_{-m,q-p}^+)$, $0 \leq p \leq q$, $1 \leq n < (k+1) + \frac32p$ such that 
\begin{multline}\label{Asymptoticsmq}
s_{-m,q}^+ = (r^+)^{-\frac{2m}{3}-q}w_{-m,q,+}(\theta^+,\ln r^+) \\
+ \sum_{p=0}^{q}\; \sum_{1\leq n < (k+1) + \frac{3}{2} p} \ell_n^+(s_{-m,q-p}^+) (r^+)^{\lambda_{n}-p} w_{n,p,+}(\theta^+,\ln r^+) \; + \; r_{-m,q,k,+}^+\ ,
%
\end{multline}
where  
$w_{n,0,+}(\theta^+,\ln r^+) = w_{n,0,+}(\theta^+)$ are given in Proposition~\ref{PropositionAsymptoticEspaceAPoids}
and, for $p\geq 1$, $w_{n,p,+}(\theta^+,
\ln r^+)$ are polynomials in
$\ln r^+$ whose coefficients (functions of $\theta^+$) belong to $\mathcal{C}^\infty([0, \pi])
\cap \mathcal{C}^\infty( [\pi, \frac{3\pi}{2}])$ (here $[a, b]$ denotes the
closure of the intervall $(a,b)$).
%
\noindent \item For any $k\in \N$, there exists a function  $r_{-m,q,k,-}^+$
belonging to $V_{2,\beta'}^2(\OmegaTop)\cap V_{2,\beta'}^2(\OmegaBottom)$
for any $\beta' > 1 -\frac{2(k+1)}{3}$ 
and real constants $\ell_n^-(s_{-m,q-p}^+)$, $0 \leq p \leq q$, $1 \leq n < (k+1) + \frac32p$
such that 
\begin{equation}\label{AsymptoticsmqLeftCorner}
s_{-m,q}^+ =  \sum_{p=0}^{q}
\, \sum_{1 \leq n < (k+1) + \frac{3}{2} p} \ell_n^-(s_{-m,q-p}^+) (r^-)^{\lambda_{n}-p} w_{n,p,-}(\theta^-,
\ln r^-) \; + \; r_{-m,q,k,-}^+,
\end{equation}
where 
$w_{n,0,-}(\theta^-,\ln r^-) = w_{n,0,-}(\theta^-)$ (see Prop.~\ref{PropositionAsymptoticEspaceAPoids})
and $w_{n,p,-}(\theta^-,
\ln r^-)$ are  polynomials in 
$\ln r^-$ whose coefficients (functions of $\theta^-$) belong to $\mathcal{C}^\infty([0, \frac{\pi}{2}])
\cap \mathcal{C}^\infty( [\frac{\pi}{2}, \frac{3\pi}{2}])$.
\end{itemize}
\end{prop}  

\noindent The proof of Proposition~\ref{PropDefinitionsmq} is in Appendix~\ref{AppendixPreuvePropositionsmq}.  It
is strongly based on the explicit resolution of
the Laplace equation in so-called infinite conical domains for particular right-hand sides of the form $r^\lambda (\ln r)^n$,
$\lambda \in \R$, $n \in \N$ (see Section 6.4.2 in \cite{MR1469972} for similar
results). The proof consists of constructing an explicit lift of
the singular part of the jump values \eqref{Asymptoticsmq} in order to reduce the problem to a
variational one (as already done for $s_{-m,0}^+$).
\begin{rem}\label{remarqueDefinitionskqmoins}
 \noindent In the same way, for each $m \in \N^\ast$ we can define by induction a sequence of functions
$(s_{-m,q}^-)_{q\in \N^\ast}$ as follows: $s_{-m,q}^-$ is the unique function belonging to $V^{2}_{2, \beta}(\OmegaTop)\cap V^{2}_{2, \beta}(\OmegaBottom)$ for any
$\beta> 1+\frac{2m}{3} +q$ that satisfies
 the  transmission problem obtained from~\eqref{Problemsmq} by substituting $s_{-m,q-p}^+$ for
 $s_{-m,q-p}^-$ in the jump conditions, and  the asymptotic
 expansions
 obtained interchanging \eqref{Asymptoticsmq} and \eqref{AsymptoticsmqLeftCorner}, replacing the
 superscripts plus by superscripts minus.
\end{rem}
\subsubsection{Annotations to the singular functions}

Let us comment the results of the previous proposition and of Proposition~\ref{PropExistenceUniquenesmMF2}:
\begin{itemize}
\item[--] For $m>0$ fixed, the family $(s_{-m,q}^+)_{q\in\N}$ provides 
  particular singular solutions to~\eqref{MacroscopicProblem1}. 
%
\item[--] The exponents $\lambda$ of $r^-$ and $r^+$ appearing in the asymptotic expansions~\eqref{Asymptoticsmq},\eqref{AsymptoticsmqLeftCorner}
are singular exponents $\lambda \in \Lambda$ as well as 'shifted' singular
exponents of the form $\lambda = \lambda_n - p$, $\lambda_n \in \Lambda$, where the integer $p$ is between
$1$ and~$q$. The most singular part of $s_{-m,q}^+$ in the vicinity of
$\mx_{O}^+$ is 
$
(r^{+})^{-\frac{2}{3}m -q} w_{-m,q,+}(\theta^+, \ln r^+),
$ 
while the most singular part of $s_{-m,q}^+$ in the vicinity of
$\mx_{O}^-$ is 
$
\ell_{1}^-(s_{-m,0}^+) (r^{-})^{\frac{2}{3} -q} w_{1,q,-}(\theta^-, \ln r^-).
$
Consequently $s_{-m,q}^+$ is 'less singular' in the vicinity of the
left corner than in the vicinity of the right one.
\item[--] The function $s_{-m,q}^+$ depends only on the
functions $s_{-m, p}^+$ for $p\leq q$. In others words, providing that
$n\neq m$, the definition
of the families $\{ s_{-m,q}^+, q \in \N\}$,  $\{ s_{-n,q}^+, q \in \N\}$ are 
independently defined.
\item[--] The functions $w_{n,q,\pm}$ are defined
in~\eqref{Definitionwnpplus},\eqref{Definitionwnpmoins}. 
However, for the forthcoming derivation of the asymptotic expansion and its analysis 
their explicit expression is not important. 
Even so these functions will appear again in the definition of the near field singularities
(see Lemma~\ref{LemmeSuperDecroissance}).
\item[--] Problem~\eqref{Problemsmq} alone does not uniquely determine the
function $s_{-m,q}^+$. Indeed, in view of Remark~\ref{remarqueUnicitesm0} the solution of~\eqref{Problemsmq} is defined only up to a linear combination of $\left\{ s_{-n,0}^{\pm},  n\leq m +
  \frac{3}{2}q\right\}$.
However, imposing additionally 
the singular behavior close to the corners given by
\eqref{Asymptoticsmq} and \eqref{AsymptoticsmqLeftCorner} (using the fact
that the functions $w_{n,q,\pm}$ are uniquely determined) restores the
uniqueness (\cf Remark~\ref{remarqueUnicitesm0}).
\item[--] For a given $k \in \N$, the constants
$\ell_{n}^\pm(s_{-m,q}^+)$, $1 \leq q \leq k$, and the remainders
$r_{-m,q,k,\pm}^+$ satisfy an estimate of the
form~\eqref{Estimationsm0} that has been omitted for the sake of
concision.
\item[--] The constants $\ell_{n}^\pm(s_{-m,q}^+)$ are intrinsic to the singularity functions and can be obtained by similar kind of formulas as~\eqref{eq:lq:sm0}. %
Each constant $\ell_{n}^\pm(s_{-m,q}^+)$ appears in the decomposition of several singularity functions in~\eqref{Asymptoticsmq} and~\eqref{AsymptoticsmqLeftCorner}.
\end{itemize}

\subsection{An explicit expression for the macroscopic terms}\label{SubsectionExpliciteMacro}
These part is dedicated to the derivation of a quasi-explicit formula
of the macroscopic terms $u_{n,q}^\delta$ by introducing particular solutions
to Problem~\eqref{MacroscopicProblem1}. As mentioned before, we shall
allow $u_{n,q}^\delta$ to be singular. In view of the previous construction, we shall impose that
\begin{equation*}
\begin{aligned}
& u_{n,q}^\delta \in {V}_{2, \beta}^2(\OmegaTop)\cap {V}_{2, \beta}^2 (\OmegaBottom) \quad
\mbox{for any} \, \beta > 1 + \frac{2 n}{3} + q, \quad n > 0\ ,\\
& u_{0,q}^\delta \in {V}_{2, \beta}^2(\OmegaTop)\cap {V}_{2, \beta}^2 (\OmegaBottom) \quad
\mbox{for any} \, \beta > \frac{1}{3} + q\ .
\end{aligned}
\end{equation*}

\subsubsection{The macroscopic terms $u_{0,q}^\delta$, $q \in \N $}
We remind that the limit macroscopic field 
$u_{0,0}$ (remember that $u_{0,0}^\delta = u_{0,0}\in H^1_{0, \Gamma_D}(\Omega)$) satisfies
Problem~\eqref{MacroscopicProblemModel} with $\mathfrak{f}=
f \in L^2(\OmegaTop)\cap L^2(\OmegaBottom)$, $\mathfrak{g} = 0$ and
$\mathfrak{h}=0$ (see Section~\ref{SubsubLimitMacro}). 
In this subsection we define the functions $u_{0,q}^\delta = u_{0,q}$ 
in the large class of possible singular solutions of~\eqref{MacroscopicProblem1}
by imitating the iterative procedure of the previous subsection for the
definition of the singular functions $s_{-m,q}^+$
(\ie, by 'propagating' $u_{0,0}$ (recursively) through the transmission
conditions~\eqref{Defgnq},\eqref{Defhnq}), by which in turn no additional singular functions are added.

\begin{prop}\label{PropDefinitions0q}
For any $q \in \IN^\ast$ there exists a unique function $u_{0,q}^\delta = u_{0,q} \in V^{2}_{2, \beta}(\OmegaTop)\cap V^{2}_{2, \beta}(\OmegaBottom)$ for 
$\beta> \frac{1}{3}+q$ of~\eqref{MacroscopicProblem1}
which admits the following decompositions: 
\begin{itemize}
\item For any $k\in \N$, there exists a function  $r_{0,q,k,+}^+$
belonging to $V_{2,\beta'}^2(\OmegaBottom) \cap V_{2,\beta'}^2(\OmegaTop)$
for $\beta' > 1 -\frac{2(k+1)}{3}$
and real constants $\ell_n^+(u_{0,q-p})$, $1 \leq p \leq q$, $1 \leq n < (k+1) + \frac32p$
such that
\begin{equation}\label{Asymptotics0q}
u_{0,q} = \sum_{p=0}^q \;
\sum_{1 \leq n < (k+1) + \frac{3}{2} p} \ell_n^+(u_{0,q-p}) (r^+)^{\lambda_{n}-p} w_{n,p,+}(\theta^+,
\ln r^+) \; + \; r_{0,q,k,+}^+(r^+, \theta^+)\ .
\end{equation}
\item Analogously, for any $k\in \N$, there exists a function  $r_{0,q,k,-}^+$
belonging to $V_{2,\beta'}^2(\OmegaTop)\cap V_{2,\beta'}^2(\OmegaBottom)$
for any $\beta' > 1 -\frac{2(k+1)}{3}$
and real constants $\ell_n^-(u_{0,q-p})$, $1 \leq p \leq q$, $1 \leq n < (k+1) + \frac32p$
such that
\begin{equation}\label{Asymptotics0qLeftCorner}
u_{0,q} =  \sum_{p=0}^{q}
\, \sum_{1 \leq n \leq (k+1) + \frac{3}{2} p} \ell_n^-(u_{0,q-p}) (r^-)^{\lambda_{n}-p} w_{n,p,-}(\theta^-,
\ln r^-) \; + \; r_{0,q,k,-}^+(r^-, \theta^-)\ .
\end{equation}
\end{itemize}
\end{prop}
\noindent Note that the functions $w_{n,p,\pm}$ in~\eqref{Asymptotics0q} and \eqref{Asymptotics0qLeftCorner}
were used already in Proposition~\ref{PropDefinitionsmq} and are defined in~\eqref{Definitionwnpplus},\eqref{Definitionwnpmoins}. 
Similar to the singular functions the constants $\ell_n^\pm(u_{0,p})$ are intrinsic and fixed, when $u_{0,0}$ is fixed. %
From now on we consider the macroscopic terms $u^\delta_{0,q} = u_{0,q}$ to be defined by Proposition~\ref{PropDefinitions0q}.

\subsubsection{The macroscopic terms $u_{n,q}^\delta$, $n \in \IN \setminus \{0\}$, $q \in \N$}

We construct $u_{n,q}^\delta$ as follows: 
\begin{prop}
  \label{PropExplicitMacro}
Let $n >0$. For any $p \in
\N$, let $\ell_{-k}^\pm (u_{n,p}^\delta)$, $1\leq k \leq n$  be $2n$ 
given real constants.  Then, the family of functions
\begin{equation}\label{Definitionunq}
u_{n,q}^\delta =  \sum_{\pm} \sum_{p=0}^q \sum_{k=1}^{n}  \ell_{-k}^\pm
(u_{n,p}^\delta) s_{-k, q-p}^\pm, \quad  q \in \N,
\end{equation}
satisfies the family of macroscopic
problems~\eqref{MacroscopicProblem1}. Moreover, the term $u_{n,q}^\delta$
belongs to ${V}_{2, \beta}^2(\OmegaTop)\cap {V}_{2,
  \beta}^2 (\OmegaBottom)$ for any $\beta > 1 + \frac{2n}{3} +q$.
\end{prop}
\noindent We remind that the functions $s_{-m,q}^+$ are defined in Proposition~\ref{PropExistenceUniquenesmMF} for $q = 0$ and
Proposition~\ref{PropDefinitionsmq} for $q > 0$ and $s_{-m,q}^-$ are defined in
\eqref{Defsmmoins} ($q = 0$) and Remark~\ref{remarqueDefinitionskqmoins} ($q > 0$).

\begin{proof}
The 'most singular' part of $u_{n,q}^\delta$ $n > 0$, defined by~\eqref{Definitionunq}
corresponds to $\sum_{\pm}\ell_{-n}^\pm(u_{n,0}^\delta) s_{-n,q}^\pm$, which, 
in view of Proposition~\ref{PropDefinitionsmq} belongs
to ${V}_{2, \beta}^2(\OmegaTop)\cap {V}_{2, \beta}^2 (\OmegaBottom)$ for any $\beta > {1} + \frac{2n}{3}  + q$. %
As a consequence $u_{n,q}^\delta$ belongs to ${V}_{2, \beta}^2(\OmegaTop)\cap {V}_{2,\beta}^2 (\OmegaBottom)$ for any $\beta > {1} + \frac{2n}{3}  + q$. %
Next, let us show by induction on $q$ that the family $(u_{n,q}^\delta)_{q \in \N}$ is a particular solution
to the family of problems~\eqref{MacroscopicProblem1}. The base step ($q = 0$) is trivial. %
For the induction step, it is clear that $u_{n,q}^\delta$  is harmonic in $\OmegaTop$ and in
$\OmegaBottom$ and fulfills homogeneous Dirichlet boundary conditions on~$\Gamma_D$. %
It remains to show the jump conditions across $\Gamma$. Substituting the definition~\eqref{Definitionunq} into~\eqref{Defgnq} we can assert that
 \begin{eqnarray*}
   [u_{n,q}^\delta]_{\Gamma} & = &  \sum_{\pm} \sum_{p=0}^q \sum_{k=1}^{n}  \ell_{-k}^\pm
(u_{n,p}^\delta) \sum_{r=1}^{q-p} \left( \mathcal{D}_r^{\mathfrak{t}} \, \partial_{x_1}^r \langle
 s_{-k,q-p-r} \rangle_{\Gamma}+ \mathcal{D}_r^{\mathfrak{n}}
 \,  \partial_{x_1}^{r-1} \langle \partial_{x_2} s_{-k,q-p-r} \rangle_{\Gamma} \right)\ .
 \end{eqnarray*}
 Interchanging the sum over $r$ and $p$, using the induction
 hypothesis, we get the expected jump: 
\begin{equation*}
  [u_{n,q}^\delta]_{\Gamma}  = \sum_{r=1}^{q} \left(
    \mathcal{D}_r^{\mathfrak{t}} \, \partial_{x_1}^r +
    \mathcal{D}_r^{\mathfrak{n}} \, \partial_{x_1}^{r-1} \right)  \sum_{\pm} \sum_{p=0}^{q-r} \sum_{k=1}^{n} \ell_{-k}^\pm
(u_{n,p}^\delta) \langle
 s_{-k,q-p-r} \rangle_{\Gamma} 
 = \sum_{r=1}^{q} \left(
    \mathcal{D}_r^{\mathfrak{t}} \, \partial_{x_1}^r +
    \mathcal{D}_r^{\mathfrak{n}} \, \partial_{x_1}^{r-1} \right) \langle u_{n,q-r}^\delta \rangle_\Gamma \ .
 \end{equation*}
The condition for the normal jump follows accordingly, and the proof is complete.
\end{proof}
\noindent Let $n>0$ and $q \in \N$ be fixed. The function
$u_{n,q}^\delta$ (defined by~\eqref{Definitionunq}) is
determined up to the specification of the  $2 n$ constants
$\ell_{-k}^\pm (u_{n,q}^\delta)$, $1\leq k \leq n$ (There are $2 n$
degrees of freedom). The matching procedure will provide a way to choose these
constants in order to ensure the matching of far and near field
expansions in the matching areas.
\subsubsection{Expression for the boundary layers correctors}
\noindent Assume now that $u_{n,q}^\delta$ is defined
by~\eqref{Definitionunq}.Then, inserting this definition into
the formula~\eqref{definitionPiq} defining the boundary layer correctors~$\Pi_{n,q}^\delta$,
we find them to be given by
\begin{equation*}
\Pi_{0,q}^\delta = \Pi_{0,q} =\sum_{p=0}^q \, W_p^\mathfrak{t} \, \partial_{x_1}^p  \langle
  u_{0,q-p}(x_1, 0) \rangle_\Gamma \;+\;  \sum_{p=1}^q W_p^{\mathfrak{n}}\,
\partial_{x_1}^{p-1}  \langle
  \partial_{x_2} \, u_{0,q-p}(x_1, 0) \rangle_\Gamma\ ,
\end{equation*}
and, for $n>0$,
\begin{multline*}
\Pi_{n,q}^\delta =  \sum_{\pm} \sum_{k=1}^n \left(\sum_{p=0}^q  W_p^{\mathfrak{t}} \sum_{r=0}^{q-p}
  \ell_{-k}^\pm(u_{n,r}^\delta) \partial_{x_1}^p \langle
  s_{-k,q-p-r}^\pm \rangle_\Gamma  
 + \sum_{p=1}^q  W_p^{\mathfrak{n}} \sum_{r=0}^{q-p}
  \ell_{-k}^\pm(u_{n,r}^\delta) \partial_{x_1}^{p-1} \langle
\partial_{x_2}  s_{-k,q-p-r}^\pm \rangle_\Gamma \right)\ .
\end{multline*}
\subsubsection{Asymptotic of the far field terms close to the corners}
Thanks to the previous formulas, we have a complete asymptotic
expansion describing the behavior of  both macroscopic and boundary layer correctors terms in
the vicinity of the reentrant corners: 
for any $k \in \N$ there exists
a function $u_{n,q,k,+}$ belonging to $V_{2,\beta}^2(\OmegaTop)\cap V_{2,\beta}^2(\OmegaBottom)$
for any $\beta > 1 -\frac{2 (k+1)}{3}$ such that
\begin{equation}\label{AsymptoticMacroFieldMacthingareas}
  u_{n,q}^{\delta}  =  \sum_{r=0}^q  \sum_{m= -n}^{k + \frac{3}{2}r }
  a_{n,q-r,m,+}^\delta \, \left(
    r^+\right)^{\frac{2m}{3}-r} w_{m,r,+}(\theta^{+}, \ln r^{+}) + u_{n,q,k,+}\ ,
\end{equation}
where, for any $(n,j,m) \in \N^2 \times \Z$, 
\begin{equation}\label{eq:anjm}
a_{0,j,m,+}^\delta= \ell_{m}^+(u_{0, j}),\quad
a_{n,j,m,+}^\delta= \sum_{\pm} \sum_{k=\max(1,-m)}^n
\sum_{p=0}^{j}\ell_{-k}^\pm(u_{n,p}^\delta)\ell_{m}^+(s_{-k, j-p}^\pm),
\quad n >0.
\end{equation}
Here, we have used the convention that $w_{0,r,\pm}=0$ for any $r \in
\N$, that $\ell_{m}^\pm(s_{-k,p}^+) =0$ for any $m<-k$ and $p\in
\N$. Moreover, the notation $\sum_{m= -n}^{k + \frac{3}{2}r }$ denotes the sum over the integers $m \in \Z$ such that $-n
  \leq m \leq k + \frac{3}{2}r $ (the  integer index $m$ does not exceed
$\lfloor k + \frac{3}{2} r \rfloor$, where $\lfloor a \rfloor$ and $\lceil a \rceil$ denote
the largest integer not greater or the smallest integer not less than $a$, respectively).\\

\noindent For $m<0$, the  expression of $a_{n,j,m,+}$ can be simplified
since $\ell_{m}^+(s_{-k, q-p-r}^-)$ vanishes, and $\ell_{m}^+(s_{-k,
  q-p-r}^{+}) =0$  unless $k=-m$ and $q-p-r=0$ (in the latter case,
$\ell_{m}^+(s_{m,0}^+)= 1$):
\begin{equation}\label{ValeurRemarquableanjm}
a_{n,j,m,+}^\delta=  \ell_{m}^+(u_{n,j}^\delta).
\end{equation}
One can also give an asymptotic expansion for $\Pi_{n,q}^\delta$ for
$(\mathbf{x}_O^+)_1-x_1$ sufficiently small (we remind that $\mathbf{x}_O^\pm = ((\mathbf{x}_O^\pm)_1,
(\mathbf{x}_O^\pm)_2)$ denotes the coordinates of the vertice $\mx_O^\pm$):
\begin{equation}\label{AsymptoticBoundaryLayerMacthingareas}
\Pi_{n,q}^\delta = \sum_{r=0}^q \sum_{m=- n}^{k+\frac{3}{2}r}
\left((\mathbf{x}_O^+)_1-x_1\right)^{\frac{2 m}{3} -r} \;
a_{n,q-r,m,+}^\delta  \; p_{m,r,+}\left(\ln ((\mathbf{x}_O^+)_1-x_1), \frac{x_1}{\delta}, \frac{x_2}{\delta}\right) + \Pi_{n,q,k,+},
\end{equation}
where,
\begin{multline} \label{definitionPmr}
p_{m,r,+}(\ln t, \frac{x_1}{\delta}, \frac{x_2}{\delta}) =
\sum_{p=0}^r \, \;
g_{m,r-p,p,+}^{\mathfrak{t}}(\ln t)\;
  W_p^{\mathfrak{t}} \left(\frac{x_1}{\delta}, \frac{x_2}{\delta} \right)
  +\sum_{p=1}^r \;
  g_{m,r-p,p,+}^{\mathfrak{n}}(\ln t) \;
  W_p^{\mathfrak{n}}\left(\frac{x_1}{\delta}, \frac{x_2}{\delta}
  \right) .
\end{multline}
The functions $g_{m,r,q,+}^{\mathfrak{n}}$ and
$g_{m,r,q,+}^{\mathfrak{t}}$ are polynomials in $\ln
t$. Their definitions are given in~(\ref{gnrqplust}),(\ref{gnrqplusn}). The remainder $\Pi_{n,q,k,+}$ can be
  written as 
\begin{equation}
\Pi_{n,q,k,+}(x_1, \mX) = \sum_{p=0}^q \langle w_{n,q,p}^{\mathfrak{t}}(x_1 ,0) \rangle
W_{p}^{\mathfrak{t}}(\mX) +  \sum_{p=0}^q \langle
w_{n,q,p}^{\mathfrak{n}}(x_1 ,0) \rangle W_{p}^{\mathfrak{n}}(\mX),
\end{equation}
where one can verify (using a weighted elliptic regularity argument,
see \cite[Corollary 6.3.3]{MR1469972}) that the functions $w_{n,q,p}^{\mathfrak{t}}$ and $w_{n,q,p}^{\mathfrak{n}}$ belong to $V_{2,\beta}^2(\OmegaTop)\cap V_{2,\beta}^2(\OmegaBottom)$
for any $\beta > 1 -\frac{2 (k+1)}{3}$.\\

\noindent Naturally, similar asymptotic expansions occur in the vicinity of the left
corner.


\section{Analysis of the near field equations and near field singularities} \label{SectionNF}
The near field terms $U_{n,q,\pm}^\delta$ satisfy Laplace problems
(see~\eqref{NearFieldEquation}) posed in
the unbounded domain $\widehat{\Omega}^\pm$ (defined
in~\eqref{definitionOmegaHatpm})
of the form
\begin{equation}\label{ProblemeModeleChampProche}
 \left\lbrace\quad
\begin{aligned}
-\Delta u & = &  f  &\quad \mbox{in} \; \widehat{\Omega}^\pm,\\
u & = &  0  &\quad \mbox{on} \; \partial \mathcal{K}^\pm,\\
\partial_n u &  = &  g  &\quad \mbox{on} \;  \widehat{\Gamma}^\pm_{\hole}=\partial \widehat{\Omega}^\pm \setminus \partial \mathcal{K}^\pm.
\end{aligned}
\right. 
\end{equation}
In this section, we first present a functional framework to solve the model
problem~\eqref{ProblemeModeleChampProche} (Subsection~\ref{SubsectionNearFieldGeneral}). We pay particular attention to the
asymptotic behavior of the solutions at infinity
(Proposition~\ref{propositionAsymptoticNearField}). Based on this result, we
construct two families $S_q^\pm$, $q \in \N^\ast$, of 'near field'
singularities, \ie, solutions to \eqref{ProblemeModeleChampProche} with
$f=0$ but growing at infinity as $(R^\pm)^{\frac{2q}{3}}$ (Subsection~\ref{SubsectionNearFieldSingularities}).  Finally, we
use these singularities to write a quasi-explicit formula
(see~\eqref{FormeExplicitEUnqpm}) for the near fields
terms $U_{n,q,\pm}^\delta$ (Subsection~\ref{SectionQuasiExpliciteNF}).  Here again, most of the results are
explained for the problems posed in $\widehat{\Omega}^+$ but similar results
hold for $ \widehat{\Omega}^-$.
\subsection{General results of existence and asymptotics of the
  solution at infinity} \label{SubsectionNearFieldGeneral}
\subsubsection{Variational framework}
As fully described in Section~3.3 in
\cite{CalozVial}, the standard space to solve
Problem~\eqref{ProblemeModeleChampProche} is 
\begin{equation}
\mathfrak{V}(\widehat{\Omega}^+) = \left\{ v \in H^1_{\text{loc}}(\widehat{\Omega}^+), \;
  \nabla v \in L^2(\widehat{\Omega}^+), \quad \frac{v}{\sqrt{1 + {(R^+)}^2}}
\in  L^2(\widehat{\Omega}^+),  v = 0 \; \mbox{on} \;\partial \mathcal{K}^+\right\},
\end{equation}
which, equipped with the norm
\begin{equation}
\| v\|_{\mathfrak{V} (\widehat{\Omega}^+) } = \left( \left\|  \frac{v}{\sqrt{1 + (R^+)^2}} \right\|_{
    L^2(\widehat{\Omega}^+)}^2 + \|  \nabla v\|_{
    L^2(\widehat{\Omega}^+)}^2 \right)^{1/2}\ , 
\end{equation}
is a Hilbert space. The variational problem associated with
Problem~\eqref{ProblemeModeleChampProche} is the following:
$$
\text{find} \; u \in
\mathfrak{V}(\widehat{\Omega}^+)\; \mbox{ such that}\;\;
\mathfrak{a}(u,v) =
\int_{\widehat{\Omega}^+} f(\mX) \, v(\mX) \, d\mX \; + \; \int_{\widehat{\Gamma}^+_{\hole}} g(\mX)\, v(\mX) d\sigma,\quad \forall \,v \in
\mathfrak{V}(\widehat{\Omega}^+),
$$
where
$
 \mathfrak{a}(u,v)
=\int_{\widehat{\Omega}^+} \nabla u(\mX) \cdot \nabla v(\mX)  d\mX.
$
It is proved in \cite[Proposition 3.6]{CalozVial} (\cf also \cite[Lemma 2.2]{VialThese}),  that
$$
\int_{\widehat{\Gamma}^+_{\hole}} g(\mX)\, v(\mX)\, d\sigma \leq \, C \, \|(1 +
  {(R^+)}^2)^{1/4} g \|_{ L^2(\widehat{\Gamma}^+_{\hole})} \| v\|_{\mathfrak{V}(\widehat{\Omega}^+)},
$$
and that the bilinear form  $\mathfrak{a}$ is coercive on
$\mathfrak{V}(\widehat{\Omega}^+)$ (the seminorm of the
gradient $\| \nabla v \|_{L^2(\widehat{\Omega}^+)}$ is a norm on
$\mathfrak{V}(\widehat{\Omega}^+)$). As a consequence, the following well-posedness
result holds:
\begin{prop}\label{PropositionProblemeChampProche}
Assume that $\sqrt{1 + {(R^+)}^2} f \in L^2(\widehat{\Omega}^+)$ and $(1 +
  {(R^+)}^2)^{1/4} g  \in L^2(\widehat{\Gamma}^+_{\hole})$. Then,
Problem~\eqref{ProblemeModeleChampProche} has a unique solution $u \in \mathfrak{V}(\widehat{\Omega}^+)$.
\end{prop}
\subsubsection{Asymptotic expansion at infinity}
As usual when dealing with matched asymptotic expansions, it is important (for
the matching procedure) to be able to write an asymptotic expansion of the near
fields as $R^\pm$ tends to infinity. In the present case, because of the
presence of the thin layer of periodic holes this is far from being trivial:
there is no separation of variables. However, Theorem 4.1 in
\cite{Nazarov205} helps to answer this
difficult question.\\

\noindent For the statement of the next results, we
need to consider a new family of weighted Sobolev spaces. For $\ell
\in \N$ (in the sequel, we shall only consider $\ell \in \{ 0,1,2\}$), we introduce the space  $\mathfrak{V}_{\beta, \gamma
 }^\ell(\widehat{\Omega}^+)$ defined as the completion of
 $C_{c}^\infty(\overline{\widehat{\Omega}^+})$ with respect to the
 norm
\begin{equation}\label{DefinitionNormeDoublePoids}
\left\| v \right\|_{\mathfrak{V}_{\beta, \gamma
 }^\ell(\widehat{\Omega}^+) }= \sum_{p=0}^\ell \left\|
 (1+R^+)^{\beta-\gamma-\delta_{p,0}} \rho^{\gamma -\ell +p
   +\delta_{p,0} } \nabla^p v \right\|_{L^2(\widehat{\Omega}^+)} \quad
\rho = 1 + (1 + R^+) | \theta^+ -\pi|.
\end{equation}
The norm $\| \cdot \|_{\mathfrak{V}_{\beta, \gamma
 }^2(\widehat{\Omega}^+)}$ is a non-uniform  weighted norm. The weight
varies with the angle $\theta^+$. %
Away from the periodic layer, \ie, for $|\theta^+ - \pi| \geq \varepsilon$ for some $\varepsilon > 0$ and ${R^+}$ sufficiently large, 
we recover
the classical weighted Sobolev norm $V_{\beta}^2 (\mathcal{K}^+)$ (\cf~\eqref{definitionV2betal}) :
\begin{equation}\label{reecritureNormeV2betal}
\| v \|_{V_{\beta}^\ell(\mathcal{K}^+)} =  \left( \sum_{p=0}^{\ell} \|
{(R^+)}^{\beta-\ell +p } \nabla^p v \|_{L^2(\mathcal{K}^+)}^2 \right)^{1/2}.
\end{equation}
Indeed, in this part $\rho \sim 1+{R^+}$ for $R^+ \to \infty$. 
In contrast, close to the layer, \ie, for $\theta^+ \to \pi$ for $R^+$ fixed, we have $\rho \to 1$, and the global weight in
\eqref{DefinitionNormeDoublePoids} becomes
$(1+{R^+})^{\beta-\gamma-\delta_{p,0}}$. \\

\noindent In the classical weighted Sobolev norm~\eqref{reecritureNormeV2betal}, the weight
${(R^+)}^{\beta-\ell +p }$ depends on the derivative ($p = 0$ or $p = 1$) under consideration. It
increases by one at each derivative. This is
linked to the fact that the gradient of a function of the form
${(R^+)}^\lambda g(\theta)$, which is given by ${(R^+)}^{\lambda-1} \left( \lambda
  g(\theta) {\be}_r +  g'(\theta) {\be}_\theta \right)$, decays more rapidly
than the function itself as $R^+$ tends to $+\infty$ (comparing ${(R^+)}^{\lambda-1}$ and ${(R^+)}^{\lambda}$). %
This property does not hold anymore for a
function of the form $(X_1^+)^\lambda g(X_1^+, X_2^+)$ where 
$\mathbf{X}^+ = (X_1^+, X_2^+) = R^+ (\cos \theta^+, \sin \theta^+)$ and
$g \in
\mathcal{V}^+(\mathcal{B})$ ($g$ is periodic with respect to $X_1^+$ and
exponentially decaying with respect to $X_2^+$). Indeed, in this case 
$$
\nabla \left( (X_1^+)^\lambda g\right) = \left( \lambda (X_1^+)^{\lambda-1} g
+ (X_1^+)^\lambda \partial_{X_1^+} g  \right) \eOne + 
(X_1^+)^\lambda  \partial_{X_2^+} g\; \eTwo,  
$$ 
which does not decrease as $(X_1^+)^{\lambda-1}$. %
This remark gives a first intuition of the necessity to introduce a weighted space with a weight
adjusted in the vicinity of the periodic layer, \ie, for $\theta^+ \to \pi$. %
Note that in the case of Dirichlet boundary conditions on the holes, the appropriate weighted
space to consider is slightly different (see~\cite{Nazarov143}).   \\

\noindent To be used later we prove the following properties of these new function spaces.%
\begin{prop}\label{propositionTechniqueSobolevPoids}
Let $\gamma \in \left( \frac{1}{2},1 \right)$, $p \in \{1, 2\}$,
$\lambda \in \R$, $q \in \N$. Let $\chi$ and $\chi_\pm$ be the cut-off
functions defined in~\eqref{defchi} and \eqref{DefinitionChiplusmoins}.
\begin{itemize}
\item[-] The function $v_1 = \chi^{(p)}(X_2^+) \, \chi_-(X_1^+)
  \, (R^+)^\lambda \,(\ln R^+)^q$ belongs to
  $\mathfrak{V}_{\beta, \gamma}^0(\widehat{\Omega}^+)$ providing that
$$
\beta < \gamma - \lambda + 1/2.
$$
\item[-] The function $v_2 = \chi(R^+) (R^+)^\lambda (\ln R^+)^q $
  belongs to $\mathfrak{V}_{\beta, \gamma}^2(\widehat{\Omega}^+) $ providing that
$$
\beta < 1 - \lambda.
$$
\item[-] Let $w=w(X_1^+,X_2^+)$ be a $1$-periodic function with respect to
  $X_1^+$ such that  $\| w \, e^{|X_2^+|/2}\|_{L^2(\mathcal{B})} < +\infty$.  Then, the function $v_3 =
  \chi_-(X_1^+)  |X_1^+|^{\lambda-1} (\ln |X_1^+|)^q w(X_1^+,X_2^+)$ belongs to
  $\mathfrak{V}_{\beta, \gamma}^0(\widehat{\Omega}^+) $ providing that
$$
\beta <  \gamma - \lambda + 3/2.
$$
\item[-] Let $w \in \mathcal{V}^+(\mathcal{B})\cap
  H^2_{\text{loc}}(\mathcal{B})$ such that
$$
\int_{\mathcal{B}} \left( | \partial_{X_1^+}^2 w|^2 + | \partial_{X_2^+}^2
  w|^2    + |\partial_{X_1^+} \partial_{X_2^+} w|^2 \right) e^{|X_2^+|} d\mX^+
< + \infty.
$$ 

 Then, the function $v_4 =
  \chi_-(X_1^+)  |X_1^+|^{\lambda-1} (\ln R^+)^q w(X_1^+,X_2^+)$ belongs to
  $\mathfrak{V}_{\beta, \gamma}^2(\widehat{\Omega}^+) $ providing that
$$
\beta < \gamma - \lambda + 1/2.
$$
\end{itemize}
\end{prop}


\noindent In absence of the periodic layer the solutions of the near field equations 
might be written as linear combination of harmonic functions $(R^+)^{\lambda_m} w_{m,0,+}(\ln R^+, \theta^+)$, $m \in \N^\ast$ for $R^+ \to \infty$
where $w_{m,r,+}$ have been defined~\eqref{Definitionwnpplus}. %
With the periodic layer the behavior far above the layer remains the same, but has to be corrected by
$|X_1^+|^{\lambda_m} p_{m,0,+}(\ln |X_1^+|, X_1^+, X_2^+)$ (with $p_{m,r,+}$ defined in~\eqref{definitionPmr})
to fulfill the homogeneous boundary conditions on $\widehat{\Gamma}_\hole$. %
This correction is not harmonic and has a particular decay rate for $R^+ \to \infty$.
It can be (macroscopically) corrected by $(R^+)^{\lambda_{m}-1} w_{m,1,+}$
and in the neighbourhood of the layer by $|X_1^+|^{\lambda_m-1} p_{m,1,+}(\ln |X_1^+|, X_1^+, X_2^+)$.
Then, through a consecutive correction in the form (the cut-off function $\chi$ has been defined in~\eqref{defchi}
\begin{equation}
 \sum_{r = 0}^{p} \left(  (R^+)^{\lambda_m-r} w_{m,r,+}(\ln R^+, \theta^+)
  \chi(X_2^+) + \chi(X_1^+) |X_1^+|^{\lambda_m -r} p_{m,r,+}(\ln |X_1^+|, X_1^+, X_2^+)  \right)
  \label{eq:explicationLMX1negatif:formal}
\end{equation}
the Laplacian becomes more and more decaying for $R^+ \to \infty$, where
any decay rate can be achieved, which becomes, at least formally, zero for $p \to \infty$. %
The previous observation will be justified in a more rigorous form in the following lemma
 which turn out to be very useful in the sequel.\\
 
\noindent
For this let us introduce a smooth cut-off function $\chi_{\text{macro},+}$ (see Fig.~\ref{DessinTroncature}) that satisfies
\begin{equation}
  \label{eq:chi_macro_+}
  \chi_{\text{macro},+}(X_1^+,X_2^+) = \begin{cases} \chi(X_2^+) & \text{for} \; X_1^+ < -1,\\
    1 & \text{for} \; X_1^+>-\frac{1}{4}\ ,\\
    1 & \text{for} \; X_1^+>-1, |X_2^+| > 3\ ,
  \end{cases}
\end{equation}
and for $m \in \Z \setminus \{0\}$ the asymptotic block (we adopt this notion from~\cite{Nazarov205})
  \begin{multline}\label{DefAsymptoticBlocs}
    \mathcal{U}_{m,p,+}(X_1^+, X_2^+) = \chi(R^+) \sum_{r=0}^p \Big(
     \chi_{\text{macro},+}(X_1^+,X_2^+)  \, (R^+)^{\lambda_{{m}} -r} \, w_{m,r,+}(R^+, \theta^+) 
      + \\[-0.8em] 
     \chi_-(X_1^+) \, |X_1^+|^{\lambda_{{m}} -r} p_{m,r,+}(\ln |X_1^+|,X_1^+,X_2^+)
       \Big)\ ,
  \end{multline}
  where the cut-off function $\chi_-$ has been defined in~\eqref{DefinitionChiplusmoins}.
\begin{lema}\label{LemmeSuperDecroissance}
Under the condition $\gamma \in (1/2,1)$ the Laplacian $\Delta
\mathcal{U}_{m,p,+}$ of the asymptotic block $\mathcal{U}_{m,p,+}$ belongs to $\mathfrak{V}_{\beta, \gamma
 }^0 (\widehat{\Omega}^+)$ for any $\beta$ that satisfies
\begin{equation}
\beta < 2 - \lambda_m + p.
\end{equation}
\end{lema}
\begin{figure}[tbh]
   \centering
       \includegraphics[width=0.35\textwidth]{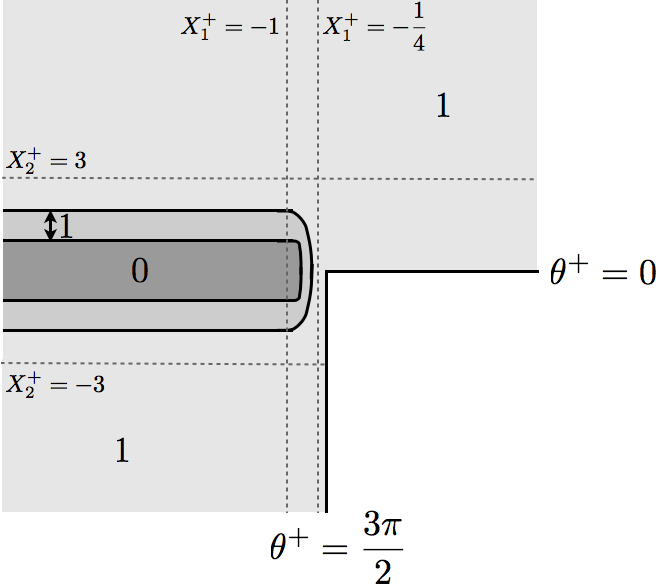}
   \caption{Schematic representation of the cut-off function $\chi_{\macro,+}$
       defined in (\ref{eq:chi_macro_+})}
 \label{DessinTroncature}
 \end{figure}
\noindent 
The functions $p_{m,q,+}$ are polynomials in $\ln |X_1^+|$. They are periodic with
respect to $X_1^+$ and exponentially decaying as $X_2^+$ tends to $\pm
\infty$. The functions $w_{m,q,+}$ are polynomials in
$\ln R^+$. Note that, for $R^+$ large enough, {if $X_1^+<-1$,}
\begin{equation}\label{explicationLMX1negatif}
    \mathcal{U}_{m,p,+}  =  \sum_{r=0}^p \Big(
     (R^+)^{\lambda_{m} -r} w_{m,r,+}(\ln R^+, \theta^+) \chi(X_2^+)   +
     |X_1^+|^{\lambda_{{m}} -r} p_{m,r,+}(\ln
      |X_1^+|,X_1^+,X_2^+) \Big) \quad
\end{equation}
and, if $X_1^+>-1$,
\begin{equation}\label{explicationLMX1positif}
\mathcal{U}_{m,p,+}  =  \sum_{r=0}^p 
(R^+)^{\lambda_m -r} w_{m,r,+}\ .
\end{equation}
The usage of the cut-off functions $\chi_{\text{macro},+}$, 
{$\chi_-(X_1^+)$}, and $\chi(R^+)$ in \eqref{DefAsymptoticBlocs} is simply a technical way to construct
function defined on the whole domain $\widehat{\Omega}^+$ (and more precisely
belonging to $H^2_{\text{loc}}(\widehat{\Omega}^+)$) that coincides
with $\mathcal{U}_{m,p,+}$ for large $R^+$ as given in~\eqref{explicationLMX1negatif} and~\eqref{explicationLMX1positif}.\\

\begin{rem}
In view of Proposition~\ref{propositionTechniqueSobolevPoids},
reminding that $p_{m,0}(\ln |X_1^+|, X_1^+, X_2^+)$ is proportional to $(1 -\chi(X_2^+))$, for any $\gamma \in (1/2,1)$, the asymptotic block
$\mathcal{U}_{m,p,+} $ belongs to $\mathfrak{V}_{\beta, \gamma
 }^2 (\widehat{\Omega}^+)$ for any $\beta < 1 - \lambda_m$.
\end{rem} 
\noindent Defining $\chi_{\macro,-}(X_1^-,X_2^-) =\chi_{\macro,+}(-X_1^-,X_2^-)$, we
can also define the asymptotic blocks associated with
$\widehat{\Omega}^-$ as follows:
 \begin{multline}\label{DefAsymptoticBlocs_moins}
\mathcal{U}_{m,p,-} = \chi(R^-) \sum_{r=0}^p \Big(
  \chi_{\text{macro},-}(X_1^-,X_2^-)  \, (R^-)^{\lambda_m -r} \, w_{m,r,-}(\ln R^-, \theta^-)  + \\[-0.8em] 
 \chi_{+}(X_1^-) \, |X_1^-|^{\lambda_m -r} p_{m,r,-}(\ln |X_1^- |,X_1^-,X_2^-)
\Big)  .
\end{multline}

\noindent We are now in a position to write the main result of this subsection, which proves that for $R^+$ large and
for sufficiently decaying right-hand sides, the
solutions of Problem~\eqref{ProblemeModeleChampProche} can be decomposed
into a sum of radial contributions corrected by periodic
exponentially decaying correctors in the vicinity of the layer of equispaced holes. In the following, a real number $\beta$
is said to be admissible if $\beta - 1 \notin \Lambda$. 

\begin{prop}\label{propositionAsymptoticNearField}
  Let $k \in \N$ and $\gamma \in (1/2, 1)$.  Assume that $f \in
  \mathfrak{V}_{\beta, \gamma }^0(\widehat{\Omega}^+)$ for some admissible
  $\beta >\max(3, 1 + \frac{2(k+1)}{3})$ (and so $\sqrt{1 + (R^+)^2}f \in L^2(\widehat{\Omega}^+)$) and that $g \in H^{1/2}(\widehat{\Gamma}_\hole)$
  is compactly supported.  Then, the unique solution $u$ of
  Problem~\eqref{ProblemeModeleChampProche} belongs to $
  \mathfrak{V}(\widehat{\Omega}^+)\cap  \mathfrak{V}_{\beta', \gamma
  }^2(\widehat{\Omega}^+)$ for any $\beta'< \frac{1}{2}$. Moreover,  
  $u$ admits, for $\gamma-1$ sufficiently small, the decomposition
\begin{equation}\label{AsymptoticExpansionSolutionCP}
u =  \sum_{n =1}^k  \mathscr{L}_{-n}(u) \,
{\mathcal{U}_{-n, p(n),+}}
+ \;\tilde{u}, \quad p(n) = \max(1, \lceil \frac{2}{3} (k-n) \rceil),
\end{equation} 
where the asymptotic blocks $\mathcal{U}_{-n, p(n), +}$ are defined
in~\eqref{DefAsymptoticBlocs}, $\mathscr{L}_{-n}(u)$,
$1\leq n \leq k$, denote $k$ constants, and, the remainder $\tilde{u}
\in \mathfrak{V}_{\beta^0, \gamma }^2(\widehat{\Omega}^+)$ for any
$\beta^0$ such that $\beta^0<\beta$. 
In addition, there exists a constant $C>0$ such that
\begin{equation}\label{EstimationAsymptoticExpansionSolutionCP}
\| \tilde{u} \|_{\mathfrak{V}_{\beta^0, \gamma
  }^2(\widehat{\Omega}^+)} + \sum_{n =1}^k \left| \mathscr{L}_{-n}(u)  \right| \quad \leq
\quad C  \left( \| f \|_{\mathfrak{V}_{\beta, \gamma
  }^0(\widehat{\Omega}^+)} + \| g \|_{H^{1/2}(\widehat{\Gamma}_\hole)} \right).
\end{equation}

\end{prop}
\noindent We remind that $\lceil a \rceil$ (used in~\eqref{AsymptoticExpansionSolutionCP})
stands for the smallest integer not less than $a$.\\

\noindent The proof of
Proposition~\ref{propositionAsymptoticNearField}, postponed in Appendix~\ref{AppendixNFPreuvePropAsymptotic}, deeply relies on 
successive applications 
of the following lemma, which is a direct adaptation of Theorem 4.1
in~\cite{Nazarov205}. To long to be presented in this paper, its proof requires the use of
involved tools of complex analysis 
that are fully described in~\cite{Nazarov205}. 

\begin{lema}\label{LemmeAsymptoticAppliNazarov} Let $\gamma \in
  (\frac{1}{2},1)$. Let $\beta^1$ and $\beta^2$ be two admissible
  exponents such that $\beta^1< \beta^2$ and \mbox{$\beta^2 - \beta^1
  <1$}. Assume that $u \in \mathfrak{V}_{\beta^1, \gamma
 }^2(\widehat{\Omega}^+)$ satisfies
 Problem~\eqref{ProblemeModeleChampProche} with $f \in \mathfrak{V}_{\beta^2, \gamma
 }^0(\widehat{\Omega}^+)$ and a compactly supported $g \in H^{1/2}(\widehat{\Gamma}_\hole)$. %
 Then, for $\gamma-1<0$ sufficiently small $u$ admits the decomposition 
\begin{equation}\label{sommeUx}
u = \sum_{\frac32(1- \beta^2) < k < \frac32(1- \beta^1)} c_k
\, \mathcal{U}_{k,3,+} \,
 + \tilde{u},
\end{equation}
where the asymptotic blocks $\mathcal{U}_{k,3,+}$ are defined in~\eqref{DefAsymptoticBlocs} and 
for any admissible $\beta^0 \in (\beta^1, \beta^2)$ the remainder $\tilde{u} \in \mathfrak{V}_{\beta^0, \gamma
 }^2(\widehat{\Omega}^+)$. In this case, there exists a positive constant $C>0$ such that
\begin{equation}
\| \tilde{u} \|_{\mathfrak{V}_{\beta^0, \gamma
  }^2(\widehat{\Omega}^+)} + \sum_{\frac32(1- \beta^2) < k < \frac32(1- \beta^1)} |c_k| \quad \leq \quad C \left( \| {u} \|_{\mathfrak{V}_{\beta^1, \gamma
  }^2(\widehat{\Omega}^+)} + \| f \|_{\mathfrak{V}_{\beta^2, \gamma
  }^0(\widehat{\Omega}^+)} + \| g \|_{H^{1/2}(\widehat{\Gamma}_\hole)} \right).
\end{equation}
\end{lema}

\noindent We emphasize that the powers of $R^+$ (or $X_1^+$) appearing
in~\eqref{AsymptoticExpansionSolutionCP}
(see~\eqref{DefAsymptoticBlocs}) are of the form $-\frac{2}{3}
n -q$, $n \in \N^\ast$, $q \in \N$. Thus, they coincide with the ones
obtained for the far field part (see for instance
Proposition~\ref{PropDefinitionsmq}). Moreover, as can be expected, the 'leading' singular
exponents ${\frac{2n}{3}}$ correspond to those of the problem
without periodic layer (here again, as for the macroscopic terms). \\

\begin{rem}
The assumption on $g$ in
Proposition~\ref{propositionAsymptoticNearField} and Lemma~\ref{LemmeAsymptoticAppliNazarov}
could be weakened by using the trace spaces associated with the weighted Sobolev spaces
$\mathfrak{V}_{\beta,\gamma}^\ell(\widehat{\Omega}^+)$ (see~\cite{Nazarov205}). 
\end{rem}
\subsection{Two families of near field singularities}\label{SubsectionNearFieldSingularities}
This subsection is dedicated to the construction of two families of functions  $S_m^\pm$, $m \in
\N^\ast$, hereinafter referred to as the near field singularities for the
right and left corner, satisfying the homogeneous Poisson problems
\begin{equation}\label{ProblemeS_q}
\left \lbrace
\begin{aligned}
-\Delta S_m^\pm & \; = & 0 &\quad \mbox{in} \; \widehat{\Omega}^\pm,\\
S_m^\pm  &\;= & 0 &\quad \mbox{on} \; \partial \mathcal{K}^\pm, \\
\partial_n  S_m^\pm  & \; = & 0 &\quad \mbox{on} \; \partial \widehat{\Omega}^\pm \setminus \partial \mathcal{K}^\pm\ ,
\end{aligned}
\right. 
\end{equation}
and behaving like
$ (R^{\pm})^{\lambda_m} w_{m,0,\pm}(\theta^{{^\pm}})$
 for $R^{\pm}$ large, where $w_{m,0,+} = \sin( \frac{2 n}{3} {\theta^+})$ and $w_{m,0,-} = \sin(
\frac{2 n}{3} (\theta^- -\frac{\pi}{2}))$ (defined in Proposition~\ref{PropositionAsymptoticEspaceAPoids}).
\begin{prop}\label{PropositionExistenceSmplus}
 There exists a unique function $S_m^+ \in
 \mathfrak{V}_{\beta,\gamma}^2(\widehat{\Omega}^+)$ for any $\beta < 1 -
 \lambda_m$ and $\gamma \in (1/2,1)$, satisfying  the homogeneous
 equation~\eqref{ProblemeS_q}  such that the function
\begin{equation}\label{Ceiling}
\tilde{S}_m^+ = S_m^+ - \mathcal{U}_{m,\lceil 1 +
\lambda_m \rceil,+ } 
\end{equation}
belongs to $\mathfrak{V}(\widehat{\Omega}^+)$. Moreover, for any $k \in \N^\ast$, choosing
$1-\gamma$ sufficiently small, there exists a function
$\mathcal{R}_{m,k} \in \mathfrak{V}_{\beta^0,\gamma}^2(\widehat{\Omega}^+)$ for any admissible $\beta^0 <
1 + \frac{2(k+1)}{3}$ and $k$ constants $\mathscr{L}_{-n}(S_m^+)$ ($1\leq k \leq
n$) such that $S_m^+$ admits the decomposition 
\begin{equation}\label{conditionAsymptoticSmplus}
S_m^+ = \mathcal{U}_{m,\lceil \frac{2(k+m)}{3} \rceil ,+ } + \sum_{n=1}^k \mathscr{L}_{-n}(S_m^+)
\mathcal{U}_{-n, \lceil \frac{2(k-n)}{3} \rceil,+} + \mathcal{R}_{m,k}.
\end{equation}
In addition, for any $\beta < 1 -\lambda_m$,  there is a constant
$C>0$ such that
\begin{equation}
\sum_{n=1}^k \left| \mathscr{L}_{-n}(S_m^+) \right| + \left\|
  \mathcal{R}_{m,k}\right\|_{\mathfrak{V}_{2,\beta^0}^{2}(\widehat{\Omega}^+)} \leq C \,\left\|
   S_m^+ \right\|_{\mathfrak{V}_{2,\beta}^2(\widehat{\Omega}^+)}.
\end{equation}
\end{prop}
\noindent In the same way as the near field singularities for the
right corner $S_m^+$, the near field singularities for the
left corner $S_m^-$ can be defined. 
Note that for $m \geq 2$ there are several functions satisfying 
 the homogeneous
 equations~\eqref{ProblemeS_q} and behaving like $(R^+)^{\lambda_m}
 w_{m,0,+}$ (leading term) for large $R^+$. Indeed, admitting the
 existence of the functions $S_m^+$, any function of the
 form $S_m^+ + \sum_{k=1}^{m-1} a_k S_{k}^+$ would also fulfill these
 requirements. Nevertheless, the \eqref{conditionAsymptoticSmplus}
 restores the uniqueness by fixing (arbitrary) $a_k$, $k=1,\ldots,m-1$ to $0$.
\begin{proof}
The proof is classical and is very similar to the proof of
Propositions~\ref{PropExistenceUniquenesmMF},
\ref{PropExistenceUniquenesmMF2} and \ref{PropDefinitionsmq}.  We
first prove the existence of $S_m^+$. The function $\tilde{S}_m^+$
satisfies 
\begin{equation*}
\left\lbrace
\begin{aligned}
-\Delta \tilde{S}_m^+ & \;= &  \tilde{f}_m &\quad \mbox{in} \; \widehat{\Omega}^+,\\
 \tilde{S}_m^+  &\;= & 0 &\quad\mbox{on} \; \partial \mathcal{K}^+,\\
\partial_n   \tilde{S}_m^+ &\;= &0 & \quad\mbox{on} \; \partial \widehat{\Omega}^+ \setminus \partial \mathcal{K}^+,
\end{aligned}\right. \quad \tilde{f}_m = - \Delta \mathcal{U}_{m,\lceil 1 +
\lambda_m \rceil,+ }.
\end{equation*}  
In view of Lemma~\ref{LemmeSuperDecroissance}, $\tilde{f}_m$ belongs
to $\mathfrak{V}_{2,\beta}^0(\widehat{\Omega}^+)$ for $\beta< 2 - \lambda_m
+\lceil 1 +
\lambda_m \rceil$. Noting that $2 - \lambda_m
+ \lceil 1 +
\lambda_m \rceil {\geqslant} 3$,
Proposition~\ref{PropositionProblemeChampProche} ensures the
existence and uniqueness of $\tilde{S}_m^+ \in \mathfrak{V}(\widehat{\Omega}^+)$,
and, hence, the existence of $S_m^+$.
Uniqueness of $S_m^+$ follows directly from the fact that difference of two possible solutions
is in the variational space $\mathfrak{V}(\widehat{\Omega}^+)$ and satisfies~\eqref{ProblemeS_q}.
Finally, the asymptotic behavior for large $R^+$ results from a direct
application of Proposition~\ref{propositionAsymptoticNearField} to the function 
$\tilde{S}_m^+  - \mathcal{U}_{m,p,+}$
choosing $p$ sufficiently large so that $\Delta (\tilde{S}_m^+  -
\mathcal{U}_{m,p,+})$ belongs to $
\mathfrak{V}_{\beta,\gamma}^2(\widehat{\Omega}^+)$ for a real number $\beta > \max(3, 1 +
\frac{2(k+1)}{3})$ (which is, thanks to Lemma~\ref{LemmeSuperDecroissance}, always possible). \end{proof}

\subsection{An explicit expression for the near field terms}\label{SectionQuasiExpliciteNF}
As done for the macroscopic terms in Section~\ref{SubsectionExpliciteMacro}, we can write a quasi-explicit formula
for the near field terms $U_{n,q,\pm}^\delta$. We shall impose that the functions
$U_{n,q,\pm}^\delta$ do not blow up faster than $(R^+)^{\lambda_n}$ for $R^+ \to \infty$. Since
$U_{n,q,\pm}^\delta$ satisfies the near field equations~\eqref{NearFieldEquation},
it is natural to construct $U_{n,q,\pm}^\delta$ as a linear
combination of the near field singularities $S_k^\pm$, $1 \leq k \leq
n$, namely  
\begin{equation}\label{FormeExplicitEUnqpm}
U_{n,q,\pm}^\delta= \sum_{k=1}^n \mathscr{L}_k(U_{n,q,\pm}^\delta) S_k^\pm,
\end{equation}
where  $\mathscr{L}_k(U_{n,q,\pm}^\delta) $ are constants that will be
determined by the matching procedure and that might depend on
$\delta$. Naturally, the functions $U_{n,q,\pm}$ (defined by~\eqref{FormeExplicitEUnqpm}) satisfy the near field
equations~\eqref{NearFieldEquation} and belong to
$\mathfrak{V}_{\beta,\gamma}^2(\widehat{\Omega}^\pm)$ for any $\beta < 1 -
 \lambda_n$ and $\gamma \in (1/2,1)$. It is worth noting that the
 definition~\eqref{FormeExplicitEUnqpm} implies that 
\begin{equation}
U_{0,q,\pm}^\delta = 0  \quad \forall q \in \N.
\label{eq:U0q=0}
\end{equation}
\paragraph{Asymptotic behavior for large $R^+$}
\noindent To conclude this section, we slightly anticipate the upcoming matching procedure by
writing the behavior
of $U_{n,q,+}^\delta$ at infinity. Thanks to the asymptotic behavior of
$S_m^+$ for $R^+ \to \infty$ (Proposition~\ref{PropositionExistenceSmplus}), we see that,
for any $K \in \N^\ast$, 
\begin{equation}
U_{n,q,+}^\delta = \sum_{k=1}^n \sum_{l=-k}^K \mathscr{L}_k(U_{n,q,+}^\delta) \,
\mathscr{L}_{-l}(S_k^+) \,
{\mathcal{U}_{-l,\lceil \frac{2(K-l)}{3}\rceil,+}} \;+ \;\mathcal{R}_{n,q,K,+},
\end{equation}
where $\mathcal{R}_{n,q,K,+} \in \mathfrak{V}_{2,\beta^0}^2(\widehat{\Omega}^+)$ for any $\beta^0 <
1 + \frac{2(K+1)}{3}$. Here, for the sake of concision, we have posed
$$
\mathscr{L}_r(S_q^+) = 0   \quad \mbox{for any } r \in \N \; \text{such that} \; 0 \leq r \leq q-1,
$$
and $\mathcal{U}_{0,j,+}=0$ for any $j \in \N$.  Then, substituting the
asymptotic blocks $\mathcal{U}_{-l,k+2-l,+}$ for their explicit
expression~\eqref{DefAsymptoticBlocs}, we obtain the decomposition 
\begin{align}
\begin{aligned}
U_{n,q,+}^\delta  &= \chi_{\macro,+}(\mX) \sum_{l=-n}^K \sum_{r=0}^{{\lceil
    \frac{2(K-l)}{3}\rceil} } A_{n,q,-l,+}^\delta 
  \left( R^+ \right)^{-\frac{2 l}{3} - r} 
   w_{-l,r,+}(R^+, \theta^+) \\
& \hspace{1em} + 
\chi_-(X_1^+)  \sum_{l=-n}^K \sum_{r=0}^{\lceil
    \frac{2(K-l)}{3}\rceil } A_{n,q,-l,+}^\delta  |X_1^+|^{-\frac{2 l}{3}
    - r} p_{-l,r,+}(\ln |X_1^+|, X_1^+,X_2^+) 
+ \mathcal{R}_{n,q,K,+}\ ,  
\end{aligned}
\label{AsymptoticUnqX1negatif}
\end{align} 
where we have used the convention $w_{0,r,+}=0$ and $p_{0,r,+}=0$ for any $r \in
\Z$. Here, 
\begin{equation}\label{definitionAnql}
\forall (n,q,l) \in \N^2 \times \Z,  \quad A_{n,q,l,+}^\delta = \sum_{k=\max(1, l)}^n \mathscr{L}_k(U_{n,q,+}^\delta) \mathscr{L}_l(S_k^+).
\end{equation}
Note that for $n=0$, $A_{n,q,l,+}^\delta=0$. Moreover, for $l \geq 0$, $\mathscr{L}_l(S_k) = \delta_{l,k}$,
and consequently 
\begin{equation}\label{Anqllpositif}
A_{n,q,l,+}^\delta = \mathscr{L}_l(U_{n,q,+}^\delta).
\end{equation} 
Finally, with the change of index $-l \to m$ and summing up over $n$ and
$q$, we can \textbf{formally} obtain an asymptotic series of the near
field: For  $X_1^+<-1$,
\begin{multline}\label{FormuleNFPretePourlesraccordsX1negatif}
\sum_{(n,q) \in \N^2} \delta^{\frac{2n}{3} +q } U_{n,q,+}^\delta =\\
   \sum_{(n,q) \in \N^2} \delta^{\frac{2n}{3} +q }
\sum_{m=-\infty}^{n}  A_{n,q,m,+}^\delta  \sum_{r\in \N} \left(
  (R^+)^{\frac{2 m}{3} - r} w_{m,r,+}(\theta^+, {\ln R^+}) \chi(X_2^+)
\right. + \left. |X_1^+|^{\frac{2 m}{3}
    - r} p_{m,r,+}(\ln |X_1^+|,\mX^+) \right)
\end{multline}
and, for $X_1^+>-1$, 
\begin{equation}\label{FormuleNFPretePourlesraccordsX1positif}
\sum_{(n,q) \in \N^2} \delta^{\frac{2n}{3} +q } U_{n,q,+}^\delta = \sum_{(n,q) \in \N^2} \delta^{\frac{2n}{3} +q }
\sum_{m=-\infty}^{n}  A_{n,q,m,+}^\delta  \sum_{r\in \N} 
  (R^+)^{\frac{2 m}{3} - r} w_{m,r,+}(\theta^+, \ln R^+)\ .
\end{equation}

\noindent The near field terms $U_{n,q,-}^\delta$ can be decomposed in strictly similar way by substituting
formally the superscript plus into a superscript minus in~\eqref{AsymptoticUnqX1negatif} and \eqref{definitionAnql}.
Here, it should be noted that the way to compute $\mathscr{L}_l(U_{n,q,-}^\delta)$ is different how $\mathscr{L}_l(U_{n,q,+}^\delta)$ are computed.
\section{Matching procedure and construction of the far and near field
terms}\label{SectionMatchingProcedure}
We are now in the position to write the matching conditions that account for the asymptotic coincidence 
of the far field expansion with the near field expansion 
in the matching areas. Based on the matching conditions, we provide an iterative
algorithm to define all the terms of the far and near field expansion (to any order), which have not been fixed yet.

\subsection{Far field expansion expressed in the microscopic
  variables} 
We start with writing the formal expansion of the far field $\sum_{(n,q)\in \N^2}
\delta^{\frac{2n}{3} + q} (u_{n,q}^\delta(\mathbf{x}) \chi(x_2/\delta)  + \Pi_{n,q}^\delta(x_1,\frac{\mathbf{x}}{\delta}))$  (\cf~\eqref{def_uFFnq}) in the
 matching area located in the vicinity of the right corner (\ie for small $r^+$). %
 Collecting \eqref{AsymptoticMacroFieldMacthingareas} and
\eqref{AsymptoticBoundaryLayerMacthingareas}, summing over the pair
of indices $(n,q) \in \N^2$ and applying the change of
scale $\mx^+/\delta = (\mx - \mx_O^+)/\delta \to \mX^+$ and so $r^+/\delta \to R^+$ we \textbf{formally} obtain
for $X_1<-1$,
\begin{multline}\label{FormuleFFPretePourlesraccordsX1negatif}
\sum_{(n,q)\in \N}
\delta^{\frac{2n}{3} + q} \left(u_{n,q}^\delta(\mathbf{x})  \chi(\frac{x_2}{\delta} )+ \Pi_{n,q}^\delta(x_1,\frac{\mathbf{x}}{\delta})\right) =  \\[-0.8em]
\sum_{(n,q)\in \N^2} \delta^{\frac{2}{3} n + q} \sum_{m=-\infty}^n
a_{n-m,q,m,+}^\delta 
\sum_{r\in \N} \Big( (R^+)^{\frac{2m}{3} -r}
  w_{m,r,+}(\theta^+, \ln R^+ + \ln \delta) \chi(X_2^+)  \\[-0.8em]
+  |X_1^+|^{\frac{2m}{3} -r}
  p_{m,r,+}(\ln|X_1^+| + \ln \delta, X_1^+, X_2^+) \Big)\ ,
\end{multline}
and for $X_1 > -1$
\begin{equation}\label{FormuleFFPretePourlesraccordsX1positif}
\sum_{(n,q)\in \N}
\delta^{\frac{2n}{3} + q} u_{n,q}^\delta(\mathbf{x})  =
\sum_{(n,q)\in \N^2} \delta^{\frac{2}{3} n + q} \sum_{m=-\infty}^n
a_{n-m,q,m,+}^\delta 
\sum_{r\in \N} (R^+)^{\frac{2m}{3} -r}
  w_{m,r,+}(\theta^+, \ln R^+ + \ln \delta) \ .
\end{equation}
\noindent Note, that the coefficients $a_{n,q,m,+}^\delta$, defined in~\eqref{eq:anjm},
depend for $n > 0$ on $\delta$ only through the constants $\ell^\pm_{-k}(u^\delta_{n,p})$, which we are
going to fix in the matching process. %
In the equations~\eqref{FormuleFFPretePourlesraccordsX1negatif} and~\eqref{FormuleFFPretePourlesraccordsX1positif}
the terms $w_{m,r,+}$ and $p_{m,r,+}$ appear with a second shifted argument, \ie, $\ln R^+ + \ln \delta$ instead of $\ln R^+$ and $\ln |X_1^+| + \ln \delta$ instead of $\ln|X_1^+|$.
The following lemma is a reformulation of these terms as linear combinations of non-shifted ones and will prove very useful in the matching procedure.
It is based essentially on the fact that the terms $w_{m,r,+}$ are polynomials in the second argument and $p_{m,r,+}$ in the first. 
The proof of the lemma finds itself in Appendix~\ref{sec:ProofLemmeTechniqueChangementEchelle}.
\begin{lema}\label{LemmeTechniqueChangementEchelle}The equalities
\begin{equation}\label{FormuleMagiqueChangementEchelle}
w_{m,r,\pm}(\theta^{{\pm}}, \ln R^{{\pm}}+\ln \delta ) = \sum_{k=0}^{\lfloor r/2 \rfloor}
w_{m-3k,r-2k,\pm}(\theta^{{\pm}}, \ln R^{{\pm}}) \;  \sum_{i=0}^k C_{m,2k,i,\pm} \;
  (\ln \delta)^i\ ,
\end{equation}
and
\begin{equation}\label{FormuleMagiqueChangementEchelleBL}
p_{m,r,\pm}(\ln |X_1^{{\pm}}| +\ln \delta,\mX^{{\pm}}) =  \sum_{k=0}^{\lfloor r/2 \rfloor}
p_{m-3k,r-2k,\pm}(\ln |X_1^{{\pm}}| ,\mX^{{\pm}}) \; \sum_{i=0}^k C_{m,2k,i,\pm} \;
  (\ln \delta)^i 
\end{equation}
hold true, where for any integer $i, k$ such that $0 \leq i \leq k$ and 
using the notation $w_{m, 2k,\pm}(\theta^{{\pm}}, \ln R^{{\pm}})=\sum_{i=0}^{k} (\ln R^{{\pm}})^i w_{m,
  2k,i,\pm}(\theta^+)$, the constants $C_{m,2k,i\pm}$ are given by
\begin{equation}\label{definitionCm2ki}
C_{m,2k,i,\pm} = \frac{4}{3 \pi} \int_{I^\pm} w_{m, 2k,i,\pm}(\theta^{{\pm}})
  w_{m-3k, 0,\pm}(\theta^{{\pm}}) d\theta^{{\pm}}, \; \; I^+ = (0,
  \frac{3\pi}{2}),  \, I^- = (-\frac{\pi}{2}, 
  \pi).
\end{equation}
\end{lema}

\begin{rem}\label{RemCoeffCm2ki}
With the convention that $w_{0,0,\pm}=0$  it follows that $C_{m,2k,i,\pm}=0$ for $k=m/3$. Moreover, in view of the orthogonality
conditions~\eqref{ConditionOrthogonalitePlus},\eqref{ConditionOrthogonaliteMoins},
$C_{m,2k,0,\pm}$ always vanishes if $k\neq 0$ and $C_{m,0,0,\pm}=1$. 
\end{rem}
\noindent Inserting \eqref{FormuleMagiqueChangementEchelle} into
\eqref{FormuleFFPretePourlesraccordsX1positif} and noting that
$\frac{2}{3}( m-3k) - (r-2k) = \frac{2}{3} m -r$, we obtain
\begin{multline*}
\sum_{(n,q)\in \N^2}
\delta^{\frac{2n}{3} + q} u_{n,q}^\delta  =
\sum_{(n,q)\in \N^2} \delta^{\frac{2}{3} n + q} \sum_{m=-\infty}^n 
a_{n-m,q,m,+}^\delta \\
\sum_{r\in \N} \sum_{ k=0}^{\lfloor r/2 \rfloor} (R^+)^{\frac{2(m-3k)}{3} -(r-2k)}
  w_{m-3k,r-2k,+}(\theta^+, \ln R^+)  \sum_{i=0}^k C_{m,2k,i,+} (\ln \delta)^i .
\end{multline*}
 Then, the changes of indices $r-2k \rightarrow r$ and $m -3k \rightarrow
 m$ give
\begin{equation}\label{FormuleFFPretePourlesraccordsX1positif2}
\sum_{(n,q)\in \N^2}
\delta^{\frac{2n}{3} + q} u_{n,q}^\delta  = \sum_{(n,q)\in \N^2}
\delta^{\frac{2}{3} n + q} \sum_{m=-\infty}^n
\tilde{a}_{n-m,q,m,+}^\delta  \sum_{r\in \N} (R^+)^{\frac{2m}{3} -r}
  w_{m,r}(\theta^+, \ln R^+),  \quad {X_1^{+}>0}, 
\end{equation}
where
\begin{equation}\label{definitionTildeAnqm}
\forall (n,q,l) \in \N^2 \times \Z, \quad \tilde{a}_{n,q,m,+}^\delta = \sum_{k=0}^{\lfloor n/3 \rfloor }
 a_{n-3k,q,m+3k,+}^\delta \sum_{i=0}^k C_{m+3k,2k,i,+} (\ln \delta)^i\ .
\end{equation} 
In particular, for $m<0$, thanks to \eqref{ValeurRemarquableanjm} (and using Remark~\ref{RemCoeffCm2ki}), we have 
\begin{equation}\label{FormuleRemarquableanqtilde}
\tilde{a}_{n,q,m,+}^\delta = \ell_{m}^+(u_{n,q}^\delta) + \sum_{k=1}^{\lfloor n/3 \rfloor }
 a_{n-3k,q,m+3k,+}^\delta \sum_{i=0}^k C_{m+3k,2k,i,+} (\ln \delta)^i.
\end{equation}
Analogously, for $X_1^+<-1$, we obtain,
\begin{multline}\label{FormuleFFPretePourlesraccordsX1negatif2}
\sum_{(n,q)\in \N}
\delta^{\frac{2n}{3} + q} (u_{n,q}^\delta(\mathbf{x})  \chi(x_2{/\delta})+ \Pi_{n,q}^\delta(x_1, \mathbf{x}/\delta) =  \\
\sum_{(n,q)\in \N^2} \delta^{\frac{2}{3} n + q} \hspace{-0.3em}\sum_{m=-\infty}^n
\tilde{a}_{n-m,q,m,+}^\delta 
\sum_{r\in \N} \left( (R^+)^{\frac{2m}{3} -r}
  w_{m,r,+}(\theta^+, \ln R^+) \chi(X_2) \right. 
+ \left. |X_1^+|^{\frac{2m}{3} -r}
  p_{m,r,+}(\ln|X_1^+|,\mX^+) \right).
\end{multline}
The previous two expressions have to be compared with
formula~\eqref{FormuleNFPretePourlesraccordsX1negatif} and\eqref{FormuleNFPretePourlesraccordsX1positif}, 
in which the coefficients $A_{n,q,m,+}^\delta$ are still not determined, since the constants $\mathscr{L}_{m}(U_{n,q,+}^\delta)$, $m = 1,\ldots,n$ are not fixed yet.
We aim to match the expansions in the matching zone and, hence, define these constants uniquely.


\subsection{Derivation of the matching conditions} 

Arrived at this point, the derivation of the matching conditions is
straight-forward. It suffices to identify \textbf{formally} all terms of the expansions~\eqref{FormuleFFPretePourlesraccordsX1negatif2} and
\eqref{FormuleFFPretePourlesraccordsX1positif2}  of the far field with all terms of
the expansions~\eqref{FormuleNFPretePourlesraccordsX1negatif} and
~\eqref{FormuleNFPretePourlesraccordsX1positif} for the near field. We
end up with the following set of conditions:
\begin{equation}\label{matchingCondition1}
A_{n,q,m,+}^\delta = \tilde{a}_{n-m,q,m,+}^\delta, \quad \forall (n,q) \in
\N^2, \;\text{ and }\;m \in \Z, \; m \leq n,
\end{equation}
where $A_{n,q,m,+}^\delta$ and $\tilde{a}_{n-m,q,m,+}^\delta$ were defined 
in \eqref{definitionAnql}  and \eqref{definitionTildeAnqm}. 
As the coefficients $A_{n,q,m,+}^\delta$ are linear combinations of the constants $\mathscr{L}_{m}(U_{n,q,+}^\delta)$, $m = 1,\ldots,n$
and the coefficients $\tilde{a}_{n,q,m,+}^\delta$ are linear combinations of the constants $\ell^+_{-m}(u_{n,p}^\delta)$, $m = 1,\ldots,n$ for some $p \in \IN$
we aim now to obtain conditions between those constants. Here, we will proceed separately for the
the cases $m>0$, $m=0$ and $m <0$.\\


\noindent For $0<m\leq n$, using the equality~\eqref{Anqllpositif}, we have for any
$(n,q) \in \N^\ast \times \N$, 
\begin{equation}\label{MactchingChampProche}
\mathscr{L}_{m}(U_{n,q,+}^\delta) =  \tilde{a}_{n-m,q,m,+}^\delta \ .
\end{equation}

\noindent For $m=0$, for any $(n,q) \in \N$, there is nothing to be matched. Indeed,  both left and
right-hand sides of \eqref{matchingCondition1} vanish
($\tilde{a}_{n,q,0,+} = 0$ because $C_{3k,2k,i,+}=0$, see Remark~\ref{RemCoeffCm2ki}).\\ 

\noindent For $m<0$, in view of~\eqref{FormuleRemarquableanqtilde} and
substituting $A_{n,q,m,+}$ for its definition~\eqref{definitionAnql},
we have for any $(n,q) \in \N^2$, 
 \begin{equation}
  \ell_{m}^+(u_{n-m,q}^\delta) = -\sum_{k=1}^{\lfloor (n-m)/3 \rfloor }
 a_{n-m-3k,q,m+3k,+}^\delta \sum_{i=0}^k C_{m+3k,2k,i,+} (\ln \delta)^i
 + \sum_{k=1}^n \mathscr{L}_k(U_{n,q,+}^\delta) \mathscr{L}_m(S_k^+),
 \end{equation} 
 which may also be red as follows: for any $(n,q,m) \in \N^3$ such
 that, $n \geq 1$, and $1\leq m \leq n$, 
\begin{equation}\label{matchingChampLointain}
 \ell_{-m}^+(u_{n,q}^\delta) = -\sum_{k=1}^{\lfloor n/3 \rfloor }
 a_{n-3k,q,-m+3k,+}^\delta \sum_{i=0}^k C_{-m+3k,2k,i,+} (\ln \delta)^i
 + \sum_{k=1}^{n-m} \mathscr{L}_k(U_{n-m,q,+}^\delta) \mathscr{L}_{{-m}}(S_k^+).
 \end{equation} 

\noindent Here again, we can write similar matching conditions for the
matching area located close to the left corner. These conditions link
the macroscopic terms $u_{n,q}^\delta$ to the near field terms
$U_{n,q,-}^\delta$: for $1 \leq m \leq n$ and for any $(n,q) \in \N^\ast
\times \N$,
\begin{equation}\label{MactchingChampProcheMoins}
\mathscr{L}_{m}(U_{n,q,-}^\delta) =  \tilde{a}_{n-m,q,m,-}^\delta,
\end{equation}
and, for any $(n,q,m) \in \N^3$ such
 that, $n \geq 1$, and $1\leq m \leq n$, 
\begin{equation}\label{matchingChampLointainMoins}
 \ell_{-m}^-(u_{n,q}^\delta) = -\sum_{k=1}^{\lfloor n/3 \rfloor }
 a_{n-3k,q,-m+3k,-}^\delta \sum_{i=0}^k C_{-m+3k,2k,i,-} (\ln \delta)^i
 + \sum_{k=1}^{n-m} \mathscr{L}_k(U_{n-m,q,-}^\delta) \mathscr{L}_{{-m}}(S_k^-).
 \end{equation} 
Here, $\tilde{a}_{n,q,m}$ are defined by
$$
\forall \,(n,q,l) \in \N^2 \times \Z, \quad \tilde{a}_{n,q,m,-}^\delta   = \sum_{k=0}^{\lfloor n/3 \rfloor }
 a_{n-3k,q,m+3k,-}^\delta \sum_{i=0}^k C_{m+3k,2k,i,-} (\ln \delta)^i\ ,$$
 with
\begin{equation*}
\forall \, (j,m) \in \N \times \Z, \; a_{0,j,m,-}^\delta=
\ell_{m}^-(s_{0, j}), \; \mbox{and} \;
a_{n,j,m,-}^\delta= \sum_{\pm} \sum_{k=\max(1,-m)}^n
\sum_{p=0}^{j}\ell_{-k}^\pm(u_{n,p}^\delta)\ell_{m}^-(s_{-k, j-p}^\pm)
\;\mbox{if }  n >0.
\end{equation*}

\subsection{Construction of the terms of the asymptotic expansions}
The matching conditions then allow us to construct the far field terms
$u_{n,q}^\delta$ and $\Pi_{n,q}^\delta$, 
and the near field terms $U_{n,q,\pm}^\delta$ by induction on $n$. The base case is obvious
since  we have seen that the macroscopic terms 
$u_{0,q}^\delta$ are entirely determined by Proposition~\eqref{PropDefinitions0q}, the
boundary layer correctors $\Pi_{0,q}$ are defined by~\eqref{definitionPiq},  and the near field
terms $U_{0,q, \pm}^\delta=0$, $q \in \IN$ by~\eqref{eq:U0q=0}.\\
 
\noindent Then, assuming that $u_{m,q}^\delta$ and $U_{m,q,\pm}^\delta$ are
constructed for any $m \leq n-1$, we will see that
\eqref{MactchingChampProche},\eqref{matchingChampLointain} and  \eqref{MactchingChampProcheMoins},\eqref{matchingChampLointainMoins} permit
to define  both $u_{n,q}^\delta$ (and consequently $\Pi_{n,q}^\delta$)
and $U_{n,q,\pm}^\delta$ for any $q \in \N$.  

\paragraph{Far field terms} We remind that,
for a given $q \in \N$, the complete definition of the macroscopic terms $u_{n,q}^\delta$
only requires the knowledge of the $\ell_{-m}^\pm(u_{n,q}^\delta)$ for any
integer $m$ between $1$ and $n$. In fact, the
conditions~\eqref{matchingChampLointain} define exactly $\ell_{-m}^+(u_{n,q}^\delta)$: in the right-hand side
of~\eqref{matchingChampLointain}, the quantities
$\mathscr{L}_{-m}(S_k^+)$ are known ($S_k^+$ is uniquely
defined) and $\mathscr{L}_k(U_{n-m,q,+}^\delta)$ are known since
$U_{n-m,q,+}^\delta$ are already defined (induction hypothesis). In
addition, since
\begin{equation*}
a_{n-3k,q,-m+3k,+}^\delta = \begin{cases}
\dsp \sum_{\pm} \sum_{r=\max(1, m-3k)}^{n-3k} \sum_{p=0}^q
\ell_{-r}^\pm(u^\delta_{n-3k,p}) \,
\ell_{-m+3k}^+(s_{-r,q-p}^\pm) & \mbox{if} \; n -  3k \neq 0,\\
\dsp \ell_{-m+3k}^+(s_{0,q})  &  \mbox{if} \; n -  3k = 0,
\end{cases}
\end{equation*}
the coefficient $a_{n-3k,q,-m+3k,+}^\delta$ is well defined for any $k$ such that $1 \leq k\leq
\lfloor n/3\rfloor$ ($\ell_{-j}^+(u_{n-3k,q})$ is known by
the induction hypothesis). Naturally the
conditions~\eqref{matchingChampLointainMoins} define
$\ell_{-m}^-(u_{n,q}^\delta)$ is the same way. Finally,
the definition of the boundary layer terms $\Pi_{n,q}^\delta$ follows from~\eqref{definitionPiq}.

\paragraph{Near field terms} Similarly, the definition
of the near field terms $U_{n,q,\pm}^\delta$ requires the specification
of the quantities  $\mathscr{L}_m(U_{n,q,\pm}^\delta) $ for any integer
$m$ between $1$ and $n$. The condition~\eqref{MactchingChampProche}
exactly provides this missing information for $U_{n,q,+}^\delta$. Indeed, in the right-hand
side of~\eqref{MactchingChampProche}, the computation of
$\tilde{a}_{n-m,q,m,+}^\delta$ requires the knowledge of  
\begin{equation*}
a_{n-m - 3k,q,m+3k,+}^\delta = \begin{cases}
\dsp \sum_{\pm} \sum_{r=0}^{n-m-3k} \sum_{p=0}^q
\ell_{-r}^\pm(u_{n-m-3k,p}^\delta) \,
\ell_{m+3k}^+(s_{-r,q-p}^\pm) & \mbox{if} \; n-m-3k\neq 0,\\
\dsp \ell_{m+3k}^+(s_{0,q})   &  \mbox{if} \; n - m - 3k = 0.
\end{cases}
\end{equation*}
for $k$ between $0$ and $\lfloor (n-m)/3 \rfloor$. But, since $m>0$, $\ell_{-r}^+(u_{n-m-3k,p}^\delta)$ are well defined  thanks to the
induction hypothesis. Then, $U_{n,q,+}^\delta$ is entirely determined.
In the same way, the condition~\eqref{MactchingChampProcheMoins} allows us
to define $U_{n,q,-}^\delta$ as well, replacing 
all occurences of $\ell^+_m(s^\pm_{-r,q-p})$ by $\ell^-_m(s^\pm_{-r,q-p})$, $m \in \IZ$
and all occurences $\ell^+_m(s_{0,q})$ by $\ell^-_m(s_{0,q})$, $m \in \IZ$
in the previous formulas.

\begin{rem} We point out that the variables $n$ and $q$ play a very
  different roles in the recursive construction of the terms of the
  asymptotic expansion. Indeed, the construction is by induction only in
  $n$. At the step $n$, we construct $u_{n,q}^\delta$,
  $\Pi_{n,q}^\delta$ and $U_{n,q,\pm}^\delta$ for any $q \in \N$.
\end{rem}



\section{Justification of the asymptotic expansion}\label{SectionErrorAnalysis}

To finish this paper, we shall prove the convergence of the asymptotic
expansion. Our main result deals with the convergence of the
truncated macroscopic series in a domain that excludes the two
corners and the periodic thin layer:

\begin{theo}\label{TheoremConvergenceMacro}
Let $N_0 > 0$ such that $3N_0$ is an integer and let $D_{N_0}$ denote the set of couples
$(n,q) \in \N^ 2$ such that $\frac{2}{3} n + q \leq N_0$. Furthermore, for a given
number $\alpha>0$, let $$\Omega_{\alpha} =   \Omega^\delta \setminus  (-L -\alpha, L+\alpha) \times (-\alpha, \alpha).$$
\noindent  Then, there exist a constant $\delta_0>0$, a constant $C = C(\alpha,\delta_0)>0$, and a constant $k = k(N_0) \geq 0$ such that
for any $\delta \in (0,\delta_0)$ 
\begin{equation}
  \| u^\delta - \sum_{(n,q)\in D_{N_0}} \delta^{\frac{2}{3} n +q}
  u_{n,q}^\delta \|_{H^1(\Omega_{\alpha})} \leq \; C \, \delta^{N_0+\frac13}(\ln\delta)^k.
\end{equation}
\end{theo}
\begin{rem}
  With more sophisticated techniques than applied in this article it is possible to prove that 
  the power of $\ln\delta$ in the previous theorem is $k = k(N_0) = \lfloor \frac12(N_0 + \frac13) \rfloor$.
  The first logarithmic term appears for $N_0 = \frac53$.
\end{rem}

\subsection{The overall procedure}
\noindent As usual for this kind of work (See \eg~\cite{fente2}
(Sect. 3), \cite{CouchesMinces} (Sect. 5.1), \cite{MR2935369} (Sect. 4)), the proof of the previous result is based on the
construction of an approximation $u_{N_0}^\delta$ of the solution $u^\delta$ on the
whole domain $\Omega^\delta$ obtained from 
the four truncated series (at order $N_0$) of the macroscopic terms, the boundary layer terms and the near field terms:
\begin{enumerate}
\item[-] The truncated macroscopic series $u_{\macro,N_0}^\delta$: we
  introduce the macroscopic cut-off function
\begin{multline}
  \label{eq:chi_macro_total}
  \chi_{\text{macro}, \text{total}}^\delta(\bx)  = \chi_+\left( \frac{x_1 - L}{\delta}\right) \chi_-\left(
      \frac{x_1+L}{\delta}\right) \chi \left(
      \frac{x_2}{\delta}\right)  \\
+  \sum_{\pm }\chi_{\text{macro},
      \pm} \left(\frac{x_1\mp L}{\delta},\frac{x_2}{\delta}\right)
    \left( 1 - \chi_\mp\left( \frac{x_1\mp L}{\delta}\right) \right),
  \end{multline}
which is equal to $1$ for $|x_1|>L$, and which coincides with
$\chi(\frac{x_2}{\delta})$ in the region $|x_1| < L-\delta$ (see
Fig.~\ref{DessinTroncatureTotal}). The cut-off functions $\chi$,
$\chi_\pm$ and $\chi_{\text{macro},
      \pm}$ are defined in
    \eqref{defchi}-\eqref{DefinitionChiplusmoins} and \eqref{eq:chi_macro_+}.  We then define the macroscopic approximation as follows:
\begin{equation}
u_{\macro,N_0}^\delta = \chi_{\text{macro}, \text{total}}^\delta(\bx) \sum_{(n,q) \in D_0} \delta^{\frac{2}{3}n+q} u_{n,q}^\delta(\mx).
\end{equation}
\item[-] The truncated periodic corrector series $\Pi_{N_0}^\delta$: it is given by
\begin{equation}
\Pi_{N_0}^\delta(\mx) = (1-\chi(x_2)) \chi_+(\frac{x_1 -L}{\delta}) \chi_{-} (\frac{x_1 + L}{\delta}) \sum_{(n,q) \in D_0} \delta^{\frac{2}{3}n+q} \Pi_{n,q}^\delta(x_1, \frac{\mx}{\delta}).
\end{equation}
The function $\chi_+(\frac{x_1 -L}{\delta}) \chi_{-} (\frac{x_1 +
  L}{\delta})$ permits us to localize the function $\Pi_{N_0}^\delta(\mx)
$ in the domain $|x_1|<L$ while the introduction of the function $(1-\chi(x_2)) $ ensures
that $\Pi_{N_0}^\delta(\mx)$ vanishes on $\Gamma_{D}$.
\item[-] The truncated near field series $U_{N_0,\pm}^\delta(\mx)$: 
\begin{equation}
U_{N_0,\pm}^\delta(\mx) =  \sum_{(n,q) \in D_0} \delta^{\frac{2}{3}n+q} U_{n,q}^\delta(\frac{\mx}{\delta}).
\end{equation}
\end{enumerate}

 \begin{figure}[tb]
   \centering
       \includegraphics[width=0.4\textwidth]{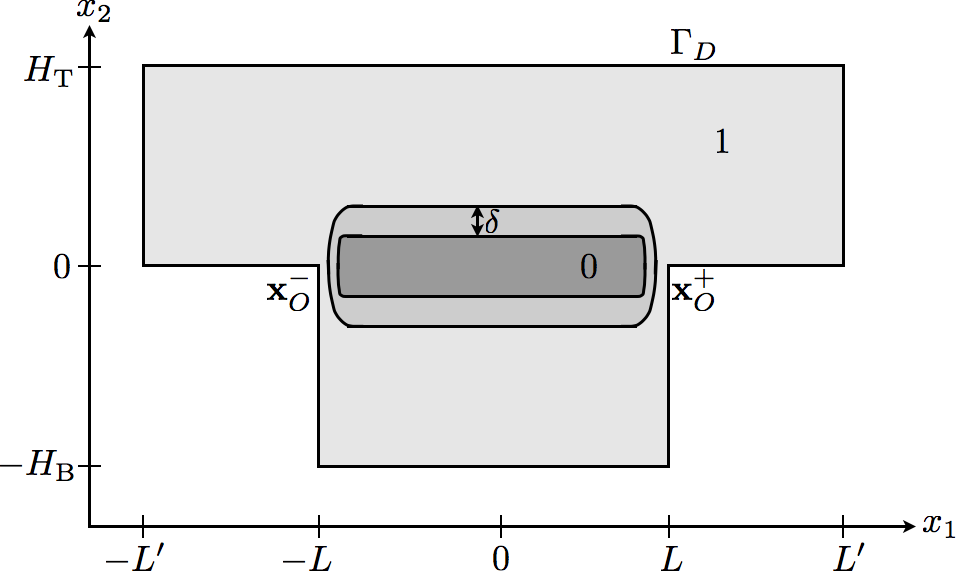}
  \caption{Schematic representation of the cut-off function $\chi_{\text{macro}, \text{total}}^\delta$.}
 \label{DessinTroncatureTotal}
 \end{figure}

\noindent We shall construct a global approximation  $u_{N_0}^\delta$ that
coincides with $U_{N_0,\pm}^\delta$ in the vicinity of the two
corners, with $\Pi_{N_0}^\delta$ in the vicinity of the periodic layer
and with $u_{\macro,N_0}^\delta$ away from the corners and the periodic
layer.  To do so, we introduce
the cut off functions 
\begin{equation}
\chi_+^\delta(\mx) =\chi\left(\frac{r^+}{\eta(\delta)}\right)\quad \mbox{and}
\quad \chi_-^\delta(\mx) =\chi\left(\frac{r^-}{\eta(\delta)}\right),
\end{equation}
where  $\eta(\delta) : \R^+ \rightarrow \R^+$ is a smooth function such that
\begin{equation}
\lim_{\delta \rightarrow 0} \eta(\delta) = 0 \; \mbox{and} \;
\lim_{\delta \rightarrow 0} \frac{\eta(\delta)}{\delta} = +\infty.
\end{equation}
For instance for $s \in (0,1)$, $\eta(\delta) = \delta^s$ satisfies these
conditions.  Finally the global approximation $u_{N_0}^\delta$
of $u^\delta$ is defined by
\begin{equation}
u_{N_0}^\delta(\mx) =\chi_+^\delta(\mx)\, U_{N_0,+} \; +\;
\chi_-^\delta(\mx) \, U_{N_0,-} \; + \; (1 -  \chi_+^\delta(\mx) -
\chi_-^\delta(\mx)) \, (u_{\macro,N_0}^\delta(x)  +\Pi_{N_0}^\delta(x) ).
\end{equation}
Note that $u_{N_0}^\delta$ belongs to $H^1_{\Gamma_D}(\Omega^\delta)$
but does not satisfy homogeneous Neumann boundary conditions on
$\Gamma^\delta$.  \\

\noindent The aim of this part is to estimate the $H^1$-norm
of the error $e_{N_0}^\delta = u^\delta - u_{N_0}^\delta$ in
$\Omega^\delta$ (We remind that $u^\delta \in \HoneGammaD(\Omega^\delta) $ is the 'exact' solution, \ie the solution of Problem~\eqref{eq:perturbed_laplace}). It is in fact
sufficient to estimate the residue
$\Delta  e_{N_0}^\delta$ and the Neumann trace $\partial_n  e_{N_0}^\delta$. Then, the estimation of $\| e_{N_0}^\delta
\|_{H^1(\Omega^\delta)}$ directly results from a straightforward
modification of the uniform stability
estimate~\eqref{eq:prop_norm_u_norm_f} (Proposition~\ref{prop:existence_uniqueness_u_delta}): there exists a constant $C>0$ such
that, for $\delta$ small enough,
\begin{equation}\label{estimationResidu}
\| e_{N_0}^\delta \|_{H^1(\Omega^\delta)} \leq C \left( \| \Delta
e_{N_0}^\delta \|_{L^2(\Omega^\delta)} + \| \partial_n e_{N_0}^\delta \|_{L^2(\Gamma^\delta)} \right) . 
\end{equation}
The main work of this part consists in proving the following proposition:
\begin{prop}\label{propositionConvergenceResidu}
There exist a constant $C>0$ and a constant $\delta_0>0$ such that,
for any $\varepsilon>0$,
for any $\delta \in (0,\delta_0)$,
\begin{equation}
  \| \Delta e_{N_0}^\delta \|_{L^2(\Omega^\delta)} + \| \partial_n e_{N_0}^\delta \|_{L^2(\Gamma^\delta)}  \leq C \left( \delta^{-2-\varepsilon}
\left( \frac{\delta}{\eta(\delta)}\right)^{N_0} +
\delta^{-1} \eta(\delta)^{N_0-\frac{1}{6} -\varepsilon}\right).
\end{equation}
\end{prop}
\noindent As a direct corollary, choosing $\eta(\delta) = \sqrt{\delta}$,
$\varepsilon= \frac{1}{2}$, we
obtain the following global error estimate:
there exist a constant $C>0$ and a constant $\delta_0>0$ such that,
for any $\delta \in (0,\delta_0)$,
\begin{equation}\label{errorGlobale}
  \| e_{N_0}^\delta \|_{H^1(\Omega^\delta)} \leq \;C \,\delta^{\frac{N_0}{2}-\frac{5}{2}}.
\end{equation}
Since $e_{N_0}^\delta $ coincides with $u^\delta-
\sum_{(n,q)\in D_{N_0}} \delta^{\frac{2}{3} n +q} u_{n,q}^\delta$ in $\Omega_\alpha$ for $\delta$ small enough,
Theorem~\ref{TheoremConvergenceMacro} follows from
\eqref{errorGlobale} and the triangular inequality. \\

\noindent The remainder of this section is dedicated to the proof of
Proposition~\ref{propositionConvergenceResidu}. Although long and
technical, the proof is rather standard. 

\begin{rem}
We emphasize that $u_{N_0}^\delta$ is certainly not the best choice to minimize
the global error. As shown in~\cite{DaugeTordeuxVialVersionLongue}, a global estimate based on the truncated far and
near field terms obtained by the compound method might provide a
better global error. Nevertheless, for the sake of simplicity and since we are mainly interested
in the macroscopic error estimate (that can always be made optimal
thanks to the triangular inequality), we prefer using here 
$u_{N_0}^\delta$.
\end{rem}
\subsection{Decomposition of the residue into a modeling error and a matching error}
Remarking that the supports of the derivatives of $\chi_+^\delta$ and
$\chi_-^\delta$ are disjoint (for $\delta$ small enough), using additionaly that $\Delta U_{N_0,\pm}^\pm=0$, we can see that
\begin{equation}\label{DecompositionErreur}
-\Delta e_{N_0}^\delta  = \mathcal{E}_{\mod} + \mathcal{E}_{\match}, 
\end{equation}
where,
\begin{equation}\label{definitionErreurRaccord}
\mathcal{E}_{\match} = - \sum_{\pm} [ \Delta , \chi_\pm^\delta(\mx)] (U_{N_0,\pm} -u_{\macro,N_0}^\delta(x)  -\Pi_{N_0}^\delta(x)),
\end{equation}
and
\begin{equation}
 \mathcal{E}_{\mod}   = f  - (1 -  \chi_+^\delta(\mx) - \chi_-^\delta(\mx))   \Delta \left( u_{\macro,N_0}^\delta(x)  +\Pi_{N_0}^\delta(x)  \right).
\end{equation}

\noindent Here, $\mathcal{E}_{\match}$ represents the matching error. Its support, which 
coincides with the union of the supports of $\nabla   \chi_+^\delta$
and  $\nabla \chi_-^\delta$, is
included in the union of the rings $\eta(\delta)<|r^\pm|<2 \eta(\delta)$. It
measures the mismatch between the far and near field truncated
expansions in the matching zones.  $\mathcal{E}_{\mod}$, representing the
modeling error (or consistency error), measures how the expansion fails to satisfies the
original Laplace problem.\\

\noindent Similarly,  is it easily seen that
\begin{equation}\label{erreurDeriveNormale}
\partial_{n} e_{N_0}^\delta  = \sum_{\pm}  \partial_n
\chi_\pm^\delta(\mx) \left( U_{N_0,\pm} - \Pi_{N_0}^\delta(x)\right),
\end{equation}  
so that the error on the boundary data in supported in the matching
areas. Therefore its treatment will be similar to the one of the matching error.\\

\noindent In the next two sections, we shall estimate in turn
$\mathcal{E}_{\mod}$ (Section~\ref{SectionModelingError}, Proposition~\ref{lema:global_modeling_error})
and $\mathcal{E}_{\match}$
(Section~\ref{SectionMatchingError} Proposition~\ref{lema:global_matching_error}).  The proof of
Proposition~\ref{propositionConvergenceResidu} results from~\eqref{DecompositionErreur}-\eqref{erreurDeriveNormale} and a direct application of these
two propositions.
\subsection{Estimation of the modeling error}\label{SectionModelingError}
The present section is dedicated to the proof of the following error estimate:
\begin{prop}
  \label{lema:global_modeling_error}
  There exist a positive constant $\mathcal{C}_{\mod}>0$ and 
  a positive number $\delta_0 >0$, such that, for any $\delta \in (0,
  \delta_0)$, 
  \begin{equation}
    \label{eq:global_modeling_error}
  \| \mathcal{E}_{\mod} \|_{L^2(\Omega^\delta)}  \; \leq \; \mathcal{C}_{\mod}  \left(
  \frac{\delta}{\eta(\delta)}\right)^{N_0} \,\delta^{-\varepsilon -2}.
  \end{equation}
\end{prop}
\noindent We first note that  the intersection of the supports of  $\nabla \chi_{\text{macro},
  \text{total}}^\delta(\bx)$  (and $\Delta \chi_{\text{macro},
  \text{total}}^\delta(\bx)$) and $1 -  \chi_+^\delta(\mx)
- \chi_-^\delta(\mx)$ is
included in the set $$\Omega_{\mod}^\delta = \{(x_1, x_2) \in
\Omega^\delta, |x_1| \leq L -\sqrt{ \eta(\delta) - 4
    \delta^2}\}.$$
 Moreover, $\chi_{\text{macro},
  \text{total}}^\delta(\bx) = \chi(\frac{x_2}{\delta})$ on this set. As a result,
\begin{equation}\label{ErreurModel1}
  f - (1 -  \chi_+^\delta(\mx) - \chi_-^\delta(\mx)) \Delta
  u_{\macro,N_0}^\delta(x)  = - (1 -  \chi_+^\delta(\mx) -
  \chi_-^\delta(\mx)) 1_{\Omega_{\mod}^\delta} [\Delta, \chi(\frac{x_2}{\delta}) ] \left( \sum_{(n,q) \in D_{N_0}} \delta^{\frac{2}{3}n+q} u_{n,q}^\delta(\mx)\right),
\end{equation}
where $1_{\Omega_{\mod}^\delta}$ denotes the
indicator function of $\Omega_{\mod}^\delta$.  In the previous formula, we used the macroscopic
equations~\eqref{FFVolum} (the functions $u_{n,q}^\delta$ are harmonic in $\OmegaTop \cup \OmegaBottom$ unless
for $n=q=0$ where $-\Delta u_{0,0}=0$). 
On the other hand,
\begin{equation}\label{ErreurModel2}
  -(1 -  \chi_+^\delta(\mx) - \chi_-^\delta(\mx)) \Delta
  \Pi_{N_0}^\delta(x)  =  -(1 -  \chi_+^\delta(\mx) -
  \chi_-^\delta(\mx))   1_{\Omega_{\mod}^\delta} \sum_{(n,q)\in D_{N_0}} \delta^{\frac{2}{3} n +q}
  \Delta \Pi_{n,q}^\delta  + \mathcal{E}_{\mod, 1}
\end{equation}
where,
\begin{multline}
 \mathcal{E}_{\mod, 1}= - (1 -  \chi_+^\delta(\mx) -
\chi_-^\delta(\mx)) \, \left(\chi_+(\frac{x_1 -L}{\delta}) \chi_{-} (\frac{x_1 + L}{\delta}) (1 -\chi(x_2)) -  1_{\Omega_{\mod}^\delta}   \right)  \sum_{(n,q)\in D_{N_0}} \delta^{\frac{2}{3} n +q} \Delta \Pi_{n,q}^\delta \\
  + (1 -  \chi_+^\delta(\mx) - \chi_-^\delta(\mx)) [ \Delta, \chi_+(\frac{x_1 -L}{\delta}) \chi_{-} (\frac{x_1 + L}{\delta}) (1 -\chi(x_2))] \sum_{(n,q)\in D_{N_0}} \delta^{\frac{2}{3} n +q} \Pi_{n,q}^\delta.
\end{multline}
Collecting~\eqref{ErreurModel1} and \eqref{ErreurModel2}, we end
up with
\begin{equation}\label{decompositionErreurModel}
\mathcal{E}_{\mod} = \mathcal{E}_{\mod, 1} +  \mathcal{E}_{\mod, 2},
\end{equation}
with
\begin{equation}
 \mathcal{E}_{\mod, 2} = -(1 -  \chi_+^\delta(\mx) -
  \chi_-^\delta(\mx))   1_{\Omega_{\mod}^\delta} \sum_{(n,q)\in D_{N_0}}
  \delta^{\frac{2}{3} n +q} \left ( \Delta  \Pi_{n,q}^\delta(x_1, \frac{\mx}{\delta})  +
    [\Delta, \chi(\frac{x_2}{\delta}) ] u_{n,q}^\delta(\mx)\right).
\end{equation}
$\mathcal{E}_{\mod, 1}$ is supported in a domain where
$|x_2|> C \eta(\delta)$ wherein the periodic
correctors $\Pi_{n,q}^\delta$ are exponentially decaying. As a result, $\mathcal{E}_{\mod, 1}$ converges super-algebraically
fast to $0$. More precisely, we can prove
the following lemma, whose proof is left to the reader:
\begin{lema}\label{LemmeErreurMod1}
For any $N \in \N$, there exists a positive
constant $C_{N}>0$ and a positive number $\delta_0>0$, such that, for any $\delta<\delta_0$,
\begin{equation}
 \left\| \mathcal{E}_{\mod, 1} \right\|_{L^2(\Omega^\delta)} \leq C_{N}\delta^N. 
\end{equation} 
\end{lema}
\noindent It remains to estimate $\mathcal{E}_{\mod, 2}$. Analogously to the case of an
infinite thin periodic layer, we naturally use the 
periodic correctors
equations~\eqref{PeriodicCorrectorEquations}. Nevertheless, the
estimation requires a careful analysis because the fields
$\Pi_{n,q}^\delta$ and $u_{n,q}^\delta$ are singular. 
We prove the
following lemma, whose proof is postponed in Appendix~\ref{AppendixProofModellingError}:
\begin{lema}\label{LemmeErreurMod2}
For any $\varepsilon>0$  sufficiently small, there exists a positive
constant $C>0$ and a positive number $\delta_0>0$, such that, for any $\delta<\delta_0$,
\begin{equation}
\| \mathcal{E}_{\mod, 2} \|_{L^2(\Omega^\delta)} \leq  C \left(
  \frac{\delta}{\eta(\delta)}\right)^{N_0} \,\delta^{-\varepsilon -2}.
\end{equation} 
\end{lema}
\noindent Obviously, Proposition~\ref{lema:global_modeling_error} is a straightforward consequence of \eqref{decompositionErreurModel},
Lemma~\ref{LemmeErreurMod1}, and Lemma~\ref{LemmeErreurMod2}.

\subsection{Estimation of the matching
  error}\label{SectionMatchingError}
We now turn to the estimation of the matching error:
\begin{prop}
  \label{lema:global_matching_error}
  There exist a positive constant $\mathcal{C}_{\text{match}}>0$ and a
  positive number $\delta_0 >0$, such that, for any $\delta <
  \delta_0$, 
  \begin{equation}
    \label{eq:global_modeling_error}
  \| \mathcal{E}_{\match} \|_{L^2(\Omega^\delta)}  \leq \mathcal{C}_{\text{match}} \left(
  \eta(\delta)^{-1} \, \left( \frac{\delta}{\eta(\delta)} \right)^{N_0
  -1- \varepsilon} + \delta^{-1} \eta(\delta)^{N_0 +
  \frac{1}{3} -\varepsilon} \right),
  \end{equation}
and
\begin{equation} \label{eq:global_modeling_errorTrace}
\| \partial_n
  e_{N_0}^\delta \|_{L^2(\Gamma^\delta)}   \leq \mathcal{C}_{\text{match}} \left(
  \eta(\delta)^{-\frac{1}{2}} \, \left( \frac{\delta}{\eta(\delta)} \right)^{N_0
  -\varepsilon-5/3} + \delta^{-\frac{1}{2}} \eta(\delta)^{N_0 -
  \frac{2}{3} -\varepsilon} \right).
\end{equation}
\end{prop}

\noindent We shall evaluate $\| \mathcal{E}_{\match} \|_{L^2(\Omega^\delta)}$
and $\| \partial_n
  e_{N_0}^\delta \|_{L^2(\Gamma^\delta)} $
in turn. Both of these functions are supported in the overlapping
areas. The function
$\mathcal{E}_{\match}$ is supported in the overlapping
areas $\Omega_{\match}^\pm$ 
$$
\Omega_{\match}^\pm = \left\{ (x_1,x_2) \in \Omega^\delta,
  \eta(\delta) \leq r^\pm \leq 2 \eta(\delta)\right\}.
$$
and $\| \partial_n
  e_{N_0}^\delta \|_{L^2(\Gamma^\delta)}$ is supported in
$$
\Gamma_{\match}^\pm  = \left\{ (x_1,x_2) \in \Gamma^\delta,
  \eta(\delta) \leq r^\pm \leq 2 \eta(\delta)\right\}.
$$
We shall estimate $\| \mathcal{E}_{\match}\|_{L^2(\Omega_{\match}^+)}$
(resp. $\| \partial_n
  e_{N_0}^\delta \|_{L^2(\Gamma_{\match}^+)}$)
but a similar analysis can be made for $\|
\mathcal{E}_{\match}\|_{L^2(\Omega_{\match}^-)}$ (resp. $\| \partial_n
  e_{N_0}^\delta \|_{L^2(\Gamma_{\match}^-)}$)). \\

\noindent We start with the computation of $U_{N_0, +}^\delta -
u_{\macro,N_0}^\delta(x) - \Pi_{N_0}^\delta$, which, thanks to the
matching conditions \eqref{matchingCondition1}, is expected to be
small. The following computation is based on an expansion of the
truncated series of far and near field terms in the overlapping area.

\subsubsection{Expansion of $u_{\macro,N_0}^\delta$,
  $\Pi_{BL,N_0}^\delta$ and $U_{N_0,+}^\delta$ in the overlapping area $\Omega_{\match}^+$}

For any couple $(n,q) \in \N^2$, we consider the integer $k(n,q,N_0)$
given by
$$
k(n,q,N_0) = \left\{ \begin{array}{ll}
 \frac{3}{2} (N_0-q) -n  & \mbox{if $N_0$ and $q$ have the same parity},\\[1ex]
 \frac{3}{2} (N_0-q) -n - \frac{1}{2} & \mbox{otherwise}.
\end{array} \right.
$$
\paragraph{Macroscopic truncated series $u_{\macro,N_0}^\delta$.} In view of \eqref{AsymptoticMacroFieldMacthingareas},  for any $(n,q) \in D_{N_0}$,  
there exists a function $\mathcal{R}_{\macro, n,q,\delta} \in V_{2,\beta}^2(\OmegaTop)\cap V_{2,\beta}^2(\OmegaBottom)$
for any $\beta > 1 -\frac{2 (k(n,q,N_0)+1)}{3}$
\begin{equation}\label{definitionUtildenq}
u_{n,q}^\delta  = \tilde{u}_{n,q}^\delta + \mathcal{R}_{\macro,
  n,q,\delta} \quad \quad \tilde{u}_{n,q}^\delta = \sum_{r=0}^q \sum_{m=-n}^{\frac{3}{2}(N_0
  +r-q) -n } a_{n,q-r,m,+}^\delta (r^+)^{\frac{2m}{3} -r} w_{m,r,+}(\theta^+,
\ln r^+).
\end{equation}
The reader may verify that,
  for non negative integer $r$,
$k(n,q,N_0)+ \lfloor  \frac{3 r}{2}\rfloor \leq \lfloor \frac{3}{2}(N_0
  +r-q) -n \rfloor$. Then, a direct computation shows that
\begin{multline*}
\sum_{(n,q) \in D_{N_0}} \delta^{\frac{2n}{3}+q}
\tilde{u}_{n,q}^\delta(R^+\delta, \theta^+) = \\  \sum_{(n,q) \in
  D_{N_0}} \delta^{\frac{2}{3} n +q} 
\sum_{m=-\alpha_{n,q}}^{n} 
\sum_{r=0}^{\frac{2}{3} (\alpha_{n,q} +m)}   a_{n,q,m-n,+}^\delta (R^+)^{\frac{2m}{3} -r} w_{m,r,+}(\theta^+,
\ln (R^+ \delta)), 
\end{multline*}
where $\alpha_{n,q } = \frac{3}{2} (N_0 -q)-n$. Then, using
Lemma~\ref{LemmeTechniqueChangementEchelle}, reproducing the
calculations of \eqref{FormuleFFPretePourlesraccordsX1positif2} and
using the matching conditions~\eqref{matchingCondition1}, we
see that
\begin{equation}
\sum_{(n,q) \in D_{N_0}} \delta^{\frac{2n}{3}+q}
\tilde{u}_{n,q}^\delta(R^+\delta, \theta^+) =\sum_{(n,q) \in D_{N_0}}
\delta^{\frac{2}{3} n +q}  \sum_{m=-\alpha_{n,q}}^n
{A}_{n,q,m,+}^\delta \sum_{r=0}^{\frac{2}{3} (\alpha_{n,q} +m)}
(R^+)^{\frac{2}{3} m -r } w_{m,r,+}(\theta^+, \ln R^+).
\end{equation}
Finally, noticing that $\chi_{\macro,\text{total}}^\delta(x)  =
\chi_{\macro,+}^\delta\left(\frac{x_1 -L}{\delta}, \frac{x_2}{\delta}\right)$ in $\Omega_{\match}^+$, the truncated
macroscopic series $u_{\macro,N_0}^\delta$ can be written as
\begin{multline}\label{AsymptoticuFF}
u_{\macro,N_0}^\delta(\mx) = \chi_{\macro,+}(\mX^+) \left( \sum_{(n,q) \in D_{N_0}}
\delta^{\frac{2}{3} n +q}  \sum_{m=-\alpha_{n,q}}^n
{A}_{n,q,m,+}^\delta \sum_{r=0}^{\frac{2}{3} (\alpha_{n,q} +m)}
(R^+)^{\frac{2}{3} m -r } w_{m,r,+}(\theta^+, \ln R^+)\right)\\+ \mathcal{R}_{\macro,N_0}^\delta(\mx)
\end{multline}
where
\begin{equation}
\mathcal{R}_{\macro,N_0}^\delta(\mx) =  \chi_{\macro,+}(\mX^+) \sum_{(n,q) \in D_{N_0}}
\delta^{\frac{2}{3} n +q} \mathcal{R}_{\macro, n,q,\delta}(\mx).
\end{equation}
\paragraph{Boundary layer correctors series $\Pi_{BL,N_0}^\delta$.}
Similarly, the asymptotic
formula~\eqref{AsymptoticBoundaryLayerMacthingareas} for the periodic
corrector $\Pi_{n,q}^\delta$ associated with the matching
conditions~\eqref{matchingCondition1} gives
\begin{multline}\label{AsymptoticPiF}
\Pi_{BL,N_0}^\delta(\mx) = \chi_{-}(\mX^+) \left( \sum_{(n,q) \in D_{N_0}}
\delta^{\frac{2}{3} n +q}  \sum_{m=-\alpha_{n,q}}^n
{A}_{n,q,m,+}^\delta \sum_{r=0}^{\frac{2}{3} (\alpha_{n,q} +m)}
|X_1^+|^{\frac{2}{3} m -r } p_{m,r,+}( \ln |X_1^+|, \mX^+)\right)\\+ \mathcal{R}_{BL,N_0}^\delta(\mx).
\end{multline}
The remainder $\mathcal{R}_{BL,N_0}^\delta(\mx)$  can be written as
\begin{equation}
\mathcal{R}_{BL,N_0}^\delta(\mx) = \chi_{-}(\mX^+)  \sum_{(n,q) \in
  D_{N_0}} \delta^{\frac{2}{3} n + p} \mathcal{R}_{\macro,
  n,q,\delta}(\mx), \quad \mathcal{R}_{BL, n,q,\delta}(\mx) = \sum_{j=0}^{K}  \langle
w_{n,q,j}(x_1, 0)
\rangle  W_{n,q,j}(\mX)
\end{equation} 
where the functions $w_{n,q,j} \in V_{2,\beta+1}^3(\OmegaTop)\cap V_{2,\beta+1}^3(\OmegaBottom)$
for any $\beta > 1 -\frac{2 (k(n,q,N_0)+1)}{3}$, $W_{n,q,j}\in
\mathcal{V}^+(\mathcal{B})$ and $K$ is a given integer depending on
$n$ and $q$.
\paragraph{Near field truncated series $U_{N_0,+}^\delta$.}
\noindent The derivation of  the expansion  of the truncated near field
expansion $U_{N_0,+}^\delta$ is much more direct. It may be directly obtained using
formula~\eqref{AsymptoticUnqX1negatif} (taking $K = \lfloor \alpha_{n,q}
\rfloor$):
\begin{eqnarray}
U_{N_0,+}^\delta & = & \chi_{\macro,+}(\mX^+)   \left(\sum_{(n,q) \in D_{N_0}} \delta^{\frac{2}{3} n + q}  \sum_{m=-\alpha_{n,q}}^n
{A}_{n,q,m,+}^\delta \sum_{r=0}^{\frac{2}{3} (\alpha_{n,q} +m)}
(R^+)^{\frac{2}{3} m -r } w_{m,r,+}(\theta^+, \ln R^+)\right) \nonumber \\
&& + 
\chi_{-}(\mX^+) \left( \sum_{(n,q) \in D_{N_0}} \delta^{\frac{2}{3} n + q}    \sum_{m=-\alpha_{n,q}}^n
{A}_{n,q,m,+}^\delta \sum_{r=0}^{\frac{2}{3} (\alpha_{n,q} +m)}
|X_1^+|^{\frac{2}{3} m -r } p_{m,r,+}( \ln |X_1^+|, \mX^+)\right) \nonumber\\
&& + \mathcal{R}_{NF, N_0}^{\delta}, \label{AsymptoticUNF}
\end{eqnarray}
where
\begin{equation}
\mathcal{R}_{NF, N_0}^{\delta} = \sum_{(n,q) \in D_{N_0}}
\delta^{\frac{2}{3} n + q} \mathcal{R}_{NF, n,q}(\mX)  
\end{equation}
$\mathcal{R}_{NF,
  n,q}(\mX) $ belonging to $\mathfrak{V}_{\beta, \gamma}^2(\widehat{\Omega}^+)$ for any $\beta <
1 + \frac{2(\lfloor\alpha_{n,q} \rfloor+1)}{3}$, $\gamma \in (1/2,1)$, 
$\gamma -1$ sufficiently small.
\subsubsection{Evaluation of the remainder}
Subtracting ~\eqref{AsymptoticuFF} and \eqref{AsymptoticPiF}
to \eqref{AsymptoticUNF} gives
\begin{equation}\label{resteMactchingTotal}
U_{NF,N_0}^\delta -u_{\macro,N_0}^\delta-\Pi_{N_0}^\delta   = \mathcal{R}_{NF, N_0}^{\delta} - \mathcal{R}_{BL, N_0}^{\delta} - \mathcal{R}_{\macro, N_0}^{\delta}. 
\end{equation}
To evaluate the matching error, we shall consider separately the
three terms of the right hand side of the previous equality, estimating their $L^2$ norm and
the $L^2$ norm of their gradient over the domain $\Omega_{\match}^+$.  The proof of the following three lemmas
can be found in Appendix~\ref{AppendixProofMacthingError}.
\begin{lema}[Estimation of the macroscopic matching remainder]\label{LemmeMatchingResiduMacro}
For any $\varepsilon>0$, there is a positive constant $C>0$ such that
\begin{equation}
\| \mathcal{R}_{\macro, N_0}^{\delta}\|_{L^2(\Omega_{\match}^+)} \leq C
\eta(\delta)^{N_0 -\varepsilon+\frac{4}{3}}, \quad \mbox{and} \quad
\|\nabla \mathcal{R}_{\macro, N_0}^{\delta}\|_{L^2(\Omega_{\match}^+)}
\leq C \delta^{-1}
\eta(\delta)^{N_0 -\varepsilon+\frac{4}{3}}.
\end{equation}
\end{lema}

\begin{lema}[Estimation of the periodic corrector remainder]\label{LemmeMatchingResiduBL}
For any $\varepsilon>0$, there is a positive constant $C>0$ such that
\begin{equation}\label{ResiduBL}
\| \mathcal{R}_{BL, N_0}^{\delta} \|_{L^2(\Omega_{\match}^+)}
\leq \eta(\delta)^{N_0 + \frac{5}{3} - \varepsilon}
\quad  \| \nabla \mathcal{R}_{BL, N_0}^{\delta}
\|_{L^2(\Omega_{\match}^+)}  + \| \mathcal{R}_{BL, N_0}^{\delta} \|_{L^2(\partial \Omega_{\match}^+
  \cap \Gamma^\delta)} 
\leq C \delta^{-1/2} \eta(\delta)^{N_0 + \frac{1}{3} - \varepsilon}.
\end{equation} 
\end{lema}

\begin{lema}[Estimation of the near field matching remainder]\label{LemmeMatchingResiduNF}
For any $\varepsilon>0$, there is a positive constant $C>0$ such that
\begin{equation}
\| \mathcal{R}_{NF, N_0}^{\delta}\|_{L^2(\Omega_{\match}^+)} \leq C
\delta \left( \frac{\delta}{ \eta(\delta)} \right)^{N_0 -\varepsilon-2}, \quad
\|\nabla \mathcal{R}_{NF, N_0}^{\delta}\|_{L^2(\Omega_{\match}^+)} \leq C
\left( \frac{\delta}{ \eta(\delta)} \right)^{N_0 -\varepsilon-1},
\end{equation}
and 
\begin{equation}
\|\mathcal{R}_{NF,
  N_0}^\delta   \|_{L^2(\partial{\Omega_\match^+} \cap \Gamma^\delta)} \leq C
\delta^{\frac{1}{2}}    \left(
  \frac{\delta}{\eta(\delta)}\right)^{N_0  - \frac{13}{6}-\varepsilon}.
\end{equation}
\end{lema}
\noindent Finally, the proof of Proposition~\ref{lema:global_matching_error} is
a direct consequence of  \eqref{resteMactchingTotal},
Lemma~\ref{LemmeMatchingResiduMacro},
Lemma~\ref{LemmeMatchingResiduBL} and
Lemma~\ref{LemmeMatchingResiduNF}, and the following inequalities:
$$\| \nabla \chi_+^\delta(\mx)
\|_{L^\infty(\Omega_{\match}^+)} \leq C \eta(\delta)^{-1}, \quad
\| \Delta \chi_+^\delta(\mx)
\|_{L^\infty(\Omega_{\match}^+)} \leq C \eta(\delta)^{-2}\ .
$$

\begin{appendix}
\section{Far and near field equations: procedure of separation of the
  scales}\label{SectionScaleSeparation}

  In this appendix, we explain how, using the separation of scales, we
  \textbf{formally} get the
  macroscopic equations (\ref{FFVolum}, \ref{FFDirichlet}) for the functions
  $u_{n,q}^\delta$ and the the boundary layer equations
  \eqref{PeriodicCorrectorEquations} for the functions $\Pi_{n,q}^\delta$ with
  the right-hand side $G_{n,q}^\delta$ defined in \ref{eq:def_Gnq_delta}.

  \subsection{Derivation of the macroscopic equations}
  Let us consider a point $\mx_0 \in \OmegaTop \cup \OmegaBottom$, and let us
  consider a vicinity $V(\mx_0)$ of $\mx_0$ such that $V(\mx_0) \subset \OmegaTop
  \cup \OmegaBottom$. Then, if $| x_{0,1} | > L$, the macroscopic expansion
  \eqref{def_uFFnq} gives $u_{FF,n,q}^\delta (\mx) = u_{n,q}^\delta (\mx)$, and
  inserting \eqref{FFExpansion} in \eqref{eq:perturbed_laplace} gives
  \begin{equation}
    \label{eq:perturbed_laplace_expand_large_x1}
    \sum_{(n,q) \in \N^2} \delta^{\frac{2n}{3}+q} \Delta u_{n,q}^\delta = f
  \end{equation}
  To obtain~\eqref{FFVolum}, we  identify the different powers of $\delta^{\frac{2n}{3}+q}$, \emph{treating
    indexes $(n+3k,q)$ and $(n,q+2k)$ as different powers of $\delta$}.\\

  \noindent If $| x_{0,1} | < L$, the macroscopic expansion \eqref{def_uFFnq} gives
  $u_{FF,n,q}^\delta (\mx) = \chi\left( \frac{x_2}{\delta} \right)
  u_{n,q}^\delta (\mx) + \Pi_{n,q}^\delta(x_1, \mx / \delta)$, and
  inserting \eqref{FFExpansion} in \eqref{eq:perturbed_laplace} gives
  \begin{equation}
    \label{eq:perturbed_laplace_expand_small_x1}
    \sum_{(n,q) \in \N^2} \delta^{\frac{2n}{3}+q} \Delta \left( \chi\left(
        \frac{x_2}{\delta} \right) u_{n,q}^\delta + \Pi_{n,q}^\delta \left(x_1,
        \frac{\mx}{\delta} \right) \right) = f
  \end{equation}
  Then, we consider $\delta$ such that for any $\mx \in V(\mx_0)$, $|x_2| > 2
  \delta$; then the cut-off function $\chi\left( \frac{x_2}{\delta} \right)$ is
  equal to $1$, and thanks to \eqref{eq:exponential_decaying}, we can neglect
  the terms $\Delta\Pi_{n,q}^\delta(x_1, \mx / \delta)$. Finally, we
  extract then the same powers of $\delta^{\frac{2n}{3}+q}$, and we get
  equations \eqref{FFVolum}.
  \subsection{Derivation of the boundary layer equations}
  We consider again the equation \eqref{eq:perturbed_laplace_expand_small_x1}
  for any $|x_1| < 1$ (this time, we put no restriction on $x_2$). Using
  \eqref{FFVolum}, we
  \begin{equation}
    \label{eq:perturbed_laplace_expand_small_x1_r}
    \sum_{(n,q) \in \N^2} \delta^{\frac{2n}{3}+q}  \left( \frac{2}{\delta} \chi'\left(
        \frac{x_2}{\delta} \right) \partial_{x_2} u_{n,q}^\delta +
      \frac{1}{\delta^2} \chi''\left( 
        \frac{x_2}{\delta} \right) u_{n,q}^\delta + \Delta \Pi_{n,q}^\delta \left(x_1,
        \frac{\mx}{\delta} \right) \right) = 0
  \end{equation}
On the one hand, on the support of $\chi'(x_2 / \delta)$, we see that $x_2$ is small; then we
  can use an infinite Taylor expansion for both $u_{n,q}^\delta$ and its
  $x_2$-derivative:
  \begin{align}
    \label{eq:perturbed_laplace_expansion}
    u_{n,q}^\delta & = \sum_{p=0}^\infty \frac{x_2^p}{p!} \partial_{x_2}^p
    u_{n,q}(x_1, 0^\pm), \quad \pm x_2 > 0, \\
    \label{eq:perturbed_laplace_expansion_d}
    \partial_{x_2} u_{n,q}^\delta & = \sum_{p=0}^\infty
    \frac{x_2^p}{p!} \partial_{x_2}^{p+1}
    u_{n,q}(x_1, 0^\pm), \quad \pm x_2 > 0
  \end{align}
  We insert \eqref{eq:perturbed_laplace_expansion} and
  \eqref{eq:perturbed_laplace_expansion_d} in
  \eqref{eq:perturbed_laplace_expand_small_x1_r}. Since the functions
  $\Pi_{n,q}^\delta$ depends on the fast variable $X_2 = x_2 / \delta$, we make
  the variable change in the Taylor expansions as well. Then, reordering with
  the same powers of $\delta^{\frac{2n}{3}+q}$, we get
  \begin{multline}
    \label{eq:perturbed_laplace_expand_small_x1_r2}
      \sum_{(n,q) \in \N^2} \delta^{\frac{2n}{3}+q-2} \Big( \sum_{\pm}
      \sum_{p=0}^q \chi_\pm''(X_2) \frac{(X_2)^p}{p!} \partial_{x_2}^p
      u_{n,q-p}(x_1,0^\pm) \\ + \sum_{\pm} \sum_{p=0}^{q-1} \chi_\pm''(X_2)
      \frac{(X_2)^p}{p!} \partial_{x_2}^{p+1} u_{n,q-1-p}(x_1,0^\pm) \Big) +
      \sum_{(n,q) \in \N^2} \delta^{\frac{2n}{3}+q} \Delta \Pi_{n,q}^\delta
      \left(x_1, \frac{\mx}{\delta} \right)  = 0.
  \end{multline}
{On the other hand,}  expanding the Laplacian of each function $\Pi_{n,q}^\delta$, separating the
  slow scale $x_1$ from the fast scale $\mX = \mx / \delta$, we get
  \begin{equation*}
    \Delta \Pi_{n,q}^\delta \left(x_1, \frac{\mx}{\delta} \right)
    = \partial_{x_1}^2 \Pi_{n,q}^\delta(x_1,\mX) +
    \frac{2}{\delta} \partial_{x_1} \partial_{X_1} \Pi_{n,q}^\delta(x_1,\mX) +
    \frac{1}{\delta^2} \Delta_{\mX} \Pi_{n,q}^\delta(x_1,\mX)
  \end{equation*}
  We use this last relation in \eqref{eq:perturbed_laplace_expand_small_x1_r2}
  and we identify with the same powers of $\delta^{\frac{2n}{3}+q}$, then we
  get \eqref{PeriodicCorrectorEquations} with the right-hand
  side defined in \eqref{eq:def_Gnq_delta}.

\section{Technical results associated with the analysis of the boundary layer problems: transmission
  conditions} 
\subsection{Proof of Proposition \ref{PropFormeGeneraleTransmission}}
\label{ProofPropositionGeneralTransmission}
To prove Proposition~\ref{PropFormeGeneraleTransmission}, it is
sufficient
to prove by induction that, any sequence $(u_{k},\Pi_k)$ solution to
Problem~(\ref{ProblemeSectionTransmission}) satisfies the following three
properties:
 \begin{align}
\Pi_q(x_1, \mX) & \;\;  = &  & \sum_{p=0}^q \partial_{x_1}^p \langle
u_{q-p}(x_1,0)
\rangle_{\Gamma} W_{p}^{\mathfrak{t}}(\mX) +\sum_{p=1}^{q} 
\partial_{x_1}^{p-1} \langle \partial_{x_2}
u_{q-p}(x_1,0) \rangle_{\Gamma} W_{p}^{\mathfrak{n}}(\mX), \label{DefPiqAnnexe}\\ 
[u_q(x_1,\mX)]_{\Gamma} & \;\; = &  & 
\sum_{p=1}^q \mathcal{D}_p^{\mathfrak{t}} \, \partial_{x_1}^p \langle
u_{q-p}(x_1,0)
\rangle_{\Gamma}+\sum_{p=1}^{q} 
\mathcal{D}_p^{\mathfrak{n}} \,  \partial_{x_1}^{p-1} \langle \partial_{x_2}
u_{q-p}(x_1,0) \rangle_{\Gamma},\label{DefSautuqAnnexe}
\\
[\partial_{x_2}u_{q-1}(x_1,\mX)]_{\Gamma} &\;\;  = & &  
\sum_{p=2}^q \mathcal{N}_p^{\mathfrak{t}} \, \partial_{x_1}^p \langle
u_{q-p}(x_1,0)
\rangle_{\Gamma}+\sum_{p=2}^{q} 
\mathcal{N}_p^{\mathfrak{n}} \,  \partial_{x_1}^{p-1} \langle \partial_{x_2}
u_{q-p}(x_1,0) \rangle_{\Gamma}. \label{DefSautDuqm1Annexe}
\end{align}
Here, we posed $W_p^{\mathfrak{t}} =0$ for any negative integer $p$, and, for $p\geq0$, $W_{p}^{\mathfrak{t}} \in \mathcal{V}^+(\mathcal{B})$ is the unique decaying solution to
\begin{equation}\label{ProblemWpt}
\left \lbrace
\begin{aligned}
\dsp - \Delta_{\mX} W_{p}^{\mathfrak{t}}(\mX) & \; = &&  F_{p}^{\mathfrak{t}}(\mX) + \frac{\mathcal{D}_{p}^{\mathfrak{t}}  }{2}
\dsp [g_0(\mX)] +  \frac{\mathcal{N}_{p}^{\mathfrak{t}}}{2}
\dsp [g_1(\mX)]  \quad \mbox{in} \; \mathcal{B}, \\
\partial_n W_{p}^{\mathfrak{t}} &\;  = &&  0  \; \mbox{on} \;  \partial
\widehat{\Omega}_\hole,\\[1ex]
\partial_{X_1} W_{p}^{\mathfrak{t}} (0,X_2) & \; = &&  \partial_{X_1} W_{p}^{\mathfrak{t}}(1,X_2), \quad X_2 \in
\R,
\end{aligned} 
\right.
\end{equation}
where
\begin{multline}\label{DefFpt}
F_{p}^{\mathfrak{t}}(\mX) = 2 \partial_{X_1} W_{p-1}^{\mathfrak{t}}(\mX) \,+\, W_{p-2}^{\mathfrak{t}}(\mX)
\,+\, (-1)^{\lfloor p/2 \rfloor} \left(
  2\, \langle g_p(\mX)\rangle \, \delta_{p}^{\text{even}} \right)  \\+ \sum_{k=2}^{p-1}
(-1)^{\lfloor k/2 \rfloor}  \frac{[g_k(\mX)]}{2}
\delta_{k}^{\text{even}}  \mathcal{D}_{p-k}^{\mathfrak{t}} + \sum_{k=2}^{p-1}
(-1)^{\lfloor k/2 \rfloor}  \,
\frac{[g_k(\mX)]}{2} \, \delta_{k}^{\text{odd}} \, \mathcal{N}_{p-k+1}^{\mathfrak{t}},
\end{multline}
and the constants $\mathcal{D}_p^{\mathfrak{t}}$ and $\mathcal{N}_p^{\mathfrak{t}}$ are given by
\begin{equation}\label{DefDptNpt}
\mathcal{D}_p^{\mathfrak{t}} = \int_{\mathcal{B}} F_{p}^{\mathfrak{t}}(\mX)
\mathcal{D}(\mX) d\mX, \quad \mathcal{N}_p^{\mathfrak{t}} = -\int_{\mathcal{B}} F_{p}^{\mathfrak{t}}(\mX)
\mathcal{N}(\mX) d\mX.
\end{equation}
In formula~\eqref{DefFpt}, $\delta_p^{\text{odd}}$ is equal to the remainder of the euclidian
division of $p$ by $2$  (\ie $\delta_p^{\text{odd}} $ is equal to $1$ if $p$ is odd
and equal to $0$ if $p$ is even)  and $\delta_p^{\text{even}} = 1 -
\delta_p^{\text{odd}}$ ($\delta_p^{\text{even}} $ is equal to $1$ if $p$ is even
and equal to $0$ otherwise). Moreover, $\lfloor r \rfloor$ denotes the floor of a
real number $r$. \\

\noindent Similarly, $W_p^{\mathfrak{n}} = 0$, for $p\leq 0$, and, for $p\geq 1$, $W_{p}^{\mathfrak{n}} \in \mathcal{V}^+(\mathcal{B})$ is the unique decaying solution 
\begin{equation}\label{ProblemWpn}
\left\lbrace
\begin{aligned}
\dsp - \Delta_{\mX} W_{p}^{\mathfrak{n}}(\mX) & \; = &&  F_{p}^{\mathfrak{n}}(\mX) + \frac{\mathcal{D}_{p}^{\mathfrak{n}}  }{2}
\dsp [g_0(\mX)] +  \frac{\mathcal{N}_{p}^{\mathfrak{n}}}{2}
\dsp [g_1(\mX)]  \quad \mbox{in} \; \mathcal{B}, \\
\partial_n W_{p}^{\mathfrak{n}} & \;= &&  0  \; \mbox{on} \;  \partial
\widehat{\Omega}_\hole,\\
\partial_{X_1} W_{p}^{\mathfrak{n}} (0,X_2) & \;= &&  \partial_{X_1} W_{p}^{\mathfrak{n}}(1,X_2), \quad X_2 \in
\R,
\end{aligned} 
\right.
\end{equation}
where
\begin{multline}\label{DefFpn}
F_{p}^{\mathfrak{n}}(\mX) = 2 \partial_{X_1} W_{p-1}^{\mathfrak{n}}(\mX) \,+\, W_{p-2}^{\mathfrak{n}}(\mX)
\,+\, (-1)^{\lfloor p/2 \rfloor} \left(
  2 \, \langle g_p(\mX)\rangle \, \delta_{p}^{\text{odd}} \right)  \\+ \sum_{k=2}^{p-1}
(-1)^{\lfloor k/2 \rfloor}  \frac{[g_k(\mX)]}{2}
\delta_{k}^{\text{even}} \mathcal{D}_{p-k}^{\mathfrak{n}} + \sum_{k=2}^{p-1}
(-1)^{\lfloor k/2 \rfloor}  \, \frac{[g_k(\mX)]}{2}
\delta_{k}^{\text{odd}}  \, \mathcal{N}_{p-k+1}^{\mathfrak{n}},
\end{multline}
and the constants $\mathcal{D}_p^{\mathfrak{n}}$ and $\mathcal{N}_p^{\mathfrak{n}}$ are given by
\begin{equation}\label{DefDpnNpn}
\mathcal{D}_p^{\mathfrak{n}}= \int_{\mathcal{B}} F_{p}^{\mathfrak{n}}(\mX)
\mathcal{D}(\mX) d\mX, \quad \mathcal{N}_p^{\mathfrak{n}} = -\int_{\mathcal{B}} F_{p}^{\mathfrak{n}}(\mX)
\mathcal{N}(\mX) d\mX.
\end{equation}
\begin{rem} The well posedness of Problem~\eqref{ProblemWpn} and
  Problem~\eqref{ProblemWpt} results from the application of
  Proposition~\ref{prop:layer_existence_uniqueness_problem_strip}. By
  construction, the right-hand sides of Problem~\eqref{ProblemWpn} and
  Problem~\eqref{ProblemWpt} satisfy the compatibility
  conditions~\eqref{eq:prop_layer_existence_uniqueness_compatibility_D} and
  ~\eqref{eq:prop_layer_existence_uniqueness_compatibility_N}, ensuring that
  $W_{p}^{\mathfrak{t}}$ and $W_{p}^{\mathfrak{n}}$ belong to
  $\mathcal{V}^+(\mathcal{B})$.
\end{rem}
\subsection{Base cases: $p =0$ and $p=1$.}
The base case $p=0$ has been done in
Subsection~\ref{SubsubStep0}. It is easily verified that, for $p=0$,
Problem~\eqref{ProblemWpt} and Problem~\eqref{ProblemW0t} coincide since
$\mathcal{D}_0^p = \mathcal{N}_0^p=0$. Besides, formula~\eqref{decompositionPi0}
(resp. formula~\eqref{Sautu0}) exactly
corresponds to~\eqref{DefPiqAnnexe}
(resp. \eqref{DefSautuqAnnexe}).\\

\noindent The case $p=1$ has also be treated (Subsection~\ref{SubsubStep1}). We remark that the profile
function $W_1^{\mathfrak{t}}$ 
vanishes, and that Problems~\eqref{ProblemWpn} and~\eqref{ProblemW1n}, which both define $W_1^{\mathfrak{n}}$,
are similar.  As a result, formula~\eqref{decompositionPi1}
(resp. the jump conditions~\eqref{Sautu1} and ~\eqref{Sautdu0}), exactly
corresponds to~\eqref{DefPiqAnnexe}
(resp. \eqref{DefSautuqAnnexe} and \eqref{DefSautDuqm1Annexe}).
\subsection{Inductive step}
Let $q\in \N$. We now assume that
formulas~\eqref{DefPiqAnnexe}, \eqref{DefSautuqAnnexe} and
\eqref{DefSautDuqm1Annexe} hold for any non negative integer $p$ such that
$p\leq q-1$. We prove that
they are still valid for $q$. We shall follow the procedure described
in Subsection~\ref{SubSectionTransmissionGeneral}.
\subsubsection{Computation of $G_q(x_1, \mX)$}
This is by far the most technical step of the proof. We remind that $G_q(x_1, \mX)$,
defined in~\eqref{SecondMembreGq}, is given by
\begin{multline}\label{GqAppendix}
G_q(x_1, \mX) = 2 \langle g_0(\mX) \rangle \langle u_q(x_1,0)
\rangle_{\Gamma} + \frac{1}{2} [ g_0(\mX) ] [u_q(x_1,0) ]_{\Gamma}  \\
+ 2
\langle g_1(\mX) \rangle  \langle \partial_{x_2} u_{q-1}(x_1,0)
\rangle_{\Gamma} + \frac{1}{2} [ g_1(\mX) ] [\partial_{x_2}u_{q-1}(x_1,0)
]_{\Gamma} \\
+ 2 \partial_{x_1} \partial_{X_1} \Pi_{q-1} +   2 \partial_{x_1}^2
\Pi_{q-2} + \mathcal{A}_q(x_1, \mX) +  \mathcal{B}_q(x_1, \mX), 
\end{multline}
where,
\begin{equation}\label{DefAq}
\mathcal{A}_q(x_1, \mX)  =   \sum_{p=2}^q 2 \langle g_p(\mX) \rangle \langle \partial^p_{x_2} u_{q-p}(x_1,0) \rangle_{\Gamma},
\end{equation}
\begin{equation}\label{DefBq}
\mathcal{B}_q(x_1,\mX) =\sum_{p=2}^q 
 \frac{1}{2} \left[ g_p(\mX) \right] \left[  \partial^p_{x_2}
   u_{q-p}(x_1,0) \right]_{\Gamma}.
\end{equation}
We shall rewrite the different terms of the third line
of~\eqref{GqAppendix} using the following substitution rules:
\begin{enumerate}
\item[-] We replace the normal derivatives
  of the macroscopic terms
  with their corresponding tangential
  derivatives using the following two formulas (we remind that for any $p\in \N$,
  $\Delta u_{p} =0$ in a lower (and upper) vicinity of the interface $\Gamma$): for any non negative integers $k$,
  $\ell$ and $p$, 
\begin{align}
& \partial_{x_1}^\ell \partial_{x_2}^{2k}  u_{p}(x_1,0^\pm) =
(-1)^k  \partial_{x_1}^{\ell+2k} u_{p}(x_1,0^\pm),  \label{FormulePaire} \\
& \partial_{x_1}^\ell \partial_{x_2}^{2k+1}   u_{p}(x_1,0^\pm) =
(-1)^k  \partial_{x_1}^{\ell+2k} \partial_{x_2} u_{p}(x_1,0^\pm) \label{FormuleImpaire}.
\end{align}
\item[-] For any $p < q$, we substitute the jump of the traces $[u_p(x_1,
  0)]_{\Gamma}$ (or their tangential derivatives $\partial_{x_1}^k [u_p(x_1,
  0)]_{\Gamma}$) for their explicit expression~\eqref{DefSautuqAnnexe}.
\item[-] For any $p<q-1$, we replace the jumps of the normal traces
 $[\partial_{x_2} u_p(x_1,
  0)]_{\Gamma}$ (or their tangential derivatives $\partial_{x_1}^k [\partial_{x_2} u_p(x_1,
  0)]_{\Gamma}$) with their explicit formula~\eqref{DefSautDuqm1Annexe}.
\item[-] We replace $\Pi_{q-1}$ and $\Pi_{q-2}$ with their tensorial
  representation~\eqref{DefPiqAnnexe}.
 \end{enumerate}
\paragraph{Computation of $\mathcal{A}_q(x_1, \mX)$}
We divide  the sum \eqref{DefAq} into its even ($p$ even) and odd ($p$
odd) components, and we use
formulas~\eqref{FormulePaire}-\eqref{FormuleImpaire} to obtain
\begin{equation}\label{CalculAq}
\mathcal{A}_q(x_1, \mX) = \sum_{p=2}^q  A_{p}^{\mathfrak{t}}(\mX) \partial_{x_1}^p \langle u_{q-p}(x_1,
  0) \rangle_{\Gamma}\\
+ \sum_{p=2}^q A_p^{{\mathfrak{n}}}(\mX) \partial_{x_1}^{p-1} \langle \partial_{x_2} u_{q-p}(x_1,
  0) \rangle_{\Gamma}.
\end{equation}
where
\begin{equation}\label{AptApn}
A_p^{\mathfrak{t}}(\mX) = 2 \langle g_p(\mX)
\rangle (-1)^{\lfloor p/2 \rfloor } \delta_{p}^{\text{even}} \quad
\mbox{and} \quad A_p^{\mathfrak{n}}(\mX) =  2 \langle g_p(\mX)
\rangle (-1)^{\lfloor p/2 \rfloor} 
\delta_{p}^{\text{odd}}.
\end{equation}
\paragraph{Computation of $\mathcal{B}_q(x_1, \mX)$}
As previously, we first divide the sum
\eqref{DefBq} into its odd and even components $\mathcal{B}_{q,1}$ $\mathcal{B}_{q,2}$, 
\begin{equation}
\mathcal{B}_q(x_1, \mX) =  \mathcal{B}_{q,1}(x_1, \mX) +  \mathcal{B}_{q,2}(x_1, \mX), 
\end{equation}
where, using 
formulas~\eqref{FormulePaire}-\eqref{FormuleImpaire},
\begin{equation}
\mathcal{B}_{q,1}(x_1, \mX)  = \sum_{k=2}^q \left(  (-1)^{\lfloor k/2\rfloor }\,\frac{[g_k(\mX)]}{2}\,
  \delta_{k}^{\text{even}} \right) \partial_{x_1}^k [u_{q-k}(x_1, 0)]_\Gamma,
\end{equation}
and
\begin{equation}
\mathcal{B}_{q,2}(x_1, \mX)  = \sum_{k=2}^q \left(  (-1)^{\lfloor
    k/2\rfloor } \, \frac{[g_k(\mX)]}{2} 
  \delta_{k}^{\text{odd}} \right) \partial_{x_1}^{k-1} [\partial_{x_2}u_{q-k}(x_1, 0)]_\Gamma.
\end{equation}
We shall evaluate $\mathcal{B}_{q,1}$ and $\mathcal{B}_{q,2}$ in turn, using
the induction hypotheses \eqref{DefSautuqAnnexe}-\eqref{DefSautDuqm1Annexe}.
In view of \eqref{DefSautuqAnnexe}, we have,
\begin{multline*}
\mathcal{B}_{q,1}(x_1, \mX)  = \sum_{k=2}^q  \sum_{p=1}^{q-k}
(-1)^{\lfloor k/2\rfloor } \, \frac{  [g_k(\mX)] }{2} \,
  \delta_{k}^{\text{even}} \,   \mathcal{D}_{p}^{\mathfrak{t}} \, \partial_{x_1}^{p+k}
  \langle u_{q-k-p}(x_1, 0) \rangle_\Gamma \;+ \\
\; \sum_{k=2}^q  \sum_{p=1}^{q-k}  (-1)^{\lfloor k/2\rfloor }\, \frac{ [g_k(\mX)] }{2}\,
  \delta_{k}^{\text{even}} \,    \mathcal{D}_{p}^{\mathfrak{n}} \, \partial_{x_1}^{p+k-1}
  \langle \partial_{x_2} u_{q-k-p}(x_1, 0) \rangle_\Gamma.
\end{multline*}
We remark that the last term of each sum over $k$, corresponding to $k = q$, 
can be removed (since the inner sum is empty in this case). Then, the
change of index $p \leftarrow  p+k$ gives
\begin{equation}\label{B1qfinal}
\mathcal{B}_{q,1}(x_1, \mX)  = \sum_{p=3}^q  B_{p,1}^{\mathfrak{t}}(\mX) \,  \partial_{x_1}^{p}
  \langle u_{q-p}(x_1, 0) \rangle_{\Gamma}+ \sum_{p=3}^q B_{p,1}^{\mathfrak{n}}(\mX)
  \, \partial_{x_1}^{p-1}
  \langle \partial_{x_2} u_{q-p}(x_1, 0) \rangle_\Gamma,
\end{equation}
with
\begin{equation}\label{DefB1ptDefB1pn}
{B}_{p,1}^{\mathfrak{t}}(\mX)  = \sum_{k=2}^{p-1} (-1)^{\lfloor k/2\rfloor } \, \frac{ [g_k(\mX)] }{2}\,
 \delta_{k}^{\text{even}} \,    \mathcal{D}_{p-k}^{\mathfrak{t}}
\quad\mbox{and} \quad  B_{p,1}^{\mathfrak{n}}(\mX)  = \sum_{k=2}^{p-1}  (-1)^{\lfloor k/2\rfloor
  }\frac{[g_k(\mX)]}{2} \,
 \delta_{k}^{\text{even}} \,   \mathcal{D}_{p-k}^{\mathfrak{n}}.
\end{equation}
Similarly, thanks to formula~\eqref{DefSautDuqm1Annexe}, 
\begin{multline*}
\mathcal{B}_{q,2}(x_1, \mX)  = \sum_{k=2}^q  \sum_{p=1}^{q-k}  \frac{ (-1)^{\lfloor k/2\rfloor }}{2}
  \delta_{k}^{\text{odd}} \,  [g_k(\mX)] \,   \mathcal{N}_{p+1}^{\mathfrak{t}} \, \partial_{x_1}^{p+k}
  \langle u_{q-k-p}(x_1, 0) \rangle_\Gamma \;+ \\
\; \sum_{k=2}^q  \sum_{p=1}^{q-k}  \frac{ (-1)^{\lfloor k/2\rfloor }}{2}
  \delta_{k}^{\text{odd}} \,    [g_k(\mX)] \, \mathcal{N}_{p+1}^{\mathfrak{n}} \, \partial_{x_1}^{p+k-1}
  \langle \partial_{x_2} u_{q-k-p}(x_1, 0) \rangle_\Gamma,
\end{multline*}
which, using the change of index $p\leftarrow p+k$, yields
\begin{equation}\label{B2qfinal}
\mathcal{B}_{q,2}(x_1, \mX)  = \sum_{p=3}^q  B_{p,2}^{\mathfrak{t}}(\mX) \,  \partial_{x_1}^{p}
  \langle u_{q-p}(x_1, 0) \rangle_{\Gamma}+ \sum_{p=3}^q B_{p,2}^{\mathfrak{n}}(\mX)
  \, \partial_{x_1}^{p-1}
  \langle \partial_{x_2} u_{q-p}(x_1, 0) \rangle_\Gamma,
\end{equation}
Here,
\begin{equation}\label{DefB2ptDefB2pn}
 {B}_{p,2}^{\mathfrak{t}}(\mX)  = \sum_{k=2}^{p-1} (-1)^{\lfloor k/2\rfloor }\frac{  [g_k(\mX)]}{2}\,
 \delta_{k}^{\text{odd}} \,  \mathcal{N}_{p-k+1}^{\mathfrak{t}}
\quad\mbox{and} \quad  B_{p,2}^{\mathfrak{n}}(\mX)  = \sum_{k=2}^{p-1} (-1)^{\lfloor k/2\rfloor
  } \,\frac{ [g_k(\mX)]  }{2} \,
 \delta_{k}^{\text{odd}} \, \mathcal{N}_{p-k+1}^{\mathfrak{n}}.
\end{equation}
Finally the sum of \eqref{B1qfinal} and \eqref{B2qfinal} leads to
\begin{multline}\label{formule2}
B_q(x_1, \mX)  = \sum_{p=2}^q  (B_{p,1}^{\mathfrak{t}}(\mX) + B_{p,2}^{\mathfrak{t}}(\mX) ) \,  \partial_{x_1}^{p}
  \langle u_{q-p}(x_1, 0) \rangle_{\Gamma} \\+ \; \sum_{p=2}^q (B_{p,1}^{\mathfrak{n}}(\mX)  + B_{p,2}^{\mathfrak{n}}(\mX) )
  \, \partial_{x_1}^{p-1}
  \langle \partial_{x_2} u_{q-p}(x_1, 0) \rangle_\Gamma,
\end{multline}
Here, we have artificially added the term
corresponding to $p=2$ in the two summations, using the convention that the constants
$B_{2,1}^{\mathfrak{t}}$ and $B_{2,1}^{\mathfrak{n}}$, $B_{2,2}^{\mathfrak{t}}$ and $B_{2,2}^{\mathfrak{n}}$ vanish (in
the definitions~\eqref{DefB1ptDefB1pn} and~\eqref{DefB2ptDefB2pn}, the
sums are empty). 
\paragraph{Computation of $2 \partial_{x_1} \partial_{X_1} \Pi_{q-1}$
  and $\partial_{x_1}^2  \Pi_{q-2}$}
These computations are less technical. The differentiation of 
formula~\eqref{DefPiqAnnexe} (recursive hypothesis on the tensorial
representation of $\Pi_q$) with respect to both $x_1$ and $X_1$ gives
\begin{equation*}
\partial_{x_1} \partial_{X_1} \Pi_{q-1}(x_1, \mX) = \sum_{p=0}^{q-1} \partial_{x_1}^{p+1} \langle
u_{q-1- p}(x_1,0)
\rangle_{\Gamma} \partial_{X_1} W_{p}^{\mathfrak{t}}(\mX) +\sum_{p=1}^{q-1} 
\partial_{x_1}^{p} \langle \partial_{x_2}
u_{q-1-p}(x_1,0) \rangle_{\Gamma} \partial_{X_1} W_{p}^{\mathfrak{n}}(\mX).
\end{equation*}
Then, making the change of index $p \mapsto p+1$ and using the fact
  that $\partial_{X_1} W_{0}^{\mathfrak{t}} =0$ (see formula
  \eqref{ProprieteW0t}), we get
\begin{multline}\label{formule3}
2 \partial_{x_1} \partial_{X_1} \Pi_{q-1}(x_1, \mX) = \sum_{p=2}^{q} \partial_{x_1}^{p} \langle
u_{q-p}(x_1,0)
\rangle_{\Gamma}  \left( 2 \partial_{X_1} W_{p-1}^{\mathfrak{t}}(\mX) \right)  \\ +
\,\sum_{p=2}^{q} 
\partial_{x_1}^{p-1} \langle \partial_{x_2}
u_{q-p}(x_1,0) \rangle_{\Gamma}\left( 2  \partial_{X_1} W_{p-1}^{\mathfrak{n}}(\mX)
  \right).
\end{multline}

\noindent Analogously (differentiating formula~\eqref{DefPiqAnnexe} twice with
respect to $x_1$, then making the change of index $p \leftarrow p+2$),
\begin{equation}\label{formule4}
\partial_{x_1}^2 \Pi_{q-2}(x_1, \mX) = \sum_{p=2}^{q} \partial_{x_1}^{p} \langle
u_{q-p}(x_1,0)
\rangle_{\Gamma} W_{p-2}^{\mathfrak{t}}(\mX) +\sum_{p=2}^{q} 
\partial_{x_1}^{p-1} \langle \partial_{x_2}
u_{q-p}(x_1,0) \rangle_{\Gamma} W_{p-2}^{\mathfrak{n}}(\mX).
\end{equation}
Here, we have use the fact that $W_{0}^{\mathfrak{n}}$ vanishes.
\paragraph{Summary}

Collecting formulas~\eqref{GqAppendix}, \eqref{CalculAq}, \eqref{formule2},
\eqref{formule3}, \eqref{formule4}, we end up with
\begin{multline}\label{GqFinal}
G_q(x_1, \mX) = 2 \langle g_0(\mX) \rangle \langle u_q(x_1,0)
\rangle_{\Gamma} + \frac{1}{2} [ g_0(\mX) ] [u_q(x_1,0) ]_{\Gamma} + 2
\langle g_1(\mX) \rangle  \langle \partial_{x_2} u_{q-1}
\rangle_{\Gamma}  \\+ \frac{1}{2} [ g_1(\mX) ] [\partial_{x_2}u_{q-1}(x_1,0)
]_{\Gamma} 
+ \sum_{p=2}^q F_p^{\mathfrak{t}}(\mX) \partial_{x_1}^{p} \langle
u_{q-p}(x_1,0) 
\rangle_{\Gamma} 
 + \sum_{p=2}^q F_p^{\mathfrak{n}}(\mX) \partial_{x_1}^{p-1} \langle \partial_{x_2}
u_{q-p}(x_1,0) \rangle_{\Gamma},
\end{multline}
where,
\begin{equation*}
F_p^{\mathfrak{t}}(\mX) = A_{p}^{\mathfrak{t}}(\mX) + B_{p,1}^{\mathfrak{t}}(\mX)
  +B_{p,2}^{\mathfrak{t}}(\mX) + 2 \partial_{X_1} W_{p-1}^{\mathfrak{t}}(\mX)  + W_{p-2}^{\mathfrak{t}}(\mX),
\end{equation*}
and
\begin{equation*}
F_{p}^{\mathfrak{n}}(\mX) =   A_{p}^{\mathfrak{n}}(\mX) + B_{p,1}^{\mathfrak{n}}(\mX)
  +B_{p,2}^{\mathfrak{n}}(\mX) + 2 \partial_{X_1} W_{p-1}^{\mathfrak{n}}(\mX)  + W_{p-2}^{\mathfrak{n}}(\mX).
\end{equation*}
Here, the functions $A_p^{\mathfrak{t}}$, $A_p^{\mathfrak{n}}$ are defined in~\eqref{AptApn}, the
functions $B_{p,1}^{\mathfrak{t}}$, $B_{p,1}^{\mathfrak{n}}$ in~\eqref{DefB1ptDefB1pn} and the
functions $B_{p,1}^{\mathfrak{t}}$, $B_{p,2}^{\mathfrak{t}}$
in~\eqref{DefB2ptDefB2pn}. Of course, it is
easily verified that the preceding definitions of $F_{p}^{\mathfrak{t}}$ and
$F_{p}^{\mathfrak{n}}$ coincide with the ones given in~\eqref{DefFpt} and ~\eqref{DefFpn}.\\

\noindent It is important to note that, for any fixed $x_1\in
(-\LBottom,\LBottom)$, ${G}_q(x_1, \cdot)$ belongs to $\left(\mathcal{V}^-(\mathcal{B})\right)'$ because
it is the combination of exponentially decaying terms (more precisely,
functions or first
derivative of functions belonging to $\mathcal{V}^+(\mathcal{B})$ and compactly
supported ones).
\subsubsection{Computation of the normal jump $[\partial_{x_2}
  u_{q-1}(x_1, 0)]_{\Gamma}$}
As in Subsections~\ref{SubsubStep0}, \ref{SubsubStep1} and
\ref{SubsubStep2} (base steps), for any fixed $x_1 \in (-\LBottom, \LBottom)$,
Proposition~\ref{prop:layer_existence_uniqueness_problem_strip}
ensures the existence of a periodic corrector $\Pi_q(x_1, \cdot) \in
\mathcal{V}^+(\mathcal{B})$ satisfying~(\ref{ProblemeSectionTransmission}-right) if and only if 
${G}_q(x_1, \cdot)$ is orthogonal to both $\mathcal{N}$ and $\mathcal{D}$
(that is to say ${G}_q$ satisfies the two compatibility
conditions~\eqref{eq:prop_layer_existence_uniqueness_compatibility_D}-\eqref{eq:prop_layer_existence_uniqueness_compatibility_N}). In view of the second and fourth lines of
Lemma~\ref{CalculConditionCompatibilite},  we have
\begin{multline}
\int_{\mathcal{B}} {G}_q(x_1, \mX) \mathcal{N}(\mX) d\mX =  [\partial_{x_2}u_{q-1}
]_{\Gamma} +  \sum_{p=2}^q \partial_{x_1}^{p} \langle
u_{q-p}(x_1,0) \rangle_{\Gamma}  \int_{\mathcal{B}}  F_p^{\mathfrak{t}}(\mX) \mathcal{N}(\mX) d\mX ,
\\+ \sum_{p=2}^q \partial_{x_1}^{p-1} \langle \partial_{x_2}
u_{q-p}(x_1,0) \rangle_{\Gamma}\int_{\mathcal{B}} F_p^{\mathfrak{n}}(\mX) \mathcal{N}(\mX) d\mX ,
\end{multline}
In the previous formula, we recognize the constants $\mathcal{N}_p^{\mathfrak{t}}$
and  $\mathcal{N}_p^{\mathfrak{n}}$ defined in
\eqref{DefDptNpt}-\eqref{DefDpnNpn}. Finally, the compatibility condition
\eqref{eq:prop_layer_existence_uniqueness_compatibility_N} is
fulfilled if and only if
\begin{equation}
  [\partial_{x_2}u_{q-1}
]_{\Gamma}  = \sum_{p=2}^q \partial_{x_1}^{p} \langle
u_{q-p}(x_1,0) \rangle_{\Gamma}  \mathcal{N}_p^{\mathfrak{t}}+ \sum_{p=2}^q \partial_{x_1}^{p-1} \langle \partial_{x_2}
u_{q-p}(x_1,0) \rangle_{\Gamma} \mathcal{N}_p^{\mathfrak{n}}.
\end{equation}
and formula~\eqref{DefSautDuqm1Annexe} is proved.
\subsubsection{Computation of the jump $[
  u_{q}(x_1, 0)]_{\Gamma}$}
Similarly, the first and third lines of
Lemma~\ref{CalculConditionCompatibilite} give
\begin{multline*}
\int_{\mathcal{B}} {G}_q(x_1, \mX) \mathcal{D}(\mX) d\mX =
- [u_{q}(x_1,0)
]_{\Gamma} + \langle
\partial_{x_2} u_{q-1}(x_1,0) \rangle_{\Gamma}  \int_{\mathcal{B}}  2
\langle g_1(\mX) \rangle
\mathcal{D}(\mX) d\mX +\\  \sum_{p=2}^q \partial_{x_1}^{p} \langle
u_{q-p}(x_1,0) \rangle_{\Gamma}  \int_{\mathcal{B}}  F_p^{\mathfrak{t}}(\mX) \mathcal{D}(\mX) d\mX ,
+ \sum_{p=2}^q \partial_{x_1}^{p-1} \langle \partial_{x_2}
u_{q-p}(x_1,0) \rangle_{\Gamma}\int_{\mathcal{B}} F_p^{\mathfrak{n}}(\mX) \mathcal{D}(\mX) d\mX ,
\end{multline*}
\noindent Here again, we recognize the constants $\mathcal{D}_p^{\mathfrak{t}}$
and $\mathcal{D}_p^{\mathfrak{n}}$ defined in
\eqref{DefDptNpt}-\eqref{DefDpnNpn}. If we additionaly remind that $
 F_{1}^{\mathfrak{n}}(\mX) = 2 \langle g_1(\mX)\rangle$ (cf. \eqref{DefFpn}) and that
$W_1^{\mathfrak{t}}=0$ (see ~\ref{DefFpt}), the compatibility condition
\eqref{eq:prop_layer_existence_uniqueness_compatibility_D} is
fulfilled if and only if
\begin{equation*}
[u_{q}(x_1,0)
]_{\Gamma}  = \sum_{p=1}^q \mathcal{D}_p^{\mathfrak{t}} \, \partial_{x_1}^{p} \langle
u_{q-p}(x_1,0) \rangle_{\Gamma}  + \sum_{p=1}^q \mathcal{D}_p^{\mathfrak{n}} \partial_{x_1}^{p-1} \langle \partial_{x_2}
u_{q-p}(x_1,0) \rangle_{\Gamma}
\end{equation*}
 which ends the proof of formula~\eqref{DefSautuqAnnexe}.
\subsubsection{Tensorial representation of $\Pi_{q}$}
Under the conditions~\eqref{DefSautDuqm1Annexe} and \eqref{DefSautuqAnnexe},
Problem~(\ref{ProblemeSectionTransmission}-right) has a unique
solution in $\mathcal{V}^+(\mathcal{B})$.  Moreover, substituting
$[u_q(x_1,0)]_\Gamma$ and $[\partial_{x_1} u_{q-1}(x_1,0)]_\Gamma$ 
for~\eqref{DefSautuqAnnexe} and~\eqref{DefSautDuqm1Annexe} in formula
\eqref{GqFinal}, using the linearity of
Problem~(\ref{ProblemeSectionTransmission}-right), and reminding the definitions of $W_p^{\mathfrak{n}}$ and
$W_p^{\mathfrak{t}}$~\eqref{ProblemWpt}-\eqref{ProblemWpn}, we see that 
\begin{equation*}
\Pi_q(x_1, \mX) = \sum_{p=0}^q \partial_{x_1}^p \langle
u_{q-p}(x_1,0)
\rangle_{\Gamma} W_{p}^{\mathfrak{t}}(\mX) +\sum_{p=1}^{q} 
\partial_{x_1}^{p-1} \langle \partial_{x_2}
u_{q-p}(x_1,0) \rangle_{\Gamma} W_{p}^{\mathfrak{n}}(\mX).
\end{equation*} 
which is nothing but~formula~\eqref{DefPiqAnnexe}.

\section{Technical results associated with Analysis of the macroscopic problems (macroscopic singularities): proof of Proposition~\ref{PropDefinitionsmq}}\label{AppendixPreuvePropositionsmq}
The proof of the Proposition~\ref{PropDefinitionsmq} is by
induction. Before starting the proof, we first need to present a
technical Lemma which will be usefull to define the functions $w_{n,p,\pm}$.
\subsection{A prelimininary technical Lemma}
We consider the infinite cone $\mathcal{K}$ of angle $3
\pi/2$ (\ie in 2D, $\mathcal{K}$ is an infinite angular sector of angle $3 \pi/2$).
\begin{equation} 
\mathcal{K} = \left\{ (r \cos \theta, r \sin \theta) \in \R^2 ,\; r > 0, \;
  \theta \in (0, \frac{3 \pi}{2}) \right\},
\end{equation}
that we divide into two disjoint cones $\mathcal{K}_{{\gamma}}^1$ and $\mathcal{K}_{{\gamma}}^2$ of openings $\gamma$ and
$\frac{3 \pi}{2} - \gamma$, $ 0 < \gamma < \frac{3 \pi}{2}$ :
\begin{equation} \label{DefinitonK1}
\mathcal{K}_{{\gamma}}^1 = \left\{ (r \cos \theta, r \sin \theta)  \in \R^2,\; r > 0, \;
  \theta \in \mathcal{I}_{{\gamma}}^1  \right\} \;
\mathcal{K}_{{\gamma}}^2 = \left\{ (r \cos \theta, r \sin \theta)  \in \R^2,\; r > 0, \;
  \theta \in  \mathcal{I}_{{\gamma}}^2\right\},
\end{equation}
where $\mathcal{I}_{{\gamma}}^1 = (0, \gamma)$ and
$\mathcal{I}_{{\gamma}}^2 = (\gamma, \frac{3\pi}{2})$.  We denote by
$\Gamma_{\gamma}$ the interface between $\mathcal{K}_{{\gamma}}^1$ and
$\mathcal{K}_{{\gamma}}^2$, \ie, $\Gamma_{\gamma} = \partial \mathcal{K}_{{\gamma}}^1
\cap \partial \mathcal{K}_{{\gamma}}^2$. In the cases we are interested in the cone
$\mathcal{K}$ is given by $\mathcal{K}_{\mx_O^+}$ or $\mathcal{K}_{\mx_O^-}$ (defined
by~\eqref{DefinitionKpmW}), so that $\gamma=\frac{\pi}{2}$ (for
$\mathcal{K}^-$) or $\gamma =\pi$ (for $\mathcal{K}^+$).  In the upcoming
proof, we shall use the following technical Lemma, which provides an
explicit formula for a transmission problem in a cone for a particular
right-hand side (we remind that $\Lambda$ defined
in~\eqref{EnsembleDesCoefficientDeSingularite} denotes the set of singular
exponents).
\begin{lema}\label{LemmeTechniquesmq}
Let $q\in \N$, $\lambda \in \R$, $(\tta,\ttb) \in \R^2$ and 
$$
N = \begin{cases} 
q & \mbox{if}  \; \lambda \notin \Lambda \\
q+1  & \mbox{if} \; \lambda \in \Lambda.
\end{cases} 
$$
There exist $N+1$ functions $g_p \in  \mathcal{C}^\infty(\overline{\mathcal{I}^1}) \cap \mathcal{C}^\infty(\overline{\mathcal{I}^2}
)$, ($0\leq p \leq N$), such that the function 
\begin{equation}\label{ProblemeLemmeTechniqueCone}
\ttv(r, \theta) = r^\lambda \left(  \sum_{p=0}^{N} \ln r^p g_p(\theta) \right)
  \quad \mbox{satisfies}\;\;
\left \lbrace
\quad \begin{aligned}
\Delta \ttv &\; = &&  0 \; \mbox{in}\; \mathcal{K}^1 \cap \mathcal{K}^2  , \\
\ttv(r,0) & \; = &&  0, \\ 
\ttv(r,\frac{3 \pi}{2}) & \; = &&  0, \\
[\ttv]_{\Gamma_\gamma} & \; = &&  \tta \, r^\lambda \, \ln r^q,  \\
[\partial_\theta \ttv]_{\Gamma_\gamma}  & \; =&&   \ttb  \, r^\lambda \, \ln r^q. 
\end{aligned} 
\right.
\end{equation}
\end{lema}
\noindent We refer the reader to the Chapter 3 in~\cite{TheseAbdelkader} and the
Section 6.4.2 in \cite{MR1469972} for the proof.
\begin{rem}
It is possible to show that the functions $g_p$ are uniquely determined if $\lambda \notin
\Lambda$. If $\lambda \in \Lambda$, the functions $g_p$ are uniquely
determined except for $p=0$. However, the uniqueness may be restored by imposing the following orthogonality condition:
$$
\int_{0}^{\frac{3 \pi}{2}} g_0(\theta) w_{\lambda}(\theta) =0  \quad
w_\lambda(\theta) = \sin( \lambda \theta).
$$
\end{rem}
\begin{rem}
 Note that the set of singular exponents $\lambda$ (for which we have to take $N
    = q + 1$) coincide with the singular exponents of the problem
    without transmission condition on $\Gamma_\gamma$. In particular,
    these exponents
   do not depend on the location
    of the interface $\Gamma_\gamma$.
 \end{rem}

\subsection{Recursive definition of the families $w_{n,p, \pm}$}

Based on Lemma~\ref{LemmeTechniquesmq}, we shall be able to define the set of functions
$\left\{ w_{n,p,\pm}, n \in \Z^\ast, p \in \N \right\}$ of the
Proposition~\ref{PropDefinitionsmq}. We begin with the construction of $w_{n,p,+}$. We remind that the functions $w_{n,p,+}$ are polynomial
with respect to $\ln r^+$, \ie,
\begin{equation}\label{Definitionwnpplus}
 w_{n,p, +}(\theta^+, \ln r^+) = \sum_{q=0}^p w_{n,p,q,+}(\theta^+) (\ln
r^+)^q,  \quad n \in \Z^\ast, p \in \N, \quad    w_{n,p,q,+} \in
\mathcal{C}^\infty([0, \pi])\cap \mathcal{C}^\infty([ \pi, \frac{3 \pi}{2}]),
\end{equation} 
Their construction is by
induction on $p$. The function $w_{n,0,+}$ has already been defined in Proposition~\ref{PropositionAsymptoticEspaceAPoids}:
\begin{equation}
 w_{n,0,+}(\theta^+, \ln r^+) = \sin (\lambda_n \theta^+). \\
\end{equation}
 For $p \geq 1$,  we construct
 $w_{n,p,+}$ of the form~\eqref{Definitionwnpplus}, such that the function $$\ttv_{n,p, +}(r^+, \theta^+)= (r^+)^{\lambda_n -p}
w_{n,p, +}(\theta^+, \ln r^+)$$ satisfies
\begin{equation}
\left \lbrace
 \begin{aligned}\label{Problemunqplus}
\Delta \ttv_{n,p,+} &  \; = &&  0 \; \mbox{in}\; \mathcal{K}^{+,1} \cap \mathcal{K}^{+,2}      , \\
\ttv_{n,p,+}(0) & \; = && 0, \\
\ttv_{n,p,+}(\frac{3 \pi}{2}) &  \; = &&  0, \\
\dsp [\ttv_{n,p,+}(r^+, \pi)]_{\partial \mathcal{K}^{+,1}
  \cap \partial \mathcal{K}^{+,2} } & \;  = &&  (r^+)^{\lambda_n -p}  \, \tta_{n,p,+}(\ln r^+)  ,  \\
\dsp [\partial_{\theta^{+}} \ttv_{n,p,+}(r^+, \pi)]_{\partial
  \mathcal{K}^{+,1} \cap \partial \mathcal{K}^{+,2} } & \; =  &&  (r^+)^{\lambda_n -p} \, \ttb_{n,p,+}(\ln r^+), 
\end{aligned} \right. \quad \; \forall n \in \Z^\ast,\;  \forall p \in \N^\ast,
\end{equation}
where
\begin{equation}\label{definitionCones1et2}
\mathcal{K}^{+,1}  = \left\{ (r^+ \cos \theta^+,  r^+ \sin
  \theta^+) \in \mathcal{K}^+, 0 < \theta^+ < \pi \right\} \;
\mathcal{K}^{+,2}  = \left\{ (r^+ \cos \theta^+,  r^+ \sin
  \theta^+) \in \mathcal{K}^+, \pi < \theta^+ < \frac{3 \pi}{2} \right\},
\end{equation}
and
\begin{align}
& \tta_{n,p,+}(\ln r^+)= \sum_{r=0}^{p-1} \left( \mathcal{D}_{p-r}^{\mathfrak{t}}
  g_{n,r,p-r,+}^{\mathfrak{t}}(\ln r^+) +  \mathcal{D}_{p-r}^{\mathfrak{n}}
  g_{n,r,p-r,+}^{\mathfrak{n}}(\ln r^+) \right), \label{Defanpplus} \\
&  \ttb_{n,p,+}(\ln r^+)= \sum_{r=0}^{p-1} \left( \mathcal{N}_{p+1 - r}^{\mathfrak{t}}
  h_{n,r,p-r,+}^{\mathfrak{t}}(\ln r^+) +  \mathcal{N}_{p+ 1 -r}^{\mathfrak{n}}
  h_{n,r,p-r,+}^{\mathfrak{n}}(\ln r^+) \right) \label{Defbnpplus}.
\end{align}
The functions $g_{n,r,q,+}^{\mathfrak{t}}$,
$g_{n,r,q,+}^{\mathfrak{n}}$, $h_{n,r,q,+}^{\mathfrak{t}}$,
$h_{n,r,q,+}^{\mathfrak{n}}$ are defined by the following relations:
for $n \in \Z^\ast$, $r\in \N$, $q \in \N$, 
\begin{align}
& (r^+)^{\lambda_n -r -q}  \; g_{n,r,q,+}^{\mathfrak{t}}(\ln r^+) = (-1)^q
\frac{\partial^q}{(\partial r^+)^q}  \left[ (r^+)^{\lambda_n -r} \, \langle w_{n,r,+}(\pi, \ln r^+) \rangle
_{\partial \mathcal{K}^{+,1} \cap \partial \mathcal{K}^{+,2}} \right],
\label{gnrqplust}\\
& (r^+)^{\lambda_n -r -q} \; g_{n,r,q,+}^{\mathfrak{n}}(\ln r^+) = (-1)^q
\frac{\partial^{q-1}}{( \partial r^+)^{q-1}}  \left[ (r^+)^{\lambda_n -r-1} \,
  \langle \partial_{\theta^+} w_{n,r,+}(\pi, \ln r^+) \rangle
_{\partial \mathcal{K}^{+,1} \cap \partial \mathcal{K}^{+,2}}
\right], (q\geq 1)\label{gnrqplusn}\\
&  (r^+)^{\lambda_n -r -q} \; h_{n,r,q,+}^{\mathfrak{t}}(\ln r^+) = r^+ (-1)^q 
\frac{\partial^{q+1}}{(\partial r^+)^{q+1}}  \left[ (r^+)^{\lambda_n -r} \, \langle w_{n,r,+}(\pi, \ln r^+) \rangle
_{\partial \mathcal{K}^{+,1} \cap \partial \mathcal{K}^{+,2}} \right] \\
&  (r^+)^{\lambda_n -r -q} \; h_{n,r,q,+}^{\mathfrak{n}}(\ln r^+)= r^+
(-1)^q 
\frac{\partial^{q}}{(\partial r^+)^q}  \left[ (r^+)^{\lambda_n -r-1} \, \langle \partial_{\theta^+} w_{n,r,+}(\pi, \ln r^+) \rangle
_{\partial \mathcal{K}^{+,1} \cap \partial \mathcal{K}^{+,2}} \right],
 \end{align} 
and $g_{n,r,0,+}^{\mathfrak{n}} =0$.\\

\noindent The recursive procedure to construct the terms $w_{n,p,+}$ is the
following: assume that the terms $w_{n,q,+}$ are known for any $n \in
\Z^\ast$ and $q \leq p-1$. To construct $w_{n,p,+}$ we need to know the source terms
$\tta_{n,p,+}$ and $\ttb_{n,p,+}$, which require the 
computation of $g_{n,r,p-r,+}^\mathfrak{t}$ and
$g_{n,r,p-r,+}^\mathfrak{n}$ for $r \leq p-1$. But,
$g_{n,r,p-r,+}^\mathfrak{t}$ and $g_{n,r,p-r,+}^\mathfrak{n}$ only  depend
on the function $w_{n,r,+}$, which, thanks to the induction hypothesis is
known. Similarly, to construct $\ttb_{n,p,+}$, we
need to define $h_{n,r,p-r,+}^\mathfrak{t}$ and
$h_{n,r,p-r,+}^\mathfrak{n}$ for $r\leq p-1$. These terms only depend on $w_{n,r,+}$, and, as
a consequence, are known.  Lemma~\ref{LemmeTechniquesmq} then ensures
the existence of $w_{n,p,+}$.\\

\noindent Of course, we could have written a recursive formula to
obtain $g_{n,r,q,+}^{\mathfrak{t}}$ (resp.  $g_{n,r,q,
  +}^{\mathfrak{n}}$,$h_{n,r,q,
  +}^{\mathfrak{t}}$,$h_{n,r,q,
  +}^{\mathfrak{n}}$) starting
form $g_{n,r,0, +}^{\mathfrak{t}}$ (resp.  $g_{n,r,0,
  +}^{\mathfrak{n}}$,$h_{n,r,0,
  +}^{\mathfrak{t}}$,$h_{n,r,0,
  +}^{\mathfrak{n}}$) but we shall not need to explicit this
formula. However, we can prove the following useful relation:
\begin{equation}\label{relationMagique1}
h_{n,r,q,+}^{\mathfrak{t}} = - g_{n,r,q+1,+}^{\mathfrak{t}}
\quad\mbox{and} \quad h_{n,r,q,+}^{\mathfrak{n}} = - g_{n,r,q+1,+}^{\mathfrak{n}}.
\end{equation}

\begin{rem} If $\lambda_n -p \in \Lambda$ (which is the case as
soon as $p$ is even, except if $\lambda_n -p =0$), the functions $w_{n,p,+}$ is not
uniquely defined by \eqref{Problemunqplus}. In that
case, we add the somehow arbitrary condition
\begin{equation}\label{ConditionOrthogonalitePlus}
\int_{0}^{\frac{3 \pi}{2}} w_{n, p, 0,+ }(\theta^+,\ln r^+) w_{n,0,+} (\theta^+) d \theta^+ = 0
\end{equation}
to restore the uniqueness (see \eqref{Definitionwnpplus} for the
definition of $w_{n, p, 0,+ }$). As a consequence, we can see that the sum over $q$
in~\eqref{Definitionwnpplus} goes from $0$ to $\lfloor p/2\rfloor$
(and not $p$): in other words,
$w_{n,p,+}$ is a polynomial of degree at most $\lfloor p/2\rfloor$.  \\
\end{rem}
\noindent Similarly,  we can define by induction the functions $w_{n,p,-}$ of
the form
\begin{equation}\label{Definitionwnpmoins}
w_{n,p, -}(\theta^-, \ln r^-) = \sum_{q=0}^{\lfloor p/2 \rfloor} w_{n,p,q, -}(\theta^-)
(\ln r^-)^q \quad n \in \Z^\ast, p \in \N, \quad    w_{n,p,q,-} \in
\mathcal{C}^\infty([-\frac{\pi}{2}, 0])\cap \mathcal{C}^\infty([ 0, \pi]).
\end{equation} 
Again, the function $w_{n,0,-}$ has already been defined in Proposition~\ref{PropositionAsymptoticEspaceAPoids}:
\begin{equation}
w_{n,0,-}(\theta^-, \ln r^-) = \sin (\lambda_n (\theta^-- \pi/2)),
\end{equation}
For $p\geq 1$, we construct $w_{n,p,-}$ such that $\ttv_{n,p, -}= (r^-)^{\lambda_n -p}
w_{n,p, -}(\theta^-, \ln r^-)$ satisfies
\begin{equation}\label{Problemunqmoins}
\left \lbrace
\begin{aligned}
\Delta \ttv_{n,p,-} &  \; = &&  0 \; \mbox{in}\; \mathcal{K}^{-,1} \cap \mathcal{K}^{-,2}  , \\
 \ttv_{n,p,-}(-\frac{\pi}{2}) &  \;  =  &&  0,  \\
  \ttv_{n,p,-}(\pi) & \; = &&  0, \\
[\ttv_{n,p,-}(r^-, 0)]_{\partial \mathcal{K}^{-,1} \cap \partial
  \mathcal{K}^{-,2}} &  \;  = &&  (r^-)^{\lambda_n -p}  \; \tta_{n,p,-}(\ln r^-)  ,  \\
[\partial_{\theta^-} \ttv_{n,p,-}(r^-, 0)]_{\partial \mathcal{K}^{-,1}
  \cap \partial \mathcal{K}^{-,2}} &  \;  = &&   (r^-)^{\lambda_n -p}\; \ttb_{n,p,-}(\ln r^-), 
\end{aligned} \right. \quad \;  \forall n \in \Z^\ast, \;\forall  p \in \N^\ast.
\end{equation} 
Here,
\begin{equation}
\mathcal{K}^{-,1}  = \left\{ (r^- \cos \theta^-,  r^- \sin
  \theta^-) \in \mathcal{K}^-, -\pi < \theta^- < 0 \right\} \;
\mathcal{K}^{-,2}  = \left\{ (r^- \cos \theta^-,  r^- \sin
  \theta^-) \in \mathcal{K}^-, 0 < \theta^- < \pi \right\},
\end{equation}
and 
\begin{align}
& \tta_{n,p,-}(\ln r^-)= \sum_{r=0}^{p-1} \left( \mathcal{D}_{p-r}^{\mathfrak{t}}
  g_{n,r,p-r,-}^{\mathfrak{t}}(\ln r^-) +  \mathcal{D}_{p-r}^{\mathfrak{n}}
  g_{n,r,p-r,-}^{\mathfrak{n}}(\ln r^-) \right) \label{Defanpmoins}\\
& \ttb_{n,p,-}(\ln r^-)= \sum_{r=0}^{p-1} \left( \mathcal{N}_{p+1 - r}^{\mathfrak{t}}
  h_{n,r,p-r,-}^{\mathfrak{t}}(\ln r^-) +  \mathcal{N}_{p+ 1 -r}^{\mathfrak{n}}
  h_{n,r,p-r,-}^{\mathfrak{n}}(\ln r^-) \right) \label{Defbnpmoins}
\end{align}
The functions $g_{n,r,q,-}^{\mathfrak{t}}$,
$g_{n,r,q,-}^{\mathfrak{n}}$, $h_{n,r,q,-}^{\mathfrak{t}}$,
$h_{n,r,q,-}^{\mathfrak{n}}$ are defined by the following relations
for $n \in \Z \setminus\{0\}$, $r\in \N$, $q \in \N$,
\begin{align*}
&\dsp  (r^-)^{\lambda_n -r -q}   \;g_{n,r,q,-}^{\mathfrak{t}}(\ln r^-)= 
\frac{\partial^q}{(\partial r^-)^q}  \left[ (r^-)^{\lambda_n -r} \, \langle w_{n,r,-}(0, \ln r^-) \rangle
_{\partial \mathcal{K}^{+,1} \cap \partial \mathcal{K}^{+,2}} \right]  \\
&\dsp  (r^-)^{\lambda_n -r -q}  \; g_{n,r,q,-}^{\mathfrak{n}}(\ln r^-) = 
\frac{\partial^{q-1}}{(\partial r^-)^{q-1}}  \left[ (r^-)^{\lambda_n -r-1} \,
  \langle \partial_{\theta^-} w_{n,r,-}(0, \ln r^-) \rangle
_{\partial \mathcal{K}^{-,1} \cap \partial \mathcal{K}^{-,2}} \right]
\; (q\geq 1)\\
& \dsp (r^-)^{\lambda_n -r -q}  \; h_{n,r,q,-}^{\mathfrak{t}}(\ln r^-)= r^-
\frac{\partial^{q+1}}{(\partial r^-)^{q+1}}  \left[ (r^-)^{\lambda_n -r} \, \langle w_{n,r,-}(0, \ln r^-) \rangle
_{\partial \mathcal{K}^{-,1} \cap \partial \mathcal{K}^{-,2}} \right] \\
& \dsp (r^-)^{\lambda_n -r -q} \; h_{n,r,q,-}^{\mathfrak{n}}(\ln r^-)=r^-  
\frac{\partial^{q}}{(\partial r^-)^q}  \left[ (r^-)^{\lambda_n -r-1} \, \langle \partial_{\theta^-} w_{n,r,-}(0, \ln r^-) \rangle
_{\partial \mathcal{K}^{-,1} \cap \partial \mathcal{K}^{-,2}} \right] 
 \end{align*} 
and $g_{n,r,0,-}^{\mathfrak{n}} =0$. Again, we have:
\begin{equation}
h_{n,r,q,-}^{\mathfrak{t}} =  g_{n,r,q+1,-}^{\mathfrak{t}}
\quad\mbox{and} \quad h_{n,r,q,-}^{\mathfrak{n}} =  g_{n,r,q+1,-}^{\mathfrak{n}}.
\end{equation}

\noindent To restore the uniqueness of $w_{n, p, -}$ if $\lambda_n -p \in
\Lambda$ (\ie $p$ even), we impose the orthogonality condition (see \eqref{Definitionwnpmoins} for the
definition of $w_{n, p, 0,-}$))
\begin{equation}\label{ConditionOrthogonaliteMoins}
\int_{-\frac{\pi}{2}}^{\pi} w_{n, p,0, - }(\theta^-) w_{n,0,-}(\theta^-) d \theta^- = 0.
\end{equation}

\subsection{Construction of the family $s_{-m,q}^+$ of Proposition~\ref{PropDefinitionsmq}}
\noindent We are now in a position to start the heart of the induction proof of
Proposition~\ref{PropDefinitionsmq}. The proof of the base case and
the proof of the inductive case essentially rely on the same arguments, the inductive step
being as usual more technical.

\subsubsection{{Base case}}
$s_{-m,1}^+$ satisfies the following problem
 \begin{equation}\label{DefProblemeSmmm1}
\left\{
 \begin{aligned}
 \dsp - \Delta s^+_{-m,1}  & \; = &&  0 \quad \mbox{in} \; \OmegaTop \cap
 \OmegaBottom\\
s^+_{-m,1}  & \;= && 0 \quad \mbox{on} \; \Gamma_D,\\
 \dsp [s^+_{-m,1} (x_1,0)]_{\Gamma} & \;= &&  
 \mathcal{D}_1^{\mathfrak{t}} \, \partial_{x_1} \langle
 s^+_{-m,0}(x_1,0)
 \rangle_{\Gamma}+ 
 \mathcal{D}_1^{\mathfrak{n}} \,  \langle \partial_{x_2}
 s^+_{-m,0} (x_1,0) \rangle_{\Gamma}, \\
 \dsp [\partial_{x_2}s^+_{-m,1} (x_1,0)]_{\Gamma} &\; = &&
  \mathcal{N}_{2}^{\mathfrak {t}} \, \partial_{x_1}^{2} \langle
 s^+_{-m,0} (x_1,0)
 \rangle_{\Gamma}+ 
 \mathcal{N}_{2}^{\mathfrak{n}} \,  \partial_{x_1} \langle \partial_{x_2}
 s^+_{-m,0} (x_1,0) \rangle_{\Gamma}.
 \end{aligned}
\right.
 \end{equation}
 We first observe that the right-hand side of~\eqref{DefProblemeSmmm1} is singular
 (see the asymptotic \eqref{jumpPlus1}-\eqref{djumpPlus1} below). As a
 consequence, Proposition~\ref{PropExistenceUniquenessFF} does
 not apply, at least directly. However, we can write an explicit asymptotic
 representation of the jump values of $s_{m,q}^+$ in the vicinity of the two
 corners. Then, we shall see that we are able to make an appropriate lift of
 the singular part of the jumps to reduce the problem to a variational one,
 which can of course be solved using
 Proposition~\ref{PropExistenceUniquenessFF} (or Proposition~\ref{PropositionAsymptoticEspaceAPoids}).\\

\noindent Using the asymptotic
expansion~\eqref{Asymptoticsm0CoinPositif} of $s_{-m,0}^+$, in the
vicinity of $\mx_O^+$, and thanks to the two formulas 
$\partial_{x_1} v = -
 \partial_{r^+} v$, $\partial_{x_2} = -
(r^+)^{-1} \partial_\theta^+ v$ valid on $\Gamma$,  
we can write an explicit expansion of
$[s^+_{-m,1} (x_1,0)]_{\Gamma}$ in the vicinity of  $\mx_O^+$: for any
$k\in \N^\ast$, for
$r^+$ sufficiently small
\begin{multline}\label{jumpPlus1}
\left[ s^+_{-m,1} (r^+,\pi)\right]_{\Gamma} = \sum_{n=-m}^{k+1}
\ell_{n}^+(s_{-m,0}^+) \, (r^+)^{\lambda_n -1} \;
\; \left( - \lambda_n \, \mathcal{D}_1^{\mathfrak{t}} \, w_{n,0,+}(\pi) \; -\;
\mathcal{D}_1^{\mathfrak{n}} \, \partial_{\theta^+} w_{n,0,+}(\pi) \right) + \;
r_{1,k,+} (r^+) \\
=  \sum_{n=-m}^{k+1} \, \tta_{n,1,+} \;
\ell_{n}^+(s_{-m,0}^+) \; (r^+)^{\lambda_n -1}  \; + 
r_{1,k,+} (r^+), 
\end{multline}
where, $ r_{1,k,+} \in V_{2,
  \beta}^{3/2}(\Gamma)$ for any $\beta > 1 -
\frac{2(k+1)}{3}$. We used the convention $\ell_{-m}^+(s_{-m,0}^+) =  1$ and
$\ell_{n}^+(s_{-m,0}^+) =  0$ for the integers $n$ such that $-m+1 \leq
n \leq 0$. $\tta_{n,1,+}$ is defined in~\eqref{Defanpplus}. \\

\noindent Based on the same technical ingredients, we can prove that there exists
a function $ r_{2,k,+}$ in $V_{2, \beta}^{1/2}(\Gamma)$ for any $\beta > 1-
\frac{2(k+1)}{3}$ such that
\begin{equation}\label{djumpPlus1}
\left[ \frac{1}{r^+}\partial_{\theta^+} s^+_{-m,1} (r^+,\pi)\right]_{\Gamma} = \sum_{n=-m}^{k+1}
 \ttb_{n,1,+} \;  \ell_{n}^+(s_{-m,0}^+) \; (r^+)^{\lambda_n -2} \;+ \; r_{2,k,+} (r^+), \\
\end{equation}
$\ttb_{n,1,+}$ being defined in~\eqref{Defanpplus}. Of course, a similar analysis provides an explicit expression for the jump
values of $s^+_{-m,1}$ close to the left corner:
\begin{align}
& \left[ s^+_{-m,1} (r^-,0)\right]_{\Gamma} = \sum_{n=1}^{k+1}
\tta_{n,1,-} \;\ell_{n}^-(s_{-m,0}^+)  \; (r^-)^{\lambda_n -1}
\; + \; r_{1,k,-} (r^-), \\
& \left[\frac{1}{r^-} \partial_{\theta^-} s^+_{-m,1} (r^-,0)\right]_{\Gamma} = \sum_{n=1}^{k+1}
 \ttb_{n,1,-} \; \ell_{n}^-(s_{-m,0}^+) (r^-)^{\lambda_n -2} \;+ \; r_{2,k,-} (r^+),
\end{align}
where, for any $\beta > 1 -
\frac{2(k+1)}{3}$, $r_{1,k,-} \in V_{2, \beta}^{3/2}(\Gamma)$ and $r_{2,k,+}V_{2,
  \beta}^{1/2}(\Gamma)$. The constants $\tta_{n,1,+}$ and $\ttb_{n,1,-}$ are
defined in~\eqref{Defanpmoins} and~\eqref{Defbnpmoins}.\\

Now, let us consider the problem satisfied by a function 
\begin{equation}\label{DefSm1kTilde}
\widetilde{s}_{-m,1,k}^+ = s_{-m,1}^+ - \sum_{n=-m}^{k+1} \ell_n^+(s_{-m,0}^+)
\chi_{\LBottom}^+  \ttv_{n,1,+}(r^+, \theta^+) - \sum_{n=1}^{k+1} \ell_n^-(s_{-m,0}^+)
\chi_{\LBottom}^- \ttv_{n,1,-}(r^-, \theta^-),
\end{equation}
defined on $\OmegaTop \cap  \OmegaBottom$. In this definition, we have lifted the
most singular part of the jumps of ${s}_{-m,1}^+$ across $\Gamma$ (up to a given
order $k$) using the functions $\ttv_{n,1,\pm}$ defined in
\eqref{Problemunqplus}-\eqref{Problemunqmoins}. The cut-off functions
$\chi_{\LBottom}^+$ are defined in the sentence following the definition~\eqref{definitionV2betal}.
$\widetilde{s}_{-m,1,k}^+$ satisfies
\begin{equation}\label{DefProblemeSmmm1Tilde}
\left \lbrace
 \begin{aligned}
 \dsp - \Delta \widetilde{s}_{-m,1,k}^+ & \;  =  && f_{k} \quad \mbox{in} \;
 \OmegaTop \cap \OmegaBottom,\\
\widetilde{s}_{-m,1,k}^+   & \; = &&  0 \quad \mbox{on} \,\Gamma_D,\\
 \dsp [\widetilde{s}_{-m,1,k}^+(x_1,0)]_{\Gamma} & \; = &&  g_{k}, \\
 \dsp [\partial_{x_2} \widetilde{s}_{-m,1,k}^+(x_1,0)]_{\Gamma} & \; =
 &&h_{k}.
\end{aligned}
\right.
 \end{equation}
A direct computation shows that $f_{k} \in
V_{2,\beta}^{3/2}(\OmegaTop)\cap V_{2,\beta}^{3/2}(\OmegaBottom)$ ($f$ is compactly supported away from
the two corners), $g_{k} \in
V_{2,\beta}^{3/2}(\Gamma)$ and $h_{k} \in
V_{2,\beta}^{1/2}(\Gamma)$ for any $\beta > 1 -
\frac{2(k+1)}{3}$.  Thus, for $k\geq 1$, and choosing $\beta -1 \notin
\Lambda$, Proposition~\eqref{PropositionAsymptoticEspaceAPoids} guarantees
that Problem~\eqref{DefProblemeSmmm1} has a unique solution belonging to
$H^1_{0, \Gamma_D}(\OmegaTop)\cap H^1_{0, \Gamma_D}(\OmegaBottom)$ that
satisfies the following asymptotic close to the two corners: there
exists  $r_{n,1,k}^\pm \in V_{2,\beta}^2(\OmegaTop)\cap
V_{2,\beta}^2(\OmegaBottom)$ for any $\beta >  1 -
\frac{2(k+1)}{3}$, $\beta - 1 \notin \Lambda$, such that, for $r^\pm$
sufficiently small,
\begin{equation}\label{formuleAsymptoticTilde}
\widetilde{s}_{-m,1,k}^+ = \sum_{n=1}^k
\ell_n^\pm({s}_{-m,1,k}^+) (r^\pm)^{\lambda_n} w_{n}^\pm(\theta^\pm) +   r_{n,1,k}^\pm.
\end{equation}
The existence of $s_{-m,1}^+$ is then a 
direct consequence of
formula~\eqref{DefSm1kTilde}. Finally, the asymptotic formula for
$s_{-m,1}^+$ directly follows from \eqref{formuleAsymptoticTilde} and
the asymptotic formulas for $\ttv_{n,1,\pm}$ (See
\eqref{Problemunqplus}-\eqref{Problemunqmoins}). The uniqueness of
$s_{-m,1}^+$ follows form Remark~\ref{remarqueUnicitesm0}.
\subsubsection{Inductive step}
We assume that the function $s_{-m, p}^+$ are constructed up to $p=q-1$
and that the asymptotic formulas \eqref{Asymptoticsmq}-\eqref{AsymptoticsmqLeftCorner} hold for $p\leq q-1$. We
shall prove that we can construct $s_{-m, q}^+ \in V^2_{2,
  \beta}(\OmegaTop)\cap  V^2_{2, \beta}(\OmegaBottom)$,
$\beta > 1 + \frac{2m}{3} +q$, satisfying
\eqref{Problemsmq} and admitting the asymptotic expansions
\eqref{Asymptoticsmq}-\eqref{AsymptoticsmqLeftCorner} in the
vicinity of the two corners. The proof is almost entirely similar to the one of the
base case. First, using the asymptotic formula~\eqref{Asymptoticsmq} for $s_{-m,
  q-p}^+$ and the 
definitions~\eqref{Defanpplus} of $\tta_{n,p,+}$, we obtain the following
formula for the jump value of $s_{-m,q}^+$ across $\Gamma$:  for $r^+$ sufficiently small, 
$$
[s_{-m,q}^+]_{\Gamma} =\sum_{p=1}^{q} \; \sum_{-m \leq n <
 (k+1) + \frac{3}{2}}  \ell_{n}^+(s_{-m, q-p}^+)  \tta_{n, p,+} (r^+)^{\lambda_n
 -p} + r_{1,k,+}(r^+)
$$
where $r_{1,k,+} \in V_{2,\beta}^{3/2}(\Gamma)$ for any $\beta > 1 -
\frac{2(k+1)}{3}$. Here, we have used the convention that for $p>0$, $\ell_{n}^+(s_{m,
  p}^+) =0$ if $n \geq 0$.  Similarly, using the
definition~\eqref{Defbnpplus} of $\ttb_{n,p,+}$, we see that there exists $r_{2,k,+} \in V_{2,\beta}^{1/2}(\Gamma)$ for any $\beta > 1 -
\frac{2(k+1)}{3}$ such that, for $r^+$ small enough
$$
[\frac{1}{r^+}\partial_{\theta^+} s_{-m,q}^+]_{\Gamma} =\sum_{p=1}^{q} \; \sum_{-m \leq n <
 (k+1) + \frac{3}{2}}  \ell_{n}^+(s_{-m, q-p}^+)  \ttb_{n, p,+} (r^+)^{\lambda_n
 -p-1} + r_{2,k,+}(r^+).
$$
Analogously, we obtain asymptotic representations of the jump values of
$s_{-m,q}^+$ in the vicinity of the left corner: there exist
$r_{1,k,-}$ belonging to $V_{2,\beta}^{3/2}(\Gamma)$ for any $\beta > 1 -
\frac{2(k+1)}{3}$ and $r_{2,k,-}$ in $V_{2,\beta}^{1/2}(\Gamma)$ for any $\beta > 1 -
\frac{2(k+1)}{3}$, such that, for $r^-$ small enough
\begin{align*}
& [s_{-m,q}^+]_{\Gamma} =\sum_{p=1}^{q} \; \sum_{1 \leq n <
 (k+1) + \frac{3}{2}}  \ell_{n}(s_{-m, q-p})  \tta_{n, p,-} (r^-)^{\lambda_n
 -p} + r_{1,k,-} \\
& [\frac{1}{r^-}\partial_{\theta^-} s_{-m,q}^+]_{\Gamma} =\sum_{p=1}^{q} \; \sum_{1 \leq n <
 (k+1) + \frac{3}{2}}  \ell_{n}^+(s_{-m, q-p})  \ttb_{n, p,-} (r^-)^{\lambda_n
 -p-1} + r_{2,k,-}.
\end{align*}
$\tta_{n,q,-}$ and $\ttb_{n,q,-}$ are defined by \eqref{Defanpmoins} and
\eqref{Defbnpmoins}. Then, we consider the problem satisfied by the function
\begin{multline}
\tilde{s}_{-m,q,k}^+  = {s}_{-m,q}^+ \; - \; \sum_{p=1}^{q} \; \sum_{-m \leq n <
 (k+1)+ \frac{3}{2}} \ell_n^+(s_{-m,q-p}^ +) \chi_{\LBottom}^+  \ttv_{n,p,+}(r^+,
\theta^+)   \\ - \sum_{p=1}^{q} \; \sum_{1 \leq n <
 (k+1) + \frac{3}{2}} \ell_n^-(s_{-m,q-p}^ +) \chi_{\LBottom}^- \, \ttv_{n,p,-}(r^- , \theta^-).
\end{multline}
defined on $\OmegaTop \cap \OmegaBottom$, where we have lifted (up to order
$k$) the most singular part of the jumps values of $s_{-m,q}^+$. The
functions $\ttv_{n,p,\pm}$ are defined in \eqref{Problemunqplus}-\eqref{Problemunqmoins}.  The
remainder of the proof is entirely similar to the one of the base case.

\section{Technical results associated with the analysis of the near field problems (near field
  singularities)}

\subsection{Proof of Proposition~\ref{propositionTechniqueSobolevPoids}}
Before proving the Proposition~\ref{propositionTechniqueSobolevPoids}, let us first give a preliminary technical
result (whose proof essentially follows from the fact that $\lvert t \rvert \geqslant \lvert
  \sin t \rvert$ and, for $t
  \in (-\pi/2, \pi/2)$, $ | \sin t | \geq \frac{2}{\pi} |t|$):
\begin{lema}
  \label{lmm:equivalence_rho}
  For any $\theta^+ \in (\frac{\pi}{2}, \frac{3\pi}{2})$ (\ie $X_1^+<0$), the
  following estimate holds:
  \begin{equation}
    \label{eq:minoration_rho}
    1 + \lvert X_2^+ \rvert  \; \leqslant \rho(R^+,\theta^+) \; \leqslant \; \left( \frac{\pi}{2} + 1 \right) \left( 1 + \lvert X_2^+
      \rvert \right).
  \end{equation}
\end{lema}
\noindent We can now start the proof of Proposition~\ref{propositionTechniqueSobolevPoids} Let
$\gamma \in \left( \frac{1}{2},1 \right)$, $p \in \{1, 2\}$, $\lambda \in \R$
and $q \in \N$.
\begin{itemize}
\item Let $v_1 = \chi^{(p)}(X_2^+) \chi_-(X_1^+) (R^+)^\lambda (\ln (R^+))^q$ and let us
  determine on which condition on $\beta$, $v_1$ belongs to
  $\mathfrak{V}_{\beta, \gamma}^0(\widehat{\Omega}^+)$. 
  We first note that $\chi^{(p)}(X_2^+)$ is compactly supported in the bands $1
  \leqslant \lvert X_2^+ \rvert \leqslant 2$. On this domain, we can bound
  $\chi^{(p)}$ by its $\mathrm{L}^{\infty}$-norm. Then, in view of
  Lemma~\ref{lmm:equivalence_rho} (note that due to the cut-off function
  $\chi_-(X_1^+)$, we only consider $X_1^+<0$), $\rho$ defined in
  (\ref{DefinitionNormeDoublePoids}) is bounded from below and from above.  As
  a consequence, we have to estimate
  \begin{equation}
    \label{eq:proof_62_integral_1}
    \int_{-\infty}^{-1} \int_{1}^2 (R^+)^{2\lambda} (\ln (R^+))^{2q}
    (R^+)^{2\beta-2\gamma-2} dX_1^+ dX_2^+.
  \end{equation}
  In addition, since $X_2^+$ is bounded from above and
  from below, $R^+$ is equivalent to $ \lvert X_1^+ \rvert$. Moreover, $(\ln
  R^+)^{2q}$ can be bounded  by $\lvert X_1^+ \rvert^\varepsilon$, for any $\varepsilon > 0$;
  then estimating the integral~\eqref{eq:proof_62_integral_1} is equivalent to the
  estimate 
  \begin{equation}
    \label{eq:proof_62_integral_1_2}
    \int_{-\infty}^{-1} \lvert X_1^+
    \rvert^{2\lambda+2\beta-2\gamma-2+\varepsilon} dX_1^+,
  \end{equation}
  and integral (\ref{eq:proof_62_integral_1_2}) is bounded if and only if
  \begin{equation*}
    2\lambda+2\beta-2\gamma-2+\varepsilon < -1, \forall \varepsilon > 0, 
  \end{equation*}
  Finally  $v_1$ belongs to $\mathfrak{V}_{\beta, \gamma}^0(\widehat{\Omega}^+)$ as soon as $\beta < \gamma - \lambda + \frac{1}{2}$.
\item We now consider the function $v_2 =\chi(R^+) (R^+)^\lambda (\ln R^+)^q$. Since the derivatives of $
  \chi(R^+)$ are compactly supported, the only part to estimate is 
  \begin{equation}
    \label{eq:proof_62_integral_2}
    \int_{\widehat{\Omega}^+}  \chi(R^+) (R^+)^{2\beta-2\gamma-2\delta_{p,0}}
    \rho^{2\gamma-4+2p+2\delta_{p,0}} \left\lbrace \nabla^p \left(
        (R^+)^{\lambda} (\ln R^+)^{q} \right) \right\rbrace^2 dX_1^+ dX_2^+,
  \end{equation}
  for $0 \leqslant p \leqslant 2$. First, we can exhibit a function
  $\phi_p(R^+)$ such that
  \begin{equation*}
    \nabla^p \left( (R^+)^{\lambda} (\ln R^+)^{q} \right) =
    (R^+)^{\lambda-p} \phi_p(R^+), \quad \text{with} \quad \phi_p(R^+) =
    O((R^+)^\varepsilon), \quad \forall \varepsilon > 0,
  \end{equation*}
  so that we have to study the convergence of
  \begin{equation}
    \label{eq:proof_62_integral_2_2}
    \int_{\widehat{\Omega}^+}
    \chi(R^+) (R^+)^{2\lambda+2\beta-2\gamma-2p-2\delta_{p,0}+\varepsilon}
    \rho^{2\gamma-4+2p+2\delta_{p,0}} dX_1^+ dX_2^+.
  \end{equation}
  Note that estimation of (\ref{eq:proof_62_integral_2_2}) for $p=0$ and $p=1$
  is the same. Moreover, since the origin $(0,0)$ does not belong to the
  support of $\chi$, we only have to consider the behviour for large $R^+$. To
  estimate this integral, we split the domain into three parts:
  \begin{itemize}
  \item the domain $\Omega_1 = \widehat{\Omega}^+ \cap  \left\{ 0
      \leqslant |X_2^+| \leqslant 2,X_1^+<0 \right\}$ 
    (the intersection of $\widehat{\Omega}^+$ (which has holes) with a fixed width band located on both sides of the interface $\{ (X_1^+, X_2^+) \in \R^2, X_1^+ <0
    \,\mbox{and}\, X_2^+ =0\}$): since $\Omega_1$ is
    included in the band $\left\{ 0
      \leqslant |X_2^+| \leqslant 2,X_1^+<0 \right\}$, the integral
    \eqref{eq:proof_62_integral_2_2} restricted to $\Omega_1$ is
    smaller than the same integral in the whole band. Then, again, we can use that
    $\rho$ is bounded in this part  (from above and from below), so that the integral converges if and only if
    \begin{equation*}
      2\lambda+2\beta-2\gamma-2p-2\delta_{p,0}+\varepsilon < -1, \;\varepsilon
      > 0 \quad \Rightarrow \quad \beta < \gamma - \frac{1}{2} + p +
      \delta_{p,0} - \lambda,
    \end{equation*}
    and, since $\gamma > \frac{1}{2}$ and $p + \delta_{p,0} \geqslant 1$, we
    can bound $\beta$ by $1 - \lambda$.
  \item the domain $\Omega_2 = \{ 2 < |X_2^+| < |X_1^+| , X_1^+<0\}$ (two angular
    domains). We first note that $R^+$ is equivalent to $|X_1^+|$ in $\Omega_2$.
    Indeed, $R^+ \leq ( |X_1^+| + |X_2^+|) \leq 2 |X_1^+|$. Then, thanks to lemma
    \ref{lmm:equivalence_rho}, $\rho$ is equivalent to $1
    + \lvert X_2^+ \rvert$
    To evaluate the integral~\eqref{eq:proof_62_integral_2_2} restricted to
    $\Omega_2$, we make first an explicit integration with respect to $X_2^+$
    (note that $2 \gamma -4 + 2p + 2 \delta_{p,0}\neq -1$), then we use the
    equivalence of $R^+$ and $|X_1^+|$ to obtain
    \begin{multline}
      \int_{-\infty}^{-1} \int_{1}^{|X_1^+|} (R^+)^{ 2 \lambda + 
        2 \beta - 2\gamma -2 \delta_{p,0} -2p+ \varepsilon} \rho^{2 \gamma -4
        + 2p + 2 \delta_{p,0}}
    dX_2^+   dX_1^+ \\ \leq C  \int_{-\infty}^{-1} |X_1^+|^{ 2 \lambda + 
        2 \beta - 2\gamma -2 \delta_{p,0} -2p + \varepsilon } \left( |X_1^+|^{2
          \gamma -3 + 2p + 2\delta_{p,0}} +1 \right) dX_1^+. 
    \end{multline}
    Since $2
    \gamma -3 + 2p + 2\delta_{p,0} >0$, the most singular term in the right-hand side of the previous equality is $ |X_1^+|^{2
      \gamma -3 + 2p + 2\delta_{p,0}}$. As a consequence,
    $$
    \int_{-\infty}^{-1} \int_{1}^{|X_1^+|} (R^+)^{ 2 \lambda + 
      2 \beta - 2\gamma -2 \delta_{p,0} -2p+ \varepsilon} \rho^{2 \gamma -4
      + 2p + 2 \delta_{p,0}}
    dX_1^+ dX_2^+ \leq C \int_{-\infty}^{-1} |X_1^+|^{ 2 \lambda + 
      2 \beta  - 3 + \varepsilon } dX_1^+.
    $$
    This integral converges as soon as $\beta < 1 - \lambda - \varepsilon/2$.
    Finally,
    \begin{equation}
      \int_{\Omega_{2} } (R^+)^{ 2 \lambda + 
        2 \beta - 2\gamma -2 \delta_{p,0} -2p+ \varepsilon} \rho^{2 (\gamma -1)}
      dX_1^+ dX_2^+ \\ \leq C \int_{-\infty}^{-1} |X_1^+|^{ 2 \lambda + 
        2 \beta -2 \delta_{p,0}  -1 -2p} dX_1^+ 
    \end{equation}
    converges as soon as $\beta < -1 + \lambda$.

    \item the domain $\Omega_3 = \widehat{\Omega}^+ \cap
      \{ (R^+ \cos \theta^+, R^+
      \sin \theta^+) \in \R^2,   |\theta^+-\pi| > \frac{\pi}{4} \}$: in
      this domain, $\rho$ is equivalent to
      $R^+$. Using the polar coordinates, we see that we have to estimate
      \begin{equation}
        \label{eq:proof_62_integral_2_6}
        \int_1^{+\infty} \int_0^{\frac{3\pi}{4}}
        (R^+)^{2\lambda+2\beta-2\gamma-2p-2\delta_{p,0}+\varepsilon+2\gamma-4+2p+2\delta_{p,0}}
        R^+ dR^+ d\theta^+,
      \end{equation}
      and this integral converges if and only if
      \begin{equation*}
        2\lambda+2\beta+\varepsilon-3 < -1, \varepsilon > 0 \quad \Rightarrow
        \quad \beta < 1 - \lambda.
      \end{equation*}
This ends the proof of the second point.
    \end{itemize}
  \item Let $w=w(X_1^+,X_2^+)$ be a $1$ periodic function with respect to
  $X_1^+$ such that  $\| w \, e^{|X_2^+|/2}\|_{L^2(\mathcal{B})} < +\infty$, and let $v_3 = \chi_-(X_1^+)
    |X_1^+|^{\lambda-1} (\ln (|X_1^+|))^q w(X_1^+,X_2^+)$. We have to study
    the convergence of the integral
    \begin{equation}
      \label{eq:proof_62_integral_3}
    \mathcal{I}_3 =  \int_{\widehat{\Omega}^+} (R^+)^{2\beta-2\gamma-2}
      \rho^{2\gamma+2}\, \chi_-(X_1^+) \, |X_1^+|^{2(\lambda-1)}\, (\ln |X_1^+|)^{2q}
      w^2(X_1^+,X_2^+) d\mX.
    \end{equation}
In order to use the periodicity of $w$, we split the domain
$\widehat{\Omega}^+$ into an infinite sequence of similar non overlapping domains
$\mathcal{B}_{n} = \widehat{\Omega}^+\cap \left\{ -(n+1) < X_1^+ < -n
\right\}$ ($\mathcal{B}_{n}$ consists of the shifted periodicity cell
$\mathcal{B}$, see \eqref{PeriodicityCellDef}). Then, the previous
integral becomes
$$
 \mathcal{I}_3 = \sum_{n\in \N}  \int_{\mathcal{B}_{-n}} (R^+)^{2\beta-2\gamma-2} 
      \rho^{2\gamma+2} e^{-|X_2^+|} |X_1^+|^{2(\lambda-1)} (\ln |X_1^+|)^{2q}
      \left( w^2(X_1^+,X_2^+)  e^{|X_2^+|} \right) d\mX.
$$
Note first that the inequalities~\eqref{eq:minoration_rho} are valid in
$\mathcal{B}_{-n}$ ($\rho$ is equivalent
to $(1 + |X_2^+|)$ for $X_1^+ <0$).  As a result $\rho^{2\gamma}$ is
equivalent to $(1 + |X_2^+|)^{2 \gamma}$. Besides, there
exist two constants $(C_1,C_2) \in (\R^+)^2$ such that
\begin{equation}\label{inegaliteBn}
C_1 \, n  \; \leq R^+ \; \leq \;  C_2 \, n \, ( 1+\frac{|X_2^+|^2}{n^2})^{1/2}\; \leq\; C_2 \, n \, (1 + |X_2^+|).  
\end{equation}
One the one hand, if $2\beta - 2 \gamma-2>0$, using the second inequality
in~\eqref{inegaliteBn} together with the fact that for any $\alpha
\in \R$ $(1+|X_2^+|)^\alpha e^{-|X_2^+|}$ is bounded, we have,
$$
 \mathcal{I}_3 \leq  C \| w \|^2_{\mathcal{V}^+(\mathcal{B})} \,
 \sum_{n\in \N} n^{2 \beta - 2 \gamma +  2 \lambda -4 -
   \varepsilon },
$$
which converges as soon as  
\begin{equation}\label{ConditionBetav3}
\beta < \frac{3}{2} + \gamma -\lambda - \varepsilon.
\end{equation}
One the other hand, if $2\beta - 2 \gamma-2>0$, we can use the first inequality
in~\eqref{inegaliteBn} and the fact that for any $\alpha
\in \R$, $(1+|X_2^+|)^\alpha e^{-|X_2^+|}$ is bounded to show that
$\mathcal{I}_3$ also converges under the
condition~\eqref{ConditionBetav3}. To summarize, for a given parameter
$\gamma\in (\frac{1}{2},1)$, $v_3$
belongs to $\mathcal{V}_{\beta, \gamma}^0(\widehat{\Omega}^+)$ for any
$\beta \in \R$ such that $\beta < \frac{3}{2} + \gamma -\lambda$.
   
   \item For the last equality, it is sufficient to investigate the two
    following kinds
    of integrals: 
\begin{equation}
\mathcal{I}_{4,1} =  \int_{\widehat{\Omega}^+} (R^+)^{2\beta-2\gamma-2\delta_{p,0}}
      \rho^{2\gamma -4 + 2p + 2\delta_{p,0}}\, \chi_-(X_1^+) \, |X_1^+|^{2(\lambda-1)}\, (\ln |X_1^+|)^{2q}
      \tilde{w}^2(X_1^+,X_2^+) d\mX,
\end{equation} 
and
\begin{equation}
\mathcal{I}_{4,2} =  \int_{\widehat{\Omega}^+\cap supp (\chi_-'(X_1^+))} (R^+)^{2\beta-2\gamma-2\delta_{p,0}}
      \rho^{2\gamma -4 + 2p + 2\delta_{p,0}} \, |X_1^+|^{2(\lambda-1)}\, (\ln |X_1^+|)^{2q}
      \tilde{w}^2(X_1^+,X_2^+) d\mX,
\end{equation} 
for a generic function $\tilde{w} \in
L_{\textrm{loc}}^2(\widehat{\Omega}^+)$, $1$-periodic with respect to $X_1^+$
such that $\| e^{|X_2^+|/2}\tilde{w}\|_{L^2(\mathcal{B})} < +
\infty$. Indeed, by a direct computation of the partial derivatives of
$v_4$ (up to order 2), we can see that for any $\mathbf{\alpha} = (\alpha_1, \alpha_2) \in
\N^2$ such that $\alpha_1 + \alpha_2 \leq 2$,  there exist three
functions $v_{4,0}$, $v_{4,1}$ and $v_{4,2}$ such that
$$
\partial_{\mX}^{\alpha} v_4
=  \partial_{X_1^+}^{\alpha_1} \partial_{X_2^+}^{\alpha_2} v_4 =  \chi_-''(X_1^+) v_{4,2} + \chi_-'(X_1^+) v_{4,1} + \chi_-(X_1^+) v_{4,0},   
$$
where, for $\ell \in \{0,1,2\}$, $v_{4,\ell}  = |X_1^+|^{\lambda_\ell} \ln |X_1^+|^{q_\ell}
w_{4,\ell}$ with $\lambda_\ell \leq \lambda-1$, $q_\ell \leq q$, and $w_{4,\ell} \in
L_{\textrm{loc}}^2(\widehat{\Omega}^+)$ is $1$-periodic with respect to $X_1^+$
and satisfies $\| e^{|X_2^+|/2}w_{4, \ell}\|_{L^2(\mathcal{B})} < +
\infty$.
For $\mathcal{I}_{4,1}$, repeating the arguments of the proof of the third
point, we see that
$$
\mathcal{I}_{4,1} \; \leq \; C \, \| e^{|X_2^+|/2}\tilde{w} \|_{L^2(\mathcal{B})} \,
 \sum_{n\in \N} n^{2 \beta - 2 \gamma +  2 \lambda -2 - 2\delta_{p,0} +\varepsilon}.
$$  
Consequently, $\mathcal{I}_{4,1}$ converges for any $p \in \{ 0,1, 2\}$ if $\beta < \gamma +
\frac{1}{2} - \lambda -\varepsilon/2$.  As for $\mathcal{I}_{4,2}$,
we note that on the support of $\chi_-'(X_1^+)$, $|X_1^+|$ is bounded from above and from
below so that $R^+$ is
equivalent to $|X_2^+|$. It follows that $\mathcal{I}_{4,2}$ converges for any $\beta
\in \R$.
\end{itemize}

\subsection{Proof of Lemma \ref{LemmeSuperDecroissance}}

$\mathcal{U}_{m,p,+}$ is a sum of a macroscopic contribution  
$$
\mathcal{U}_{m,p,\text{macro}} = \sum_{r=0}^p (R^+)^{\lambda_m-r} w_{n,r,+}(\ln R^+, \theta^+),
$$
modulated by the cut-off function $\chi_{\text{macro},+}$, and a boundary layer contribution (exponentially decaying with respect
to $|X_2^+|$)
$$
\mathcal{U}_{m,p,\text{BL}} = \sum_{r=0}^p (|X_1^+|)^{\lambda_m -r}
p_{m,r,+}(\ln (|X_1^+|),X_1^+,X_2^+). 
$$
modulated by the cut-off function $\chi_-(X_1^+)$. Consequently,
\begin{equation}\label{LaplaceAsymptoticBlock}
\Delta \mathcal{U}_{m,p,+}  =  
[ \Delta, \chi(R^+)] (\chi_{\text{macro},+}\mathcal{U}_{m,p,\text{macro}}  + \chi_-(X_1^+) \mathcal{U}_{m,p,\text{BL}} )+  \chi(R^+) \Delta  \left( \chi_{\text{macro},+}\mathcal{U}_{m,p,\text{BL}}   + \chi_-(X_1^+)
 \mathcal{U}_{m,p,\text{macro}} \right).
\end{equation} 
The first term in the right-hand side is compactly supported
(since $\nabla \chi(R^+)$ and $\Delta \chi(R^+)$ are compactly supported). Therefore
it belongs to $\mathfrak{V}_{\beta,\gamma}^0(\widehat{\Omega}^+)$ for any
real numbers $\beta$ and $\gamma$.  It remains to estimate the terms of
the second line. The proof is technical but the main idea to figure out
is that $\Delta \left(
\chi_{\text{macro},+}\mathcal{U}_{m,p,\text{BL}} \right)$ and $\Delta \left(\chi_-(X_1^+)
 \mathcal{U}_{m,p,\text{macro}} \right)$ counterbalance (up to a
given order).\\

\noindent We start with the explicit computation of $\Delta \left(
\chi_{\text{macro},+}\,\mathcal{U}_{m,p,\text{macro}}\right)$. We consider the function
$$\ttv_{m,r,+}(X_1^+,X_2^+) = (R^+)^{\lambda_m -r} w_{m,r,+}(\theta, \ln
R^+),$$ 
already defined in \eqref{Problemunqplus}.  $\ttv_{m,r,+}$ is 
defined in the union of  the two cones $\mathcal{K}^{+,1}$ and
$\mathcal{K}^{+,2}$ (see \eqref{definitionCones1et2}). It is smooth on $\mathcal{K}^{+,1}$ and
$\mathcal{K}^{+,2}$ and it satisfies $\Delta
\ttv_{m,r,+}=0$ in $\mathcal{K}^{+,1}$ and $\mathcal{K}^{+,2}$. Using the fact
that $\chi_{\macro,+} = \chi(X_2^+)$ on the support of $\chi_-(X_1^+)$
($\chi_{\macro,+} = \chi(X_2^+)$ for $X_1^+<-1$), we have
\begin{equation}\label{EquationPreuvePenible1}
\Delta \left(
\chi_{\text{macro},+}\,\mathcal{U}_{m,p,\text{macro}}\right)= \chi_-(X_1^+)\sum_{r=0}^p 
  \Delta (\chi(X_2^+) \ttv_{m,r,+})  +  (1 -\chi_-(X_1^+)) \sum_{r=0}^p \left\{  \Delta
  ( \chi_{\text{macro},+}(X_1^+,X_2^+)  \ttv_{m,r,+} ) \right\}.
\end{equation}
The second term
in \eqref{EquationPreuvePenible1}
belongs to $\mathfrak{V}_{\beta,\gamma}^0(\widehat{\Omega}^+)$ for any
real numbers $\beta$ and $\gamma$ because
$ (1 -\chi_-(X_1^+))\nabla\chi_{\text{macro},+}$ and is $ (1
-\chi_-(X_1^+)) \Delta\chi_{\text{macro},+}$ are compactly supported and $\Delta \ttv_{m,r,+}=0$. 
Next,
$$
\Delta \left( \chi(X_2^+) \ttv_{m,r,+} \right)= 2 \chi'(X_2^+) \partial_{X_2^+} \ttv_{m,r,+} +
\chi''(X_2^+) \ttv_{m,r,+}.
$$
Using the Taylor expansion with integral remainder,
$$
\ttv_{m, r,+}(X_1^+,X_2^+) = \sum_{k=0}^{K-1} \partial_{X_2^+}^k \ttv_{m,r}(X_1^+,
0^\pm) \frac{(X_2^+)^k}{k!}+ \frac{1}{(K-1)!} \,\int_{0}^{X_2^+} \frac{\partial^{K} \ttv_{m,r,+}(X_1^+,
t) } {\partial t^K} (X_2^+ -t)^{K-1} dt,
$$
reminding that $| \partial_{X_2^+}^K \ttv_{m,r,+} |$ is smaller than $C
|X_1^+|^{\lambda_m - r - K} \ln |X_1^+|^r$ on the support of the derivative of $\chi(X_2^+)$, 
it is verified that
\begin{multline}\label{formulePenible1}
\Delta (\chi(X_2^+) \ttv_{m,r,+})  = \sum_{k=0}^{p-r} \delta_{k}^{\text{even}}
(-1)^{\lfloor k/2\rfloor} \left(2 \partial_{X_1^+}^k  \langle \ttv_{m,r,+}
  \rangle \langle  g_k \rangle  + \frac{1}{2} \partial_{X_1^+}^k [
  \ttv_{m,r,+} ] [g_k]\right) \\
+  \sum_{k=0}^{p-r} \delta_{k}^{\text{odd}}
(-1)^{\lfloor k/2\rfloor} \left(2 \partial_{X_1^+}^{k-1}
  \langle \partial_{X_2^+} \ttv_{m,r,+}
  \rangle \langle  g_k \rangle  + \frac{1}{2}
\partial_{X_1^+}^{k-1} [\partial_{X_2^+}
  \ttv_{m,r,+} ]\,   [g_k] \right) + r_{m,r,p}(X_1^+, X_2^+),
\end{multline} 
where $r_{m,q,p}$ is a smooth function whose support is included in the band
$ 1 \leq |X_2^+| \leq 2$ that satisfies   
\begin{equation}
 | r_{m,r,p} | \leq  C |X_1^+|^{\lambda_m - p - 1} \ln | X_1^+|^r,
\end{equation}
for a constant $C$ depending on $m$, $r$ and $p$.  We remind that $\langle \ttv_{m,r,+}\rangle$ and $[ \ttv_{m,r,+}]$ denote the mean and jump
values of $\ttv_{m,r,+}$ across the interface $\partial
\mathcal{K}^{+,1}\cap \partial \mathcal{K}^{+,2}$ and that
the functions $\langle g_k\rangle$ and $[g_k]$ are defined
in~\eqref{SautMoyennegp}. These functions only depend on $X_2^+$ and
their support coincides with the support of $\chi'(X_2^+)$. To obtain
\eqref{formulePenible1}, we have used the fact that $\Delta
\ttv_{m,r,+}=0$ on the support of $\langle g_k\rangle$ and $[g_k]$ so
that we can use the formulas~\eqref{FormulePaire}-\eqref{FormuleImpaire}. 
 Noting that
$\partial_{X_1^+}^k \ttv_{m, r,+}(X_1^+, 0^\pm) = (-1)^k \partial_{R^+}^k
\left\{ (R^+)^{\lambda_m-r}
w_{m,r,+}(\ln R^+, \pi^\pm) \right\}$, and using the definition
\eqref{gnrqplust} of $g_{m,r,k,+}^{\mathfrak{t}}$, we see that
\begin{equation}
\partial_{X_1^+}^k\langle \ttv_{m,r,+} \rangle = (-1)^k  \partial_{R^+}^k
\left\{ 
\langle  (R^+)^{\lambda_m -r} w_{m,r,+}(\ln R^+) \rangle \right\} = |X_1^+|^{ \lambda_m -
  r-k} g_{m,r,k,+}^{\mathfrak{t}}(\ln |X_1^+|).  
\end{equation}
Analogously, since  $\partial_{X_2^+}\ttv_{m,r,+}(X_1^+,0^ \pm) = -  (R^+)^{\lambda_m-r-1} 
\partial_\theta \ w_{m,r,+}(\ln R^+, \pi^ \pm)$, using the definition~\eqref{gnrqplusn} of
$g_{m,r,k,+}^{\mathfrak{n}}$, we have
\begin{equation}
\partial_{X_1^+}^{k-1} \langle \partial_{X_2^+} \ttv_{m,r,+} \rangle = |X_1^+|^{ \lambda_m -
  r-k} g_{m,r,k,+}^{\mathfrak{n}}(\ln |X_1^+|).  
\end{equation}
Next, we replace the $[\ttv_{m,r,+}] = [
(R^+)^{\lambda_m -r } w_{m,r,+}]$ with
its explicit expression~\eqref{Defanpplus}, taking into account that  $\mathcal{D}_{0}^{\mathfrak{t}}=\mathcal{D}_{0}^{\mathfrak{n}}=0$, we obtain 
\begin{equation}
\partial_{X_1^+}^k[ \ttv_{m,r,+} ] = (|X_1^+|)^{\lambda_m - k - r } \left(
  \sum_{j=0}^{r} \mathcal{D}_{r-j}^{\mathfrak{t}}
  g_{m,j,k+r-j,+}^{\mathfrak{t}}(\ln |X_1^+|) + \mathcal{D}_{r-j}^{\mathfrak{n}}
  g_{m,j,k+r-j,+}^{\mathfrak{n}}(\ln |X_1^+|) \right).
\end{equation} 
Similarly, substituting $[\partial_{X_2^+}\ttv_{m,r,+}] =
-  (R^+)^{\lambda_m-r-1} [
\partial_\theta \ w_{m,r,+}(\ln R^+, \pi)]$ for
its explicit expression~\eqref{Defbnpplus}, using the
relation~\eqref{relationMagique1} and the fact that
$\mathcal{N}_{1}^{\mathfrak{n}}$ and $\mathcal{N}_{1}^{\mathfrak{t}}$ vanish, we get
\begin{equation}
\partial_{X_1^+}^{k-1} [ \partial_{X_2^+} \ttv_{m,r,+} ] = (|X_1^+|)^{\lambda_m - k - r } \left(
  \sum_{j=0}^{r} \mathcal{N}_{r+1-j}^{\mathfrak{t}}
  g_{m,j,k+r-j,+}^{\mathfrak{t}}(\ln |X_1^+|) + \mathcal{N}_{r+1-j}^{\mathfrak{n}}
  g_{m,j,k+r-j,+}^{\mathfrak{n}}(\ln |X_1^+|) \right).
\end{equation}
Substituting the left-hand sides of the last four equalities for their
right-hand sides in~\eqref{formulePenible1},
summing the contribution~\eqref{formulePenible1} from $r=0$ to $p$, we end up with
\begin{multline}\label{LaplacienRmpmacro}
\Delta \left(
\chi_{\text{macro},+}\,\mathcal{U}_{m,p,\text{macro}}\right)  =
 R_{m,p,\text{macro}}(X_1^+,X_2^+)
\\ + \chi_-(X_1^+)\sum_{r=0}^{p} |X_1^+|^{\lambda_m-r}
\sum_{k=0}^r \left( g_{m, r-k,k,+}^{\mathfrak{t}}(\ln |X_1^+|)
 \mathcal{A}_{k}^{\mathfrak{t}} (X_2^+) + g_{m, r-k,k,+}^{\mathfrak{n}}(\ln |X_1^+|)
  \mathcal{A}_{k}^{\mathfrak{n}} (X_2^+) \right) 
\end{multline} 
where,
\begin{equation}
R_{m,p,\text{macro}} = \chi_-(X_1^+)\sum_{r=0}^p r_{m,r,p}(X_1^+, X_2^+) + (1 -\chi_-(X_1^+)) \sum_{r=0}^p \left\{  \Delta
  ( \chi_{\text{macro},+}(X_1^+,X_2^+)  \ttv_{m,r,+} ) \right\},
\end{equation} 
\begin{equation}
  \mathcal{A}_{k}^{\mathfrak{t}} (X_2^+) = 2 \delta_{k}^{\text{even}}
  (-1)^{\lfloor k/2 \rfloor} < g_k > + \sum_{\ell=0}^{k-1} \left(
  \delta_{\ell}^{\text{even}}(-1)^{\lfloor \ell/2 \rfloor} \frac{[g_\ell]}{2}
  \mathcal{D}_{k-\ell}^{\mathfrak{t}}  + 
  \delta_{\ell}^{\text{odd}}(-1)^{\lfloor \ell/2 \rfloor} \frac{[g_\ell]}{2}
  \mathcal{N}_{k-\ell+1}^{\mathfrak{t}}   \right),
\end{equation}
and
\begin{equation}
  \mathcal{A}_{k}^{\mathfrak{n}} (X_2^+) = 2 \delta_{k}^{\text{odd}}
  (-1)^{\lfloor k/2 \rfloor} < g_k > + \sum_{\ell=0}^{k-1} \left(
  \delta_{\ell}^{\text{ even}}(-1)^{\lfloor \ell/2 \rfloor} \frac{[g_\ell]}{2}
  \mathcal{D}_{k-\ell}^{\mathfrak{n}}  + 
  \delta_{\ell}^{\text{odd}}(-1)^{\lfloor \ell/2 \rfloor} \frac{[g_\ell]}{2}
  \mathcal{N}_{k-\ell+1}^{\mathfrak{n}}   \right).
\end{equation}
Comparing the previous two formulas with \eqref{DefFpt}-\eqref{DefFpn}, it is
recognized that
\begin{equation}\label{FormuleMagiqueMacro}
\mathcal{A}_{k}^{\mathfrak{t}}  =-\Delta W_{k}^{\mathfrak{t}}-
2 \partial_{X_1^+} W_{k-1}^{\mathfrak{t}}- W_{k-2}^{\mathfrak{t}}
\quad \quad \mbox{and} \quad \quad \mathcal{A}_{k}^{\mathfrak{n}}  =-\Delta W_{k}^{\mathfrak{n}}-
2 \partial_{X_1^+} W_{k-1}^{\mathfrak{n}}- W_{k-2}^{\mathfrak{n}}.
\end{equation} 
Applying the first point of Proposition~\ref{propositionTechniqueSobolevPoids}, we note that $R_{m,p,\text{macro}}$
belongs to $\mathfrak{V}_{\beta, \gamma}^0(\widehat{\Omega}^+)$  for
$\beta < \gamma - (\lambda_m -p) + \frac{3}{2}$, and therefore for
any $\beta < 2 - (\lambda_m -p)$. \\

\noindent The computation of $\Delta \left( \chi_-(X_1^+)\mathcal{U}_{m,p,\text{BL}}
  \right)$ (last term to evaluate in~\eqref{LaplaceAsymptoticBlock}) is much easier. We remind that $p_{m,r,+}$, defined in~\eqref{definitionPmr}, has the following expression:
\begin{equation*}
p_{m,r,+}(\ln |X_1^+|,\mX^+) =
\left( 
\sum_{k=0}^r \, \;
g_{m,r-k,k,+}^{\mathfrak{t}}(\ln |X_1^+|)\;
  W_k^{\mathfrak{t}} \left(\mX^+\right)
  +\sum_{k=1}^r \;
  g_{m,r-k,k,+}^{\mathfrak{n}}(\ln |X_1^+|) \;
  W_k^{\mathfrak{n}}\left(\mX^+
  \right) \right).
\end{equation*}
Then, using the definition~\eqref{gnrqplusn} of
$g_{m,r,k,+}^{\mathfrak{t}}$ we see that
\begin{multline}
\dsp \Delta \left\{ |X_1^+|^{\lambda_{m}-r} g_{m,r-k,k,+}^{\mathfrak{t}}(\ln |X_1^+|)\;
  W_k^{\mathfrak{t}}(\mX^+) \right\}  =  |X_1^+|^{\lambda_{m}-r}\,
g_{m,r-k,k,+}^{\mathfrak{t}}(\ln |X_1^+|) \, \Delta
W_k^{\mathfrak{t}}(\mX^+)  \\ \dsp + |X_1^+|^{\lambda_{m}-r -1}  \left( 2
  \,g_{m,
  r-k,k+1,+}^{\mathfrak{t}}(\ln |X_1^+|) \,  \partial_{X_1^+}
W_k^{\mathfrak{t}}(\mX^+)  \right) + \; |X_1^+|^{\lambda_{m}-r -2}  \left( g_{m,
  r-k,k+2,+}^{\mathfrak{t}}(\ln |X_1^+|) \, W_k^{\mathfrak{t}}(\mX^+)\right)  .
\end{multline}
The reader may verify that a similar formula holds for $\Delta \left\{
  |X_1^+|^{\lambda_{m}-r} g_{m,r-k,k,+}^{\mathfrak{n}}(\ln |X_1^+|)
\right\}$, replacing the subindex $\mathfrak{t}$ with $\mathfrak{n}$ . It follows that
\begin{multline}\label{LaplacienRmpBL}
\Delta \left( \chi_-(X_1^+)\mathcal{U}_{m,p,\text{BL}}
  \right)= \chi_-(X_1^+) \left( \sum_{r=0}^p |X_1^+|^{\lambda_m -r}  \left( \sum_{k=0}^r
  g_{m,r-k,k,+}^{\mathfrak{t}}(\ln |X_1^+|) \left( \Delta W_k^{\mathfrak{t}}
    + 2 \partial_{X_1^+} W_{k-1}^{\mathfrak{t}} +
    W_{k-2}^{\mathfrak{t}}\right)  \right) \right) \\
+ \chi_-(X_1^+) \left(\sum_{r=0}^p |X_1^+|^{\lambda_m -r}  \left( \sum_{k=0}^r
  g_{m,r-k,k,+}^{\mathfrak{n}}(\ln |X_1^+|) \left( \Delta W_k^{\mathfrak{n}}
    + 2 \partial_{X_1^+} W_{k-1}^{\mathfrak{n}} +
    W_{k-2}^{\mathfrak{n}}\right)  \right) \right) + R_{m,p,\text{BL}}(X_1^+,
X_2^+),
\end{multline}
where,
\begin{multline}
 R_{m,p,\text{BL}}(X_1^+,X_2^+) = \\\chi_-(X_1^+) |X_1^+|^{\lambda_m-p-1} \sum_{k=0}^{p+1} \left(
   2 \,
   g_{m,p+1-k,k,+}^{\mathfrak{t}}(\ln |X_1^+|)\,\partial_{X_1^+} W_{k-1}^{\mathfrak{t}}\;+ \;2
   g_{m,p+1-k,k,+}^{\mathfrak{n}}(\ln |X_1^+|)  \, \partial_{X_1^+}  W_{k-1}^{\mathfrak{n}} \right)  \\
+\chi_-(X_1^+) \sum_{r=p+1}^{p+2} \sum_{k=0}^{r} |X_1^+|^{\lambda_m-r} \left(
  g_{m,r-k,k,+}^{\mathfrak{t}}(\ln |X_1^+|)  W_{k-2}^{\mathfrak{t}}
  +g_{m,r-k,k,+}^{\mathfrak{n}}(\ln |X_1^+|)  W_{k-2}^{\mathfrak{n}}  \right)\\ + 2 \nabla \chi_-(X_1^+) \cdot \nabla  \mathcal{U}_{m,p,\text{BL}}
+ \mathcal{U}_{m,p,BL}  \Delta (\chi_-(X_1^+)).
\end{multline}
The support of the terms of the fourth
line are included in a band $B = \text{supp}(\chi'_-(X_1^+)) \cap
\widehat{\Omega}^+$. Since, the function and the derivative of
$\mathcal{U}_{m,p,\text{macro}}$  are exponentially decaying in this band, it
is verified that these terms also belong to $\mathfrak{V}_{\beta,\gamma}^0(\widehat{\Omega}^+)$ for any
real numbers $\beta$ and $\gamma$. The terms of the second and third
lines are  of the form $|X_1^+|^{\lambda_m-p-1} w_1(\ln
|X_1^+|, X_1^+, X_2^+)$ or $|X_1^+|^{\lambda_m-p-2} w_2(\ln
|X_1^+|, X_1^+, X_2^+)$ where $w_1$ and $w_2$ are polynomial functions with
respect to $\ln |X_1^+|$: for $i=1$ or $i=2$,
$
w_i(t, X_1^+, X_2^+)  = \sum_{s=0}^q t^s g_{i,s}(X_1^+,X_2^+)$. The functions $g_{i,s}$ belong to $L^2(\mathcal{B})$ and are such
that $\| g_{i,s} e^{|X_2^+|/2} \|_{L^2(\mathcal{B})} < + \infty$ (they
are exponentially decaying).
Therefore, the third point of Proposition~\ref{propositionTechniqueSobolevPoids} ensures that
$R_{m,p,\text{BL}}$ belongs to $\mathfrak{V}_{\beta, \gamma}^0(\widehat{\Omega}^+)$  for
$\beta < \gamma - (\lambda_m -p) + \frac{3}{2}$, and consequently for
any $\beta < 2 - (\lambda_m -p)$.\\

\noindent Finally, collecting \eqref{LaplacienRmpmacro}-\eqref{FormuleMagiqueMacro} and
\eqref{LaplacienRmpBL}, we see that
\begin{multline*}
\Delta \mathcal{U}_{m,p,+} =  2 \nabla \chi(R^+) \cdot \nabla
(\chi_{\text{macro},+}\mathcal{U}_{m,p,\text{macro}}  + \chi_-(X_1^+) \mathcal{U}_{m,p,\text{BL}} ) +
\Delta \chi(R^+) (\chi_{\text{macro}}\mathcal{U}_{m,p,\text{macro}}  + \chi_-(X_1^+)
\mathcal{U}_{m,p,\text{BL}} )  \\
+ R_{m,p,\text{BL}}(X_1^+,X_2^+) + R_{m,p,\text{macro}}(X_1^+,X_2^+). 
\end{multline*} 
which belongs to $\mathfrak{V}_{\beta, \gamma}^0(\widehat{\Omega}^+)$  for
for any $\beta < 2 - (\lambda_m -p)$.

\subsection{Proof of Proposition~\ref{propositionAsymptoticNearField}} \label{AppendixNFPreuvePropAsymptotic}
The first part of the proof consists in showing that Problem~\eqref{ProblemeModeleChampProche} has a unique solution $u
\in \mathfrak{V}_{\beta', \gamma}^2{(\widehat{\Omega}^+)}$ for any $\beta' < 1/2$. As usual,
the proof of existence  relies on a variational existence result, which
is given by  Proposition~\ref{PropositionProblemeChampProche} in the
present case. We first verify that the hypothesis on $f$ and $g$ of
Proposition~\eqref{propositionAsymptoticNearField} guarantee that we
can apply  Proposition~\ref{PropositionProblemeChampProche}.
Since $\frac{1}{2} < \gamma < 1$, $\rho^{\gamma+1} \geq  1$, we have
$$
\left\| \sqrt{1 + (R^+)^2} f \right\|_{L^2 (\widehat{\Omega}^+)} \leq
C \| (1 + R^+) \rho^{\gamma+1} f  \|_{L^2 (\widehat{\Omega}^+)} .
$$
Therefore,  since $\beta >3$,$\frac{1}{2} < \gamma < 1$ and $\beta - \gamma - 1>1$, $\sqrt{1 +
  (R^+)^2} f$ belongs to $L^2 (\widehat{\Omega}^+)$: 
\begin{equation}\label{funplusRL2V}
\left\| \sqrt{1 + (R^+)^2} f \right\|_{L^2 (\widehat{\Omega}^+)} \leq
C \| (1 + R^+)^{\beta -\gamma -1} \rho^{\gamma+1} f  \| \leq C \| f\|_{\mathfrak{V}_{\beta, \gamma
  }^0(\widehat{\Omega}^+)}.
\end{equation}
Then,  $g \in H^{1/2}(\widehat{\Gamma}_\hole)$ being compactly
supported, Proposition~\ref{PropositionProblemeChampProche} ensures that Problem~\eqref{ProblemeModeleChampProche} has a unique solution $u
\in \mathfrak{V}(\widehat{\Omega}^+)$.  Next, we check that the
solution $u$ belongs to $\mathfrak{V}_{\beta'-1, \gamma-1
 }^1(\widehat{\Omega}^+)$ for any $\beta' < 1/2$: since $\gamma \in (\frac{1}{2},
1)$,
$\rho^{(\gamma-1) -1 + p + \delta_{p,0}} < 1$ for $p=0$ or $p=1$,
which means in particular that
$$
\left\| v \right\|_{\mathfrak{V}_{\beta'-1, \gamma-1
 }^1(\widehat{\Omega}^+) } \leq C \left( \left\| (1+R^+)^{\beta'
     -\gamma -1} v 
\right\|_{L^2(\widehat{\Omega}^+) }   + \left\| (1+R^+)^{\beta' -
    \gamma} \nabla v
\right\|_{L^2(\widehat{\Omega}^+) }   \right).
$$
For any $\beta' < 1/2$, the right-hand side of the previous equality is
controlled by $\| v \|_{\mathcal{V}(\widehat{\Omega}^+)}$, which
confirms that 
$u \in \mathfrak{V}_{\beta'-1, \gamma-1
 }^1(\widehat{\Omega}^+)$ for any $\beta' < 1/2$.  In fact, using a weighted
elliptic regularity argument (Proposition 4.1 in~\cite{Nazarov205}),
taking into account that $\widehat{\Gamma}_\hole$ is smooth
 and that $g \in
H^{1/2}(\widehat{\Gamma}_\hole)$ is compactly supported, we deduce
that $u$ actually belongs to $\mathfrak{V}_{\beta', \gamma}^2{(\widehat{\Omega}^+)}$ for
$\beta' < 1/2$. Moreover the following estimate holds: for any
$\beta'<1/2$, there exists a constant $C$ such that
\begin{equation}
\| u \|_{\mathfrak{V}_{\beta', \gamma}^2{(\widehat{\Omega}^+)}} <  C
\left( 
\| f \|_{\mathfrak{V}_{\beta, \gamma}^0{(\widehat{\Omega}^+)}} + \| g
\|_{H^{1/2}(\widehat{\Gamma}_\hole)} \right).
\end{equation}
It should be noted that the assumption on the smoothness of
$\widehat{\Gamma}_\hole$ is obviously required for having $v
\in \mathfrak{V}_{\beta', \gamma}^2{(\widehat{\Omega}^+)}$. 
Problem~\eqref{ProblemeModeleChampProche} being linear, the uniqueness
of $u \in \mathfrak{V}_{\beta', \gamma}^2{(\widehat{\Omega}^+)}$ turns out to be a direct consequence of Proposition~\ref{PropositionProblemeChampProche}.  \\

\noindent The second part of the proof consists in proving  the
asymptotic formula~\eqref{AsymptoticExpansionSolutionCP}.  It relies on a finite
number of applications of
Lemma~\ref{LemmeAsymptoticAppliNazarov} in order to exhibit the
asymptotic behavior of $u$ as $R^+$ tends to $+\infty$. We start with
applying Lemma~\ref{LemmeAsymptoticAppliNazarov} picking $\beta_1 =
\frac{2}{5}< \frac{1}{2}$ (the choice of $2/5$ is not important) and
$\beta_2 = \beta_1 +1 -\varepsilon$, $\varepsilon>0$. Then, for
$\varepsilon$ sufficiently small, $\beta^2$ is admissible and the sum in \eqref{sommeUx} is empty since
there is no integer value of $k$ such that
$$
 -\frac{2}{5} < 1 - \beta_2 < \frac{2}{3} k < 1 - \beta_1 = \frac{3}{5}.
$$  
As a result, choosing $\beta^0 = \beta^2-\varepsilon = \frac{7}{5} -2\varepsilon$, we deduce that $u$ is actually in
 $\mathfrak{V}_{\gamma, \frac{7}{5} -2\varepsilon}^2(\widehat{\Omega}^+)$ for
 any $\varepsilon>0$ and for $1-\gamma$ sufficiently small. In fact,
 applying a second time Lemma~
\ref{LemmeAsymptoticAppliNazarov}
 with $\beta^1=  \frac{7}{5}
 -2 \varepsilon$ and an admisible $\beta^2 < \beta^1 + 2\varepsilon$, we
 see that  $u$ is in
 $\mathfrak{V}_{\frac{7}{5}, \gamma}^2(\widehat{\Omega}^+)$ for
 $1-\gamma$ sufficiently small. Of course, the following estimate holds
\begin{equation}
\| u \|_{\mathfrak{V}_{7/5, \gamma}^2{(\widehat{\Omega}^+)}} <  C \left(
\| f \|_{\mathfrak{V}_{\beta, \gamma}^0{(\widehat{\Omega}^+)}} + \| g
\|_{H^{1/2}(\widehat{\Gamma}_\hole)} \right).
\end{equation}
Next, let $\beta^1 = \frac{7}{5}$ and $\beta^2 = \beta^1 + 1 -
\varepsilon = \frac{12}{5} -
\varepsilon$. For $\varepsilon$ small enough $\beta^2$ is admissible,
and we have
$$
-\frac{4}{3}<1 - \beta^2 < - \frac{2}{3} < 1 -\beta^2.
$$ 
Thanks to Lemma~\ref{LemmeAsymptoticAppliNazarov}, there exist a constant $\mathscr{L}_{-1}(u)$ and a remainder $\tilde{u}$
belonging to $\mathfrak{V}_{\beta_0,\gamma }^2(\widehat{\Omega}^+)$
for any $\beta^0<\beta^2$ (providing that $1-\gamma$ sufficiently small) such that
$$
u = \mathscr{L}_{-1}(u) \mathcal{U}_{-1,3,+} + \tilde{u} \quad
\mbox{with} \quad 
| \mathscr{L}_{-1}(u) | + \| \tilde{u} \|_{\mathfrak{V}_{\beta^0,
    \gamma}^2{(\widehat{\Omega}^+)}} <  C \left( 
\| f \|_{\mathfrak{V}_{\beta, \gamma}^0{(\widehat{\Omega}^+)}} + \| g
\|_{H^{1/2}(\widehat{\Gamma}_\hole)}  \right).
$$
Again, applying a second time
Lemma~\ref{LemmeAsymptoticAppliNazarov} with $\beta^1 = \frac{12}{5} -
\varepsilon$ and  $\beta^2 = \frac{12}{5} +
\varepsilon$ for $\varepsilon$ small enough, we see that $\tilde{u}
\in \mathfrak{V}_{12/5,\gamma}^2(\widehat{\Omega}^+)$ (the sum in
\eqref{sommeUx} is empty in this case), with $ \| \tilde{u}
\|_{\mathfrak{V}_{12/5, \gamma}^2{(\widehat{\Omega}^+)}} <  C \left( 
\| f \|_{\mathfrak{V}_{\beta, \gamma}^0{(\widehat{\Omega}^+)}} +\| g
\|_{H^{1/2}(\widehat{\Gamma}_\hole)} \right)$.\\

\noindent Now, we shall
pursue the procedure, working only on the remainder.  In fact, we
shall consider the slightly different remainder 
$$
u_1= u- \mathscr{L}_{-1}(u) \mathcal{U}_{-1,p,+},
$$ 
where the positive integer $p\geq 3$ is chosen sufficiently large in order to ensure that $\Delta
u_1$ belongs to $\mathfrak{V}_{\beta, \gamma}^0(\widehat{\Omega}^+)$ (Lemma~\ref{LemmeSuperDecroissance} guarantees the existence of such
an integer $p$ (take for instance the smallest integer $p$ such that $p> \beta - 2 - \frac{2}{3}$)). Moreover, since 
$$
u_1 = \tilde{u} +  \mathscr{L}_{-1}(u) (\mathcal{U}_{-1,3,+} - \mathcal{U}_{-1,p,+} ),
$$ 
$u_1$ belongs to $\mathfrak{V}_{12/5,\gamma}^2(\widehat{\Omega}^+)$
for $1 -\gamma$ sufficiently small. We naturally have
$$
\| u_1 \|_{\mathfrak{V}_{12/5, \gamma}^2{(\widehat{\Omega}^+)}} <  C \left(
\| f \|_{\mathfrak{V}_{\beta, \gamma}^0{(\widehat{\Omega}^+)}} +\| g
\|_{H^{1/2}(\widehat{\Gamma}_\hole)} \right).
$$
Note also that the restriction $g_1$ of $\partial_n u_1$ to
$\widehat{\Gamma}_\hole$ is compactly supported and belongs to
$H^{1/2}(\widehat{\Gamma}_\hole)$ since $\widehat{\Gamma}_\hole$ is smooth. Then, we shall
apply again Lemma~\ref{LemmeAsymptoticAppliNazarov} with
$\beta^1 = \frac{12}{5} $ and  $\beta^2 = \frac{17}{5}-\varepsilon $.
We remark that, for $\varepsilon$ sufficiently small, $\beta^2$ is
admissible and,  for $k=-2$, and $k=-3$, 
$$
1 - \beta^1 < \frac{2 k}{3}  < 1 - \beta^2.
$$
Then, Lemma~\ref{LemmeAsymptoticAppliNazarov} ensures that there are
two constants $\mathscr{L}_{-2}(u)$ and $\mathscr{L}_{-3}(u)$ such that
$$
u_1 = \mathscr{L}_{-2}(u) \mathcal{U}_{-2,3,+} +
\mathscr{L}_{-3}(u)\mathcal{U}_{-3,3,+} + \tilde{u}_1, \quad
\tilde{u}_1\in \mathfrak{V}_{\beta_0,\gamma }^2(\widehat{\Omega}^+) \mbox{
for any} \; \beta^0<\beta^2.
$$
 Moreover,
$$
| \mathscr{L}_{-2}(u) | + | \mathscr{L}_{-3}(u) | + \| \tilde{u}_2
\|_{\mathfrak{V}_{\beta^0, \gamma}^2{(\widehat{\Omega}^+)}} <  C
\left( 
\| f \|_{\mathfrak{V}_{\beta, \gamma}^0{(\widehat{\Omega}^+)}} + \| g
\|_{H^{1/2}(\widehat{\Gamma}_\hole)} \right).
$$
Applying a second time
Lemma~\ref{LemmeAsymptoticAppliNazarov} with $\beta^1 = \frac{17}{5} -
\varepsilon$ and  $\beta^2 = \frac{17}{5} +
 \varepsilon$ for $\varepsilon$ small enough, we see that
 $\tilde{u}_1$ is indeed in
 $\mathfrak{V}_{17/5,\gamma}^2(\widehat{\Omega}^+)$ and satisfies
\begin{equation*}
 \| \tilde{u}_2 \|_{\mathfrak{V}_{17/5,
     \gamma}^2{(\widehat{\Omega}^+)}} <  C \left(
\| f \|_{\mathfrak{V}_{\beta, \gamma}^0{(\widehat{\Omega}^+)}}+ \| g
\|_{H^{1/2}(\widehat{\Gamma}_\hole)} \right).
\end{equation*}

\noindent Then, we introduce the new remainder  
$$
u_2 = u_1 - \mathscr{L}_{-2}(u)\, \mathcal{U}_{-2,p,+} +  \mathscr{L}_{-3}(u)\, \mathcal{U}_{-3,p,+},
$$
where $p \geq 3$ is chosen sufficiently large in order to ensure that $\Delta
u_2$ belongs to $\mathfrak{V}_{\beta, \gamma}^0(\widehat{\Omega}^+)$. Of
course, $u_2$ belongs to
$\mathfrak{V}_{17/5,\gamma}^2(\widehat{\Omega}^+)$. Again, the restriction $g_2$ of $\partial_n u_2$ to
$\widehat{\Gamma}_\hole$ is compactly supported and belongs to $H^{1/2}(\widehat{\Gamma}_\hole)$ since $\widehat{\Gamma}_\hole$ is smooth.\\

\noindent Repeating the
procedure several times (applying
Lemma~\ref{LemmeAsymptoticAppliNazarov} at most $(2 \lfloor
\beta \rfloor+1)$ times) provides the asymptotic formula \eqref{AsymptoticExpansionSolutionCP} and the
associated estimate~\eqref{EstimationAsymptoticExpansionSolutionCP}.

\section{Technical results associated with the matching procedure:
  proof of Lemma~\ref{LemmeTechniqueChangementEchelle}}\label{sec:ProofLemmeTechniqueChangementEchelle}
We first prove~\eqref{FormuleMagiqueChangementEchelle}  by
induction. The base case $q = 0$ is obvious because $w_{m,0,\pm}$ is
independent of  $\ln R^+$ and $C_{m,0,0}=1$ (cf. Remark~\ref{RemCoeffCm2ki}).\\

\noindent Before we prove the inductive step, it is interesting to consider the cases $r=1$ and
$r=2$ separately. 
For $r=1$, we remark that $w_{m,1,+}(\theta^+, \ln R^+)$ does not
depend on $\ln R^+$. Indeed, $\lambda_m - 1 = \frac{2}{3} m -1$ does not
belong to $\Lambda$ ($\lambda_m -1$ is not a singular exponent). Then,
Lemma~\ref{LemmeTechniquesmq} ensures that $w_{m,1,+}(\theta^+, \ln
R^+)=w_{m,1,+}(\theta^+)$. It follows that $w_{m,1,+}(\theta^+, \ln
R^+ + \ln \delta) = w_{m,1,+}(\theta^+) {\;=\;} w_{m,1,+}(\theta^+,\ln
R^+)$. By contrast, For
$r=2$ and $m\neq 3$, $\lambda_m -2$ belongs to $\Lambda$. More precisely
$\lambda_{m} -2$ = $\lambda_{m-3}$. Then, we know that
$$
w_{m,2,+}(\theta^+, \ln R^+) = {w}_{m,2,0,+}(\theta^+) + 
w_{m,2,1,+}(\theta^+) \ln R^+,\; \mbox{with}\; w_{m,2,1,+}(\theta^+)= C \,w_{m-3,0,+}(\theta^+).
$$
and $C=\frac{4}{3\pi} \int_{0}^{\frac{3\pi}{2}} w_{m,2,1,+}(\theta^+)
,w_{m-3,0,+}(\theta^+)d\theta^+$. Therefore,
$$
w_{m,2,+}(\theta^+, \ln R^+ + \ln \delta) = w_{m,2,+}(\theta^+, \ln R^+) + C_{m,2,1,+}
w_{m-3,0,+}(\theta^+) \ln \delta,
$$
which proves formula~\eqref{FormuleMagiqueChangementEchelle} for
$r=2$. The previous argument may then be repeated inductively. The key
point to figure out
is that for any $q\in \N$ such that $q \neq m/3$, $\lambda_{m} -2q $ is always a singular
exponent, which  corresponds to $\lambda_{m-3q}$. \\

\noindent Let us start the inductive step.  Let $r \in \N$, and assume that
formula \eqref{FormuleMagiqueChangementEchelle} is valid up to order
$r -1$. Let us denote by $w_{LHS}^\delta$ and $w_{RHS}^\delta$
    the left and right-hand sides
of~\eqref{FormuleMagiqueChangementEchelle}, \ie,
\bds{
\begin{eqnarray*}
 w_{LHS}^\delta(\theta^+, \ln R^+) &=&  w_{m,r,+}^\delta(\theta^+, \ln R^+ + \ln \delta),\\
 w_{RHS}^\delta(\theta^+, \ln R^+) &=& \dsp \sum_{k=0}^{\lfloor r/2 \rfloor}
w_{m-3k,r-2k,+}(\theta^+, \ln R^+) \; \left( \sum_{i=0}^k C_{m,2k,i,\pm} \;
  (\ln \delta)^i \right).
\end{eqnarray*}
}
$w_{LHS}^\delta$ and $w_{RHS}^\delta$ have both a polynomial dependance
with respect to $\ln \delta$ and $\ln R^+$. 
 We shall see that $\ttv_{LHS}^\delta = (R^+)^{\lambda_m -r} w_{LHS}^\delta$
and $\ttv_{RHS}^\delta = (R^+)^{\lambda_m -r} w_{RHS}^\delta$ satisfy the
same problem and apply  Lemma~\ref{LemmeTechniquesmq} to conclude. 
First, we can rewrite $\ttv_{LHS}^\delta$ as
$$
\ttv_{LHS}^\delta(\theta^+, \ln R^+) = \delta^{r-\lambda_m}
\left( \varphi^\delta(R^+) \right)^{\lambda_m -r} \, w_{m,r,+}^\delta\left(\theta^+, \ln
\left(\varphi^\delta(R^+)\right) \right)\;\;
\mbox{with} \;\varphi^\delta(R^+) = R^+ \delta.
$$
Then, in view of the formula $ \left(R^+ \partial_{R^+}\right)\{ g(\varphi^\delta(R^+)) \}= \varphi^\delta(R^+) g'(\varphi^\delta(R^+))$ %
and reminding that the function
 $\varphi^{\lambda_m -r} w_{m,r,+}(\theta^+, \ln \varphi)$ is
harmonic in
$\mathcal{K}^{+,1}$ and $\mathcal{K}^{+,2}$, we have
\begin{equation*}
\Delta \ttv_{RHS}^\delta = \Delta \ttv_{LHS}^\delta = 0 \quad \mbox{in} \quad 
\mathcal{K}^{+,1} \cup \mathcal{K}^{+,2}. 
\end{equation*} 
 Then, we evaluate separately the jump values of $w_{LHS}^\delta$ and
 $w_{RHS}^\delta$. 
$$
[w_{LHS}^\delta] = \sum_{p=0}^{r-1} \mathcal{D}_{r-p}^{\mathfrak{t}}
g_{m,p,r-p,+}^{\mathfrak{t}} (\ln R^+ + \ln \delta) + \sum_{p=0}^{r-1}
\mathcal{D}_{r-p}^{\mathfrak{n}} g_{m,p,r-p,+}^{\mathfrak{n}} (\ln R^+ + \ln \delta)
$$
But, it is verified that
$$
(R^+)^{\lambda_n -i- j } g_{n,i,j,+}^{\mathfrak{t}}(\ln R^+ + \ln \delta)  =
(-1)^j \frac{\partial^j}{\partial (R^+)^j} \left( (R^+)^{\lambda_n -i} \langle
w_{n,i,+}(\pi, \ln R^+ + \ln \delta) \rangle \right).
$$
Then, for $p \leq r-1$, using the induction hypothesis and the fact
that $\lambda_m - p = \lambda_{m-3q} - (p-2q)$, we have
\begin{equation}\label{gtChangementEchelle}
g_{n,p,r-p,+}^{\mathfrak{t}}(\ln R^+ + \ln \delta)  = \sum_{q=0}^{\lfloor p/2 \rfloor}  \sum_{i=0}^q (\ln \delta)^i \, C_{m,2q,i,+} \,
g_{m-3q,p-2q,r-p,+}^{\mathfrak{t}}(\ln R^+),
\end{equation} 
and, analogously, 
\begin{equation}\label{gnChangementEchelle}
g_{n,p,r-p,+}^{\mathfrak{n}}(\ln R^+ + \ln \delta)  =
\sum_{q=0}^{\lfloor p/2 \rfloor} \sum_{i=0}^q (\ln \delta)^i \,
C_{m,2q,i,+} \,
g_{m-3q,p-2q,r-p,+}^{\mathfrak{n}}(\ln R^+) \quad (p \leq r-1).
\end{equation}
Therefore,
\begin{equation}
[w_{LHS}^\delta] =  \sum_{p=0}^{r-1} \sum_{q=0}^{\lfloor p/2 \rfloor} \sum_{i=0}^q (\ln \delta)^i \, C_{m,2q,i,+}
\left( \mathcal{D}_{r-p}^{\mathfrak{t}} \,
g_{m-3q,p-2q,r-p,+}^{\mathfrak{t}}(\ln R^+)  \;+ \;
\mathcal{D}_{r-p}^{\mathfrak{n}} \,
g_{m-3q,p-2q,r-p,+}^{\mathfrak{n}}(\ln R^+) \right). 
\end{equation}
On the other hand,
\begin{multline*}
[w_{RHS}^\delta] = \sum_{q=0}^{\lfloor r/2 \rfloor}  \sum_{i=0}^q (\ln
\delta)^i\, C_{m,2q,i,+}\, [w_{m-3q ,
  r-2p,+}]  \\
= \sum_{q=0}^{\lfloor r/2 \rfloor} \sum_{i=0}^q \sum_{p=0}^{r - 2q -1}  (\ln
\delta)^i\, C_{m,2q,i,+}
\left( \mathcal{D}_{r -2 q- p}^{\mathfrak{t}} g_{m-3q,p,r-2q-p,+}^{\mathfrak{t}}(\ln R^+)
+\mathcal{D}_{r -2 q- p}^{\mathfrak{n}} g_{m-3q,p,r-2q -p,+}^{\mathfrak{n}}(\ln R^+)
\right) 
\end{multline*}
Then,  using the change of index $p \leftarrow p+2q$, and
interchanging the sums over $p$ and $q$, we see that $[w_{RHS}] =
[w_{LHS}]$. Similar arguments yield $[\partial_{\theta^+} w_{RHS}] =
[\partial_{\theta^+} w_{LHS}]$. Finally
$$
\Delta (\ttv_{LHS} - \ttv_{RHS}) = 0, \quad  [\partial_{\theta^+} (w_{LHS}
-w_{RHS}) ] =0 ,\quad \mbox{and}  \quad [ w_{LHS}
-w_{RHS} ] =0.
$$
\noindent To conclude we consider separately the case $r$ odd or $r =
\frac{2m}{3}$ from the case and $r$ even and $r \neq \frac{2m}{3}$.
If $r$ is odd or $r = \frac{2m}{3}$, $\lambda_m -r$ is not a
singular exponent. Then, Lemma~\ref{LemmeTechniquesmq} ensures that
$\ttv_{LHS}^\delta  = \ttv_{RHS}^\delta$.  If $r$ is even and $r \neq \frac{2m}{3}$, then Lemma~\ref{LemmeTechniquesmq}
guarantees that 
$\ttv_{LHS}^\delta - \ttv_{RHS}^\delta$ is proportional to
$w_{m-3r/2,0,+} $. In other words, writing $$w_{LHS}^\delta(\theta^+, \ln
R^+) =\left(
\sum_{i=0}^{\lfloor r/2 \rfloor}  w_{LHS,i}^\delta(\theta^+) \, (\ln R^+)^i\right)
\;\mbox{ and } w_{RHS}^\delta(\theta^+, \ln R^+) =
\left( \sum_{i=0}^{\lfloor r/2 \rfloor}  w_{RHS,i}^\delta(\theta^+)\, (\ln
  R)^i\right),$$
we know that  $w_{LHS,i}^\delta=w_{RHS,i}^\delta$ for $i\neq0$, and $w_{LHS,0}^\delta=w_{RHS,0}^\delta + C(\delta) w_{m-r/2,0,+}$. We
 shall see that $C(\delta)$ vanishes. Indeed,
\bds{
\begin{multline}
 \int_{0}^{\frac{3 \pi}{2} } w_{LHS,0}^\delta(\theta^+)  \,
 w_{m -3r/2,0,+}(\theta^+) d \theta^+ = \sum_{i=0}^{\lfloor r/2 \rfloor } (\ln
 \delta)^i \int_{0}^{\frac{3\pi}{2}}  w_{m, r,i}(\theta^+) w_{m -3r/2,0}(\theta^+) 
d\theta^+\\ = \frac{3\pi}{4} \sum_{i=0}^{ r/2  }  C_{m,r,i,+} (\ln
 \delta)^i.
\end{multline}
}
On the other hand, using the orthogonality
condition~\eqref{ConditionOrthogonalitePlus} for  $k< r/2$, $k\neq
m/3$, the fact that $C_{m, 2k,i,\pm}=0$ for any integer $i$
such that $0 \leq i \leq k$ if $k=m/3$ (See Remark~\ref{RemCoeffCm2ki}), we
see that 
\begin{multline}
 \hspace{-0.4cm}\int_{0}^{\frac{3 \pi}{2} } w_{RHS}^\delta(\theta^+) \,  w_{m -3r/2,0,+}(\theta^+) d \theta^+ \hspace{-0.1cm}=\hspace{-0.1cm}
 \sum_{k=0}^{\lfloor r/2 \rfloor} \sum_{i=0}^{k} C_{m,2k,i,+} (\ln \delta)^i
 \int_{0}^{\frac{3 \pi}{2} }\hspace{-0.2cm} w_{m -3k,r-2k,0,+}(\theta^+) w_{m-3r/2, 0,+}(\theta^+)  d \theta^+ \\
=  \sum_{i=0}^{\lfloor r/2\rfloor} C_{m,r,i,+} (\ln \delta)^i
 \int_{0}^{\frac{3 \pi}{2} } |w_{m -3r/2,0,+}(\theta^+)|^2  d
 \theta^+ = \frac{3\pi}{4} \sum_{i=0}^{r/2}   C_{m,r,i,+} (\ln \delta)^i. 
\end{multline}
As a consequence $w_{LHS,0}^\delta = w_{RHS,0}^\delta$ and the proof of
formula~\eqref{FormuleMagiqueChangementEchelle} is complete.\\

\noindent Finally, Formula~\eqref{FormuleMagiqueChangementEchelleBL} then results from
\eqref{gtChangementEchelle}-\eqref{gnChangementEchelle} (which finally
hold for any $(n,p,r) \in \N^3$), and
\eqref{definitionPmr}:
\bds{
\begin{align*}
& p_{m,r,+}( \ln  (\delta|X_1^+|), \mX^+)  =   \sum_{p=0}^r \left(
  g_{m,r-p,p,+}^{\mathfrak{t}}(\ln  (\delta|X_1^+|)) W_p^\mathfrak{t}
  (\mX^+) + g_{m,r-p,p,+}^{\mathfrak{n}}(\ln  (\delta|X_1^+|))
  W_p^{\mathfrak{n}}(\mX^+)  \right)\\
& \quad =  \sum_{p=0}^r \sum_{q=0}^{\lfloor (r-p)/2\rfloor} \sum_{i=0}^q (\ln
\delta)^i C_{m,2q,i,+}  
\left( g_{m-3q, r-p-2q,p,+}^{\mathfrak{t}} W_p^\mathfrak{t}
  (\mX^+) +  g_{m-3q, r-p-2q,p,+}^{\mathfrak{n}} W_p^\mathfrak{n} 
  (\mX^+)\right) \\
& \quad =  
\quad \sum_{q=0}^{\lfloor r/2\rfloor} \left(\sum_{i=0}^q (\ln
\delta)^i C_{m,2q,i,+}  \right) \left( \sum_{p=0}^{r-2q}g_{m-3q, r-p-2q,p,+}^{\mathfrak{t}} W_p^\mathfrak{t}
  (\mX^+) +  g_{m-3q, r-2q-p,p,+}^{\mathfrak{n}} W_p^\mathfrak{n}
  (\mX^+)\right) \\
& \quad =  \sum_{q=0}^{\lfloor r/2\rfloor} p_{m-3q, r-2q,+}(\ln |X_1^+|,
\mX^+) \left( \sum_{i=0}^q (\ln
\delta)^i C_{m,2q,i,+} \right).
\end{align*}
}

\bds{
\section{Technical results for the justification of the asymptotic
  expansion}
}

\subsection{Evaluation of the modeling error: proof of Lemma~\ref{LemmeErreurMod2}}\label{AppendixProofModellingError}
On the one hand, a direct computation shows that
\begin{equation}\label{ErreurModela}
\sum_{(n,q)\in D_{N_0}}
  \delta^{\frac{2}{3} n +q} \Delta  \Pi_{n,q}^\delta =\sum_{(n,q)\in
    D_{N_0}}  \delta^{\frac{2}{3}n+q-2}\left(\Delta_X \Pi_{n,q}^\delta +
    2 \partial_{x_1} \partial_{X_1}\Pi_{n,q}^\delta + \partial_{x_1}^2
  \Pi_{n,q-2}\right)(x_1, \frac{\mx}{\delta}) + \mathcal{R}_{N_0,BL}^\delta(\mx)
\end{equation}
where 
\begin{equation}\label{definitionRBLN0}
\mathcal{R}_{N_0,BL}^\delta(\mx) = \sum_{n=0}^{\lfloor \frac{3 N_0}{2}
  \rfloor } 2 \delta^{\frac{2n}{3} + \alpha_n -1 } \partial_{x_1} \partial_{X_1} \Pi_{n,\alpha_n   } + \sum_{n=0}^{\lfloor \frac{3 N_0}{2}
  \rfloor }  \sum_{q= \alpha_n +1}^{\alpha_n +2} \delta^{\frac{2n}{3} + q -2}  \partial_{x_1}^2
\Pi_{n,q-2   } \quad \alpha_n = \lfloor  N_0 - \frac{2n}{3} \rfloor
\end{equation}
On the other hand, writing a Taylor expansion of $u_{n,q}^\delta$ for $\pm x_2>0$,
\begin{equation}
u_{n,q}^\delta(x_1, x_2) = \sum_{p=0}^{\alpha_n -q} \frac{\partial^p
 u_{n,q}^\delta}{\partial x_2^p}(x_1,0^\pm) \frac{x_2^p}{p!} +\int_{0}^{x_2} \frac{\partial^{\alpha_n -q+1}
 u_{n,q}^\delta}{\partial t^{\alpha_n -q+1}}(x_1,t)
\frac{(x_2-t)^{\alpha_n -q}}{(\alpha_n -q)!} dt,
\end{equation} 
we see that
\begin{multline}\label{ErreurModelb}
\hspace{-0.3cm}\sum_{(n,q)\in D_{N_0}} \hspace{-0.2cm}
  \delta^{\frac{2}{3} n +q}   [\Delta, \chi(\frac{x_2}{\delta}) ]
  u_{n,q}^\delta(\mx) = \sum_{(n,q)\in D_{N_0}}
  \delta^{\frac{2}{3} n +q-2}  \sum_{\pm }\sum_{p=0}^q  \frac{\partial^p
  u_{n,q-p}^\delta}{\partial x_2^p}(x_1,0^\pm) \left( [\Delta_{\mX}, \chi_\pm(X_2) ]
    \frac{X_2^ p}{p!} \right) (\frac{\mx}{\delta}) \\+  \mathcal{R}_{N_0,\text{taylor}}^\delta(\mx)
\end{multline}
where
\begin{equation}
 \mathcal{R}_{N_0,\text{taylor}}^\delta(\mx) = \sum_{\pm} \sum_{(n,q)\in D_{N_0}}
 \delta^{\frac{2}{3} n +q}   [\Delta, \chi_\pm(\frac{x_2}{\delta}) ]
 \int_{0}^{x_2} \frac{\partial^{\alpha_n -q+1}
  u_{n,q}^\delta}{\partial t^{\alpha_n -q+1}}(x_1,t) \frac{(x_2-t)^{\alpha_n -q}}{(\alpha_n -q)!} dt.
\end{equation}
Subtracting equalities~\eqref{ErreurModela} and \eqref{ErreurModelb},
taking into account the periodic corrector equations~\eqref{PeriodicCorrectorEquations}, we obtain
\begin{equation}\label{DecompositionEmod2}
\mathcal{E}_{\mod, 2} =  -(1 -  \chi_+^\delta(\mx) -
  \chi_-^\delta(\mx))   1_{\Omega_{\mod}^\delta}  \left(\mathcal{R}_{N_0,\text{taylor}}^\delta(\mx)  + \mathcal{R}_{N_0,BL}^\delta(\mx)\right). 
\end{equation}

\noindent It remains to estimate $\mathcal{R}_{N_0,\text{taylor}}^\delta(\mx) $ and
$\mathcal{R}_{N_0,BL}^\delta(\mx)$ in $\Omega_{\mod}^\delta$. This should be done precisely
because the two terms are singular close to the two corners. For
$\mathcal{R}_{N_0,\text{taylor}}^\delta(\mx)$, we prove the following estimate:
\begin{lema}\label{LemmeErreurModelFF} For any $\varepsilon >0$ small
  enough, there are two positive constant $C>0$ and $\delta_0>0$  such that,
  for any $\delta < \delta_0$,
\begin{equation}
\left\| (1 -  \chi_+^\delta(\mx) -
  \chi_-^\delta(\mx))   1_{\Omega_{\mod}^\delta}\mathcal{R}_{N_0,\text{taylor}}^\delta(\mx) \right\|_{L^2(\Omega^\delta)} \leq C
\delta^{-3/2-\varepsilon} \left( \frac{\delta}{\eta(\delta)} \right)^{N_0}.
\end{equation}
\end{lema}
\begin{proof}
We first note that, for $\delta$ sufficiently small,  the functions
$1_{\Omega_{\mod}^\delta} [\Delta, \chi_\pm(\frac{x_2}{\delta}) ]
 \int_{0}^{\pm |x_2|} \frac{\partial^{\alpha_n -q+1}
  u_{n,q}^\delta}{\partial t^{\alpha_n -q+1}}(x_1,t)
\frac{(x_2-t)^{\alpha_n -q}}{(\alpha_n -q)!}$ are supported in the
union of the bands $B_1$ and $B_2$ where
$$
 B_1 =  (-L+  \eta(\delta)/2, L -\eta(\delta)/2  ) \times
( \delta,  2 \delta) \quad  B_2 =  (-L+  \eta(\delta)/2, L -\eta(\delta)/2  ) \times
( -2 \delta,   -\delta).
$$
In what follows, we shall obtain an estimate for $\left\| (1 -  \chi_+^\delta(\mx) -
  \chi_-^\delta(\mx))   1_{\Omega_{\mod}^\delta}\mathcal{R}_{N_0,\text{taylor}}^\delta(\mx) \right\|_{L^2(B_1)}$ but a strictly similar analysis can be carried out for
$B_2$. \\

\noindent In order to separate the singular (close to the corners) and the
regular (far from the corner) behaviors of the macroscopic
terms, we localize the error, namely
$$
\left\| (1 -  \chi_+^\delta(\mx) -
  \chi_-^\delta(\mx))
  1_{\Omega_{\mod}^\delta}\mathcal{R}_{N_0,\text{taylor}}^\delta(\mx)
\right\|_{L^2(B_1)} 
\leq  C \left( \left\| \mathcal{R}_{N_0,\text{taylor}}^\delta(\mx)
\right\|_{L^2(B_1^0) } +  \sum_{\pm} \left\| \mathcal{R}_{N_0,\text{taylor}}^\delta(\mx)
\right\|_{L^2(B_1^\pm) } \right),
$$
where 
$$
B_1^+ = (-L+  \eta(\delta)/2, -L/2 ) \times
( 0, 2 \delta), B_1^- = (L/2 , L -  \eta(\delta)/2) \times
( 0, 2 \delta), B_1^0 = (-L/2, L/2 ) \times
( 0, 2 \delta).
$$
We start with the analysis of $ \left\| \mathcal{R}_{N_0,\text{taylor}}^\delta(\mx)
\right\|_{L^2(B_1^0) }$, where the functions $u_{n,q}^\delta$ do not
blow up. By a direct computation, we see that there exists a constant
$C>0$ (independent of $u_{n,q}^\delta$) such that
\begin{equation}\label{LemmaAppendixEstimation1}
\int_{B_1^0}  \left| [\Delta, \chi_\pm(\frac{x_2}{\delta}) ]
 \int_{0}^{x_2} \frac{\partial^{\alpha_n -q+1}
  u_{n,q}^\delta}{\partial t^{\alpha_n -q+1}}(x_1,t)
\frac{(x_2-t)^{\alpha_n -q}}{(\alpha_n -q)!} \right|^2 dx_1 dx_2\leq
C 
  \delta^{2 \alpha_n -2q -3}   \| \partial_{x_2}^{\alpha_n -q+1} u_{n,q}^\delta \|_{L^2(B_1^0)}^2 . 
\end{equation}
But, a standard
elliptic regularity estimates (see \eg \cite{BrezisAnglais}) shows that $u_{n,q}^\delta$ is
smooth in $B_1^0$. Reminding that $u_{n,q}^\delta$ has a polynomial dependence
with respect to $\ln \delta$, it is verified that, for any $k\in \N$, for any $\varepsilon>0$, there exists a constant $C$ such that
$
\| u_{n,q}^\delta \|_{H^k(B_1^0)} \leq C \delta^{-\varepsilon}$.
Summing~\eqref{LemmaAppendixEstimation1} over $(n,q) \in D_{N_0}$,
using the fact that $\frac{2}{3} n - \lfloor \frac{2}{3} n \rfloor
>0$, we see that
\begin{equation}\label{estimationB0}
 \left\| \mathcal{R}_{N_0,\text{taylor}}^\delta(\mx)
\right\|_{L^2(B_1^0) } \leq    C \delta^{ N_0   -3/2 - \varepsilon}.
\end{equation}
 We turn to the estimation of  $\left\| \mathcal{R}_{N_0,\text{taylor}}^\delta(\mx)
\right\|_{L^2(B_1^+) }$ .We remind that
  the restriction $u_{n,q}^\delta$ to $\OmegaTop$
 belongs to
  $V_{2,\beta_{n,q}}^2(\OmegaTop)$ where
  $\beta_{n,q} = 1 + \frac{2n}{3} +q + \varepsilon$ ($\varepsilon>0$). Using a weighted elliptic
  regularity argument (see Corollary 6.3.3 in \cite{MR1469972}),  $u_{n,q}^\delta$ belongs to
  $V_{2,\beta_{n,q}+\ell}^{2+\ell}(\OmegaTop)$, for any $\ell \in \N$,
  and, as
  a direct consequence $\partial_{x_2}^\alpha
  u_{n,q}^\delta$ is in
  ${V}_{2,\beta_{n,q}-2+\alpha}^0(\OmegaTop)$.  The
  estimation~\eqref{LemmaAppendixEstimation1} (replacing $B_1^0$ with $B_1^+$) remains valid unless
  $\|\partial_{x_2}^{\alpha_n -q+2} u_{n,q}^\delta \|_{L^2(B_1^+)}$
  and  $\|\partial_{x_2}^{\alpha_n -q+1} u_{n,q}^\delta
  \|_{L^2(B_1^+)}$ are not uniformly bounded anymore. Nevertheless, since $
  \frac{2n}{3} + \varepsilon + \alpha_n>N_0,$
\begin{equation}\label{LemmaAppendixEstimation2}
\int_{B_1^+} |\partial_{x_2}^{\alpha_n -q+1} u_{n,q}^\delta|^2 d\mx =
\int_{B_1^+} |\partial_{x_2}^{\alpha_n -q+1} u_{n,q}^\delta|^2 (r^+)^{2(
  \frac{2n}{3} + \varepsilon + \alpha_n)} 
(r^+)^{-2(
  \frac{2n}{3} + \varepsilon + \alpha_n) } d\mx  \leq C  \eta(\delta)^{-2(
  \frac{2n}{3} + \varepsilon + \alpha_n) }.
\end{equation}
Introducing~\eqref{LemmaAppendixEstimation2} into
\eqref{LemmaAppendixEstimation1} and summing over $(n,q) \in D_{N_0}$
yields
\begin{equation}
 \left\| \mathcal{R}_{N_0,\text{taylor}}^\delta(\mx)
\right\|_{L^2(B_1^+) } \leq C \sum_{(n,q)\in D_{N_0}} \delta^{-3/2 -\varepsilon} 
  \left(\frac{\delta}{\eta(\delta)} \right)^{\frac{2n}{3} +
    \varepsilon + \alpha_n}.
\end{equation}
Since $\delta < \eta(\delta)$ and that $\frac{2n}{3} +
    \varepsilon + \alpha_n = N_0 +\frac{2n}{3}-\lfloor \frac{2n}{3}
    \rfloor >0$, we end up with
\begin{equation}\label{estimationBplus}
 \left\| \mathcal{R}_{N_0,\text{taylor}}^\delta(\mx)
\right\|_{L^2(B_1^+) } \leq C \delta^{   -3/2 - \varepsilon}  \left(\frac{\delta}{\eta(\delta)} \right)^{N_0 +
    \varepsilon}. 
\end{equation}
Naturally a similar estimation holds for $B_1^-$.
\noindent Collecting the results of~\eqref{estimationB0},
\eqref{estimationBplus} 
 gives the desired result.
\end{proof}
\noindent The analysis of the $\mathcal{R}_{N_0,BL}^\delta(\mx)$ leads to the
following lemma:
\begin{lema}
\label{LemmeErreurModelBL} For any $\varepsilon >0$ small
  enough, there are two positive constant $C>0$ and $\delta_0>0$  such that,
  for any $\delta < \delta_0$,
\begin{equation}
\left\| (1 -  \chi_+^\delta(\mx) -
  \chi_-^\delta(\mx))   \,1_{\Omega_{\mod}^\delta}
  \,\mathcal{R}_{N_0,BL}^\delta(\mx) \right\|_{L^2(\Omega^\delta)}
\leq C \, \left(
  \frac{\delta}{\eta(\delta)}\right)^{N_0+3/2} \,\delta^{-\varepsilon -2}.
\end{equation}
\end{lema}
\begin{proof}
As in the proof of Lemma~\ref{LemmeErreurModelFF}, we shall decompose
the error into three terms, which corresponds to evaluate the error in
three distinct regions of  $\Omega_{\mod}^\delta $ (where
$\Pi_{n,q}^\delta$ are either regular or singular):
\begin{equation*}
\left\| (1 -  \chi_+^\delta(\mx) -
  \chi_-^\delta(\mx))   \,1_{\Omega_{\mod}^\delta}
  \,\mathcal{R}_{N_0,BL}^\delta(\mx) \right\|_{L^2(\Omega^\delta)}
\leq \|
\mathcal{R}_{N_0,BL}^\delta(\mx)\|_{L^2(\Omega_{\mod,0})} +
\sum_{\pm}\| \mathcal{R}_{N_0,BL}^\delta(\mx)\|_{L^2(\Omega_{\mod,\pm}^\delta)}
\end{equation*}
where
\begin{equation*}
 \Omega_{\mod,0}^\delta = \left\{ (x_1,x_2) \in \Omega_{\mod}^\delta, \; -\frac{L}{2} < |x_1| <\frac{L}{2}   \right\}  \quad
\Omega_{\mod,\pm}^\delta = \left\{ (x_1,x_2) \in \Omega_{\mod}^\delta,
   \frac{\eta(\delta)}{2} < \pm (L-x_1) < \frac{L}{2}   \right\}. 
\end{equation*}
We start with the estimation of $\|
\mathcal{R}_{N_0,BL}^\delta(\mx)\|_{L^2(\Omega_{\mod,0}^\delta)}$. 
In view of the definition~\eqref{definitionPiq},
\begin{equation}\label{LemmaRModBL1}
\partial_{x_1} \partial_{X_1} \Pi_{n,q}^\delta(x_1, \frac{\mx}{\delta}) = \sum_{p=0}^q \partial_{x_1}^{p+1} \langle
u_{n,q-p}^\delta(x_1,0)
\rangle_{\Gamma} \partial_{X_1}W_{p}^{\mathfrak{t}}(\frac{\mx}{\delta}) +\sum_{p=1}^{q} 
\partial_{x_1}^{p} \langle \partial_{x_2}
u_{n,q-p}^\delta(x_1,0) \rangle_{\Gamma} \partial_{X_1} W_{p}^{\mathfrak{n}}(\frac{\mx}{\delta}).
\end{equation}
so that
\begin{multline}\label{LemmaRModBL2}
\| \partial_{x_1} \partial_{X_1}
\Pi_{n,q}^\delta \|_{{L^2(\Omega_{\mod,0}^\delta)} }^2 \leq C \big( \sum_{p=0}^q  \|\partial_{x_1}^{p+1} \langle
u_{n,q-p}^\delta(x_1,0)
\rangle_{\Gamma}  \|_{L^\infty(\Gamma_{\mod,0})}
\|\partial_{X_1}W_{p}^{\mathfrak{t}}(\frac{\mx}{\delta})\|_{L^2(\Omega_{\mod,0}^\delta)}
\\
+ \sum_{p=1}^{q} \| \partial_{x_1}^{p} \langle \partial_{x_2}
u_{n,q-p}^\delta(x_1,0) \rangle_{\Gamma}\|_{L^\infty(\Gamma_{\mod,0})} \| \partial_{X_1} W_{p}^{\mathfrak{n}}(\frac{\mx}{\delta})\|_{L^2(\Omega_{\mod,0}^\delta)}\big).
\end{multline}
where $\Gamma_{\mod,0} = \{ (x_1,x_2) \in \Gamma, |x_1| <\frac{L}{2} \}$.
In $\Omega_{\mod,0}^\delta$,  the function $u_{n,q}^\delta$ and its traces over
$ \partial \Omega_{\mod,0}^\delta \cap \Gamma$ are
smooth (they have a polynomial dependance with respect to $\ln \delta$). Consequently, for any $k\in \N$, for any $\varepsilon>0$,
there exists a constant $C>0$ such that  
\begin{equation}\label{LemmaRModBL3}
\|\partial_{x_1}^{p+1} \langle
u_{n,q-p}^\delta(x_1,0)
\rangle_{\Gamma}  \|_{L^\infty(\Gamma_{\mod,0})}
\leq C \delta^{-\varepsilon} \quad \mbox{and} \quad 
 \| \partial_{x_1}^{p} \langle \partial_{x_2}
u_{n,q-p}^\delta(x_1,0) \rangle_{\Gamma}\|_{L^\infty(\Gamma_{\mod,0})}  \leq C
\delta^{-\varepsilon}.
\end{equation}
Then, using the periodicity of the boundary layer terms, we obtain
\begin{equation}\label{LemmaRModBL4}
\int_{\Omega_{\mod,0} }  | \partial_{X_1}
W_{p}^{\mathfrak{n}}(\frac{\mx}{\delta}) |^2 d\mx  \leq \sum_{\ell
  =0}^{\lceil \frac{2L-2}{\delta} \rceil} \delta^2 
\| \partial_{X_1}
W_{p}^{\mathfrak{n}} \|^2_{L^2(\mathcal{B})}  \leq C \delta,
\end{equation}
Similarly,
\begin{equation}\label{LemmaRModBL5}
\int_{\Omega_{\mod,0} }  |
W_{p}^{\mathfrak{t}}(\frac{\mx}{\delta}) |^2 d\mx  \leq \delta^2 \sum_{\ell
  =0}^{\lceil \frac{2L-2}{\delta} \rceil} \| 
W_{p}^{\mathfrak{n}} \|^2_{L^2(\mathcal{B})}   \leq C \delta .
\end{equation}
Collecting the results of
\eqref{LemmaRModBL1}-\eqref{LemmaRModBL2}-\eqref{LemmaRModBL3}-\eqref{LemmaRModBL4}-\eqref{LemmaRModBL5}, we obtain
\begin{equation*}
\| \partial_{x_1} \partial_{X_1}
\Pi_{n,q}^\delta \|_{{L^2(\Omega_{\mod,0})} } \leq C \delta^{1/2-\varepsilon}
\end{equation*}
A similar analysis leads to 
\begin{equation*}
\| \partial_{x_1}^2 
\Pi_{n,q}^\delta \|_{{L^2(\Omega_{\mod,0})} } \leq C \delta^{1/2-\varepsilon}.
\end{equation*}
Then, summing over $n$ and $q$ in~\eqref{definitionRBLN0} gives
\begin{equation}\label{LemmaRModEst1}
\|
\mathcal{R}_{N_0,BL}^\delta(\mx)\|_{L^2(\Omega_{\mod,0})} \leq C
\delta^{N_0 -1/2 - \varepsilon}.
\end{equation}
We can now evaluate $\|
\mathcal{R}_{N_0,BL}^\delta(\mx)\|_{L^2(\Omega_{\mod,+}^\delta)}$ (the
evaluation of $\|
\mathcal{R}_{N_0,BL}^\delta(\mx)\|_{L^2(\Omega_{\mod,-}^\delta)}$ being
similar). Here again, the inequality \eqref{LemmaRModBL1} (replacing $\Gamma_{\mod,0}$ with $\Gamma_{\mod,+}  =\left\{ (x,_1,x_2) \in
  \Gamma, \eta(\delta) < x_1 < \frac{L}{2} \right\}$ and $\Omega_{\mod,0}^\delta$ with $\Omega_{\mod,+}^\delta$) is valid, but the
norm of the tangential traces of the macroscopic fields are not
uniformly bounded anymore. However, for any positive integer $\alpha>0$, $\partial_{x_1}^{\alpha} \partial_{x_2}
u_{n,q}^\delta$ and  $\partial_{x_1}^{\alpha+1} 
u_{n,q}^\delta$ belonging to both $V_{\beta_{n,q}+\alpha+1}^2(\OmegaTop)$ and
$V_{\beta_{n,q}+\alpha+1}^2(\OmegaBottom)$ ($\beta_{n,q} = 1
+\frac{2n}{3} + q + \varepsilon$). Consequently
$\partial_{x_1}^{\alpha} \langle \partial_{x_2}
u_{n,q}^\delta \rangle_{\Gamma}$ and  $\partial_{x_1}^{\alpha+1} \langle 
u_{n,q}^\delta \rangle_{\Gamma}$ belong to
$V_{\beta_{n,q}+\alpha+1}^{3/2}(\Gamma)$, which means in particular that
(see Lemma 6.1.2 in \cite{MR1469972}) that
$\| |x_1-L|^{\beta_{n,q}+\alpha+1/2}\partial_{x_1}^{\alpha} \langle \partial_{x_2}
u_{n,q}^\delta \rangle_{\Gamma} \|_{L^\infty(\Gamma_{\mod,+})}$ and $\| |x_1-L|^{\beta_{n,q}+\alpha+1/2}\partial_{x_1}^{\alpha} \langle \partial_{x_2}
u_{n,q}^\delta \rangle_{\Gamma} \|_{L^\infty(\Gamma_{\mod,+}) }$ are uniformly bounded.
As a consequence,
\begin{equation*}
\| \partial_{x_1}^{p} \langle \partial_{x_2}
u_{n,q-p}^\delta \rangle_{\Gamma} \|_{L^\infty(\Gamma_{\mod,+})} \leq
\eta(\delta)^{-(\frac{3}{2} + \varepsilon + \frac{2}{3} n + q)},\;
\mbox{and}\;
\| \partial_{x_1}^{p+1} \langle 
u_{n,q-p}^\delta \rangle_{\Gamma} \|_{L^\infty(\Gamma_{\mod,+})}   \leq
\eta(\delta)^{-(\frac{3}{2} + \varepsilon + \frac{2}{3} n + q)}.
\end{equation*}
Finally,
\begin{equation*}
\| \partial_{x_1} \partial_{X_1}
\Pi_{n,q}^\delta \|_{{L^2(\Omega_{\mod,+}^\delta)} } \leq C \delta^{1/2}
\eta(\delta)^{-(\frac{3}{2} + \varepsilon + \frac{2n}{3} +q)}.
\end{equation*}
Similarly,
\begin{equation*}
\| \partial_{x_1}^2 
\Pi_{n,q}^\delta \|_{{L^2(\Omega_{\mod,+}^\delta)} } \leq C \delta^{1/2}
\eta(\delta)^{-(\frac{5}{2} + \varepsilon + \frac{2n}{3} +q)}.
\end{equation*}
Then, summing over $n$ and $q$ in~\eqref{definitionRBLN0} gives
\begin{equation}\label{LemmaRModEst2}
\|
\mathcal{R}_{N_0,BL}^\delta(\mx)\|_{L^2(\Omega_{\mod,+}^\delta)} \leq C
 \left(
  \frac{\delta}{\eta(\delta)}\right)^{N_0+3/2} \,\delta^{-\varepsilon -2}.
\end{equation}
Collecting \eqref{LemmaRModEst1} and \eqref{LemmaRModEst2} finishes the proof.
\end{proof}
\noindent Finally, the estimate of the modeling error of Lemma~\ref{LemmeErreurMod2} results from~\eqref{DecompositionEmod2},
Lemma~\ref{LemmeErreurModelFF} and Lemma~\ref{LemmeErreurModelBL}. 
\subsection{Evaluation of the matching error}
\label{AppendixProofMacthingError}
\subsubsection{Proof of Lemma~\ref{LemmeMatchingResiduMacro}}
Since  $\mathcal{R}_{\macro, n,q,\delta} \in
V_{2,\beta}^2(\OmegaTop)\cap V_{2,\beta}^2(\OmegaBottom)$ for any
$\beta > 1 -\frac{2 (k(n,q,N_0)+1)}{3}$, using the fact that
$k(n,q,N_0) \geq \frac{3}{2} (N_0 -q) -n-\frac{1}{2}$  and the fact
that $r^+ \leq 2 \eta(\delta)$ on $\Omega_{\match}^+$, one can verify that, for any
$\varepsilon>0$, there is a positive constant $C>0$ such that
\begin{equation}
\|\mathcal{R}_{\macro, n,q,\delta}\|_{L^2(\Omega_{\match}^+)}
  \leq \eta(\delta)^{N_0 - \frac{2}{3} n - q -\varepsilon+\frac{4}{3}},
  \quad \|\nabla \mathcal{R}_{\macro, n,q,\delta}\|_{L^2(\Omega_{\match}^+)}
  \leq \eta(\delta)^{N_0  - \frac{2}{3} n - q -\varepsilon+\frac{1}{3}}. 
\end{equation}
As a consequence, summing over $(n,q)
\in D_{N_0}$ remarking that  $\delta < \eta(\delta)$ and $\nabla
\chi_{\macro,+}(\mx^+/\delta) \leq \delta^{-1}$ completes the proof.

\subsubsection{Proof of Lemma~\ref{LemmeMatchingResiduBL}}

The function $\mathcal{R}_{BL, N_0}^{\delta} $ is supported in the
domain $-L+\delta<x_1 <L- \delta$. We shall evaluate it in the domain
$$
\Omega_{\match,BL}^+ = \Omega_{\match}^+ \cap \{ (x_1,x_2) \in \R^2, x_1-L
< -\delta\}.
$$
that we separate into two parts $\Omega_{\match,BL,1}^+$ and
$\Omega_{\match,BL,2}^+$ defined as follows:
\begin{eqnarray*}
&& \Omega_{\match,BL,1}^+ = \Omega_{\match}^+ \cap \{ (x_1,x_2) \in
\R^2, -\frac{\eta(\delta)}{2} <x_1-L
< -\delta\} \\
&& \Omega_{\match,BL,2}^+ = \Omega_{\match}^+ \cap \{ (x_1,x_2) \in
\R^2,-2 \eta(\delta) <x_1-L
< -\frac{\eta(\delta)}{2}\} \\
\end{eqnarray*}
We shall now estimate $\mathcal{R}_{BL,n,q,\delta}$ (or, more
precisely we study each term of the form $\langle w_{n,q,j}(x_1,0) \rangle
W_{n,q,j}$) in $\Omega_{\match,BL,1}^+$ and $\Omega_{\match,BL,2}^+$. In
the domain $\Omega_{\match,BL,1}^+$, the profile functions $W_{n,q,j}$
are exponentially decaying (since $|x_2|> \frac{\sqrt{3}}{2}
\eta(\delta)$ on this domain). As a consequence, for any $N \in \N$ there exists a
constant $C>0$ such that 
\begin{equation}\label{ResiduBLOmega1}
\| \mathcal{R}_{BL, N_0}^{\delta} \|_{L^2(\Omega_{\match,BL,1}^+)} + \|
\nabla \mathcal{R}_{BL, N_0}^{\delta} \|_{L^2(\Omega_{\match,BL,1}^+)}
\leq C \eta(\delta)^N.
\end{equation}
Next, we remark that for any $\beta > 1 - \frac{2 (k(n,q,N_0) + 1)}{3}$, $\|
|x_1-L|^{\beta-\frac{1}{2} } \langle \bds{w_{n,q,j}}\rangle \|_{L^\infty(\Gamma_{\match}^+)} $ is bounded ($\langle w_{n,q,j}(x_1,0)
\rangle  \in V_{2,\beta+1}^{5/2}(\Gamma)$). As a consequence,for any $\varepsilon>0$,
$$
\| \langle w_{n,q,j}(x_1,0) \rangle \|_{L^\infty(\Gamma_{\match}^+)}  \leq C \eta(\delta)^{N_0 -\frac{1}{6} - q-
  \frac{2}{3} n - \varepsilon}.
$$
Then, using the previous estimation and  the periodicity of the profile function $W_{n,q,j} $, we have
\begin{multline}
\| \langle w_{n,q,j}(x_1,0) \rangle
W_{n,q,j}  \|_{L^2(\Omega_{\match,BL,2}^+)}\leq C \|
\langle w_{n,q,j}(x_1,0) \rangle \|_{L^\infty(\Gamma_{\match}^+)} \| W_{n,q,j} \|_{L^2(\Omega_{\match,BL,2}^+)} \\
\leq  C \delta^{\frac{1}{2}} \eta(\delta)^{N_0 + \frac{1}{3} -  q-
  \frac{2}{3} n - \varepsilon}.
\end{multline}
Analogously,  for any $\beta > 1 - \frac{2 k(n,q,N_0) + 1}{3}$, $\|
|x_1-L|^{\beta+\frac{1}{2} } \langle w_{n,q,j}(x_1,0) \rangle \|_{L^\infty(\Gamma_{\match}^+)} $ is bounded. As a consequence, for any $\varepsilon>0$,
 \begin{equation}
\| \nabla \left(  \langle w_{n,q,j}(x_1,0) \rangle 
W_{n,q,j} \right)  \|_{L^2(\Omega_{\match,BL,2}^+)} \leq \delta^{-\frac{1}{2}}\eta(\delta)^{N_0 + \frac{1}{3} -  q-
  \frac{2}{3} n - \varepsilon}.
\end{equation}
Summing up over $j \in (0, K)$ and $(n,q) \in D_{N_0}$, noting that
$\chi_-(\frac{x_1-L}{\delta})=1$ on $\Omega_{\match,BL,2}^+$, we obtain,
\begin{equation}\label{ResiduBLOmega2}
\| \mathcal{R}_{BL, N_0}^{\delta} \|_{L^2(\Omega_{\match,BL,2}^+)}
\leq \delta^{1/2} \eta(\delta)^{N_0 + \frac{1}{3} - \varepsilon} \quad
\quad \| \nabla \mathcal{R}_{BL, N_0}^{\delta} \|_{L^2(\Omega_{\match,BL,2}^+)}
\leq \delta^{-1/2} \eta(\delta)^{N_0 + \frac{1}{3}  - \varepsilon}.
\end{equation} 
\noindent Collecting \eqref{ResiduBLOmega1} and \eqref{ResiduBLOmega2}, and
carrying out an entirely similar analysis for $\| \mathcal{R}_{BL, N_0}^{\delta} \|_{L^2(\partial \Omega_{\match}^+
  \cap \Gamma^\delta)}$  leads to ~\eqref{ResiduBL}.

\subsubsection{Proof of Lemma~\ref{LemmeMatchingResiduNF}}
As previously, we first investigate separately $\mathcal{R}_{NF,
  n,q}$. We consider the scaled matching domains 
  $$\widehat{\Omega}_{\match}^\delta =  \{ \mX^+ \in
  \widehat{\Omega}^+,\frac{\eta(\delta)}{\delta} <R^+ <  2
  \frac{\eta(\delta)}{\delta} \}, \quad \hat{\Gamma}_\match^\delta = \{ \mX^+ \in
  \partial \widehat{\Omega}_\hole,\frac{\eta(\delta)}{\delta} <R^+ <  2
  \frac{\eta(\delta)}{\delta} \}.$$
Then, making the change of scale $\mX^+ = \mx^+/\delta$, we have
\begin{align*}
& \| \mathcal{R}_{NF,
  n,q} \|_{L^2(\Omega_{\match}^+)}^2  \leq   \delta^2 \int_{\hat{\Omega}_{\match}^\delta }
|  \mathcal{R}_{NF,
  n,q}|^2 d\mX^+  \\
 & \qquad \leq  \delta^2 \| (1+R^+)^{-(\beta -\gamma -1)} \rho^{-(\gamma-1)}
\|_{L^\infty(\widehat{\Omega}_{\match}^\delta )}^2 \| (1+R^+)^{\beta
  -\gamma-1} \rho^{\gamma-1}  \mathcal{R}_{NF,
  n,q} \|_{L^2(\widehat{\Omega}^+)}^2 \\
 & \qquad \leq   \delta^2 \| (1+R^+)^{-(\beta -\gamma -1)} \rho^{-(\gamma-1)}
\|_{L^\infty(\widehat{\Omega}_{\match}^\delta )}^2 \|  \mathcal{R}_{NF,
 n,q} \|_{\mathfrak{V}_{\beta,\gamma}^2(\hat{\Omega}^+)}^2.
\end{align*}
 If $\gamma-1<0$,
$
 \rho^{-(\gamma-1)}  \leq (1 + R^+)^{-(\gamma -1)}. 
$ 
As a result, for any  $\beta < 1 + \frac{2\lfloor
  \alpha_{n,q}\rfloor +1 }{3}$, $\gamma \in (\frac{1}{2},1)$, $\gamma -1$  sufficiently small, 
\begin{equation*}
\| \mathcal{R}_{NF,
  n,q} \|_{L^2(\Omega_{\match}^+)} \leq C \delta \left( \frac{\delta}{\eta(\delta)} \right)^{\beta-2}.
\end{equation*}
Taking $\beta = \frac{5}{3} +  (N_0 -q) -
\frac{2}{3} n  - 1 -\varepsilon$,  $\varepsilon>0$ ($\frac{2}{3} \lfloor \alpha_{n,q} \rfloor \geq (N_0 -q) -
\frac{2}{3} n -1$), we have
\begin{equation*}
\| \mathcal{R}_{NF,
  n,q} \|_{L^2(\Omega_{\match}^+)} \leq C \delta \left(
  \frac{\eta(\delta)}{\delta} \right)^{N_0 -q -\frac{2}{3} n - 2-\varepsilon}.
\end{equation*}
Summing over $(n,q) \in \mathcal{D}_{N_0}$, we obtain
\begin{equation*}
\| \mathcal{R}_{NF, N_0}^{\delta}\|_{L^2(\Omega_{\match}^+)} \leq C 
\sum_{(n,q) \in D_{N_0}} \delta^{\frac{2}{3} n+q } \delta \left( \frac{\delta}{\eta(\delta)} \right)^{N_0 -q -\frac{2}{3} n -2-
\varepsilon} \leq C \delta\left( \frac{\delta}{\eta(\delta)} \right)^{N_0 -2-
\varepsilon}.
\end{equation*}
The gradient of $\mathcal{R}_{NF, N_0}^{\delta}$  can be estimated in
the same way noticing that
\begin{align*}
& \| \nabla \mathcal{R}_{NF,
  n,q} \|_{L^2(\Omega_{\match}^+)}^2  \leq \int_{\widehat{\Omega}_{\match}^\delta }
|  \nabla_{\mX^+}\mathcal{R}_{NF,
  n,q}|^2 d\mX^+  \\
 & \qquad \leq  \| (1+R^+)^{-(\beta -\gamma )} \rho^{-(\gamma-1)}
\|_{L^\infty(\hat{\Omega}_{\match}^\delta )}^2 \| (1+R^+)^{\beta
  -\gamma} \rho^{\gamma-1}  \nabla_{\mX^+} \mathcal{R}_{NF,
  n,q} \|_{L^2(\widehat{\Omega}^+)}^2 \\
 & \qquad \leq   \| (1+R^+)^{-(\beta -\gamma )} \rho^{-(\gamma-1)}
\|_{L^\infty(\widehat{\Omega}_{\match}^\delta )}^2 \|  \mathcal{R}_{NF,
 n,q} \|_{\mathfrak{V}_{\beta,\gamma}^2(\widehat{\Omega}^+)}^2.
\end{align*}
As for the trace estimate, we note that $\| (R^+)^\gamma \mathcal{R}_{NF,
  n,q} \|_{ H^2(\widehat{\Omega}_{\match}^\delta)}$ is uniformly bounded  for any $
\gamma \leq \frac{2\lfloor
  \alpha_{n,q}\rfloor +1 }{3} -1$. Consequently, since $\hat{\Gamma}_\match^\delta$ is smooth, under the same
condition for $\gamma$, $\|\mathcal{R}_{NF,
  n,q} (R^+)^\gamma  \|_{L^\infty(\hat{\Gamma}_\match^\delta)}$ is uniformly bounded. Then
\begin{equation*}
\|\mathcal{R}_{NF,
  n,q}   \|_{L^2(\partial {\Omega_\match^+} \cap \Gamma^\delta)} \leq C \delta^{\frac{1}{2}} \|\mathcal{R}_{NF,
  n,q}   \|_{L^2(\hat{\Gamma}_\match^\delta)} \leq C
\delta^{\frac{1}{2}} \left(
  \frac{\delta}{\eta(\delta)}\right)^{\theta-\frac{1}{2}} \leq C
\delta^{\frac{1}{2}}    \left(
  \frac{\delta}{\eta(\delta)}\right)^{N_0 -q - \frac{2}{3} n - \frac{13}{6}-\varepsilon}.  
\end{equation*}   

\end{appendix}
\bibliographystyle{plain}
\bibliography{ResearchReport}
\end{document}